\providecommand{\main}{.} 
\documentclass[twoside,11pt]{article}

\usepackage{blindtext}

\usepackage[abbrvbib]{jmlr2e}   

\usepackage{lastpage}

\usepackage[margin=1in]{geometry}

\usepackage{amssymb}
\usepackage{amsmath}
\usepackage{amsfonts}
\usepackage{tensor}
\usepackage{mathtools}
\usepackage{cancel}
\usepackage{bm}
\usepackage{url}
\usepackage{textgreek}

\usepackage{booktabs}
\usepackage{siunitx}

\usepackage{float}
\usepackage{placeins} 
\usepackage{algorithm}
\usepackage[noend]{algpseudocode}

\usepackage{footnote}
\makesavenoteenv{algorithm}

\algrenewcommand\algorithmicrequire{\textbf{Input:}}
\algrenewcommand\algorithmicensure{\textbf{Output:}}

\newcommand*{\secref}[1]{\hyperref[{#1}]{Section~\ref*{#1}}}
\newcommand*{\secrefwithtitle}[1]{\hyperref[{#1}]{Section~\ref*{#1}~(\nameref*{#1})}}

\newcommand*{\lemref}[1]{\hyperref[{#1}]{Lemma~\ref*{#1}}}
\newcommand*{\propref}[1]{\hyperref[{#1}]{Proposition~\ref*{#1}}}
\newcommand*{\thmref}[1]{\hyperref[{#1}]{Theorem~\ref*{#1}}}
\newcommand*{\defref}[1]{\hyperref[{#1}]{Definition~\ref*{#1}}}
\newcommand*{\corref}[1]{\hyperref[{#1}]{Corollary~\ref*{#1}}}
\newcommand*{\figref}[1]{\hyperref[{#1}]{Figure~\ref*{#1}}}
\newcommand*{\tableref}[1]{\hyperref[{#1}]{Table~\ref*{#1}}}
\newcommand*{\algoref}[1]{\hyperref[{#1}]{Algorithm~\ref*{#1}}} 
\newcommand*{\appref}[1]{\hyperref[{#1}]{Appendix~\ref*{#1}}} 
\newcommand*{\apprefwithtitle}[1]{\hyperref[{#1}]{Appendix~\ref*{#1}~(\nameref*{#1})}}

\newtheorem{prop}[theorem]{Proposition}

\usepackage{framed}

\usepackage{xcolor}
\usepackage{listings}
\usepackage{soul}
\usepackage{tikz}
\usetikzlibrary{backgrounds}
\usetikzlibrary{math}
\usetikzlibrary{quotes}
\usepackage{multirow}

\usetikzlibrary{
    positioning, 
    calc, 
    patterns, 
    arrows.meta, 
    shapes.arrows, 
    shapes.misc}
\usepackage{tikz-3dplot}

\usepackage{pgfplots}
\usepgfplotslibrary{colormaps}
\usepgfplotslibrary{groupplots}
\pgfplotsset{compat=1.18}

\definecolor{tabblue}{HTML}{1F77B4}
\definecolor{taborange}{HTML}{FF7F0E}
\definecolor{tabgreen}{HTML}{2CA02C}
\definecolor{tabred}{HTML}{D62728}
\definecolor{tabpurple}{HTML}{9467BD}
\definecolor{tabbrown}{HTML}{8C564B}
\definecolor{tabpink}{HTML}{E377C2}
\definecolor{tabgray}{HTML}{7F7F7F}
\definecolor{tabolive}{HTML}{BCBD22}
\definecolor{tabcyan}{HTML}{17BECF}

\colorlet{tab1}{tabblue}
\colorlet{tab2}{taborange}
\colorlet{tab3}{tabgreen}
\colorlet{tab4}{tabred}
\colorlet{tab5}{tabpurple}
\colorlet{tab6}{tabbrown}
\colorlet{tab7}{tabpink}
\colorlet{tab8}{tabgray}
\colorlet{tab9}{tabolive}
\colorlet{tab10}{tabcyan}

\pgfplotscreateplotcyclelist{matplotlib}{
    {tabblue},
    {taborange},
    {tabgreen},
    {tabred},
    {tabpurple},
    {tabbrown},
    {tabpink},
    {tabgray},
    {tabolive},
    {tabcyan},
}

\pgfplotsset{
    colormap={grey_r}{
            gray(0cm)=(1); 
            gray(1cm)=(0); 
    },
    colormap={tabblues}{
        HTML(20cm)=(0F3A57);   
        HTML(41.4cm)=(1F77B4); 
        HTML(80cm)=(A8D3F0);   
    },
    colormap={taboranges}{
        HTML(20cm)=(663000);  
        HTML(52.7cm)=(FF7F0E); 
        HTML(80cm)=(FFC999);   
    },
    colormap={tabgreens}{
        HTML(20cm)=(165016);   
        HTML(40cm)=(2CA02C);   
        HTML(80cm)=(AFE9AF);   
    },
    colormap={tabreds}{
        rgb=(0.941, 0.659, 0.659) 
        rgb=(0.839, 0.153, 0.157) 
        rgb=(0.337, 0.063, 0.063) 
    },
    colormap={YlOrBr}{
        HTML(0cm)=(FFFFE5);  
        HTML(2cm)=(FFF7BC);  
        HTML(4cm)=(FEE391);  
        HTML(6cm)=(FEC44F);  
        HTML(8cm)=(D95F0E);  
        HTML(10cm)=(662506); 
    },
    colormap={magma}{
        HTML(0cm)=(000004);  
        HTML(2cm)=(3B0F70);  
        HTML(4cm)=(8C2981);  
        HTML(6cm)=(DE4968);  
        HTML(8cm)=(FE9F6D);  
        HTML(10cm)=(FCFDB1); 
    },
    colormap={plasma}{
        HTML(0cm)=(0D0887);  
        HTML(2cm)=(6A00A8);  
        HTML(4cm)=(B12A90);  
        HTML(6cm)=(E16462);  
        HTML(8cm)=(FCA636);  
        HTML(10cm)=(F0F921); 
    },
    colormap={flare}{
        rgb=(0.918, 0.592, 0.467) 
        rgb=(0.855, 0.435, 0.431) 
        rgb=(0.741, 0.286, 0.455) 
        rgb=(0.588, 0.196, 0.463) 
        rgb=(0.424, 0.133, 0.416) 
        rgb=(0.271, 0.098, 0.337) 
    },
    colormap={rainbow}{
        rgb=(0.47, 0.00, 1.00) 
        rgb=(0.20, 0.35, 1.00) 
        rgb=(0.00, 0.80, 1.00) 
        rgb=(0.25, 0.98, 0.75) 
        rgb=(0.60, 1.00, 0.50) 
        rgb=(0.90, 0.85, 0.30) 
        rgb=(1.00, 0.55, 0.15) 
        rgb=(1.00, 0.00, 0.00) 
    }
}

\usepackage{caption}
\usepackage{subcaption}

\DeclareFontFamily{U}{mathx}{}
\DeclareFontShape{U}{mathx}{m}{n}{<-> mathx10}{}
\DeclareSymbolFont{mathx}{U}{mathx}{m}{n}
\DeclareMathAccent{\widehat}{0}{mathx}{"70}
\DeclareMathAccent{\widecheck}{0}{mathx}{"71}

\newcommand{\vb}[1]{{\bm{#1}}}

\newcommand{\nhphantom}[1]{\sbox0{#1}\hspace{-\the\wd0}}

\DeclareMathOperator*{\argmin}{arg\,min}

\newcommand{\param}{\theta}
\newcommand{\paramzero}{{\theta_0}}

\newcommand{\state}{u}
\newcommand{\adj}{v}
\newcommand{\res}{R}
\newcommand{\qoi}{Q}

\newcommand{\sqrtcov}{C}
\newcommand{\inputbasis}{U}
\newcommand{\outputbasis}{V}
\newcommand{\pto}{q}
\newcommand{\ptoW}{f}
\newcommand{\paramspace}{X}
\newcommand{\obsspace}{Y}

\newcommand{\parammultiset}{\widehat{\Theta}}

\newcommand{\ppsleft}{\vb{\sigma}}
\newcommand{\ppsright}{\vb{\tau}}
\newcommand{\pushleft}{\vb{\mu}}
\newcommand{\pushright}{\vb{\nu}}
\newcommand{\incrleft}{\widetilde{\vb{\sigma}}}
\newcommand{\incrright}{\widetilde{\vb{\tau}}}

\newcommand{\kron}{\otimes_K}

\newcommand{\leftcore}{\vb{P}}
\newcommand{\rightcore}{\vb{Q}}
\newcommand{\leftunfolding}[1]{{{#1}^L}}
\newcommand{\rightunfolding}[1]{{{#1}^R}}

\newcommand{\sampleind}{i}
\newcommand{\nummodes}{d}
\newcommand{\taylororder}{k}
\newcommand{\derivativeorderactionsection}{j}
\newcommand{\derivativeorderrmgnsection}{j}
\newcommand{\jac}{\mathcal{J}}
\newcommand{\jaczero}{J}
\newcommand{\numactions}{n_s}
\newcommand{\tuple}[1]{{\mathbb{#1}}}
\newcommand{\loss}{\Phi}

\newcommand{\adjres}{\res^{\text{adj}}}

\newcommand{\innerproduct}[2]{\left\langle #1, #2 \right\rangle}

\newcommand{\norm}[1]{\left\| #1 \right\|}
\newcommand{\abs}[1]{\left\lvert #1 \right\rvert}

\algrenewcommand\algorithmicrequire{\textbf{Input:}}
\algrenewcommand\algorithmicensure{\textbf{Output:}}

\def\onedot{$\mathsurround0pt\ldotp$}
\def\cdddot#1{
  \mathbin{\vcenter{\baselineskip.67ex
    \hbox{\onedot}\hbox{\onedot}\hbox{\onedot}%
  }}%
}

\DeclarePairedDelimiter{\floor}{\lfloor}{\rfloor}

\makeatletter
\newcommand*\@MyHelperProd[1]{%
    \mathop{\hbox{#1$\m@th\displaystyle\prod$}}\limits
}
\newcommand*\smallprod{\@MyHelperProd\small}
\newcommand*\footnoteprod{\@MyHelperProd\footnotesize}
\makeatother

\usepackage{subfiles}

\ShortHeadings{Tucker Tensor Train Taylor Series}{Alger, Christierson, Chen and Ghattas}
\firstpageno{1}

\begin{document}

\title{Tucker Tensor Train Taylor Series}

\author{\name Nick Alger \email nalger@oden.utexas.edu \\
       \addr Oden Institute for Computational Engineering and Sciences\\
       The University of Texas at Austin\\
       Austin, TX 78712, USA
       \AND
       \name Blake Christierson \email bechristierson@utexas.edu \\
       \addr Oden Institute for Computational Engineering and Sciences\\
       The University of Texas at Austin\\
       Austin, TX 78712, USA
       \AND
       \name Peng Chen \email pchen402@gatech.edu \\
       \addr School of Computational Science and Engineering\\
       Georgia Institute of Technology\\
       Atlanta, GA 30308, USA
       \AND
       \name Omar Ghattas \email omar@oden.utexas.edu \\
       \addr Oden Institute for Computational Engineering and Sciences and\\
       Walker Department of Mechanical Engineering\\
       The University of Texas at Austin\\
       Austin, TX 78712, USA}

\editor{}

\maketitle

\begin{abstract}


Learning derivative-accurate surrogates for implicit simulators is a key challenge in scientific machine learning.
High-order Taylor surrogates have long been considered intractable in high dimensions, because the derivative tensors are enormous and accessible only through probes. We make such surrogates tractable with the Tucker tensor train Taylor series (T4S), a local surrogate that represents each derivative tensor of a truncated Taylor expansion as a Tucker tensor train.
T4S targets a different learning problem than global operator learning: rather than training from input-output pairs at many parameter values, it is trained from random directionally symmetric derivative probes at a single expansion point. Computing $m$ probes of the $k$th derivative requires only $O(mk)$ linearized solves sharing one operator, cheaper than the $O(m)$ nonlinear solves for function evaluations or $O(m\,2^k)$ linearized solves for asymmetric probes.
We develop derivative-informed dimension reduction, Riemannian Gauss-Newton and Cauchy SGD fitting algorithms with rank continuation, requiring little hyperparameter tuning, and fast sweeping routines for the Riemannian Jacobian. We prove representational guarantees under spectral decay of the input covariance. Experiments show that our methods match quasi-optimal T3-SVD accuracy on random tensors from probes alone, up to data-limited ranks, and recover high-order Taylor structure in Poisson PDE examples.

\end{abstract}




\newpage



\section{Introduction}
\label{sec:intro}

Learning surrogate models for expensive simulators is a central challenge in scientific machine learning. In many scientific and engineering applications, a model evaluation requires solving an implicit system, such as a parameterized partial differential equation (PDE), rather than applying an explicit formula. These model evaluations are often embedded in outer-loop tasks, such as inverse problems, prediction and control under uncertainty, optimal experimental design, and digital-twin updating. Such tasks may require many model evaluations, and many of them also need gradients, Hessians, or higher derivatives. Even when derivatives are not used directly, derivative accuracy can determine the quality of local optimization steps, posterior approximations, proposal distributions, preconditioners, and uncertainty estimates.

This paper studies local, derivative-accurate surrogate learning for smooth parameter-to-output maps defined by such simulators. The input is a high- or infinite-dimensional parameter located near a fixed reference, and the output is a high- or infinite-dimensional quantity of interest depending on the parameter through the solution of a state equation. We work throughout with a covariance-whitened form of the map; precise definitions are given in \secref{sec:setting}. In the PDE setting, the parameter may be an initial condition, boundary condition, source, or material coefficient, and the output a field, boundary trace, flux, or other high-dimensional quantity. \figref{fig:overall_mapping_example} illustrates one such map.

\begin{figure}
\centering
\newcommand{\adddolfinplot}[5][]{
    \nextgroupplot[
        title={#3},
        #1
    ]
    \addplot graphics [xmin=0, xmax=1, ymin=0, ymax=1] {\main/figures/poisson_components/#2};
}

\begin{tikzpicture}

\begin{groupplot}[
    group style={
        group size=4 by 1,      
        horizontal sep=44pt     
    },
    scale only axis,
    enlargelimits=false,
    width=0.15\textwidth,    
    height=0.15\textwidth,
    title style={yshift=-6pt},      
    xtick=\empty, 
    yticklabel style={
        font=\footnotesize,
        /pgf/number format/fixed, 
        /pgf/number format/precision=0
    }
]

    \adddolfinplot[
        ytick=\empty
    ]{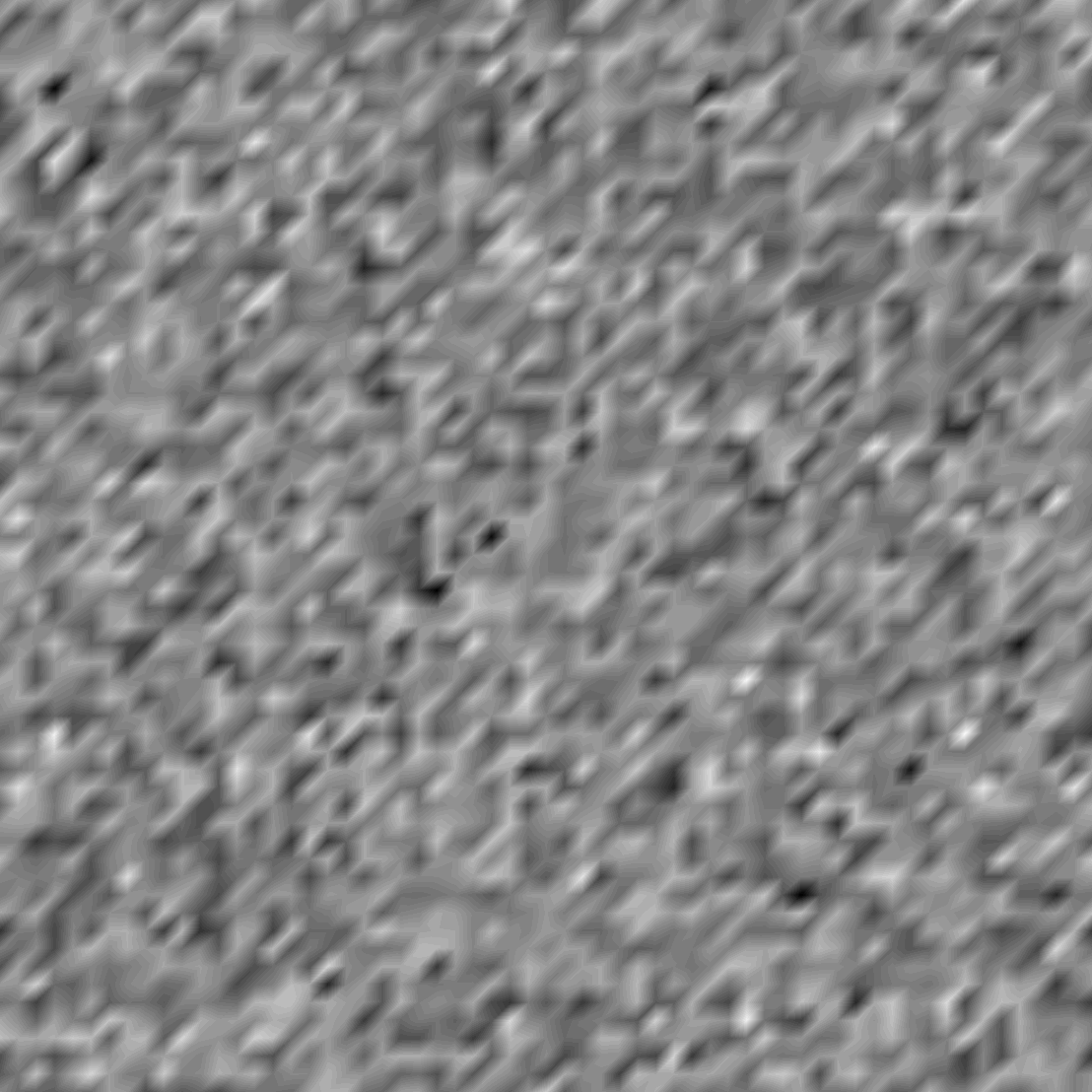}{$x$}{-3.4}{3.8}
    
    \adddolfinplot[ytick=\empty]{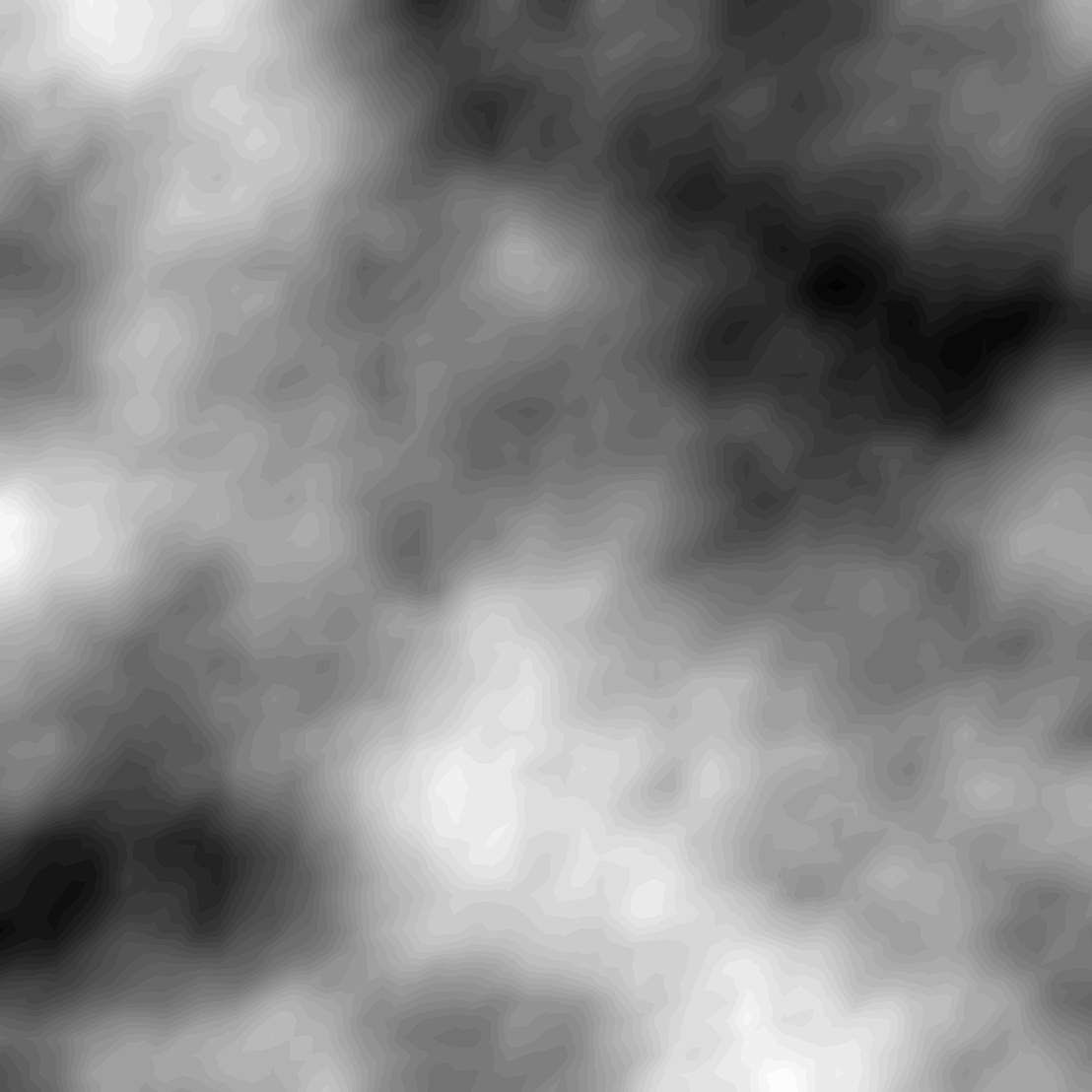}{$\theta$}{-1.76}{1.52}
    
    \adddolfinplot[
        ytick=\empty,
        clip=false,
        execute at end axis={
            \draw[black, thick] (-0.02, 0.98) rectangle (1.02, 1.02);
        }
    ]{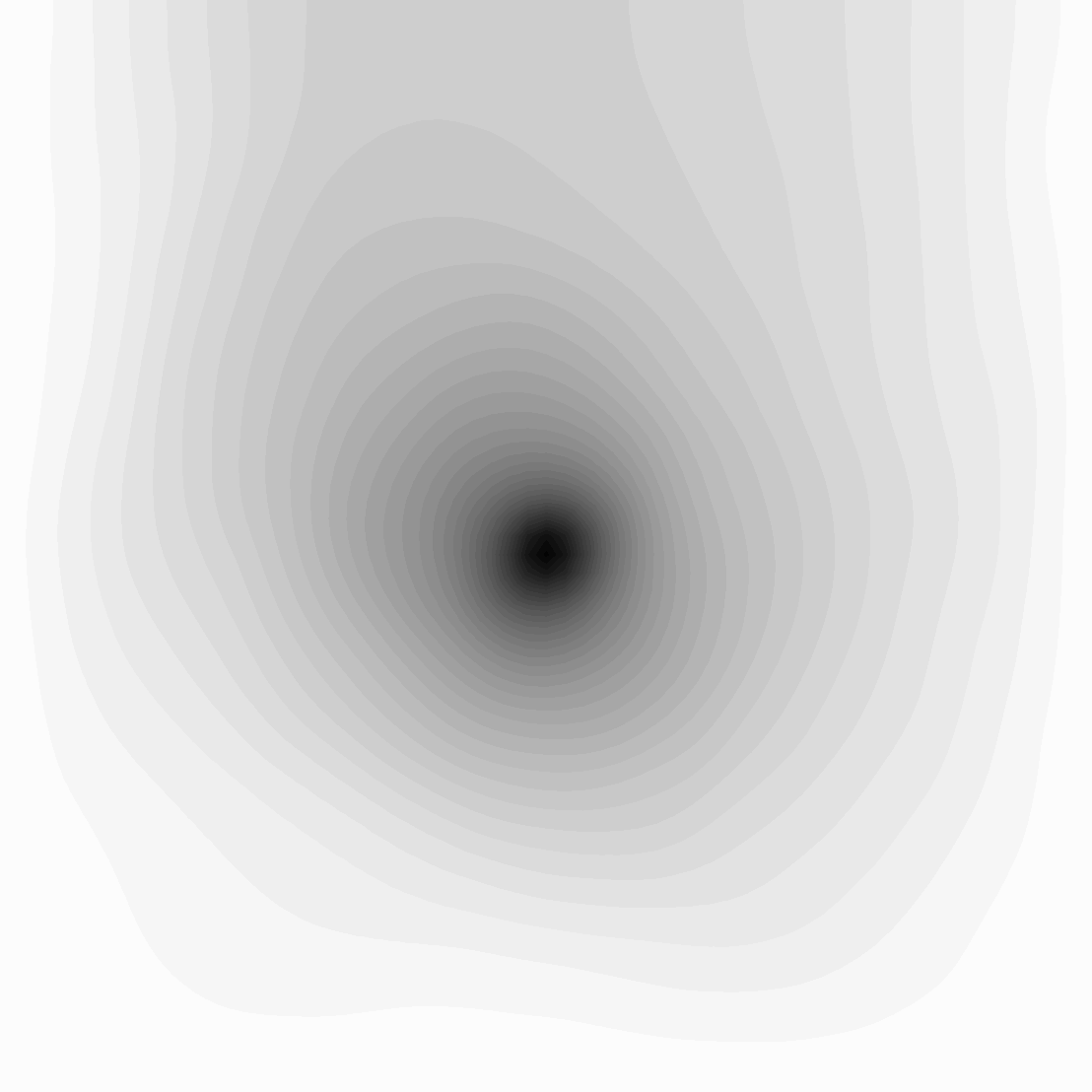}{$u$}{0}{0.78}

    \nextgroupplot[
        title={$q$},
        hide axis=false,
        colorbar=false,
        xmin=0,
        xmax=1,
        xtick={0, 1},
        minor xtick={0.25, 0.5, 0.75},
        xticklabel=\empty,
        ymin=0, ymax=0.21,
        ytick={0, 0.05, 0.10, 0.15, 0.20},
        yticklabel=\empty,
        xminorgrids=true,
        ymajorgrids=true,
        grid style={gray!20, solid},       
        minor grid style={gray!20, solid}
    ]
    \addplot [
        thick,
        black,
        mark=none
    ] table [x index=0, y index=1] {\main/figures/poisson_components/fig2a_1.txt};
    
\end{groupplot}

\node at ($(group c1r1.east)!0.5!(group c2r1.west)$) [yshift=-0.5em] {
    $\xrightarrow[\substack{\text{Evaluate} \\[0.25em] \theta_0 + Cx}]{}$
};
\node at ($(group c2r1.east)!0.5!(group c3r1.west)$) [yshift=-0.5em] {
    $\xrightarrow[\substack{\text{Solve} \\[0.25em] 0 = R(\theta, u)}]{}$
};
\node at ($(group c3r1.east)!0.5!(group c4r1.west)$) [yshift=-0.5em] {
    $\xrightarrow[\substack{\text{Evaluate} \\[0.25em] Q(\theta, u)}]{}$
};

\end{tikzpicture}
\caption{Illustration of the components of the covariance-whitened mapping $f:x \mapsto q$ for the example in \secref{sec:transmission}. The state equation is a Poisson PDE modeling steady state heat conduction from a source at the center of the domain $\Omega=[0,1]^2$. The parameter $\theta$ is a random spatially varying log conductivity coefficient, and the output $q$ is the trace of the temperature, $u$, along the top boundary of the domain. The noise function $x$ is the covariance-whitened version of $\theta$. We approximate $f$ using a surrogate model based on a Taylor series.
}
\label{fig:overall_mapping_example}
\end{figure}

The learning problem considered here differs from the standard operator-learning setting. Rather than training from input-output pairs at many parameter values, we train from randomized derivative tensor probes at a single expansion point. The learned object is a local Taylor surrogate whose derivative tensors are compressed in a low-rank format, fit to match the observed probes. Increasing the data then means probing more random directions and, when useful, higher derivative orders. This regime is natural for implicit simulators: a new function evaluation may require a nonlinear solve, whereas derivative probes require only linearized forward or adjoint solves that share a common coefficient matrix and differ only in right-hand sides.

Taylor series make natural local surrogates: their derivatives are built into the model, their approximation error can be analyzed, and their structure is explicit. For high-dimensional implicit maps, however, Taylor surrogates have traditionally been viewed as intractable. The derivative tensors grow exponentially with derivative order, so they cannot be stored or manipulated as dense arrays, and because $f$ depends on an implicitly defined state, individual tensor entries are not directly accessible. This makes conventional tensor approximation methods, built on sampling scattered entries, unsuitable. The tensors can be queried only through probes: contractions of the tensor with vectors in all but one index (see \secref{sec:intro_probes} and \figref{fig:general_tensor_actions}).

\begin{figure}
\centering
\begin{tikzpicture}[scale=0.75, every node/.style={scale=0.75}]
    \node[draw, rounded corners, minimum size=1.0cm, inner sep=0.0, line width=0.3mm] (Ta) at (0.0,0.0) {$A$};

    \node (u1a) at (-1.25,0.0) {};

    \node[draw, rounded corners, minimum size=0.6cm, inner sep=0.0, line width=0.3mm] (u2a) at (0.0,-1.0) {$x_2$};

    \node[draw, rounded corners, minimum size=0.6cm, inner sep=0.0, line width=0.3mm] (u3a) at (1.0,0.0) {$x_3$};

    \node[draw, rounded corners, minimum size=0.6cm, inner sep=0.0, line width=0.3mm] (u4a) at (0.0,1.0) {$x_4$};

    \draw[line width=0.3mm] (Ta) -- (u1a);
    \draw[line width=0.3mm] (Ta) -- (u2a);
    \draw[line width=0.3mm] (Ta) -- (u3a);
    \draw[line width=0.3mm] (Ta) -- (u4a);

    \node (y1) at (0.0,-2.0) {$\psi_1$};

    \node[draw, rounded corners, minimum size=1.0cm, inner sep=0.0, line width=0.3mm] (Tb) at (4+0.0,0.0) {$A$};

    \node[draw, rounded corners, minimum size=0.6cm, inner sep=0.0, line width=0.3mm] (u1b) at (4+-1.0,0.0) {$x_1$};

    \node (u2b) at (4+0.0,-1.25) {};

    \node[draw, rounded corners, minimum size=0.6cm, inner sep=0.0, line width=0.3mm] (u3b) at (4+1.0,0.0) {$x_3$};

    \node[draw, rounded corners, minimum size=0.6cm, inner sep=0.0, line width=0.3mm] (u4b) at (4+0.0,1.0) {$x_4$};

    \draw[line width=0.3mm] (Tb) -- (u1b);
    \draw[line width=0.3mm] (Tb) -- (u2b);
    \draw[line width=0.3mm] (Tb) -- (u3b);
    \draw[line width=0.3mm] (Tb) -- (u4b);

    \node (y2) at (4.0,-2.0) {$\psi_2$};

    \node[draw, rounded corners, minimum size=1.0cm, inner sep=0.0, line width=0.3mm] (Tc) at (8+0.0,0.0) {$A$};

    \node[draw, rounded corners, minimum size=0.6cm, inner sep=0.0, line width=0.3mm] (u1c) at (8+-1.0,0.0) {$x_1$};

    \node[draw, rounded corners, minimum size=0.6cm, inner sep=0.0, line width=0.3mm] (u2c) at (8+0.0,-1.0) {$x_2$};

    \node (u3c) at (8+1.25,0.0) {};

    \node[draw, rounded corners, minimum size=0.6cm, inner sep=0.0, line width=0.3mm] (u4c) at (8+0.0,1.0) {$x_4$};

    \draw[line width=0.3mm] (Tc) -- (u1c);
    \draw[line width=0.3mm] (Tc) -- (u2c);
    \draw[line width=0.3mm] (Tc) -- (u3c);
    \draw[line width=0.3mm] (Tc) -- (u4c);

    \node (y3) at (8.0,-2.0) {$\psi_3$};

    \node[draw, rounded corners, minimum size=1.0cm, inner sep=0.0, line width=0.3mm] (Td) at (12+0.0,0.0) {$A$};

    \node[draw, rounded corners, minimum size=0.6cm, inner sep=0.0, line width=0.3mm] (u1d) at (12+-1.0,0.0) {$x_1$};

    \node[draw, rounded corners, minimum size=0.6cm, inner sep=0.0, line width=0.3mm] (u2d) at (12+0.0,-1.0) {$x_2$};

    \node[draw, rounded corners, minimum size=0.6cm, inner sep=0.0, line width=0.3mm] (u3d) at (12+1.0,0.0) {$x_3$};

    \node (u4d) at (12+0.0,1.25) {};

    \draw[line width=0.3mm] (Td) -- (u1d);
    \draw[line width=0.3mm] (Td) -- (u2d);
    \draw[line width=0.3mm] (Td) -- (u3d);
    \draw[line width=0.3mm] (Td) -- (u4d);

    \node (y4) at (12.0,-2.0) {$\psi_4$};

\end{tikzpicture}
\caption{Probes $\psi_1, \psi_2, \psi_3, \psi_4$ of tensor $A$ by vectors $x_1, x_2, x_3, x_4$.}
\label{fig:general_tensor_actions}
\end{figure}

We address these obstacles with the Tucker tensor train Taylor series (T4S), defined in \secref{sec:t4s_model}. T4S represents each covariance-whitened derivative tensor in a truncated Taylor series as a Tucker tensor train, a structured low-rank tensor network defined in \secref{sec:t3} and illustrated in \figref{fig:tucker_tensor_train}. The Tucker factors decouple large physical discretization dimensions from smaller internal ranks, while the tensor train cores reduce the storage complexity from exponential to linear in derivative order. We fit the Tucker tensor trains directly from directionally symmetric derivative tensor probes---those in which all probing vectors are identical---which reduces the number of linearized state and adjoint solves per probe from exponential to linear in derivative order. An initial derivative-informed dimension reduction step compresses the input and output spaces to moderate dimensions before the reduced derivative tensors are fit using Riemannian manifold optimization with rank continuation.

\begin{figure}
\centering
\begin{tikzpicture}[scale=0.7, every node/.style={scale=0.6}]
    \node[draw, rounded corners, minimum size=1.75cm, inner sep=2.0, line width=0.3mm] (t) at (0-0.3,0) {\Large $D^4 \ptoW(0)$};
    \node (tout1) at (0.882*1.25-0.3,-1.214*1.25) {};
    \node (tout2) at (-0.882*1.25-0.3,-1.214*1.25) {};
    \node (tout3) at (1.427*1.25-0.3,0.464*1.25) {};
    \node (tout4) at (-1.427*1.25-0.3,0.464*1.25) {};
    \node (tout5) at (0*1.25-0.3,1.5*1.25) {};
			
    \draw[line width=0.3mm] (t) -- (tout1);
    \draw[line width=0.3mm] (t) -- (tout2);
    \draw[line width=0.3mm] (t) -- (tout3);
    \draw[line width=0.3mm] (t) -- (tout4);
    \draw[line width=0.3mm] (t) -- (tout5);
			
    \node (equals) at (2.25,0) {\Large $\approx$};

    \node[draw, rounded corners, minimum size=1.25cm, line width=0.3mm] (c) at (5.0,0) {\large $\vb{K}$};
    \node[draw, rounded corners, minimum size=0.8cm, inner sep=0.0, line width=0.3mm] (Q1) at (3.811179354631058, 0.38627124296868437) {$\vb{U}_{4,1}$};
    \node[draw, rounded corners, minimum size=0.8cm, inner sep=0.0, line width=0.3mm] (Q2) at (4.265268434634408, -1.011271242968684) {$\vb{U}_{4,2}$};
    \node[draw, rounded corners, minimum size=0.8cm, inner sep=0.0, line width=0.3mm] (Q3) at (5.734731565365591, -1.0112712429686845) {$\vb{U}_{4,3}$};
    \node[draw, rounded corners, minimum size=0.8cm, inner sep=0.0, line width=0.3mm] (Q4) at (6.188820645368942, 0.386271242968684) {$\vb{U}_{4,1}$};
    \node[draw, rounded corners, minimum size=0.8cm, inner sep=0.0, line width=0.3mm] (Q5) at (5.0, 1.25) {$\vb{U}_{4,5}$};

    \draw[line width=0.3mm] (c) -- (Q1);
    \draw[line width=0.3mm] (c) -- (Q2);
    \draw[line width=0.3mm] (c) -- (Q3);
    \draw[line width=0.3mm] (c) -- (Q4);
    \draw[line width=0.3mm] (c) -- (Q5);

    \node (Q1e) at (3.002781315780178, 0.6489356881873898) {};
    \node (Q2e) at (3.765650970185806, -1.6989356881873894) {};
    \node (Q3e) at (6.234349029814194, -1.69893568818739) {};
    \node (Q4e) at (6.997218684219822, 0.6489356881873891) {};
    \node (Q5e) at (5.000000000000001, 2.1) {};

    \draw[line width=0.3mm] (Q1e) -- (Q1);
    \draw[line width=0.3mm] (Q2e) -- (Q2);
    \draw[line width=0.3mm] (Q3e) -- (Q3);
    \draw[line width=0.3mm] (Q4e) -- (Q4);
    \draw[line width=0.3mm] (Q5e) -- (Q5);

    \node (equals) at (7.75,0) {\Large $\approx$};

    \node[draw, rounded corners, minimum size=0.5cm, inner sep=0.0, line width=0.3mm] (oneL) at (7.75+1.25,0.5) {$1$};
    \node[draw, rounded corners, minimum size=0.5cm, inner sep=0.0, line width=0.3mm] (oneR) at (15.25+1.25,0.5) {$1$};
    
    \node[draw, rounded corners, minimum size=0.9cm, inner sep=0.5, line width=0.3mm] (G1) at (9.0+1.25,0.5) {$\vb{G}_{4,1}$};
    \node[draw, rounded corners, minimum size=0.9cm, inner sep=0.5, line width=0.3mm] (G2) at (10.25+1.25,0.5) {$\vb{G}_{4,2}$};
    \node[draw, rounded corners, minimum size=0.9cm, inner sep=0.5, line width=0.3mm] (G3) at (11.5+1.25,0.5) {$\vb{G}_{4,3}$};
    \node[draw, rounded corners, minimum size=0.9cm, inner sep=0.5, line width=0.3mm] (G4) at (12.75+1.25,0.5) {$\vb{G}_{4,4}$};
    \node[draw, rounded corners, minimum size=0.9cm, inner sep=0.5, line width=0.3mm] (G5) at (14.0+1.25,0.5) {$\vb{G}_{4,5}$};

    \draw[line width=0.3mm] (oneL) -- (G1) -- (G2) -- (G3) -- (G4) -- (G5) -- (oneR);
    
    \node[draw, rounded corners, minimum size=0.8cm, inner sep=0.0, line width=0.3mm] (Q1b) at (9.0+1.25,-0.5) {$\vb{U}_{4,1}$};
    \node[draw, rounded corners, minimum size=0.8cm, inner sep=0.0, line width=0.3mm] (Q2b) at (10.25+1.25,-0.5) {$\vb{U}_{4,2}$};
    \node[draw, rounded corners, minimum size=0.8cm, inner sep=0.0, line width=0.3mm] (Q3b) at (11.5+1.25,-0.5) {$\vb{U}_{4,3}$};
    \node[draw, rounded corners, minimum size=0.8cm, inner sep=0.0, line width=0.3mm] (Q4b) at (12.75+1.25,-0.5) {$\vb{U}_{4,4}$};
    \node[draw, rounded corners, minimum size=0.8cm, inner sep=0.0, line width=0.3mm] (Q5b) at (14.0+1.25,-0.5) {$\vb{U}_{4,5}$};
			
    \draw[line width=0.3mm] (G1) -- (Q1b);
    \draw[line width=0.3mm] (G2) -- (Q2b);
    \draw[line width=0.3mm] (G3) -- (Q3b);
    \draw[line width=0.3mm] (G4) -- (Q4b);
    \draw[line width=0.3mm] (G5) -- (Q5b);

    \node (Q1c) at (9.0+1.25,-1.25) {};
    \node (Q2c) at (10.25+1.25,-1.25) {};
    \node (Q3c) at (11.5+1.25,-1.25) {};
    \node (Q4c) at (12.75+1.25,-1.25) {};
    \node (Q5c) at (14.0+1.25,-1.25) {};

    \draw[line width=0.3mm] (Q1b) -- (Q1c);
    \draw[line width=0.3mm] (Q2b) -- (Q2c);
    \draw[line width=0.3mm] (Q3b) -- (Q3c);
    \draw[line width=0.3mm] (Q4b) -- (Q4c);
    \draw[line width=0.3mm] (Q5b) -- (Q5c);
\end{tikzpicture}
\caption{Approximation of the $4$th derivative tensor, $D^4 f(0)$, (left) by a Tucker decomposition (middle) and a Tucker tensor train (right), shown in graphical tensor notation. Note that the $4$th derivative tensor has $5$ indices: one for each input perturbation, and one for the output.}
\label{fig:tucker_tensor_train}
\end{figure}

T4S is complementary to global operator-learning methods such as neural operators \citep{kovachki.li.ea.2023}, including Fourier neural operators \citep{Li2021fourier}, DeepONets \citep{lu2021learning}, derivative-informed neural operators (DINO) \citep{cao2025derivative,luo2025efficient,olearyroseberry.chen.ea.2024,o2022derivative,qiu2024derivative}, and related neural surrogate models. Those methods learn global input-output maps from many sampled function evaluations at different parameter values, and can be highly effective when such data are available; approximation-theoretic results exist for several architectures \citep{kovachki.lanthaler.ea.2021,lanthaler.mishra.ea.2022}. T4S targets a different regime: local learning from derivative information at one expansion point, with the goal of a mathematically structured surrogate whose derivatives at the expansion point are fit directly. DINO and the derivative-enhanced DeepONet of \cite{qiu2024derivative}, for instance, train operator surrogates using first-order derivative information together with parameter dimension reduction; T4S carries this idea to derivatives of all orders up to $k$ in a structured low-rank tensor format. 

Relatedly, classical local surrogates in large-scale inverse problems and uncertainty quantification rely on low-rank approximations of the Jacobian of $f$ or the Hessian of associated scalar-valued functions \citep{bashir2008hessian,bui2013computational,chen2019hessian,cui2016dimension,flath2011fast,IsaacPetraStadlerEtAl15short,petra2014computational,spantini2015optimal,wu2021efficient}. Because $f$ is vector-valued, its $k$th derivative tensor has $k+1$ modes (Section 3.1): the linear ($k=1$) term is a Jacobian, the analogue of the Gauss-Newton Hessian used in scalar-valued function approximations. The $k=2$ term is already a third-order tensor, rarely formed with existing methods, and $k \ge 3$ lies beyond them. The Tucker tensor train format of T4S reduces to a low-rank approximation of the Jacobian at $k=1$ and a Tucker decomposition of the third-order tensor at $k=2$. 

T4S is also complementary to sparse polynomial and sparse-grid methods \citep{adcock2022sparse,babuvska2007stochastic,bonizzoni2020regularity,chen2015new,chernov2013first,cohen2015approximation,kunoth2016sparse,xiu2002wiener}, transport maps \citep{brennan2020greedy,marzouk2016sampling,papamakarios2016fast}, and low-rank tensor and hybrid low-rank/sparse-polynomial methods \citep{bachmayr2023low,bachmayr2025sparse,khoromskij2018tensor}. Many of these have useful approximation theory but rely on function evaluations at scattered input points, or on access patterns poorly suited to implicit derivative tensors; T4S instead exploits derivative probes, a data source natural for implicit simulators and valuable for derivative-dependent outer-loop tasks. The closest relative on the learning side is the tensor train polynomial model of \cite{novikov.trofimov.ea.2018}, developed for supervised learning, which represents the $2 \times \dots \times 2$ coefficient tensor of an all-order global polynomial predictor over binary features as a tensor train and fits it to labeled examples by stochastic Riemannian gradient descent.

This paper builds on our previous tensor train Taylor series work \citep{alger.chen.ea.2020}, which demonstrated that high-order Taylor surrogates can be useful for PDE-based parameter-to-output maps. The present work changes the learning and computational picture in four ways. First, we use directionally symmetric derivative probes, which are far cheaper than the asymmetric probes used previously. Second, we replace plain tensor trains with Tucker tensor trains, improving conditioning and separating large physical dimensions from internal ranks. Third, we fit the tensors directly from probe data using Riemannian optimization and rank continuation, rather than using a complicated multi-stage construction. Fourth, we prove representational approximation guarantees for covariance-whitened derivative tensors, showing how spectral decay of the input covariance leads to explicit Tucker and tensor train rank-error tradeoffs.

The proposed approach is local, by design. A Taylor surrogate is appropriate when the map is sufficiently smooth near the expansion point and the relevant inputs remain in a neighborhood where the truncated expansion is accurate. T4S is therefore best suited to local uncertainty quantification, posterior approximation, design, and optimization concentrated near a nominal parameter or inferred state---the local forms of the outer-loop tasks noted at the outset, such as deterministic and Bayesian inverse problems \citep{BanksKunisch89,cao2023residual,KaipioSomersalo05,Stuart10a,Tarantola05,Vogel02}, prediction and PDE-constrained control under uncertainty \citep{AlexanderianPetraStadlerEtAl17,IsaacPetraStadlerEtAl15short}, optimal design and optimal experimental design \citep{AlexanderianNicholsonPetra24,du2023learning,go2025accurate,go2025sequential,gong2026shape,o2022learning,Pironneau84,wu2023large}, and digital-twin updating \citep[pages 61--68]{of2024foundational}. It is not intended to replace global operator learning. If global coverage is needed, one possible extension is to combine multiple local T4S models in a mixture or adaptive multipoint construction, for example by a higher-order generalization of the method in \cite{LuoChenChenEtAl2024}. The method also relies on low-rank structure in covariance-whitened derivative tensors: when the input covariance spectrum decays slowly and the derivative tensors have high intrinsic rank, the required Tucker and tensor train ranks may become large.

\subsection{Contributions}

The main contribution of this paper is a derivative-probe learning framework for constructing local, derivative-accurate surrogates of expensive implicit maps. Within this framework, we make the following specific contributions:

\begin{description}
\item[Tucker tensor train Taylor series surrogate model.] We introduce T4S, a local Taylor surrogate in which each derivative tensor is represented as a Tucker tensor train. The Tucker factors decouple the large external index dimension $N$ from the internal tensor train rank structure, replacing the $O(dNr^2)$ storage of full physical-dimension tensor train cores by $O(dNr+dr^3)$ storage in the equal-rank setting and avoiding the ill-conditioning that arises from over-parameterization in the full tensor train format.
\item[Derivative-probe training data.] We formulate surrogate construction as a learning problem from random derivative tensor probes at a single expansion point, and give algorithms for computing these probes efficiently. The algorithms exploit directional symmetry and amortize the cost of probing different derivative orders with the same input vectors. Directionally symmetric probes align with Taylor series evaluation and reduce the number of linearized state and adjoint solves per probe from exponential to linear in derivative order.
\item[Derivative-informed dimension reduction.] We develop a dimension reduction step that constructs compact input and output bases from higher-order derivative information at the expansion point. This reduces the infinite- or very-high-dimensional learning problem to a moderate-dimensional tensor fitting problem while preserving directions important for the local Taylor approximation.
\item[Riemannian fitting algorithms with rank continuation.] We identify trust-region Riemannian manifold Gauss-Newton (TR-RMGN) and Riemannian manifold Cauchy SGD (MC-SGD) as effective methods for fitting Tucker tensor trains from probe data. MC-SGD uses Cauchy step sizes, computing directional curvature along the gradient by a Hessian-vector product to set the step size automatically. We also introduce a rank continuation scheme based on singular values of tensor unfoldings and matricizations, which incrementally grows Tucker and tensor train ranks during optimization and reduces sensitivity to rank choice. The methods require little to no hyperparameter tuning, an advantage that is important under rank continuation, which would otherwise require retuning at each rank change.
\item[Fast sweeping algorithms for the Riemannian Jacobian.] Riemannian optimization requires repeated applications of the Riemannian Jacobian of the probe residual and its transpose. We develop fast sweeping algorithms for these operations, enabling efficient gradients and Gauss-Newton Hessian-vector products without forming large dense Jacobians. Simplifications of these algorithms may be used to compute gradients and Hessian-vector products in the Euclidean (non-Riemannian) context.
\item[Representational approximation guarantees.] We prove that each covariance-whitened derivative tensor can be approximated by a Tucker tensor train with quantified ranks and errors (\thmref{thm:ttt_main_thm}). Under power-law spectral decay, \corref{cor:error_vs_rank_poly_decay} gives explicit rank-error tradeoffs in the large-rank asymptotic regime. These guarantees use only spectral decay of the input covariance operator and are therefore conservative; additional low-rank structure of derivative tensors can improve performance further.
\item[Numerical evidence.] We test the tensor fitting algorithms on random tensor benchmarks and 
construct T4S surrogates for two PDE-based parameter-to-output maps: a linear\footnote{Linear in the source but nonlinear in the material coefficient, which is the relevant quantity for our approximation.} Poisson PDE with a Gaussian parameter and boundary-value observable, and a nonlinear one with a logistic parameter and a boundary-flux observable.
The experiments show that TR-RMGN and MC-SGD track quasi-optimal T3-SVD accuracy up to data-limited ranks, that rank continuation is important for stable fitting, and that the learned Taylor surrogates recover higher-order Taylor structure as the number of derivative probes increases. These are validation studies of the learning method 
rather than demonstrations of a specific outer-loop application, which we leave to future work.
\end{description}

\tableofcontents

\section{Problem setting and T4S model}
\label{sec:setting}

In this paper, we are concerned with the following learning problem:

\begin{framed}
\textbf{Derivative-probe surrogate learning problem.}
Let $f : X \rightarrow Y$ be an expensive smooth map between high- or infinite-dimensional spaces, defined implicitly by a simulator, and let $x = 0$ be an expansion point. For derivative orders
$j = 1, \dots, k$,
suppose we can query forward and reverse probes of $D^jf(0)$ but cannot efficiently access dense derivative-tensor entries. Given a computational budget for such probes, construct a local surrogate $f_k$ that can be evaluated cheaply and whose quality is assessed by prediction accuracy near the expansion point, derivative-probe accuracy on held-out directions, and performance in downstream derivative-dependent tasks.
\end{framed}

The remainder of this section makes the elements of this problem precise. \secref{subsec:setting} specifies the target map $f$, \secref{sec:intro_probes} defines the probe data, and \secref{sec:t4s_model} defines our proposed T4S surrogate $f_k$.

\subsection{Setting}
\label{subsec:setting}

We seek to approximate a map $\ptoW: X \rightarrow Y$ of the form
\begin{equation*}
    f(x) = q(\theta_0 + \sqrtcov x).
\end{equation*}
Here, $X$ and $Y$ are Hilbert spaces, $\theta_0 \in X$ is a point of interest, $\sqrtcov:X \rightarrow X$ is a self-adjoint positive-semidefinite Hilbert-Schmidt operator, and $q:X \rightarrow Y$,
\begin{equation}
\label{eq:f_of_m_defn}
q(\theta) = Q(\theta, u(\theta)),
\end{equation}
is a mapping that depends on a parameter $\theta$ implicitly through a state variable $u=u(\param)$ that solves a \emph{state equation}
\begin{equation}
R(\theta,u) = 0. \label{eq:pde_constraint}
\end{equation}
We call $q$ the \emph{parameter-to-output} map and $f$ the \emph{covariance-whitened} parameter-to-output map; $\sqrtcov$ plays the role of a covariance square root, so that in finite-dimensional discretizations, if $x$ is a white noise perturbation and $\theta=\theta_0+\sqrtcov x$, the corresponding parameter perturbation has covariance $\sqrtcov^2$. We assume that the state $u$ resides in a Banach space, the residual $R$ maps into a Banach space, $Q$ and $R$ are smooth, \eqref{eq:pde_constraint} is uniquely solvable for $u$ for all $\theta$ in a neighborhood of $\theta_0$, and the linearized state operator $\partial_u R(\theta_0, u_0)$ is continuously invertible, where $u_0 := u(\theta_0)$. We write $u(\theta)$ for the solution of \eqref{eq:pde_constraint} at a given $\theta$.

We are particularly interested in the setting where $R(\theta,u)=0$ is a PDE or system of PDEs, $\theta$ is a high-dimensional random vector or infinite-dimensional random field with (exact or approximate) mean $\theta_0$ and covariance $C^2$, $u$ is the PDE solution, and $q$ is a high- or infinite-dimensional quantity of interest. For instance, $\theta$ could be the initial condition, boundary conditions, distributed source, or material coefficient in the PDE, and $q$ could be the trace of $u$ or its normal derivative along a boundary. A typical $\sqrtcov^2$ in this setting is a covariance operator in the Mat\'ern class (\citealp[Section 3.4.1]{diggle1998model}; \citealp{lindgren2011explicit,Roininen14}). An illustration of such a map $f$ is shown in \figref{fig:overall_mapping_example}.

\subsection{Tensor probes as training data}
\label{sec:intro_probes}

Our learning problem is complicated by the fact that array entries of discretized derivative tensors are not easily accessible, a consequence of the implicit dependence of $u(\theta)$ on $\theta$. Instead, derivative tensors are accessible only indirectly, by \emph{probing}\footnote{Note: we previously \citep{alger.chen.ea.2020} used the term ``tensor actions'' for tensor probes. We break from this prior terminology because tensor probes are not group actions.} the tensor.

\begin{definition}[Tensor probes]
\label{def:tensor_actions}
Let $A:X_1 \times \dots \times X_j \rightarrow \mathbb{R}$ be a scalar-valued multilinear function, and let $(x_i)_{i=1}^j$ be a collection of vectors $x_i \in X_i$.
The $i$th \emph{probe} of $A$ by $x_1,\dots,x_{i-1}, x_{i+1}, \dots, x_{j}$ is the linear functional $\psi_i:X_i \rightarrow \mathbb{R}$,
\begin{equation*}
\psi_i(v) = A(x_1, \dots,x_{i-1}, v, x_{i+1}, \dots, x_j).
\end{equation*}
We say that  $(\psi_i)_{i=1}^j$ are the \emph{probes} of $A$ by $(x_i)_{i=1}^j$. 
Probes are \emph{directionally symmetric} if $X_1=\dots=X_j$ and $x_1=\dots=x_j$.
\end{definition}

Tensor probing is illustrated in \figref{fig:general_tensor_actions}.
The definition extends to vector-valued multilinear functions via the standard scalar/vector correspondence by currying (see \secref{sec:tensor_correspondences}).
For the vector-valued derivative tensor $D^j f(0): X^j \rightarrow Y$, we use two types of probes. A \emph{forward probe} evaluates the tensor on $j$ input directions,
\begin{equation*}
y = D^j f(0)(x_1, \dots, x_j) \in Y.
\end{equation*}
A \emph{reverse probe} pairs the output with a functional $\omega \in Y'$ and leaves one input slot open,
\begin{equation*}
\psi(v) = \omega\!\left(D^j f(0)(x_1, \dots, x_{j-1}, v)\right).
\end{equation*}
In this paper the input directions are usually directionally symmetric: $x_1 = \dots = x_j$ for forward probes and $x_1 = \dots = x_{j-1}$ for reverse probes. Here ``directionally symmetric'' refers only to the repeated input directions; the output mode is either evaluated directly in a forward probe or paired with an independent output functional $\omega$ in a reverse probe. Forward and reverse probes are defined precisely in \secref{sec:tensor_correspondences}, and their computation for derivatives of implicitly defined maps is detailed in \secref{sec:derivative_actions}.

In finite dimensions, multilinear functions correspond to arrays, and the $i$th probe is the vector obtained by contracting the array with vectors in all but the $i$th index. An array entry can be recovered by probing with appropriately chosen unit basis vectors and selecting one entry of the result, but this is inefficient when no other entries of that probe are needed. Conventional tensor approximation methods, which sample scattered tensor entries, are therefore unsuitable for tensors accessible only through probing.

Probing derivatives of an implicitly defined map requires solving many linearized versions of the state equation \eqref{eq:pde_constraint}. The number of solves, and the cost of assembling their right-hand sides, depends on the repetition (directional symmetry) in the probing vectors $x_i$: for an order-$j$ derivative it grows exponentially in $j$ for directionally asymmetric probes (no repetition) but only linearly for directionally symmetric probes (total repetition), with precise counts in \secref{sec:cost_vs_symmetry}.
Because Fr\'echet derivatives are symmetric in their input arguments, their action on repeated directions determines the full symmetric multilinear map through polarization; finite random symmetric probes therefore provide a natural training-data model for fitting derivative tensors.

\subsection{T4S model}
\label{sec:t4s_model}

We build a surrogate model of $f$ based on a truncated Taylor series:
\begin{equation}
    f(x) \approx f(0) + Df(0)x + \frac{1}{2}D^2f(0)x^2 + \dots + \frac{1}{k!}D^kf(0)x^k, \label{eq:ptoW taylor}
\end{equation}
where
$
    D^jf(0): X^j \rightarrow Y
$
is the $j$th Fr\'{e}chet derivative of $f$ with respect to $x$, evaluated at zero. 
We write
$$
D^jf(0)x^j := D^jf(0)\underbrace{(x, \dots, x)}_{j\text{ copies of }x}
$$ 
to denote the application of the multilinear function $D^jf(0)$ to $j$ copies of $x$.

Such Taylor series approximations are traditionally considered intractable in high dimensions because the derivative tensors are extremely large: after discretization of $X$ and $Y$ by, e.g., finite elements, $D^jf(0)$ becomes a $(j+1)$-dimensional array with shape $N \times \dots \times N \times M$, where $N$ and $M$ are the numbers of degrees of freedom in the discretizations of $X$ and $Y$ (e.g., the number of gridpoints in a mesh). 
To make Taylor series tractable, we approximate the derivative tensors with a type of tensor network called a \emph{Tucker tensor train}, equivalently termed an \emph{extended tensor train}. A Tucker tensor train consists of a Tucker decomposition (\citealp{de2000multilinear,tucker1966some}; \citealp[Section 4]{kolda.bader.2009}) composed with a tensor train (TT) decomposition \citep{fannes1992finitely,oseledets.2011,oseledets2009breaking} of the central Tucker core, illustrated using graphical tensor notation\footnote{For an introduction to tensor networks and graphical tensor notation, see, e.g., \cite{bridgeman2017hand,cichocki2016tensor}.} in \figref{fig:tucker_tensor_train}. When its ranks are small, a Tucker tensor train requires low storage and permits fast multilinear algebra operations. The Tucker bases separate the large physical discretization dimensions from the smaller internal ranks, while the tensor train cores control complexity with respect to derivative order; the format, its relationship to other low-rank formats, and its manifold structure are detailed in \secref{sec:t3}, where we also define Tucker tensor trains and their ranks formally (\defref{def:t3}).

\begin{definition}[T4S model]
\label{def:tttt}
A $k$th-order Tucker tensor train Taylor series (T4S) is an approximation $f_k \approx f$ of the form:
\begin{equation}
    f_k(x) = f(0) +  
    \sum_{j=1}^k \frac{1}{j!}  T_j\!\underbrace{\left(x, \dots, x\right)}_{j\text{ copies of }x}, \label{eq:taylor_model_infinite}
\end{equation}
where $T_j$ is the vector-valued multilinear function corresponding to a Tucker tensor train
\begin{equation*}
\tuple{T}_j = \left((U_{j,i})_{i=1}^{j+1},(G_{j,i})_{i=1}^{j+1}\right).
\end{equation*}
Here, $U_{j,i}$ are the Tucker basis cores and $G_{j,i}$ are the tensor train cores which define the Tucker tensor train. The T4S model is specified by the tuple $\left(f(0), \tuple{T}_1, \dots, \tuple{T}_k\right)$.
\end{definition}

If $T_j = D^j f(0)$ for all $j \le k$, then $f_k$ matches $f$ and all of its derivatives through order $k$ at the expansion point. In practice, $T_j$ is a learned low-rank approximation of $D^j f(0)$, so the derivative accuracy of the T4S model is affected by low-rank approximation error.
The derivative tensor $D^j f(0)$ is symmetric in its $j$ input arguments, but the Tucker tensor train parameterization in \defref{def:tttt} does not explicitly enforce this symmetry. This is intentional: the Taylor polynomial depends only on the symmetrized input action $T_j(x, \dots, x)$, and the directionally symmetric probe data used here identify this symmetric action. Enforcing shared input factors or explicitly symmetric cores is a possible refinement, but is not required for evaluating the learned Taylor surrogate.

\defref{def:tttt} describes the full-space model. The practical construction in \secref{sec:method} first builds reduced input and output bases and fits Tucker tensor trains for the derivatives of the reduced map. The reduced tensors may be viewed either as a compressed T4S model or as full-space tensors when composed with the learned input and output bases.

Beyond the prior work \citep{alger.chen.ea.2020} discussed in the introduction, the T4S model is connected to \cite{chkifa2013sparse}, which constructs a sparse polynomial approximation for stochastic PDEs based on Taylor approximation; \cite{kressner2015low}, which forms tensor train approximations of higher-order correlations of Gaussian random fields in a truncated KL basis; and \cite{BONIZZONI2016349,bonizzoni2013analysis}, which form tensor train approximations of higher-order moments for stochastic PDEs by expanding the PDE in a Taylor series.

Building the T4S model requires selecting the Taylor series order $k$.
In a Taylor expansion, contributions typically decrease with increasing order, yielding diminishing returns from higher-order terms; moreover, the low-order terms are themselves subject to low-rank approximation error, which can overwhelm any gains from including higher-order terms. While there is no universal best order,
our practical experience suggests that a modest order of $k=3$ or $k=4$ (corresponding to derivative tensors with $4$ or $5$ total modes, including the output mode) often offers a favorable balance between accuracy and computational expense.


\section{Preliminaries}
\label{sec:tensor_background_real}

This section summarizes required definitions and background on multilinear functions, multidimensional arrays, and the correspondences between them. We also define tensor trains and Tucker tensor trains.
Following standard usage in machine learning and computational science, we use the term \emph{tensor} as a general term referring to either a multilinear function (scalar- or vector-valued) or its corresponding array.
Readers familiar with this material may skip this section.
As a reader's guide to notation, throughout the paper blackboard-bold symbols (e.g., $\tuple{T}$) denote tensor-network core representations, standard symbols (e.g., $T$) denote the multilinear functions they represent, and bold symbols (e.g., $\vb{T}$) denote finite-dimensional arrays. We use bracket notation for array entries: $\vb{X}[i_1, \dots, i_d]$ denotes the $(i_1, \dots, i_d)$-entry of a $d$-dimensional array $\vb{X}$. The order of a derivative tensor $D^j f(0)$ is $j+1$ when the output mode is included.



\subsection{Tensor background}
\label{sec:tensor_correspondences}


\paragraph{Scalar- and vector-valued multilinear functions} Scalar- and vector-valued multilinear functions are closely related and can be converted into one another. A scalar-valued multilinear function $A:X_1 \times \dots \times X_k \times Z \to \mathbb{R}$ corresponds to the vector-valued multilinear function
\begin{equation*}
\widetilde{A}:X_1 \times \dots \times X_k \to Z', \qquad \widetilde{A}(x_1, \dots, x_k)(z) := A(x_1, \dots, x_k, z).
\end{equation*}
A vector-valued multilinear function $B:X_1 \times \dots \times X_k \to Y$ corresponds to the scalar-valued multilinear function 
\begin{equation*}
\widetilde{B}:X_1 \times \dots \times X_k \times Y' \to \mathbb{R}, \qquad \widetilde{B}(x_1, \dots, x_k, \omega) := \omega(B(x_1, \dots, x_k)).
\end{equation*}
Applying these transformations in sequence (scalar $\rightarrow$ vector $\rightarrow$ scalar) yields
$
\widetilde{\widetilde{A}}(x_1, \dots, x_k, E_z) = A(x_1, \dots, x_k, z),
$
for any evaluation map $E_z \in Z''$,
$E_z(\omega) = \omega(z)$.
The reverse sequence (vector $\rightarrow$ scalar $\rightarrow$ vector) yields
$
\widetilde{\widetilde{B}}(x_1, \dots, x_k) = E_{B(x_1, \dots, x_k)},
$
where $E_{B(x_1, \dots, x_k)}$ is the evaluation map associated with $B(x_1, \dots, x_k) \in Y$. 
If $Y$ and $Z$ are reflexive---as is the case for Hilbert spaces, the setting for the derivative tensors here---then $\widetilde{\widetilde{A}} = A$ and $\widetilde{\widetilde{B}} = B$ under the canonical identifications $Y \cong Y''$ and $Z \cong Z''$.

The probes of a vector-valued multilinear function $B$ 
by vectors $(x_1, \dots, x_k, \omega)$ 
are defined as the probes of the corresponding scalar-valued multilinear function $\widetilde{B}$ by those same vectors.
For the first $k$ probes of $B$, one of the `$x_j$ arguments' of $\widetilde{B}$ is left unevaluated, i.e.,
\begin{equation*}
\psi_j(~\cdot~) = \omega(B(x_1, \dots, x_{j-1}, ~\cdot~, x_{j+1}, \dots, x_k)). \qquad (j\text{th reverse probe})
\end{equation*}
We call these \emph{reverse probes} of $B$.
The last probe of $B$, in which the `$\omega$ argument' of $\widetilde{B}$ is left unevaluated, is the \emph{forward probe} of $B$. The forward probe is the evaluation map $\omega \mapsto \omega(B(x_1, \dots, x_k))$, which we identify with the vector $y \in Y$ given by
$$
y = B(x_1, \dots, x_k). \qquad (\text{forward probe})
$$


\paragraph{Arrays and multilinear functions} Given a scalar-valued multilinear function $A:\mathbb{R}^{N_1} \times \dots \times \mathbb{R}^{N_d} \to \mathbb{R}$, the corresponding array $\vb{A} \in \mathbb{R}^{N_1 \times \dots \times N_d}$ is defined by
\begin{equation*}
\vb{A}[i_1, \dots, i_d] := A(\vb{e}_{1,i_1}, \dots, \vb{e}_{d,i_d}),
\end{equation*}
where
$\vb{e}_{a,b}=(0,\dots,0,1,0,\dots,0)$ denotes the $b$-th standard basis vector in $\mathbb{R}^{N_a}$. Given such an array $\vb{A}$, the corresponding scalar-valued multilinear function $A$ is obtained as:
\begin{equation*}
A(\vb{x}_1, \dots, \vb{x}_d) = \sum_{i_1=1}^{N_1} \dots \sum_{i_d=1}^{N_d} \vb{A}[i_1, \dots, i_d]\, \vb{x}_1[i_1] \dots \vb{x}_d[i_d].
\end{equation*}
For a vector-valued multilinear function $B:\mathbb{R}^{N_1} \times \dots \times \mathbb{R}^{N_k} \to \mathbb{R}^M$, the corresponding array $\vb{B} \in \mathbb{R}^{N_1 \times \dots \times N_k \times M}$ is defined by
\begin{equation*}
\vb{B}[i_1, \dots, i_k, j] := B(\vb{e}_{1,i_1}, \dots, \vb{e}_{k,i_k})[j].
\end{equation*}
Given an array $\vb{B}$ of this shape, the corresponding vector-valued multilinear function $B$ is
\begin{equation*}
B(\vb{x}_1, \dots, \vb{x}_k)[j] = \sum_{i_1=1}^{N_1} \dots \sum_{i_k=1}^{N_k} \vb{B}[i_1, \dots, i_k, j]\, \vb{x}_1[i_1] \dots \vb{x}_k[i_k].
\end{equation*}
These definitions, combined with the scalar/vector correspondences described above,
establish one-to-one correspondences between finite-dimensional scalar-valued multilinear functions, vector-valued multilinear functions, and multidimensional arrays.

\paragraph{Linear algebra and norms for tensors}

Addition and scaling of arrays is performed elementwise. Addition and scaling of multilinear functions is defined by adding and scaling their outputs. 
The induced norm of a multilinear function $B: X_1 \times \dots \times X_k \to Y$ is
\begin{equation}
\label{eq:induced_norm_definition}
\norm{B} = \sup_{\substack{
\vb{x}_i \in X_i, \\
\norm{\vb{x}_i} = 1, \\
i=1,\dots,\taylororder
}}
\norm{B(\vb{x}_1, \dots, \vb{x}_\taylororder)},
\end{equation}
where the norms on the right-hand side are those of $X_i$ and $Y$. For an array $\vb{B}$, we define $\norm{\vb{B}}$ as the induced norm of its corresponding multilinear function. Although the induced norm is NP-hard to compute \citep{hillar2013most}, it plays a central role in our theory since it determines the boundedness of Fréchet derivatives of smooth functions (see, e.g., \cite[Section~9.4]{arbogast.bona.2008}).
We denote the Hilbert-Schmidt (also known as Frobenius) norm by $\norm{\cdot}_{\mathrm{HS}}$ and its corresponding inner product by $\innerproduct{\cdot}{\cdot}_{\mathrm{HS}}$. For one-dimensional arrays $\vb{x}, \vb{x}' \in \mathbb{R}^N$, we have
$
\norm{x}=\norm{x}_{\mathrm{HS}}=\left(\sum_{i=1}^N \vb{x}[i]^2\right)^{1/2}$
and
$
\innerproduct{\vb{x}}{\vb{x}'}_\text{HS}=\sum_{i=1}^N \vb{x}[i]\vb{x}'[i].
$
All operations described above (addition, scaling, norms, and inner products) are consistent across the equivalent representations of tensors described in \secref{sec:tensor_correspondences}.

\paragraph{Matrix unfoldings, matricizations, and the Kronecker product}


If $\vb{A}$ is an $N_1 \times \dots \times N_\nummodes$ array,  the \emph{$i$th matrix unfolding} of $\vb{A}$ is the $(N_1 \dots N_i) \times (N_{i+1} \dots N_\nummodes)$ matrix formed by reshaping $\vb{A}$ so that the first $i$ modes of $\vb{A}$ correspond to the rows of the matrix, and the last $\nummodes-i$ modes of $\vb{A}$ correspond to the columns of the matrix.
The $i$th \emph{matricization} of $\vb{A}$ is the $N_i \times (N_1 \dots N_{i-1} N_{i+1}\dots N_d)$ matrix formed by permuting the indices of $\vb{A}$ to move the $i$th index first, then reshaping the permuted version of $\vb{A}$ into a matrix such that the rows of the matrix correspond to the $i$th index, and the columns of the matrix correspond to all other indices.
Any standard array ordering convention may be used for the reshapings (e.g., row-major or column-major), so long as the same convention is used consistently, and one defines the Kronecker product, $\kron$, consistently with the ordering.

\subsection{Tucker tensor trains}
\label{sec:t3}

A Tucker tensor train combines a Tucker decomposition of a tensor with a tensor train decomposition of its central Tucker core. Here, we define tensor trains and Tucker tensor trains. The Tucker decomposition is never used directly, so its definition is omitted.
We formulate Tucker tensor trains in terms of vector-valued multilinear functions, as our interest lies in approximating vector-valued derivative tensors. Tensor trains, however, are defined in terms of scalar-valued multilinear functions, since defining them for vector-valued multilinear functions would require treating the final tensor index as a special case, complicating subsequent analysis.

\begin{definition}[Tensor train (TT)]
\label{def:tt}
A \emph{tensor train (TT) decomposition} of a scalar-valued multilinear function
$S: X_1 \times X_2 \times \dots \times X_d \rightarrow \mathbb{R}$  
is a factorization
$$
S(x_1, x_2, \dots, x_d)
= G_1(x_1)\,G_2(x_2)\,\dots\,G_d(x_d),
$$
where $G_i: X_i \to \mathbb{R}^{r_{i-1} \times r_i}$ are matrix-valued linear functions and $r_0 = r_d = 1$.
The right-hand side is a product of matrices. The tensor train representation of $S$ is the tuple  
$$
\tuple{S} = (G_1, \dots, G_d).
$$ 
Each $G_i$ is called a tensor train \emph{core}, and 
$\vb{r} = (r_0, r_1, \dots, r_d)$ are the tensor train \emph{ranks (TT-ranks)}.
\end{definition}

When $X_i = \mathbb{R}^{n_i}$, the function $S$ corresponds to an array $\vb{S} \in \mathbb{R}^{n_1 \times \dots \times n_d}$,  
and each $G_i$ corresponds to a 3-tensor  
$\vb{G}_i \in \mathbb{R}^{r_{i-1} \times n_i \times r_i}$  
with entries
$
\vb{G}_i[a,b,c] = \vb{e}_a^T\, G_i(\vb{e}_b)\, \vb{e}_c,
$
where $\vb{e}_l$ denotes the $l$th unit vector of appropriate dimension.
The array $\vb{S}$ can then be written as the contraction
$$
\vb{S}[i_1,\dots,i_d]
= \sum_{a_1=1}^{r_1} \!\dots\! \sum_{a_{d-1}=1}^{r_{d-1}}
\vb{G}_1[1,i_1,a_1]
\,\vb{G}_2[a_1,i_2,a_2]
\dots
\vb{G}_d[a_{d-1},i_d,1].
$$
Tensor trains are named for the sequential contraction of their cores, which is analogous to the way railway cars are connected in a railway train.   
When ranks remain moderate, the storage and computational costs scale linearly with the number of tensor indices, effectively mitigating the curse of dimensionality.    
For further background, see \cite[Section 3.4]{gelss.2017.dissertation}, \cite{orus2014practical,oseledets.2011}. In the physics literature, tensor trains are known as matrix product states.



\begin{definition}[Tucker tensor train (T3)]
\label{def:t3}
A \emph{Tucker tensor train (T3) decomposition} of a vector-valued multilinear function  
$T:\paramspace_1 \times \dots \times \paramspace_k \rightarrow \obsspace$  
is a factorization
$$
T(x_1, \dots, x_k)
= U_{k+1}\, S(U_1^* x_1, \dots, U_k^* x_k),
$$
where $U_i:\mathbb{R}^{n_i}\rightarrow\paramspace_i$ for $i=1,\dots,k$,  
$U_{k+1}:\mathbb{R}^{n_{k+1}}\rightarrow\obsspace$ are linear operators,  
and $S:\mathbb{R}^{n_1}\times\dots\times\mathbb{R}^{n_k}\rightarrow\mathbb{R}^{n_{k+1}}$
is the vector-valued multilinear function represented by a tensor train $\mathbb{S}=(G_1,\dots,G_{k+1})$. Here the output mode of $S$ is vector valued; it is represented as a tensor train, which \defref{def:tt} defines in terms of scalar-valued functions, via the scalar/vector correspondence of \secref{sec:tensor_correspondences}.
The Tucker tensor train is specified by the tuple
$$
\tuple{T} = ((U_1,\dots,U_{k+1}), (G_1,\dots,G_{k+1})).
$$
The operators $U_1,\dots,U_{k+1}$ are the \emph{Tucker bases} 
and the functions $G_1,\dots,G_{k+1}$ are the \emph{tensor train cores}.
The integers $\vb{n}=(n_1,\dots,n_{k+1})$ are the Tucker ranks of $\tuple{T}$.  
The TT-ranks $\vb{r}$ of $(G_1,\dots,G_{k+1})$ are the TT-ranks of $\tuple{T}$.
\end{definition}

When $\paramspace_i=\mathbb{R}^{N_i}$ and $\obsspace=\mathbb{R}^{N_{k+1}}$,
$T$ corresponds to an array $\vb{T}$, the operators $U_i$ to matrices
$\vb{U}_i$, and each $G_i$ to a 3-tensor
$\vb{G}_i \in \mathbb{R}^{r_{i-1} \times n_i \times r_i}$. A plain tensor
train representation would instead have
$\vb{G}_i \in \mathbb{R}^{r_{i-1} \times N_i \times r_i}$, so Tucker tensor
trains are far cheaper to store and manipulate when $n_i \ll N_i$.

The format admits two complementary perspectives. As a generalization of the
tensor train format, it augments each TT core with a local Tucker basis
matrix mapping an internal contracted index to the full free index,
decoupling internal rank from external dimension; \cite{oseledets.2011} noted
but did not develop this modification. Conversely, it is a special case of the
hierarchical Tucker (HT) format of \cite{hackbusch2009new}, further analyzed
by \cite{grasedyck2010hierarchical} (see \cite{grasedyck2013literature} for a
survey of low-rank formats), in which the tree of cores is linear, with TT
cores at interior nodes and Tucker factor matrices at the leaves.
\cite{dolgov2013two} gave the first systematic development as a computational
object (algebraic operations and stability analysis) in the context of quantized tensor trains. The spectral tensor-train format of
\cite{bigoni2016spectral} is a prominent scientific computing application.

Any finite-dimensional tensor $\vb{T}$ has an exact Tucker tensor train
representation $\tuple{T}$ whose Tucker ranks are the ranks of the
matricizations of $\vb{T}$ and whose tensor train ranks are the ranks of the
matrix unfoldings; these are the minimal ranks of any exact representation.
One such $\tuple{T}$ is constructed by the T3-SVD algorithm
(\algoref{alg:t3_svd_dense} in \appref{app:tensor_appendix}), a generalization
of TT-SVD \citep[Algorithm 1]{oseledets.2011}. An implicit sweeping version
(\algoref{alg:t3_svd_implicit} in \appref{app:tensor_appendix}) computes
the singular values of all matricizations and matrix unfoldings of an existing Tucker tensor train efficiently; see \appref{app:t3_unfoldings} for details.

Finite-dimensional Tucker tensor trains with fixed Tucker ranks $\vb{n}$ and
tensor train ranks $\vb{r}$ form an embedded submanifold
$\mathcal{M}_{\vb{n},\vb{r}} \subset \mathbb{R}^{N_1 \times \dots \times N_{k+1}}$,
inheriting this structure from its parent frameworks (Tucker composed with TT,
or hierarchical Tucker). \cite{koch2010dynamical} established and exploited the
smooth manifold structure of the Tucker format for dynamical approximation,
and \cite{holtz.rohwedder.ea.2012} the smooth embedded manifold structure of
the TT format. \cite{uschmajew2013geometry} established the smooth quotient
manifold structure of the HT format with explicit tangent space
characterizations, and \cite{lubich2013dynamical} the tangent space projector,
curvature estimates, and geometric foundations for optimization, treated
comprehensively in \cite{uschmajew2020geometric}. The Tucker tensor train
manifold inherits these foundations from both lines of work, while its linear
tree structure keeps manifold operations at the same complexity scale as the
TT case.

Our numerical methods manipulate tangent vectors via gauged representations,
compute singular values of matrix unfoldings and matricizations (T3-SVD), and
retract tangent vectors to the manifold; \appref{sec:tensor_train_manifold}
provides efficient algorithms for these tasks.

\section{Derivative probing}
\label{sec:derivative_actions}

The training data for the learning problem consists of derivative probes of implicitly defined maps, computed by repeated application of the chain rule and implicit function theorem (\citealp[Section 3]{alger.chen.ea.2020}; \citealp{maddison2019automated}). This section describes how to compute probes of $D^\derivativeorderactionsection \pto(\paramzero)$ and, in particular, how their cost depends on the directional symmetry (repetition) of the vectors defining them. We assume that directional \emph{partial} derivatives of $\res$ and $\qoi$ are easily computable, as holds in typical applications. Probes of $D^\derivativeorderactionsection \pto(\paramzero)$ yield the probes of $D^\derivativeorderactionsection \ptoW(0)$ needed to build the T4S model, through
\begin{equation}
\label{eq:relating_Df_and_Dq}
D^\derivativeorderactionsection \ptoW(0)(x_1, \dots, x_\derivativeorderactionsection) = D^\derivativeorderactionsection \pto(\paramzero)(\sqrtcov x_1, \dots, \sqrtcov x_\derivativeorderactionsection).
\end{equation}

Every probe reduces to a sequence of linearized solves that share a single operator $A$ (the linearized state operator) or its adjoint $A^*$, organized over the lattice of multiset subsets of the derivative directions. \algoref{alg:symbolic_diff} and \algoref{alg:probing} compute all forward and reverse probes by traversal of this lattice. The number of solves is the dominant cost, and it ranges from $O(\derivativeorderactionsection)$ for directionally symmetric probes to $O(2^\derivativeorderactionsection)$ for asymmetric ones (\secref{sec:cost_vs_symmetry}). 

The key points are:
\begin{itemize}
    \item Computing probes requires solving multiple linearized versions of the state equation. These share the same linearized state operator (or its adjoint, for reverse probes) and differ only in their right-hand sides.
    \item The cost of a probe, both in terms of the number of solves and the cost of assembling their right-hand sides, depends sharply on the repetition in the directions defining it. Probes with more repetition are far cheaper.
    \item Computing a high-order probe yields every lower-order probe whose directions are a subset of its own at little additional cost.
\end{itemize}

\subsection{Derivative notation}
\label{sec:derivative_notation}

The operator $D$ denotes the total derivative with respect to $\param$, and $\partial_w$ the partial derivative with respect to $w$. Functions of $\param$ are evaluated at $\paramzero$ and functions of $\state$ at $u=\state(\paramzero)$, the solution of $\res(\paramzero,\state(\paramzero))=0$, unless stated otherwise; we suppress these arguments where they would clutter notation, writing e.g.\ $D\pto = D\pto(\paramzero)$. Derivative directions are indicated in compressed polynomial notation, as follows.

\paragraph{Forward probing notation} Let $\parammultiset=\{\widehat{\param}_i\}_{i=1}^l$ be a collection of directions $\widehat{\param}_i \in \paramspace$, and let $\alpha$ be a multiset with entries in $\{1,\dots,l\}$ and cardinality $|\alpha|=\sum_{i=1}^l \alpha(i)=\derivativeorderactionsection$. The forward probe of $D^\derivativeorderactionsection \pto$ by directions $\widehat{\param}_i$ with multiplicities $\alpha(i)$ is
\begin{equation}
y = D^\derivativeorderactionsection \pto \widehat{\param}_1^{\,\alpha(1)} \cdots \widehat{\param}_l^{\,\alpha(l)} = D^\derivativeorderactionsection \pto \parammultiset^\alpha,
\end{equation}
e.g.\ $D^4 \pto \parammultiset^{\{1,1,2,3\}} = D^4 \pto \widehat{\param}_1^{\,2} \widehat{\param}_2 \widehat{\param}_3 = D^4 \pto(\paramzero)(\widehat{\param}_1, \widehat{\param}_1, \widehat{\param}_2, \widehat{\param}_3)$.

\paragraph{Reverse probing notation} The reverse probe of $D^\derivativeorderactionsection \pto$ inserts vectors into all but one of its arguments and pairs the result with a linear functional $\omega$; by multilinear symmetry which argument is open is irrelevant, so we speak of \emph{the} reverse probe. For a multiset $\alpha$ with $|\alpha|=\derivativeorderactionsection-1$, the reverse probe is the linear functional $\psi$ with $\psi(\nu) = \omega(D^{\derivativeorderactionsection} \pto \parammultiset^\alpha \nu)$, written
\begin{equation}
\psi = \omega(D^{\derivativeorderactionsection} \pto \parammultiset^\alpha),\label{eq:reverse_action_shorthand}
\end{equation}
with the missing direction reserved for the argument of $\psi$.

\paragraph{Partial derivative notation} Partial derivatives use the same convention. For a variable $Z$ depending on $w$, directions $\widehat{W}=\{\widehat{w}_\beta\}_{\beta \in B}$, and a multiset $\Gamma$ with entries in $B$, the $|\Gamma|$th partial directional derivative is written $\partial_w^{|\Gamma|} Z \widehat{W}^{\Gamma}$; e.g.\ $\partial_w^3 Z \widehat{W}^{\{\beta, \beta', \beta'\}} = \partial_w^3 Z(\widehat{w}_{\beta}, \widehat{w}_{\beta'}, \widehat{w}_{\beta'})$.

\subsection{Forward probing}
\label{sec:forward_actions}

\begin{figure}
    \centering
    \begin{subfigure}[b]{0.98\textwidth}
        \centering
        \def\symspace{4.75}
\def\symwidth{1.125}
\def\symheight{1.0625}

\begin{tikzpicture}
    \tikzstyle{every node}=[font=\footnotesize]
    \tikzset{incarr/.style={-{Stealth[length=4]}}}
    
    \node (a) at (\symspace-1.5,0) {};
    
    \node (111a) at ([shift=({0,3*\symheight})]a) {$\lbrace 1, 1, 1 \rbrace$};
    \node (11a) at ([shift=({0,2*\symheight})]a) {$\lbrace 1, 1 \rbrace$};
    \node (1a) at ([shift=({0,\symheight})]a) {$\lbrace 1 \rbrace$};
    \node (0a) at ([shift=({0,0})]a) {$\emptyset$};

    \node (Oa) at ([shift=({0,-0.75*\symheight})]a) {$\mathcal{O}(k)$};
    
    \node at ([shift=({0,3.75*\symheight})]a) {$D^3 u \, \widehat{\param}_1^{\,3}$};

    \draw[incarr] (0a) -- (1a);
    \draw[incarr] (1a) -- (11a);
    \draw[incarr] (11a) -- (111a);
    
    \node (b) at (0,0) {};
    
    \node (112b) at ([shift=({-0.5*\symwidth,3*\symheight})]b) {$\lbrace 1, 1, 2 \rbrace$};
    
    \node (12b) at ([shift=({0.5*\symwidth,2*\symheight})]b) {$\lbrace 1, 2 \rbrace$};
    \node (11b) at ([shift=({-1.5*\symwidth,2*\symheight})]b) {$\lbrace 1, 1 \rbrace$};

    \node (2b) at ([shift=({1.5*\symwidth,\symheight})]b) {$\lbrace 2 \rbrace$};
    \node (1b) at ([shift=({-0.5*\symwidth,\symheight})]b) {$\lbrace 1 \rbrace$};
    
    \node (0b) at ([shift=({0.5*\symwidth,0})]b) {$\emptyset$};

    \node at ([shift=({-0.5*\symwidth,3.75*\symheight})]b) {$D^3 u \, \widehat{\param}_1^{\,2} \, \widehat{\param}_2$};

    \draw[incarr] (0b) -- (1b);
    \draw[incarr] (0b) -- (2b);
    \draw[incarr] (1b) -- (11b);
    \draw[incarr] (1b) -- (12b);
    \draw[incarr] (2b) -- (12b);
    \draw[incarr] (11b) -- (112b);
    \draw[incarr] (12b) -- (112b);

    \node (c) at (-\symspace, 0) {};
    
    \node (123c) at ([shift=({0,3*\symheight})]c) {$\lbrace 1, 2, 3 \rbrace$};
    
    \node (12c) at ([shift=({-1.5*\symwidth,2*\symheight})]c) {$\lbrace 1, 2 \rbrace$};
    \node (13c) at ([shift=({0,2*\symheight})]c) {$\lbrace 1, 3 \rbrace$};
    \node (23c) at ([shift=({1.5*\symwidth,2*\symheight})]c) {$\lbrace 2, 3 \rbrace$};
    
    \node (1c) at ([shift=({-1.5*\symwidth,\symheight})]c) {$\lbrace 1 \rbrace$};
    \node (2c) at ([shift=({0,\symheight})]c) {$\lbrace 2 \rbrace$};
    \node (3c) at ([shift=({1.5*\symwidth,\symheight})]c) {$\lbrace 3 \rbrace$};
    
    \node (0c) at ([shift=({0,0})]c) {$\emptyset$};

    \node (Oc) at ([shift=({0,-0.75*\symheight})]c) {$\mathcal{O}(2^k)$};

    \node at ([shift=({0,3.75*\symheight})]c) {$D^3 u \, \widehat{\param}_1 \, \widehat{\param}_2 \, \widehat{\param}_3$};

    \draw[incarr] (0c) -> (1c);
    \draw[incarr] (0c) -> (2c);
    \draw[incarr] (0c) -> (3c);
    \draw[incarr] (1c) -> (12c);
    \draw[incarr] (1c) -> (13c);		
    \draw[incarr] (2c) -> (12c);
    \draw[incarr] (2c) -> (23c);
    \draw[incarr] (3c) -> (13c);
    \draw[incarr] (3c) -> (23c);
    \draw[incarr] (12c) -> (123c);
    \draw[incarr] (13c) -> (123c);
    \draw[incarr] (23c) -> (123c);

    \node (0) at (-\symspace-2.75,0) {};
    \node[left] (00) at ([shift=({0,0})]0) {$l\!=\!0$};
    \node[left] (1) at ([shift=({0,\symheight})]0) {$l\!=\!1$};
    \node[left] (2) at ([shift=({0,2*\symheight})]0) {$l\!=\!2$};
    \node[left] (3) at ([shift=({0,3*\symheight})]0) {$l\!=\!3$};
    
    \draw[-Latex] ([shift=({-0.1,0})]00.south west) -- ([shift=({-0.1,0})]3.north west)
        node[midway, above, align=center, rotate=90] {Incremental Order};
    
    \draw[-Latex] (Oc) -- (Oa)
        node[above, midway, align=center] (Ob) {Number of linear systems};
\end{tikzpicture}
    \end{subfigure}
    \vspace{-1.0em}
    \caption{Incremental lattices for third-order forward derivative probes with different directional symmetries. Each node corresponds to one incremental state solve; repeated directions collapse the subset lattice into a smaller multiset lattice.}
    \label{fig:directional_symmetry_derivative_lattice}
\end{figure}

Forward probes admit an expansion as a sum of directional partial derivatives of $\qoi$. Writing $\widehat{\state}_\beta := D^{|\beta|} \state \parammultiset^\beta$ for the \emph{incremental state variables} and $\widehat{U}_\alpha := \{\widehat{\state}_\beta: \beta \subset \alpha\}$ for the collection of those indexed by multiset subsets of $\alpha$, we have
\begin{equation}
D^\derivativeorderactionsection \pto \parammultiset^\alpha = \sum_{i} c_i \,\partial_\param^{|\mu_i|} \partial_\state^{|\Gamma_i|} Q \,\parammultiset^{\mu_i} \widehat{U}_\alpha^{\Gamma_i}, \label{eq:Q_partial_sum3333}
\end{equation}
where $c_i \in \mathbb{N}$ counts repeated terms, $\mu_i$ is a multiset with entries in $\{1,\dots,l\}$, and $\Gamma_i$ is a multiset of multiset subsets $\beta \subset \alpha$.

This expansion follows by repeatedly applying the chain rule for total derivatives to $\pto(\param)=\qoi(\param, \state(\param))$. Each application introduces incremental state variables $\widehat{\state}_\beta$, which quantify the sensitivity of $\state$ to perturbations of $\param$; by strong induction, $D^\derivativeorderactionsection \pto \parammultiset^\alpha$ depends on $\widehat{\state}_\beta$ for all multiset subsets $\beta\subset\alpha$. The numbers $c_i$, the multisets $\mu_i$, and the multisets of multisets $\Gamma_i$ are produced by the symbolic procedure of \secref{sec:symbolic_differentiation}. For example,
\begin{equation*}
D^2 \pto \widehat{\param}_1^2 =\partial^2_{\param \param} \qoi \widehat{\param}_1^2 + 2\partial^2_{\param \state} \qoi \widehat{\param}_1 \widehat{\state}_{\{1\}} + \partial^2_{\state \state} \qoi \widehat{\state}_{\{1\}}^2 + \partial_{\state} \qoi \widehat{\state}_{\{1,1\}},
\end{equation*}
in which the second term has $c_i=2$, $\mu_i=\{1\}$, and $\Gamma_i=\{\{1\}\}$.

Each incremental state variable solves the \emph{incremental state equation}
\begin{equation}
\label{eq:incr_state_3240198234712}
0 = D^{|\beta|} \res \parammultiset^\beta.
\end{equation}
Applying the chain rule to $\res$ as above and isolating the term containing $\widehat{\state}_\beta$ recasts \eqref{eq:incr_state_3240198234712} as the linear system
\begin{equation}
\label{eq:incr_state_Au_b}
    A \widehat{\state}_\beta = b_\beta, \qquad A = \partial_\state \res,
\end{equation}
where $A$ is the linearized state operator at the expansion point, the same for every incremental solve; only the right-hand side $b_\beta$ changes. The right-hand side $b_\beta$ is a sum of partial derivatives of $\res$ depending on lower-order incremental state variables $\widehat{\state}_\gamma$ for proper subsets $\gamma \subsetneq \beta$.

Computing $D^\derivativeorderactionsection \pto \parammultiset^\alpha$ therefore traverses the lattice of multiset subsets of $\alpha$ (\figref{fig:directional_symmetry_derivative_lattice}), solving \eqref{eq:incr_state_Au_b} for $\widehat{\state}_\beta$ at each node $\beta\subset\alpha$ and substituting the results into \eqref{eq:Q_partial_sum3333}.

\subsection{Reverse probing}
\label{sec:reverse_actions}

For $\derivativeorderactionsection=1$ the reverse probe $\psi$ in \eqref{eq:reverse_action_shorthand} is the derivative $g = D\omega(\pto) = \omega(D \pto)$ of the scalar map $\param \mapsto \omega(\pto(\param))$; after identifying $\paramspace'$ with $\paramspace$ through the Riesz map, $g$ is the gradient of this scalar map. By the adjoint method \citep{cea1986conception,gunzburger2002perspectives,plessix2006review,troltzsch2010optimal},
\begin{align}
g &= \omega(\partial_\param \qoi) + (\partial_\param \res)(\adj), \label{eq:gradient_partials} \\
0 &= \res, \label{eq:state_eq_partials} \\
0 &= \adjres, \label{eq:adjoint_eq_partials}
\end{align}
where the \emph{adjoint residual}
\begin{equation}
\label{eq:adjoint_residual_def}
\adjres := \omega(\partial_\state \qoi) + (\partial_\state \res)(\adj)
\end{equation}
is the state-derivative of the Lagrangian $\mathcal{L} = \omega(\qoi) + \res(\adj)$, paralleling the state residual $\res = \partial_\adj \mathcal{L}$ and the gradient $g = \partial_\param \mathcal{L}$. We write linear functionals to the left of their vector arguments throughout, so the residual acts on the adjoint as $\res(\adj)= \res(\theta_0, u(\theta_0))(\adj)$.
The adjoint equation \eqref{eq:adjoint_eq_partials} is the linear system $A^* \adj = c$ with $A^*=(\partial_\state \res)^*$ and $c=-\omega(\partial_\state \qoi)$; computing $g$ solves \eqref{eq:state_eq_partials} for $\state$, then $A^*\adj=c$ for $\adj$, then evaluates \eqref{eq:gradient_partials}.

For $\derivativeorderactionsection > 1$, the reverse probe is a derivative of the gradient, $\psi = \omega(D^{\derivativeorderactionsection} \pto \parammultiset^\alpha) = D^{\derivativeorderactionsection-1} g \parammultiset^\alpha$, computed by the forward procedure of \secref{sec:forward_actions} applied to $g$ in place of $\qoi$, with $\state$ replaced by the pair $(\state,\adj)$ and the state equation replaced by the combined state and adjoint system. Repeatedly differentiating $g$ expresses $\psi_\beta = D^{|\beta|} g \parammultiset^\beta$ as a sum of partial derivatives of $\qoi$ and $\res$, now also depending on the \emph{incremental adjoint variables}
\begin{equation*}
    \widehat{\adj}_\beta := D^{|\beta|} \adj \parammultiset^\beta, \qquad
    \widehat{V}_\alpha := \{\widehat{\adj}_\beta : \beta \subset \alpha\},
\end{equation*}
the reverse-mode counterparts of $\widehat{\state}_\beta$ and $\widehat{U}_\alpha$, with $\widehat{\adj}_\emptyset = \adj$. Each term pairs a linear functional with a vector---either $\omega$ acting on a derivative of $\qoi$, or a derivative of $\res$ acting on an incremental adjoint---so that
\begin{equation}
\psi_\beta = D^{|\beta|} g \parammultiset^\beta
= \sum_i c_i\, \omega\!\left(\partial_\param^{|\mu_i|+1}\partial_\state^{|\Gamma_i|} \qoi\, \parammultiset^{\mu_i}\widehat{U}_\beta^{\Gamma_i}\right)
+ \sum_i \tilde{c}_i\left(\partial_\param^{|\tilde{\mu}_i|+1}\partial_\state^{|\tilde{\Gamma}_i|} \res\, \parammultiset^{\tilde{\mu}_i}\widehat{U}_\beta^{\tilde{\Gamma}_i}\right)\!(\widehat{\adj}_{\delta_i}),
\label{eq:reverse_probe_expansion}
\end{equation}
where $\delta_i \subset \beta$ and one $\param$ argument of each $\qoi$ and $\res$ partial is left open for the argument of $\psi_\beta$. Equation \eqref{eq:reverse_probe_expansion} is the reverse-mode analog of the forward expansion \eqref{eq:Q_partial_sum3333}, with $g$ in place of $\pto$ and with the outer pairings $\omega$ and $\widehat{\adj}_{\delta_i}$.

Each incremental adjoint variable solves the \emph{incremental adjoint equation}
\begin{equation}
\label{eq:incr_adj_residual}
0 = D^{|\beta|} \adjres \parammultiset^\beta.
\end{equation}
Isolating the term $(\partial_\state \res)(\widehat{\adj}_\beta) = A^* \widehat{\adj}_\beta$ recasts \eqref{eq:incr_adj_residual} as
\begin{equation}
\label{eq:incr_adj_ATu_c}
A^* \widehat{\adj}_\beta = c_\beta,
\end{equation}
exactly as \eqref{eq:incr_state_3240198234712} is recast as \eqref{eq:incr_state_Au_b}. The right-hand side $c_\beta$ depends on $\widehat{\state}_\gamma$ for all $\gamma \subset \beta$ and on $\widehat{\adj}_\gamma$ for proper subsets $\gamma \subsetneq \beta$; its base case $c_\emptyset = -\omega(\partial_\state \qoi)$ recovers the adjoint equation \eqref{eq:adjoint_eq_partials}. Reverse probing thus mirrors forward probing, solving the two incremental systems \eqref{eq:incr_state_Au_b} and \eqref{eq:incr_adj_ATu_c} at each node of the same lattice.

\subsection{Symbolic automatic differentiation} 
\label{sec:symbolic_differentiation}

As the derivative order increases, the formulas for $D^{|\beta|}\pto\parammultiset^\beta$, $D^{|\beta|}g\parammultiset^\beta$, $D^{|\beta|}\res\parammultiset^\beta$, and $D^{|\beta|}\adjres\parammultiset^\beta$ in terms of partial derivatives of $\qoi$ and $\res$ become lengthy. Rather than deriving and transcribing them by hand, we generate them by symbolic differentiation. The symbolic differentiation procedure generates only the algebraic formulas for the incremental right-hand sides; the directional partial derivatives of $\res$ and $\qoi$ themselves are then assembled by algorithmic differentiation or by problem-specific variational forms. Each term of \eqref{eq:Q_partial_sum3333} has the form $t = \partial_\param^{|\mu|} \partial_\state^{|\Gamma|} Q \parammultiset^\mu \widehat{U}_\alpha^\Gamma$, and by the product rule
\begin{align}
Dt\widehat{\param}_i &=  \partial_\param^{|\mu|+1} \partial_\state^{|\Gamma|} Q \parammultiset^{\mu + \{i\}} \widehat{U}_\alpha^{\Gamma}
+ \partial_\param^{|\mu|} \partial_\state^{|\Gamma|+1} Q \parammultiset^\mu \widehat{U}_\alpha^{\Gamma + \{\{i\}\}} \nonumber\\
&+ \sum_{\gamma \in \Gamma} \partial_\param^{|\mu|} \partial_\state^{|\Gamma|} Q \parammultiset^\mu \widehat{U}_\alpha^{\Gamma + \{\gamma + \{i\}\} - \{\gamma\}}, \label{eq:Dt_symbolic}
\end{align}
a sum of terms of the same form.

The reverse-mode objects $g$ and $\adjres$ are likewise sums of terms, but each pairs a linear functional with a vector: either the output functional $\omega$ acting on a derivative of $\qoi$, or a derivative of $\res$ acting on an incremental adjoint $\widehat{\adj}_\delta$. Differentiating an $\omega$-paired term is done via \eqref{eq:Dt_symbolic}, since $\omega$ does not depend on $\param$. Differentiating a $\widehat{\adj}_\delta$-paired term yields
\begin{equation}
D\big(t(\widehat{\adj}_\delta)\big)\widehat{\param}_i = t(\widehat{\adj}_{\delta+\{i\}}) + \big(Dt\,\widehat{\param}_i\big)(\widehat{\adj}_\delta),
\label{eq:Dt_symbolic_outer}
\end{equation}
whose first term raises the order of the incremental adjoint, $\widehat{\adj}_{\delta+\{i\}} = D^{|\delta|+1}\adj\parammultiset^{\delta+\{i\}}$, and whose second differentiates the inner derivative of $\res$ via \eqref{eq:Dt_symbolic}.

Symbolically, a term is represented by a tuple
\begin{equation*}
\partial_\param^{|\mu|} \partial_\state^{|\Gamma|} \tau\, \parammultiset^\mu \widehat{U}_\alpha^\Gamma \ \ \text{paired with}\ \ \rho \quad\longleftrightarrow\quad (\rho,\tau,\mu,\Gamma),
\end{equation*}
where $\tau \in \{`\qoi\textrm{'}, `\res\textrm{'}\}$ tags the differentiated function and $\rho$ records the outer pairing. For forward probes and residual derivatives there is none, $\rho = \mathrm{id}$, and the tuple reduces to $(\tau,\mu,\Gamma)$; for reverse objects $\rho$ is either $\omega$ (with $\tau=`\qoi\textrm{'}$) or an incremental adjoint $\widehat{\adj}_\delta$ (with $\tau=`\res\textrm{'}$). The entire sum is stored as a dictionary (hash map) mapping tuples to their integer coefficients $c$, and is differentiated term by term using \eqref{eq:Dt_symbolic} together with \eqref{eq:Dt_symbolic_outer}.

\algoref{alg:symbolic_diff} formalizes this procedure: given the order-zero representation $\mathcal{D}_\emptyset$ of a function $F$ and a multiset $\alpha$, it traverses the subset lattice and returns the entire family $\{\mathcal{D}_\beta\}_{\beta\subset\alpha}$, obtaining each node by differentiating a parent $\beta-\{i\}$ in one further direction $i$. The family is required because the probing algorithms need $D^{|\beta|} F \parammultiset^\beta$ at every node, not only at the top order $\alpha$. The four objects in derivative probing arise as the same algorithm applied to different initial representations: the forward probe $\pto$ from $\mathcal{D}_\emptyset = \{(\mathrm{id},`\qoi\textrm{'},\emptyset,\emptyset)\mapsto 1\}$ and the residual $\res$ from $\{(\mathrm{id},`\res\textrm{'},\emptyset,\emptyset)\mapsto 1\}$; the gradient $g$ and the adjoint residual $\adjres$ both from $\{(\omega,`\qoi\textrm{'},\emptyset,\emptyset)\mapsto 1,\; (\adj,`\res\textrm{'},\emptyset,\emptyset)\mapsto 1\}$, differing only in the reserved open argument ($\param$ for $g$, $\state$ for $\adjres$). The representations of $\res$ and $\adjres$ supply the incremental right-hand sides $b_\beta$ and $c_\beta$. Once the symbolic representations are determined, the terms themselves are constructed by assembling the corresponding directional partial derivatives.

\begin{algorithm}[t]
\caption{Symbolic differentiation: representations of $D^{|\beta|} F \parammultiset^\beta$ for all $\beta\subset\alpha$.}
\label{alg:symbolic_diff}
\begin{algorithmic}[1]
\Require Order-zero representation $\mathcal{D}_\emptyset$ of $F$: a dictionary mapping symbolic terms $(\rho,\tau,\mu,\Gamma)$ to coefficients $c$; multiset $\alpha$ of derivative directions.
\Ensure Dictionaries $\{\mathcal{D}_\beta\}_{\beta\subset\alpha}$, where $\mathcal{D}_\beta$ represents $D^{|\beta|} F \parammultiset^\beta$.
\For{each nonempty $\beta \subset \alpha$, by nondecreasing $|\beta|$} \Comment{traverse the lattice of subsets}
    \State choose $i \in \beta$ and set $\beta' \gets \beta - \{i\}$ \Comment{parent $\beta'\subsetneq\beta$, already computed}
    \State $\mathcal{D}_\beta \gets \{\,\}$ \Comment{empty dictionary; absent keys default to $0$}
    \For{each term $(\rho,\tau,\mu,\Gamma)$ in $\mathcal{D}_{\beta'}$ with coefficient $c$} \Comment{differentiate in direction $i$, \eqref{eq:Dt_symbolic}}
        \State $\mathcal{D}_\beta[(\rho,\,\tau,\,\mu+\{i\},\,\Gamma)] \mathrel{{+}{=}} c$ \Comment{$\param$-derivative}
        \State $\mathcal{D}_\beta[(\rho,\,\tau,\,\mu,\,\Gamma+\{\{i\}\})] \mathrel{{+}{=}} c$ \Comment{introduce $\widehat{\state}_{\{i\}}$}
        \For{each $\gamma \in \Gamma$}
            \State $\mathcal{D}_\beta[(\rho,\,\tau,\,\mu,\,\Gamma+\{\gamma+\{i\}\}-\{\gamma\})] \mathrel{{+}{=}} c$ \Comment{raise the order of $\widehat{\state}_\gamma$}
        \EndFor
        \If{$\rho = \widehat{\adj}_\delta$ is an incremental adjoint} \Comment{differentiate the outer pairing, \eqref{eq:Dt_symbolic_outer}}
            \State $\mathcal{D}_\beta[(\widehat{\adj}_{\delta+\{i\}},\,\tau,\,\mu,\,\Gamma)] \mathrel{{+}{=}} c$ \Comment{raise the order of $\widehat{\adj}_\delta$}
        \EndIf
    \EndFor
\EndFor
\State \textbf{return} $\{\mathcal{D}_\beta\}_{\beta\subset\alpha}$
\end{algorithmic}
\end{algorithm}

\subsection{The probing algorithm}
\label{sec:probing_algorithm}

With the symbolic representations of $D^{|\beta|}\pto\parammultiset^\beta$, $D^{|\beta|}g\parammultiset^\beta$, $D^{|\beta|}\res\parammultiset^\beta$, and $D^{|\beta|}\adjres\parammultiset^\beta$ in hand (\secref{sec:symbolic_differentiation}), \algoref{alg:probing} computes all forward and reverse probes by a single traversal of the subset lattice of $\alpha$. The traversal solves for the incremental state variables $\widehat{\state}_\beta$ and incremental adjoint variables $\widehat{\adj}_\beta$ at every node $\beta\subset\alpha$, sharing the state solve and the operator $A$. At node $\beta$ the forward probe $D^{|\beta|}\pto\parammultiset^\beta$ is assembled from $\{\widehat{\state}_\gamma : \gamma\subset\beta\}$ via \eqref{eq:Q_partial_sum3333}, and the reverse probe $\omega(D^{|\beta|+1}\pto\parammultiset^\beta)$ from $\{\widehat{\state}_\gamma : \gamma\subset\beta\}$ and $\{\widehat{\adj}_\gamma : \gamma\subset\beta\}$ via \eqref{eq:reverse_probe_expansion}. Because the reverse probe leaves one argument open, at each node it probes the derivative tensor one order higher than the forward probe; at the top node $\beta=\alpha$ this yields $D^{\derivativeorderactionsection}\pto\parammultiset^\alpha$ and $\omega(D^{\derivativeorderactionsection+1}\pto\parammultiset^\alpha)$. Omitting the adjoint steps recovers pure forward probing.

\begin{algorithm}[t]
\caption{Derivative probing: all forward and reverse probes on the lattice of subsets of $\alpha$.}
\label{alg:probing}
\begin{algorithmic}[1]
\Require Multiset $\alpha$ of $\derivativeorderactionsection$ derivative directions from $\parammultiset=\{\widehat{\param}_i\}_{i=1}^l$; output functional $\omega$; routines for directional partial derivatives of $\res$ and $\qoi$.
\Ensure Forward probes $\{y_\beta = D^{|\beta|}\pto\parammultiset^\beta : \beta\subset\alpha\}$; reverse probes $\{\psi_\beta = \omega(D^{|\beta|+1}\pto\parammultiset^\beta) : \beta\subset\alpha\}$.
\State Use \algoref{alg:symbolic_diff} to form symbolic representations of $D^{|\beta|}\pto\parammultiset^\beta$, $D^{|\beta|}g\parammultiset^\beta$, $D^{|\beta|}\res\parammultiset^\beta$, $D^{|\beta|}\adjres\parammultiset^\beta$ for all $\beta\subset\alpha$
\State Extract $b_\beta$ from $D^{|\beta|}\res\parammultiset^\beta$ and $c_\beta$ from $D^{|\beta|}\adjres\parammultiset^\beta$ ($\beta\subset\alpha$) \Comment{recast as \eqref{eq:incr_state_Au_b}, \eqref{eq:incr_adj_ATu_c}}
\State Solve the state equation $0=\res$ for $\state$ \Comment{base node $\beta=\emptyset$}
\State Assemble $A=\partial_\state \res$ and its adjoint $A^*$ \Comment{shared by all incremental solves}
\State Solve the incremental adjoint equation $A^*\widehat{\adj}_\emptyset = c_\emptyset$ for $\widehat{\adj}_\emptyset = \adj$ \Comment{base node $\beta=\emptyset$; \eqref{eq:incr_adj_ATu_c}}
\For{each nonempty $\beta \subset \alpha$, by nondecreasing $|\beta|$} \Comment{single traversal of the subset lattice}
    \State Assemble $b_\beta$; solve $A\,\widehat{\state}_\beta = b_\beta$ for $\widehat{\state}_\beta$ \Comment{incremental state \eqref{eq:incr_state_Au_b}}
    \State Assemble $c_\beta$; solve $A^*\widehat{\adj}_\beta = c_\beta$ for $\widehat{\adj}_\beta$ \Comment{incremental adjoint \eqref{eq:incr_adj_ATu_c}}
\EndFor
\For{each $\beta \subset \alpha$} \Comment{assemble every probe in the lattice}
    \State Assemble $y_\beta$ from $\{\widehat{\state}_\gamma : \gamma\subset\beta\}$ \Comment{forward probe, order $|\beta|$; \eqref{eq:Q_partial_sum3333}}
    \State Assemble $\psi_\beta$ from $\{\widehat{\state}_\gamma : \gamma\subset\beta\}$ and $\{\widehat{\adj}_\gamma : \gamma\subset\beta\}$ \Comment{reverse probe, order $|\beta|+1$; \eqref{eq:reverse_probe_expansion}}
\EndFor
\State \textbf{return} $\{y_\beta\}_{\beta\subset\alpha},\ \{\psi_\beta\}_{\beta\subset\alpha}$
\end{algorithmic}
\end{algorithm}

\subsection{Cost vs.\ directional symmetry}
\label{sec:cost_vs_symmetry}

\begin{table}
    \centering
    \begin{subtable}[t]{0.12\textwidth}
        \centering
        \begin{tabular}{c}
             \\
             Order \\ \hline
             1 \\
             2 \\
             3 \\
             4 \\
             5 \\
             6 \\
             7 \\
             8 \\
             9 \\
             10
        \end{tabular}
    \end{subtable}%
    \begin{subtable}[t]{0.41\textwidth}
        \centering
        \begin{tabular}{rr|rr}
            \multicolumn{2}{c|}{Forward}
                & \multicolumn{2}{|c}{Reverse} \\
            \multicolumn{1}{c}{asym.}
                & \multicolumn{1}{c|}{sym.}
                & \multicolumn{1}{|c}{asym.}
                & \multicolumn{1}{c}{sym.} \\
            \hline
            1    & 1    & 1    & 1 \\
            3    & 2    & 3    & 3 \\
            7    & 3    & 7    & 5 \\
            15   & 4    & 15   & 7 \\
            31   & 5    & 31   & 9 \\
            63   & 6    & 63   & 11 \\
            127  & 7    & 127  & 13 \\
            255  & 8    & 255  & 15 \\
            511  & 9    & 511  & 17 \\
            1023 & 10   & 1023 & 19
        \end{tabular}
        \caption{Number of incremental solves}
    \end{subtable}%
    \begin{subtable}[t]{0.46\textwidth}
        \centering
        \begin{tabular}{rr|rr}
            \multicolumn{2}{c|}{Forward}
                & \multicolumn{2}{|c}{Reverse} \\
            \multicolumn{1}{c}{asym.}
                & \multicolumn{1}{c|}{sym.}
                & \multicolumn{1}{|c}{asym.}
                & \multicolumn{1}{c}{sym.} \\
            \hline
            2       & 2     & 2       & 2 \\
            5       & 4     & 5       & 5 \\
            15      & 7     & 15      & 7 \\
            52      & 12    & 52      & 21 \\
            203     & 19    & 203     & 38 \\
            877     & 30    & 877     & 64 \\
            4,140   & 45    & 4,140   & 105 \\
            21,147  & 67    & 21,147  & 165 \\
            115,975 & 97    & 115,975 & 254 \\
            678,570 & 139   & 678,570 & 381
        \end{tabular}
        \caption{Number of right-hand side terms}
    \end{subtable}
    
    \caption{Comparison of the number of incremental solves and number of terms in the right-hand side (RHS) for directionally asymmetric (no repetition) and directionally symmetric (complete repetition) forward and reverse derivative probes.}
    \label{tab:der_act_dir_sym}
\end{table}

The number of solves to compute a probe with directions $\parammultiset^\alpha$ equals the size of the lattice of multiset subsets of $\alpha$,
\begin{equation}
\label{eq:multiset_lattice_size}
    \vert \lbrace \beta \subset \alpha \rbrace \vert 
        = \prod_{i=1}^l (\alpha_i + 1),
\end{equation}
which ranges from $|\alpha|+1$ when all directions coincide to $2^{|\alpha|}$ when all differ. Assembling the right-hand sides $b_\beta$ and $c_\beta$, which are sums of directional partial derivatives of $\res$ and $\qoi$, adds further cost. The number of terms in these right hand sides likewise grows as repetition decreases. \tableref{tab:der_act_dir_sym} reports both counts for derivative orders $\derivativeorderactionsection=1,\dots,10$ in two extremes: \emph{directionally symmetric probes}, all directions equal (e.g.\ $\parammultiset^\alpha=\widehat{\param}_1^{\,3}$), and \emph{directionally asymmetric probes}, all distinct (e.g.\ $\parammultiset^\alpha=\widehat{\param}_1 \widehat{\param}_2 \widehat{\param}_3$). Thus for $m$ probes of derivatives up to order $\derivativeorderactionsection$, directionally symmetric probing requires $O(m\derivativeorderactionsection)$ incremental state and adjoint solves, whereas fully asymmetric probing requires $O(m\,2^\derivativeorderactionsection)$, before accounting for the larger right-hand-side assembly cost. \cite{wang2017high} studies the effect of directional symmetry on cost in a non-implicit automatic differentiation setting. Because cost depends so strongly on directional symmetry, we construct the Tucker tensor train Taylor series model using only directionally symmetric probes of $D^\derivativeorderactionsection \pto$.

Algebraic polarization and multilinear symmetry of derivatives imply that the directionally symmetric probes of $D^\derivativeorderactionsection \pto$ completely determine $D^\derivativeorderactionsection \pto$. This exact recovery uses the full (infinite) family of symmetric probes; in practice we fit the model from finitely many random symmetric probes, so constructing the T4S model is a learning problem rather than exact recovery.

While high-order probes are expensive, the same computation yields every lower-order probe almost for free. Computing the top-order probe with directions $\parammultiset^\alpha$ already requires solving the incremental state and adjoint equations at every node $\beta\subset\alpha$, since the top node depends on all lower-order incremental variables $\widehat{\state}_\beta$ and $\widehat{\adj}_\beta$. These are precisely the variables that assemble the lower-order probes $D^{|\beta|}\pto\parammultiset^\beta$ and $\psi_\beta$ for $\beta\subsetneq\alpha$, so the entire family returned by \algoref{alg:probing} costs no incremental solves beyond those already paid for the top-order probe, only their comparatively cheap assembly.


\section{Fitting the model}
\label{sec:method}

This section constructs the Tucker tensor trains $\tuple{T}_j$ that approximate the derivatives $D^j f(0)$ in the T4S model. The construction has four stages. We first reduce the infinite-dimensional (or very high-dimensional) Hilbert spaces $X$ and $Y$ to moderate Euclidean spaces $\mathbb{R}^N$ and $\mathbb{R}^M$ by a derivative-based sketching procedure that uses directionally symmetric probes to build orthonormal bases $U$ and $V$ for reduced subspaces of $X$ and $Y$ (\secref{sec:initial_dimension_reduction}). We then generate training data as directionally symmetric forward and reverse probes of the reduced derivatives (\secref{sec:symmetric_action_data}), and fit each $\tuple{T}_j$ at fixed rank by minimizing a probe-matching least-squares loss over the manifold of fixed-rank Tucker tensor trains (\secref{sec:fixed_rank_optimization_problem}). Finally, we grow the ranks by continuation, re-solving the fixed-rank problem with a warm start at each step, and select the final model by validation (\secref{sec:rank_adaptivity}). Throughout this section and the rest of the paper, $N$ and $M$ denote the reduced input and output dimensions, in contrast to the full discretization sizes of \secref{sec:t4s_model}. We fit the tensors to the reduced derivatives $D^j \widetilde{f}(0)$ and recover Tucker tensor trains on the original spaces $X$ and $Y$ afterward if desired (\secref{sec:initial_dimension_reduction}).

We present two methods for the fixed-rank problem: a trust-region Riemannian Gauss-Newton method (TR-RMGN, \secref{sec:trrmgn}) and a stochastic manifold gradient descent method with Cauchy step sizes (MC-SGD, \secref{sec:cauchy_sgd}). TR-RMGN inherits the convergence guarantees of trust-region Riemannian Newton-type methods, the gold standard among classical methods for problems of this form; MC-SGD lacks those guarantees but performs well and is much cheaper per iteration. Rank continuation wraps either method in the outer loop
$$
\text{fit } \tuple{T}_j \rightarrow \text{increase ranks} \rightarrow \text{fit } \tuple{T}_j \rightarrow \text{increase ranks} \rightarrow \dots,
$$
increasing ranks so as to keep edge condition numbers balanced across the tensor network (\secref{sec:rank_adaptivity}). Finally, \secref{sec:cost_comparison} compares the training cost of T4S with other surrogates; it is not a step of the construction.

\subsection{Derivative-based dimension reduction}
\label{sec:initial_dimension_reduction}

We build shared orthonormal bases $U : \mathbb{R}^N \to X$ and $V : \mathbb{R}^M \to Y$ for reduced subspaces of the input and output spaces, using a symmetric tensor version of randomized sketching driven by cheap directionally symmetric derivative probes. The output basis $V$ sketches the column space of each derivative tensor, viewed as a tall matrix whose rows are indexed by the grouped input indices and whose columns are indexed by the output index. The sketch is taken with a row-wise directionally symmetric Khatri-Rao random matrix, each of whose rows has the Kronecker structure $\theta \otimes \dots \otimes \theta$ for a random $\theta \in X$; this large matrix is never formed, because applying it reduces to a directionally symmetric forward probe of the derivative tensor for each row. The input basis $U$ is built analogously, viewing the derivative tensor (with its output projected onto $\operatorname{range}(V)$) as a wide matrix and sketching with reverse probes. Derivative tensors of all orders are sketched with the same random vectors, which both yields a single shared basis per space and lets incremental state solutions be reused across orders.

General (non-symmetric) Khatri-Rao sketches, with rows $\theta_1 \otimes \dots \otimes \theta_j$ for independent random $\theta_i$, are a natural choice here because they provably give Johnson--Lindenstrauss / oblivious subspace embeddings \citep{jin2021faster,bujanovic2025subspace}. The symmetric sketches we use ($\theta_1 = \dots = \theta_j$) are not covered by these guarantees; we adopt them because directionally symmetric probes are far cheaper to compute than asymmetric ones (\secref{sec:cost_vs_symmetry}).

Although this scheme is motivated by matrix sketching in finite dimensions, it can be written so that no operation requires finite dimensions, and we present the algorithms at the infinite-dimensional level. For intuition we also state, for each basis, what the algorithm computes when $X$ is finite-dimensional.

\subsubsection{Output basis}

\begin{algorithm}
\caption{Sketch shared output basis $V:\mathbb{R}^M \rightarrow Y$.}
\label{alg:output_basis}
\begin{algorithmic}[1]
\State Start with empty basis $V:\mathbb{R}^0 \rightarrow Y$, $V(0)=0$
\For{$i=1,2,\dots$}
    \State Draw a random vector $\widehat{\theta} \in X \sim N(0,C^2)$
    \For{$j=1,\dots,k$}
        \State $y_j \gets D^j q(\paramzero) \widehat{\theta}^j$
        \State $\rho_j \gets y_j - VV^* y_j$
        \If{$\norm{\rho_j} \ge \epsilon \norm{y_j}$}
            \State $V \gets \begin{bmatrix} V & \rho_j/\norm{\rho_j} \end{bmatrix}$
        \EndIf 
    \EndFor
    \State Terminate loop if $V$ has not been updated for $p_\text{stop}$ consecutive iterations
\EndFor
\end{algorithmic}
\end{algorithm}

\algoref{alg:output_basis} constructs $V : \mathbb{R}^M \to Y$. Each outer iteration draws $\widehat{\theta} \in X \sim N(0,C^2)$ and, for each order $j = 1,\dots,k$, forms the symmetric forward probe
$$
y_j = D^j q(\paramzero)\,\widehat{\theta}^j.
$$
When $X$ is finite-dimensional, $y_j$ is the symmetric forward probe of $D^j f(0)$ by a white-noise vector in $N(0,I)$, since $f(x) = q(\paramzero + Cx)$. If $\norm{y_j - VV^*y_j} < \epsilon \norm{y_j}$, then $y_j$ is well approximated by $\operatorname{range}(V)$ and is discarded; otherwise the residual $y_j - VV^*y_j$ is normalized and appended to $V$. The same $\widehat{\theta}$ is reused across all $j$ so that incremental state variables carry over between orders. The outer loop draws a new $\widehat{\theta}$ each time and stops after $p_\text{stop}$ consecutive iterations in which every $y_j$, $j = 1,\dots,k$, is well approximated; we use $p_\text{stop}=5$. For numerical stability, $V$ should be re-orthogonalized occasionally.

\subsubsection{Input basis}

\begin{algorithm}
\caption{Sketch shared input basis $U:\mathbb{R}^N \rightarrow X$.}
\label{alg:input_basis}
\begin{algorithmic}[1]
\State Start with empty basis $U:\mathbb{R}^0 \rightarrow X$, $U(0)=0$
\For{$i=1,2,\dots$}
    \State Draw a random vector $\widehat{\theta} \in X \sim N(0,C^2)$
    \State Draw a random vector $\widehat{\vb{\omega}} \in \mathbb{R}^M \sim N(0, \vb{I})$
    \State $\omega \gets \widehat{\vb{\omega}}^T V^*$
    \For{$j=1,\dots,k$}
        \State $\psi_j \gets \omega(D^j q(\paramzero) \widehat{\theta}^{j-1})$
        \State $\widehat{x}_j \gets C^* \psi_j^*$
        \State $\rho_j \gets \widehat{x}_j - UU^* \widehat{x}_j$
        \If{$\norm{\rho_j} \ge \epsilon \norm{\widehat{x}_j}$}
            \State $U \gets \begin{bmatrix} U & \rho_j / \norm{\rho_j} \end{bmatrix}$
        \EndIf
    \EndFor
    \State Terminate loop if $U$ has not been updated for $p_\text{stop}$ consecutive iterations
\EndFor
\end{algorithmic}
\end{algorithm}

\algoref{alg:input_basis} constructs $U : \mathbb{R}^N \to X$. Each outer iteration draws $\widehat{\theta} \in X \sim N(0,C^2)$ and $\widehat{\vb{\omega}} \in \mathbb{R}^M \sim N(0,I)$ and sets $\omega = \widehat{\vb{\omega}}^T V^*$. For each order $j = 1,\dots,k$ it forms the symmetric reverse probe
$$
\psi_j = \omega\!\left(D^j q(\paramzero)\,\widehat{\theta}^{j-1}\right), \qquad
\widehat{x}_j = C^* \psi_j^*.
$$
When $X$ is finite-dimensional, $\widehat{x}_j$ is the Riesz representative of the symmetric reverse probe of $VV^* D^j f(0)$ by white-noise vectors in both the input and output spaces. The accept/discard rule and the reuse of $\widehat{\theta}$, $\widehat{\vb{\omega}}$ across orders are as for the output basis: $\widehat{x}_j$ is discarded if $\norm{\widehat{x}_j - UU^* \widehat{x}_j} < \epsilon \norm{\widehat{x}_j}$, and otherwise its residual is appended to $U$, and the loop stops after $p_\text{stop}$ consecutive iterations with no update. As with $V$, $U$ should be re-orthogonalized occasionally.

\subsubsection{Dimension-reduced map}
The reduced map $\widetilde{f} : \mathbb{R}^N \to \mathbb{R}^M$ is
$$
\widetilde{f}(\vb{x}) = V^* f(U \vb{x}).
$$
Recall that $f(x) = q(\paramzero + Cx)$ absorbs the whitening transform $C$, so $\widetilde{f}$ acts on whitened, reduced coordinates; this is why the training directions in $\mathbb{R}^N$ used below are plain white noise. Its $j$th derivative at the origin satisfies
$$
D^j \widetilde{f}(0)(\vb{x}_1,\dots,\vb{x}_j)
= V^* D^j f(0)\left(U\vb{x}_1,\dots,U\vb{x}_j\right),
$$
so a probe of $D^j \widetilde{f}(0)$ is a probe of $D^j f(0)$ with its inputs lifted by $U$ and its output projected by $V$. We fit Tucker tensor trains $\tuple{T}_j$ to the reduced derivatives $D^j \widetilde{f}(0)$; if Tucker tensor trains on the original spaces are wanted, compose the first $j$ Tucker basis cores of $\tuple{T}_j$ with $U$ and the last with $V$.

\subsubsection{Alternative dimension reduction methods}
The bases $U$ and $V$ may instead be built by conventional dimension reduction that evaluates $f$ at many points: proper orthogonal decomposition / PCA / Karhunen-Lo\`eve expansion \citep{bhattacharya2021model,lanthaler2023operator,manzoni2016dimensionality}, active subspaces \citep{constantine2015active} and their vector-valued \citep{verdiere2025mollified,zahm2020gradient} and derivative-informed \citep{luo2025dimension} variants, likelihood-informed and related posterior-oriented methods \citep{baptista2022gradient,chen2024coupled,cui2014likelihood,zahm2022certified}, greedy methods \citep{lieberman2010parameter}, and Hessian-based methods \citep{chen2019hessian}. Tucker bases could also be built directly from symmetric derivative probes by the damped power iteration of \cite[Algorithm 1]{jin2024scalable} or the single-mode-sketching method of \cite{hashemi2025rtsms}, with per-order bases $U_j$, $V_j$ then merged into the shared $U$, $V$.

\subsection{Symmetric probing training data}
\label{sec:symmetric_action_data}

For each order $j$ we seek a Tucker tensor train $\tuple{T}_j$ with $T_j \approx D^j \widetilde{f}(0)$, where $T_j$ is the vector-valued multilinear function associated with $\tuple{T}_j$. The training data are directionally symmetric forward and reverse probes of $D^j \widetilde{f}(0)$ on random directions. These reduced-space probes are computed by the machinery of \secref{sec:derivative_actions} applied to $\widetilde{f}$---equivalently, full-space probes of $D^j f(0)$ lifted by $U$ and projected by $V$---and reusing the same directions across orders makes them cheap. They are distinct from the probes of \secref{sec:initial_dimension_reduction}, which served only to build $U$ and $V$.

Let $\numactions$ be the number of training samples. For $i=1,\dots,\numactions$, draw normalized white-noise directions
\begin{align*}
 \vb{x}^{(i)} = \vb{x}_0^{(i)} \big/ \norm{\vb{x}_0^{(i)}},& \qquad \vb{x}_0^{(i)} \sim N(0, \vb{I}_{N}), \\
\vb{\omega}^{(i)} = \vb{\omega}_0^{(i)} \big/ \norm{\vb{\omega}_0^{(i)}},& \qquad \vb{\omega}_0^{(i)} \sim N(0, \vb{I}_{M}),
\end{align*}
and, for $j=1,\dots,k$, compute the symmetric forward and reverse probes of $D^j \widetilde{f}(0)$,
\begin{align*}
\vb{y}_{j}^{(i)} &= D^j \widetilde{f}(0)\left(\vb{x}^{(i)}\right)^j, \\
\vb{\psi}_{j}^{(i)} &= \left(\vb{\omega}^{(i)}\right)^T D^j \widetilde{f}(0)\left(\vb{x}^{(i)}\right)^{j-1}.
\end{align*}
Unit-norm directions matter for conditioning. Since the forward probe $\vb{y}_j^{(i)}$ reuses $\vb{x}^{(i)}$ in all $j$ slots, its dependence on the raw magnitude is $\norm{\vb{x}_0^{(i)}}^{j}$, a heavy-tailed weight whose spread grows with order. Unnormalized, a few large draws dominate the least squares \eqref{eq:loss_expanded_double_sum}, shrinking the effective sample count and inflating the Gauss-Newton condition number (\secref{sec:trrmgn})---worse at higher $k$ (\secref{sec:numerical_random}). Normalizing removes the weight (or divides it out afterward), so each sample is a unit-norm rank-one measurement.

Because the same directions $\vb{x}^{(i)}$, $\vb{\omega}^{(i)}$ are used for all orders, incremental variables from lower-order probes are stored and reused at higher orders. The data used to fit $\tuple{T}_\derivativeorderrmgnsection$ are the input--output pairs
\begin{equation*}
((\vb{x}^{(i)}, \vb{\omega}^{(i)}), (\vb{\psi}_{\derivativeorderrmgnsection}^{(i)}, \vb{y}_{\derivativeorderrmgnsection}^{(i)})), \qquad i=1,\dots, \numactions, \quad \derivativeorderrmgnsection=1,\dots,k.
\end{equation*}

\subsection{Fixed-rank nonlinear least-squares problem}
\label{sec:fixed_rank_optimization_problem}

For fixed Tucker ranks $\vb{n}$ and tensor train ranks $\vb{r}$, we fit $\tuple{T}_\derivativeorderrmgnsection$ by approximately solving
\begin{equation}
\label{eq:fixed_rank_tt_optimization_problem_original}
\min_{\vb{T} \in \mathcal{M}_{\vb{n},\vb{r}}} ~ \loss_\derivativeorderrmgnsection(\vb{T}),
\end{equation}
with the least-squares loss
\begin{equation}
\label{eq:loss_expanded_double_sum}
    \loss_\derivativeorderrmgnsection(\vb{T}) = \frac{1}{2 \numactions}\sum_{i=1}^{\numactions}\sum_{l=1}^\derivativeorderrmgnsection \norm{\vb{\psi}_{\derivativeorderrmgnsection}^{(\sampleind)} - \vb{a}_{l}^{(\sampleind)}(\vb{T})}^2 + \frac{1}{2 \numactions}\sum_{i=1}^{\numactions}\norm{\vb{y}_{\derivativeorderrmgnsection}^{(\sampleind)} - \vb{a}_{\derivativeorderrmgnsection+1}^{(\sampleind)}(\vb{T})}^2.
\end{equation}
Here $\vb{a}_{l}^{(i)}(\vb{T})$ is the probe of $\tuple{T}_\derivativeorderrmgnsection$ that leaves mode $l$ open while filling the remaining modes with the sample directions: the $\derivativeorderrmgnsection$ input modes with $\vb{x}^{(i)}$ and the output mode with $\vb{\omega}^{(i)}$. For $l = 1,\dots,\derivativeorderrmgnsection$ an input mode is open, so $\vb{a}_{l}^{(i)}(\vb{T})$ is a reverse probe in $\mathbb{R}^N$, compared in \eqref{eq:loss_expanded_double_sum} against the reverse target $\vb{\psi}_{\derivativeorderrmgnsection}^{(i)}$; for $l = \derivativeorderrmgnsection+1$ the output mode is open, so $\vb{a}_{\derivativeorderrmgnsection+1}^{(i)}(\vb{T})$ is the forward probe in $\mathbb{R}^M$, compared against the forward target $\vb{y}_{\derivativeorderrmgnsection}^{(i)}$. Each map $\vb{T} \mapsto \vb{a}_{l}^{(\sampleind)}(\vb{T})$ is linear on $\mathbb{R}^{N \times \dots \times N \times M}$.

The loss compares the single reverse target $\vb{\psi}_{\derivativeorderrmgnsection}^{(i)}$ against all $\derivativeorderrmgnsection$ of the reverse probes $\vb{a}_{1}^{(i)},\dots,\vb{a}_{\derivativeorderrmgnsection}^{(i)}$. For the true derivative tensor, which is symmetric in its input arguments, these $\derivativeorderrmgnsection$ probes coincide under directionally symmetric inputs, so the $\derivativeorderrmgnsection$ residuals would be redundant; including all of them encourages the learned $\tuple{T}_\derivativeorderrmgnsection$, which is not explicitly symmetrized, to match the symmetric action across its input modes.

Nonlinearity enters only through the feasible set $\mathcal{M}_{\vb{n},\vb{r}} \subset \mathbb{R}^{N \times \dots \times N \times M}$, the manifold of Tucker tensor trains with Tucker ranks $\vb{n}$ and tensor train ranks $\vb{r}$ (\appref{sec:tensor_train_manifold}).

\subsubsection{Trust-region RMGN} 
\label{sec:trrmgn}

TR-RMGN solves \eqref{eq:fixed_rank_tt_optimization_problem_original} by repeatedly linearizing the manifold at the current iterate, minimizing the loss over a trust region in the tangent space by an inexact conjugate-gradient solve, retracting back to the manifold, and accepting or rejecting the step by a trust-region ratio test. It follows the general framework of \cite{absil2007trust}, \cite[Chapter 7]{absil2008optimization}, and \cite[Section 6.4]{boumal2023introduction}, and is guaranteed to converge to a critical point; as is typical for tensor optimization, a global minimizer is not expected.

\begin{figure}
\centering
\scalebox{0.6}{
        \begin{tikzpicture}[scale=4, every node/.style={scale=1.0}]
        \draw[fill, top color=gray, bottom color=white, shading=axis, shading angle=40.0] 
        (0.0, 1.6) [bend right=15] to (-1.0, 1.1) [bend left=35] to (0.3, 0.3) [bend left=25] to (1.2,0.9) [bend right=25] to cycle;

        \draw[fill, fill opacity=0.1, color=gray, draw=none] (0.1, 2.0) -- (-1.6, 1.3) --  (-0.2, 0.5) -- (1.3,1.6) -- cycle;
        \draw[fill=none] (0.1, 2.0) -- (-1.6, 1.3) --  (-0.2, 0.5) -- (1.3,1.6) -- cycle;

        \begin{scope}
            \clip (0.1, 2.0) -- (-1.6, 1.3) --  (-0.2, 0.5) -- (1.3,1.6) -- cycle;
            \draw[rotate around={-10:((0.85,1.40)}, fill=gray!255, opacity=0.15, draw=none] (0.85,1.40) ellipse[x radius=0.8, y radius=0.5];
            \draw[rotate around={-10:((0.85,1.40)}, fill=gray!255, opacity=0.15, draw=none] (0.85,1.40) ellipse[x radius=0.5, y radius=0.3125];
            \draw[rotate around={-10:((0.85,1.40)}, fill=gray!255, opacity=0.15, draw=none] (0.85,1.40) ellipse[x radius=0.2, y radius=0.125];
        \end{scope}
        \node[fill=black, draw=none, circle, inner sep=2pt] (qstar) at (0.85,1.40) {};
        \node[fill=black, draw=none, circle, inner sep=2pt] (pcur) at (-0.1,1.30) {};
        \node[fill=black, draw=none, circle, inner sep=2pt] (pnext) at (0.8,1.05) {};

        \draw[-{Stealth[length=2.5mm]}, shorten >=2pt,shorten <=2pt, line width=2pt] (pcur) [bend left=45] to (qstar);
        \draw[-{Stealth[length=2.5mm]}, shorten >=2pt,shorten <=2pt, line width=2pt] (qstar) [bend left=15] to (pnext);
        
        \node at (0.32, 0.65) {\Large $\mathcal{M}_{\mathbf{n},\mathbf{r}}$};
        \node at (0.93,1.43) {$\mathbf{q}$};
        \node at (0.875,0.965) {$\widetilde{\mathbf{p}}$};
        \node at (-0.08,1.22) {$\mathbf{p}$};
        
    \end{tikzpicture}
    
}
\caption{One step of the Riemannian manifold Gauss-Newton method. 
The shaded ellipsoids are level sets of the objective function restricted to $\vb{p}+T_{\vb{p}} \mathcal{M}_{\vb{n},\vb{r}}$, where $\vb{p}$ is the current iterate. The point $\vb{q}$ is an approximate minimizer of the objective function on $\vb{p}+T_{\vb{p}} \mathcal{M}_{\vb{n},\vb{r}}$. 
}
\label{fig:ttrmgn}
\end{figure}

Because $\vb{T} \mapsto \vb{a}_{l}^{(i)}(\vb{T})$ is linear and $\vb{p} + T_\vb{p}\mathcal{M}_{\vb{n},\vb{r}}$ is an affine subspace, replacing $\mathcal{M}_{\vb{n},\vb{r}}$ by this local linearization in \eqref{eq:fixed_rank_tt_optimization_problem_original} gives a linear least-squares subproblem, solved by conjugate gradients on the normal equations. Starting from $\vb{p} \in \mathcal{M}_{\vb{n},\vb{r}}$ and trust-region radius $\Delta$, each iteration (illustrated in \figref{fig:ttrmgn}) does the following:

\begin{enumerate}
\item Apply the CG–Steihaug method \citep{steihaug1983conjugate}, a conjugate-gradient variant for trust-region subproblems, to approximately minimize $\loss_\derivativeorderrmgnsection$ over the $\vb{q}$ in $\vb{p}+T_{\vb{p}} \mathcal{M}_{\vb{n},\vb{r}}$ within distance $\Delta$ of $\vb{p}$:
\begin{equation}
\label{eq:fixed_rank_tt_optimization_problem}
\vb{q} \approx \vb{p} + \argmin_{\substack{\vb{v} \in T_{\vb{p}} \mathcal{M}_{\vb{n},\vb{r}} \\ \norm{\vb{v}} \le \Delta}} \loss_\derivativeorderrmgnsection(\vb{p} + \vb{v}).
\end{equation}

\item Retract $\vb{q}$ onto $\mathcal{M}_{\vb{n},\vb{r}}$ to obtain a proposal $\widetilde{\vb{p}} = \operatorname{retract}(\vb{q})$.

\item Compute the ratio
$$
\rho = \frac{\loss_\derivativeorderrmgnsection(\vb{p}) - \loss_\derivativeorderrmgnsection(\widetilde{\vb{p}})}{\loss_\derivativeorderrmgnsection(\vb{p}) - \loss_\derivativeorderrmgnsection(\vb{q})}
$$
of actual to predicted reduction. Based on $\rho$ and the CG–Steihaug termination, accept ($\vb{p} \leftarrow \widetilde{\vb{p}}$) or reject (retain $\vb{p}$) the proposal and update $\Delta$ by the standard rules \citep[Algorithm 4.1]{nocedal1999numerical}, \citep[Algorithm 10]{absil2008optimization}, and \citep[Algorithm 6.3]{boumal2023introduction}.
\end{enumerate}

The iteration repeats until convergence. For background on trust-region methods see \cite[Chapter 4]{nocedal1999numerical}, and for their Riemannian extensions \cite[Chapter 7]{absil2008optimization} and \cite[Section 6.4]{boumal2023introduction}. We detail the two nontrivial steps, the linear least-squares subproblem and the retraction, below.

\paragraph{Linear least-squares subproblem}
Using the definition of $\loss_\derivativeorderrmgnsection$ and linearity of probing in the probed tensor, the subproblem \eqref{eq:fixed_rank_tt_optimization_problem} becomes
$$
\argmin_{\substack{\vb{v} \in T_{\vb{p}} \mathcal{M}_{\vb{n},\vb{r}} \\ \|\vb{v}\| \le \Delta}} 
\frac{1}{2 \numactions} \sum_{i=1}^{\numactions} \sum_{l=1}^{\derivativeorderrmgnsection+1}
\| \vb{b}_{l}^{(i)} - \vb{a}_{l}^{(i)}(\vb{v}) \|^2,
$$
where $\vb{b}_{l}^{(i)} := \vb{\psi}_{\derivativeorderrmgnsection}^{(i)} - \vb{a}_{l}^{(i)}(\vb{p})$ for $l = 1, \dots, \derivativeorderrmgnsection$ and $\vb{b}_{\derivativeorderrmgnsection+1}^{(i)} := \vb{y}_{\derivativeorderrmgnsection}^{(i)} - \vb{a}_{\derivativeorderrmgnsection+1}^{(i)}(\vb{p})$. Setting the gradient to zero gives the normal equations
\begin{equation}
\label{eq:sqp_normal_eq}
(\jaczero^T \circ \jaczero)(\vb{v}) = \jaczero^T(\vb{b}),
\end{equation}
where $\jaczero$ is the tangent-vector-to-probes map that stacks every probe of $\vb{v}$,
\begin{align*}
\jaczero(\vb{v}) &= \big((\vb{a}_{1}^{(1)}(\vb{v}), \dots, \vb{a}_{\derivativeorderrmgnsection+1}^{(1)}(\vb{v})), \dots, (\vb{a}_{1}^{(\numactions)}(\vb{v}), \dots, \vb{a}_{\derivativeorderrmgnsection+1}^{(\numactions)}(\vb{v}))\big), \\
\vb{b} &= \big((\vb{b}_{1}^{(1)}, \dots, \vb{b}_{\derivativeorderrmgnsection+1}^{(1)}), \dots, (\vb{b}_{1}^{(\numactions)}, \dots, \vb{b}_{\derivativeorderrmgnsection+1}^{(\numactions)})\big),
\end{align*}
and $\jaczero^T$ is its adjoint under the Hilbert--Schmidt inner product; we write $\circ$ because $\jaczero$ and $\jaczero^T$ act on collections of arrays. We solve \eqref{eq:sqp_normal_eq} by CG–Steihaug with tolerance $\epsilon = \min(0.5, \sqrt{\norm{\vb{b}}}) \norm{\vb{b}}$ \citep[Section 7.1]{nocedal1999numerical}. Fully dense tangent vectors are too large to store or use in CG, so all CG operations act on gauged representations of tangent vectors (\appref{sec:tensor_train_manifold}); applying $\jaczero$ and $\jaczero^T$ efficiently in this representation is the subject of \secref{sec:tt_operations}.

\paragraph{Retraction step} 
The point $\vb{q} \in \vb{p} + T_\vb{p} \mathcal{M}_{\vb{n},\vb{r}}$ is represented as a Tucker tensor train with doubled ranks (\appref{sec:tensor_train_manifold}). We obtain $\widetilde{\vb{p}} = \operatorname{retract}(\vb{q})$ by rounding this representation with the implicit T3-SVD algorithm (\algoref{alg:t3_svd_implicit} in \appref{app:tensor_appendix}), truncating the SVDs to the prescribed ranks $\vb{n}$ and $\vb{r}$. This generalizes the tensor train retraction of \cite[Section 4.3]{steinlechner2016riemannian}.

\subsubsection{Manifold Cauchy SGD}
\label{sec:cauchy_sgd}

When training data are plentiful, using all of it at each TR-RMGN iteration is expensive, and one would prefer a minibatch per iteration. But on a small minibatch the Gauss-Newton Hessian is singular and ill-conditioned on its nonzero eigenspace, so CG converges slowly and adding CG iterations tends to hurt generalization. We therefore use a manifold stochastic gradient descent method with a Cauchy step size (MC-SGD):
\begin{align}
\label{eq:cauchy_sgd}
\vb{q}_k &= \vb{p}_k - \frac{\norm{\vb{g}^{(B)}}^2}{\norm{J^{(B)}\left(\vb{g}^{(B)}\right)}^2} \vb{g}^{(B)}, \\
\vb{p}_{k+1} &= \operatorname{retract}(\vb{q}_k),
\end{align}
where $J^{(B)}$ is the Jacobian on a freshly sampled minibatch $B$, $\vb{g}^{(B)} = (J^{(B)})^T(\vb{b})$ is the stochastic gradient, and $\operatorname{retract}$ is the implicit truncated T3-SVD retraction. Cauchy step sizes are classical in optimization \citep[Chapter 3]{nocedal1999numerical}, going back to Cauchy's 1847 note \citep{cauchy1847methodes}, and have recently drawn interest in stochastic and machine-learning settings \citep{castera2022second,smee2025firstish}.

Writing the step as
\begin{equation*}
    -\frac{\norm{\vb{g}^{(B)}}^2}{\norm{J^{(B)}\left(\vb{g}^{(B)}\right)}^2} \vb{g}^{(B)}
    =
    -\frac{
    \innerproduct{\vb{g}^{(B)}}{\vb{g}^{(B)}}
    }{
    \innerproduct{\vb{g}^{(B)}}{\left((J^{(B)})^T \circ J^{(B)}\right)(\vb{g}^{(B)})}
    } \vb{g}^{(B)}
\end{equation*}
shows that the denominator is the Gauss-Newton Hessian applied to the gradient: the Cauchy step forms the local quadratic model of the objective along the gradient ray and jumps to its minimum, equivalently a manifold Gauss-Newton-CG step truncated after one CG iteration. It costs one extra Jacobian-vector product per iteration, cheap next to gradient assembly and retraction, and removes the step-size schedule and tuning otherwise needed in stochastic gradient descent. This contrasts with the grid-searched fixed learning rate used for stochastic Riemannian gradient descent on the tensor train manifold by \citet{novikov.trofimov.ea.2018}. On a small minibatch the denominator $\norm{J^{(B)}(\vb{g}^{(B)})}$ can be tiny or zero along nearly flat directions (especially with low-precision rounding), so in practice the step should be capped by a guard.

\paragraph{Stopping criterion}
This is the first of three distinct controls in the construction: it stops the fixed-rank MC-SGD iteration and triggers a rank increase, separate from the continuation's termination test and from the final model selection, both in \secref{sec:rank_adaptivity}. For fixed ranks, $\loss_j(\vb{p}_k)$ typically falls quickly and then fluctuates around a noisy equilibrium. We detect the equilibrium by exponentially smoothing the loss,
\begin{equation*}
s_k = \alpha \loss_j(\vb{p}_k) + (1-\alpha)\loss_j(\vb{p}_{k-1}),
\end{equation*}
and stopping when the lagged smoothed difference $\Delta s = s_k - s_{k-t}$ becomes positive. We set $\alpha = 1 - \exp(-1/\tau)$ with $\tau = C_\tau \numactions/|B|$ (the number of iterations in $C_\tau$ epochs) and lag $t = C_t \numactions/|B|$ (the iterations in $C_t$ epochs), using $C_\tau=1.0$, $C_t=3.0$, and $|B|=\floor{\numactions / 10}$. Because this training-loss criterion can stop early under minibatch noise and does not guard against overfitting, the final model across all ranks is selected by held-out probe (validation) error (\secref{sec:rank_adaptivity}).

\subsubsection{Alternative optimization methods and related work}
\label{sec:alternative_optimization_methods}

Readers following the method may skip to \secref{sec:rank_adaptivity}; this subsection records alternatives and related work.

\paragraph{Core-wise optimization}
Our method takes the full tensor as the optimization variable, using its core representation only as a computational device. An alternative is to treat the cores themselves as independent variables and minimize over them, which admits deterministic methods such as L-BFGS \citep{gorodetsky2018gradient,gorodetsky2022reverse} and stochastic methods such as SGD \citep{yuan2019high} and Adam \citep{kingma2014adam}. Its principal difficulty is that the map from cores to the represented tensor is highly nonlinear and neither one-to-one nor onto, which increases the nonlinearity and ill-conditioning of the problem; \figref{fig:t_squared} gives a one-dimensional illustration.

\begin{figure}
\centering
\begin{subfigure}[b]{0.49\textwidth}
    \centering
    \begin{tikzpicture}[scale=1.5]
        \draw[->, dashed] (-1.0,0) -- (1.0,0) node[right] {$t$};
        \draw[->, dashed] (0,-0.25) -- (0,1.0) node[above] {$f(t)$};
    
        \draw[color=black, line width=1, domain=-1:1] plot (\x,\x*\x) node[right] {};
    \end{tikzpicture}
    \caption{$f(t)=t^2$}
\end{subfigure}
\begin{subfigure}[b]{0.49\textwidth}
    \centering
    \begin{tikzpicture}[scale=1.5]
        \draw[color=black, line width=1] (-1.0,0) -- (1.0,0) node[right] {$g_1$};
        \draw[color=black, line width=1] (0,-1.0) -- (0,1.0) node[above] {$g_2$};
    
        \draw[color=black, line width=1, domain=0.05:1.0]    plot (\x, 0.05 / \x)             node[right] {};
        \draw[color=black, line width=1, domain=0.15:1.0]    plot (\x, 0.15 / \x)             node[right] {};
        \draw[color=black, line width=1, domain=0.45:1.0]    plot (\x, 0.45 / \x)             node[right] {};
    
        \draw[color=black, line width=1, domain=-1.0:-0.05]    plot (\x, -0.05 / \x)             node[right] {};
        \draw[color=black, line width=1, domain=-1.0:-0.15]    plot (\x, -0.15 / \x)             node[right] {};
        \draw[color=black, line width=1, domain=-1.0:-0.45]    plot (\x, -0.45 / \x)             node[right] {};
    
        \draw[color=black, line width=1, domain=0.05:1.0]    plot (\x, -0.05 / \x)             node[right] {};
        \draw[color=black, line width=1, domain=0.15:1.0]    plot (\x, -0.15 / \x)             node[right] {};
        \draw[color=black, line width=1, domain=0.45:1.0]    plot (\x, -0.45 / \x)             node[right] {};
    
        \draw[color=black, line width=1, domain=-1.0:-0.05]    plot (\x, 0.05 / \x)             node[right] {};
        \draw[color=black, line width=1, domain=-1.0:-0.15]    plot (\x, 0.15 / \x)             node[right] {};
        \draw[color=black, line width=1, domain=-1.0:-0.45]    plot (\x, 0.45 / \x)             node[right] {};
    \end{tikzpicture}
    \caption{Level sets of $f(g_1,g_2)= g_1^2 g_2^2$.}
\end{subfigure}
\caption{The change of variables $t=g_1 g_2$ transforms the easy to optimize function $t^2$ into the difficult to optimize function $g_1^2 g_2^2$. 
Contraction of adjacent cores in a tensor network is analogous to the multiplication $g_1g_2$.
}
\label{fig:t_squared}
\end{figure}

\paragraph{Alternating least-squares}
Alternating least-squares (ALS) methods \citep{grasedyck2015variants,holtz2012alternating,rohwedder2013local} cycle through the cores until convergence, fitting one at a time with the others fixed; performance improves when the fixed cores are orthogonalized relative to the one being updated, as in \eqref{eq:p_rep_Btilde} and \eqref{eq:p_rep_Gtilde}. Each single-core subproblem is a linear least-squares problem solvable by CG, giving ALS-CG; full ALS, which solves each subproblem exactly by dense linear algebra, is typically intractable at large scale. ALS-CG Hessian-vector products are those of the TR-RMGN Hessian (\secref{sec:tt_operations}) with certain terms zeroed, hence only slightly cheaper, so per-iteration CG cost is comparable. Unlike TR-RMGN, which updates all cores at once, ALS-CG updates one at a time.

\paragraph{Riemannian manifold optimization}
Riemannian optimization on low-rank manifolds progresses from matrices \citep{boumal2011rtrmc,kasai2018inexact} to tensors: from \cite{vandereycken2013low}, \cite{kressner2014low} extended it to fixed-rank Tucker tensors and \cite{steinlechner2016riemannian} to the tensor train manifold, with related subspace-correction \citep{kressner2014tensor}, preconditioned \citep{kressner2016preconditioned}, and hierarchical-Tucker \citep{da2015optimization} developments. \cite{luo2023low} prove local quadratic convergence of Riemannian Gauss-Newton for fixed-Tucker-rank estimation, supporting the Gauss-Newton method we use; see also \cite{novikov2021tensor,psenka2020second,uschmajew2020geometric} for the tensor train manifold, \cite{willner2025riemannian} for hierarchical Tucker, and \cite{cai2026tensor,bian2025fast} for preconditioned and quotient-geometry formulations. The most direct treatment of the Tucker tensor train manifold (there termed the extended tensor train manifold) is \cite{molozhavenko2026optimization}, which proves its smooth structure and derives efficient projection and retraction when the Tucker basis matrices are shared; we do not share Tucker bases, but since derivative tensors are symmetric in their inputs, doing so is a natural direction for future work.

The main disadvantage of Riemannian methods is the per-iteration retraction, which requires sweeping orthogonalization and SVDs and can be costly at large ranks; randomized tensor train rounding may reduce this \citep{al2023randomized,kressner2023streaming,aldaas2025adaptive}. Efficient Riemannian Hessian-vector products are a further challenge: \cite{novikov2022automatic} compute them by automatic differentiation, whereas we use a problem-specific method for Gauss-Newton Hessian-vector products (\secref{sec:tt_operations}).

\subsection{Rank continuation}
\label{sec:rank_adaptivity}

Starting from small ranks $\vb{n}$, $\vb{r}$, we solve the fixed-rank problem \eqref{eq:fixed_rank_tt_optimization_problem} by either method of \secref{sec:trrmgn} or \secref{sec:cauchy_sgd}, then incrementally increase the ranks and re-solve, warm-starting from the previous solution. Two controls govern this outer loop, both distinct from the per-rank stopping criterion of \secref{sec:cauchy_sgd}: a \emph{termination} test that ends the continuation, and a \emph{selection} rule that picks the final model.

We reserve a fraction of the probes (e.g.\ 20\%) as validation data, excluded from the fit \eqref{eq:fixed_rank_tt_optimization_problem_original} and used only to choose ranks. Validation error typically decreases as ranks grow, then stabilizes or rises from overfitting. For termination, we track how overdetermined the model is through the ratio of data dimension to $\dim(\mathcal{M}_{\vb{n},\vb{r}})$; this ratio falls as ranks grow, and the continuation stops once it drops below a threshold $\tau_\text{data}$ (e.g.\ $\tau_\text{data}=2$), before the model becomes underdetermined. For selection, we return the ranks that achieved the smallest validation error over all iterations.

Continuation also regularizes. As shown in \secref{sec:initial_guess_new_ranks}, the warm-start gradient is confined to the newly added directions, so each refit only corrects near the previous solution rather than refitting from scratch---which is why the validation error levels off rather than diverging at high rank.

\subsubsection{Choosing the new ranks}
Inspired by the adaptive-rank ALS method for tensor completion of \cite{grasedyck2019stable}, we increase ranks based on the singular values of the tensor's matrix unfoldings and matricizations, growing the most well-conditioned edges so as to keep all edges comparably conditioned.

Let $\tuple{T}$ be the Tucker tensor train obtained from \eqref{eq:fixed_rank_tt_optimization_problem} and $\vb{T}$ its array. Write the singular values of the $i$th matrix unfolding of $\vb{T}$ as $\sigma_{i,1}^\text{TT} \ge \sigma_{i,2}^\text{TT} \ge \dots \ge \sigma_{i,r_i}^\text{TT} \ge \dots$ (values beyond $\sigma_{i,r_i}^\text{TT}$ vanish by the Tucker tensor train structure), and define the \emph{tensor-train edge condition numbers}
$$
\kappa_i^\text{TT} = \frac{\sigma_{i,1}^\text{TT}}{\sigma_{i,r_i}^\text{TT}},
\qquad
\vb{\kappa}^\text{TT} = (1, \kappa_1^\text{TT}, \dots, \kappa_{d-1}^\text{TT}, 1).
$$
Similarly, from the singular values $\sigma_{i,1}^\text{Tucker} \ge \dots \ge \sigma_{i,n_i}^\text{Tucker} \ge \dots$ of the $i$th matricization, define the \emph{Tucker edge condition numbers}
$$
\kappa_i^\text{Tucker} = \frac{\sigma_{i,1}^\text{Tucker}}{\sigma_{i,n_i}^\text{Tucker}},
\qquad
\vb{\kappa}^\text{Tucker} = (\kappa_1^\text{Tucker}, \dots, \kappa_d^\text{Tucker}).
$$
All singular values, and hence condition numbers, are computed by the implicit T3-SVD algorithm (\algoref{alg:t3_svd_implicit} in \appref{app:tensor_appendix}).

Let $\kappa_\text{max}$ be the largest entry of $\vb{\kappa}^\text{TT}$ and $\vb{\kappa}^\text{Tucker}$. We propose new ranks
$$
n_i' =
\begin{cases}
n_i+n_\text{chunk}, & \kappa_i^\text{Tucker} < \kappa_\text{max}/\tau, \\
n_i, & \text{otherwise},
\end{cases}
\qquad
r_i' =
\begin{cases}
r_i+n_\text{chunk}, & \kappa_i^\text{TT} < \kappa_\text{max}/\tau, \\
r_i, & \text{otherwise},
\end{cases}
$$
increasing a rank only on a well-conditioned edge---one whose condition number is at least a factor $\tau$ below $\kappa_\text{max}$---so the increases bring all edges toward comparable conditioning. The user-defined $\tau>1$ and $n_\text{chunk}$ control how uniformly conditioned the edges stay and how fast ranks grow; we use $\tau = 10.0$ and typically $n_\text{chunk} = 1$.

Proposed ranks may be too large, producing degenerate tensor trains, just as a matrix factorization $\vb{A} = \vb{X}\vb{Y}$ with inner dimension $r$ is degenerate when $r > \min(N,M) \ge \operatorname{rank}(\vb{A})$. We therefore reduce $\vb{n}'$ and $\vb{r}'$ by the least amount needed for a non-degenerate Tucker tensor train. This ``useless rank removal'' sweeps back and forth through the cores as in the T3-SVD algorithm, adjusting ranks as needed, but uses only the shape and ranks---no linear algebra is performed. If adjustment leaves $\vb{n}' = \vb{n}$ and $\vb{r}' = \vb{r}$ (because every edge already satisfied $\kappa_i^\text{Tucker}, \kappa_i^\text{TT} \ge \kappa_\text{max}/\tau$, or because all increases were removed), we uniformly increase all ranks by $n_\text{chunk}$ and again remove useless rank.

\subsubsection{Initial guess with new ranks} 
\label{sec:initial_guess_new_ranks}
Given the updated ranks $\vb{n}'$, $\vb{r}'$, we pad the cores of $\tuple{T}$ with zeros to match them without changing the represented tensor $\vb{T}$, and use the result as the initial guess for \eqref{eq:fixed_rank_tt_optimization_problem_original} at the next iteration. This guess is necessarily degenerate, since the unfoldings and matricizations of $\vb{T}$ retain their original ranks $\vb{r}$, $\vb{n}$. Because the point lies on a lower-rank stratum at the boundary of $\mathcal{M}_{\vb{n}',\vb{r}'}$, the initial orthogonal representations of the tangent-space base point are partially arbitrary and in practice depend on the orthogonalization routine.

This degenerate guess is nonetheless harmless. Orthogonalization completes the rank-deficient point with an arbitrary orthonormal frame on the new directions, so the tangent space already spans them; the first Gauss-Newton step then moves off the lower-rank stratum into the interior of $\mathcal{M}_{\vb{n}',\vb{r}'}$, and the ranks stay non-degenerate. We prefer this deterministic start to a random restart or perturbation: a random start often stalls at a poor critical point because Gauss-Newton cannot escape along negative curvature, while the zero-padded one lies in the basin of the continued solution.

Warm-starting also tames the conditioning. The Gauss-Newton Hessian \eqref{eq:loss_expanded_double_sum} is ill-conditioned, and more so as ranks grow, since each added rank resolves a smaller singular value of $\vb{T}$ and contributes smaller eigenvalues; continuation does not fix this. Instead it deflates the right-hand side: at the warm start the lower ranks are already optimal, so the gradient 
has essentially no component along the resolved directions and is orders of magnitude smaller than a cold solve's. Each refit thus corrects only the new directions, whereas a cold solve rebuilds the whole structure from scratch and overfits once underdetermined (\secref{sec:numerical_random}).

\subsubsection{Alternative approaches for rank adaptivity} 
The simplest approach treats the ranks as hyperparameters chosen by grid search, as in the stochastic Riemannian fitting of \cite{novikov.trofimov.ea.2018}. Our scheme is inspired by the ``SALSA'' method for tensor completion \citep{grasedyck2019stable}, which continually raises and lowers tensor train ranks within a regularized ALS scheme to keep a fixed number of ``minor'' singular values per unfolding. The modified-ALS (``MALS'') method \citep{holtz2012alternating} adapts ranks by merging two consecutive cores, updating the enlarged tensor by least squares, and refactoring by truncated SVD; the same two-site scheme underlies the tensor-network classifier of \cite{stoudenmire.schwab.2016}, which updates by a gradient step and reports difficulty combining this with SGD. \cite{dolgov2014alternating} adapt tensor train ranks within an alternating scheme using an inexact gradient (residual) direction followed by truncated SVD, and \cite{dektor2021rank} give an analogous idea for time-dependent PDEs based on the velocity component orthogonal to the tangent space; such gradient-projection enrichment is an alternative rank-continuation strategy that could be used here, approximating the intractable full-space gradient by the gradient on a larger-rank tangent space. The closely related Riemannian approach of \cite{vermeylen2025riemannian} grows the rank along a descent direction in the tangent cone to the bounded-rank variety.

\subsection{Cost comparison}
\label{sec:cost_comparison}

\tableref{tab:training_cost} compares the cost of generating training data and fitting the model for T4S, our prior tensor train construction \citep{alger.chen.ea.2020}, a conventional neural network surrogate, and derivative-informed neural operators (DINO) \citep{olearyroseberry.chen.ea.2024}. We measure data-generation cost by two kinds of solves: \emph{nonlinear solves} of the state equation, and \emph{linear solves} of the linearized forward and adjoint equations, which share a common coefficient matrix at a fixed operating point and are therefore much cheaper.

\begin{table}
    \centering
    \small
    \begin{tabular}{lcccl}
        Method & $n_s$ & Nonlinear solves & Linear solves & Training cost \\
        \hline
        T4S (this work)            & probes  & $1$          & $2k\,n_s + O(nk)$   & Fit the tensors. \\
        \citet{alger.chen.ea.2020} & probes  & $1$          & $O(k^2\,r^2\,2^k)$ & Negligible. \\
        \hline
        DINO \citep{olearyroseberry.chen.ea.2024} & samples & $n_s + n_s'$ & $n\,n_s + n\,n_s'$ & Train the network. \\
        Conventional neural network               & samples & $n_s$        & $0$                & Train the network. \\
        \hline
    \end{tabular}
    \caption{Training-data generation and training cost for surrogate models of an
    implicitly defined parameter-to-output map. Here $k$ is the Taylor order, $r$ the tensor train rank, $n$ the reduced dimension,
    $n_s$ the number of training samples, and $n_s' \le n_s$ the number of samples used
    for DINO's reduced bases. For simplicity of presentation, here $n$ is assumed to be either larger or comparable in size to the reduced output dimension $m$. Nonlinear and linear solves count solutions of the state
    equation and of the linearized forward/adjoint equations, respectively; the
    training-cost column involves no further solves. See
    \secref{sec:cost_comparison}.}
    \label{tab:training_cost}
\end{table}

The probe-based methods (T4S and \citealp{alger.chen.ea.2020}) expand about a single point and need only one nonlinear solve, after which all derivative information comes from linear solves. T4S exploits directional symmetry to cut the per-probe linear-solve cost from the $O(2^k)$ of asymmetric probes to $O(k)$, and amortizes incremental solves across orders by reusing probing vectors (\secref{sec:cost_vs_symmetry}). The sampling methods evaluate $f$ at $n_s$ distinct parameter values, each a nonlinear solve; DINO additionally forms a reduced Jacobian per sample, adding $n\,n_s$ linear solves. The dimension-reduction step (\secref{sec:initial_dimension_reduction}) is lower order: for T4S it adds $O(nk)$ linear solves and no nonlinear solve, since the expansion point is reused; for DINO it repeats the data-generation pattern with a smaller count $n_s' \le n_s$. Once data are generated, fitting incurs no further solves---T4S by Riemannian manifold optimization (\secref{sec:trrmgn}, \secref{sec:cauchy_sgd}), \citet{alger.chen.ea.2020} by negligible small least-squares solves, and the neural network and DINO by network training.

The count $n_s$ means different things across methods---random probe directions for the probe-based methods, sampled function evaluations for the sampling methods---and the methods need different amounts of data to reach comparable accuracy. \tableref{tab:training_cost} therefore reports the cost of generating each method's own training data rather than placing all methods at a common $n_s$. Conventional neural networks and DINO target a global regime complementary to the local approximation considered here, and are included for cost context rather than as a direct accuracy comparison.


\ifdefined\algforw\else\newlength{\algforw}\fi

\section{Riemannian Jacobian} 
\label{sec:tt_operations}

The TR-RMGN method (\secref{sec:trrmgn}) and the MC-SGD method (\secref{sec:cauchy_sgd}) both apply the Riemannian manifold least-squares Jacobian $J$ and its transpose $J^T$ to vectors. This section presents sweeping methods for these operations. Full tangent vectors are too large to form, so we work on \emph{gauged representations} $\delta \tuple{V}$ of tangent vectors, on which $J$ and $J^T$ factor as $\jac\circ\Pi$ and $\Pi\circ\jac^T$, where $\Pi$ is a low-cost projector and $\jac$ probes the tangent vector against the training data (\secref{sec:jacobian_gauged}). The methods for $\jac$ and $\jac^T$ (\secref{sec:sweeping_methods}) consist only of array contractions and additions and cost $O(dNn + dnr^2 + Mm)$ per training sample, with the ambient dimension $N$ and tensor order $d$ entering only linearly. Proofs of correctness are given in \appref{app:finite_to_infinite}. Readers using standard core-based optimizers can instead obtain corewise (non-manifold) versions of these algorithms by a substitution (\secref{sec:corewise_jac}).

\subsection{Riemannian Jacobian for gauged representations}
\label{sec:jacobian_gauged}

The inputs to $J$ and outputs from $J^T$ are tangent vectors to the manifold of fixed-rank Tucker tensor trains (T3 manifold), which are too large to form or store explicitly. We therefore represent a tangent vector $\vb{v}$ implicitly by its gauged variations $\delta \tuple{V} = ((\delta \vb{U}_i)_{i=1}^d, (\delta \vb{G}_i)_{i=1}^d)$; see \appref{sec:tensor_train_manifold} for the T3 manifold, this representation, and the associated linear algebra.

Let $\jac$ be the linear map that mirrors $\jaczero$ but acts on variations rather than full tangent vectors, $\jac : \delta \tuple{V} \mapsto \jaczero(\vb{v})$, where $\vb{v}$ is the tangent vector represented by $\delta \tuple{V}$. Let $\Pi$ be the orthogonal projector
$$
\Pi : ((\delta \vb{U}_l)_{l=1}^{j+1}, (\delta \vb{G}_l)_{l=1}^{j+1}) \mapsto ((\widehat{\delta \vb{U}}_l)_{l=1}^{j+1}, (\widehat{\delta \vb{G}}_l)_{l=1}^{j+1}),
$$
defined by
$$
\begin{aligned}
\widehat{\delta \vb{U}}_l &= (\vb{I} - \vb{U}_l\vb{U}_l^T)\delta \vb{U}_l, && l = 1, \dots, j+1, \\
\widehat{\delta \vb{G}}_l &= (\vb{I} - \vb{P}_l^L(\vb{P}_l^L)^T)\delta \vb{G}_l, && l = 1, \dots, j, \\
\widehat{\delta \vb{G}}_{j+1} &= \delta \vb{G}_{j+1},
\end{aligned}
$$
which enforces the gauge conditions \eqref{eq:gauge_condition_tucker} and \eqref{eq:gauge_condition_tt}. The versions of $J$ and $J^T$ acting on gauged tangent vectors are then $\jac \circ \Pi$ and $\Pi \circ \jac^T$.

These factorizations let both fitting methods run entirely on gauged representations, never forming full tangent vectors. The TR-RMGN normal equations \eqref{eq:sqp_normal_eq} become $(\Pi \circ \jac^T \circ \jac \circ \Pi)(\delta \tuple{V}) = (\Pi \circ \jac^T)(\vb{b})$, and applying the CG--Steihaug method to this projected system is equivalent to applying it directly to \eqref{eq:sqp_normal_eq}. In MC-SGD, the stochastic gradient $\vb{g}^{(B)}=(J^{(B)})^T(\vb{b})$ is represented by $\tuple{G}^{(B)} = (\Pi \circ (\mathcal{J}^{(B)})^T)(\vb{b})$, and the quantity $J^{(B)}(\vb{g}^{(B)})$ in the Cauchy steplength by $(\mathcal{J}^{(B)} \circ \Pi)(\tuple{G}^{(B)})$, where $\mathcal{J}^{(B)}$ is the minibatch version of $\jac$.

\subsection{Sweeping methods}
\label{sec:sweeping_methods}

We present methods for applying $\jac^{(s)}$ and $\left(\jac^{(s)}\right)^T$ to vectors---the versions of $\jac$ and $\jac^T$ for a single training sample $s$. The full $\jac$ (or minibatch $\jac^{(B)}$) and its transpose follow by looping or vectorizing over samples. The methods consist solely of array contractions and additions, making them well suited to vectorized GPU execution; padding the Tucker tensor train cores with zeros to make the Tucker and tensor train ranks uniform further improves pipelining.

In the sweeping methods below, a hat ($\widehat{\,\cdot\,}$) marks a base-point quantity, $\delta$ a tangent-vector perturbation, and a tilde ($\widetilde{\,\cdot\,}$) an adjoint quantity arising in the transpose sweep.

The probes $\delta\vb{z}=(\delta\vb{z}_1, \dots, \delta\vb{z}_d)$ of a tangent vector depend on both the probing vectors $\vb{w}_1, \dots, \vb{w}_d$ and the representation $\delta \tuple{V}$. The operator $\jac^{(s)}$ is the partial evaluation $(\vb{w}_1,\dots,\vb{w}_d),\, \delta \tuple{V} \mapsto \delta\vb{z}$ in which the $\vb{w}_i$ are fixed, so its input is $\delta \tuple{V}$ and its output is $\delta\vb{z}$; different $\vb{w}_i$ give different realizations of $\jac^{(s)}$. In our Riemannian fitting procedures the $\vb{w}_i$ take the directionally symmetric form $\vb{x}^{(s)},\dots,\vb{x}^{(s)},\vb{\omega}^{(s)}$, but the methods below do not rely on this.

Because a tangent vector is represented by a doubled-rank Tucker tensor train (\appref{app:tangent_vectors_doubled_ranks}), we first solve the plain probing problem for a Tucker tensor train (\secref{sec:tt_actions}), then specialize it to probe a tangent vector to derive the application of $\jac^{(s)}$ (\secref{sec:perturbed_action}), and finally apply the transpose $\left(\jac^{(s)}\right)^T$ (\secref{sec:transpose_of_action_map}).

\subsubsection{Probing a Tucker tensor train}
\label{sec:tt_actions}

\begin{figure}
\centering
\begin{tikzpicture}[scale=0.75, every node/.style={scale=0.75}]

\pgfmathsetmacro{\dx}{2.5}
\pgfmathsetmacro{\dya}{2.5}
\pgfmathsetmacro{\dyb}{2.5}
\pgfmathsetmacro{\munuoffsety}{-0.15}
\pgfmathsetmacro{\munugapy}{0.3}
\pgfmathsetmacro{\munuspacerx}{0.6}

    \node[draw, rounded corners, minimum size=0.5cm, inner sep=0.0, line width=0.3mm] (oneL) at (-1*\dx,0) {$1$};
    \node[draw, rounded corners, minimum size=0.5cm, inner sep=0.0, line width=0.3mm] (oneR) at (4*\dx,0) {$1$};

    \node[draw, rounded corners, minimum size=0.8cm, inner sep=0.0, line width=0.3mm] (G1) at (0*\dx,0) {$\vb{G}_1$};
    \node[draw, rounded corners, minimum size=0.8cm, inner sep=0.0, line width=0.3mm] (G2) at (1*\dx,0) {$\vb{G}_2$};
    \node[draw, rounded corners, minimum size=0.8cm, inner sep=0.0, line width=0.3mm] (G3) at (2*\dx,0) {$\vb{G}_3$};
    \node[draw, rounded corners, minimum size=0.8cm, inner sep=0.0, line width=0.3mm] (G4) at (3*\dx,0) {$\vb{G}_4$};

    \draw[line width=0.3mm] (oneL) -- (G1) -- (G2) -- (G3) -- (G4) -- (oneR);

    \node[draw, rounded corners, minimum size=0.8cm, inner sep=0.0, line width=0.3mm] (B1) at (0*\dx,-\dya) {$\vb{U}_1$};
    \node[draw, rounded corners, minimum size=0.8cm, inner sep=0.0, line width=0.3mm] (B2) at (1*\dx,-\dya) {$\vb{U}_2$};
    \node[draw, rounded corners, minimum size=0.8cm, inner sep=0.0, line width=0.3mm] (B3) at (2*\dx,-\dya) {$\vb{U}_3$};
    \node[draw, rounded corners, minimum size=0.8cm, inner sep=0.0, line width=0.3mm] (B4) at (3*\dx,-\dya) {$\vb{U}_4$};
    
    \node[draw, rounded corners, minimum size=0.6cm, inner sep=0.0, line width=0.3mm] (u1) at (0*\dx,-\dya-\dyb) {$\vb{w}_1$};
    \node[draw, rounded corners, minimum size=0.6cm, inner sep=0.0, line width=0.3mm] (u2) at (1*\dx,-\dya-\dyb) {$\vb{w}_2$};
    \node[draw, rounded corners, minimum size=0.6cm, inner sep=0.0, line width=0.3mm] (u3) at (2*\dx,-\dya-\dyb) {$\vb{w}_3$};
    \node[draw, rounded corners, minimum size=0.6cm, inner sep=0.0, line width=0.3mm] (u4) at (3*\dx,-\dya-\dyb) {$\vb{w}_4$};
			
    \draw[line width=0.3mm] (G1) -- (B1) -- (u1);
    \draw[line width=0.3mm] (G2) -- (B2) -- (u2);
    \draw[line width=0.3mm] (G3) -- (B3) -- (u3);
    \draw[line width=0.3mm] (G4) -- (B4) -- (u4);

    \draw[line width=0.3mm, arrows={-Stealth[harpoon]}] (-1*\dx+\munuspacerx,\munuoffsety+\munugapy) to node[midway, above]{${\pushleft}_0$} (0*\dx-\munuspacerx,\munuoffsety+\munugapy);
    \draw[line width=0.3mm, arrows={-Stealth[harpoon]}] (0*\dx-\munuspacerx,\munuoffsety) to node[midway, below]{${\pushright}_0$} (-1*\dx+\munuspacerx,\munuoffsety);
    
    \draw[line width=0.3mm, arrows={-Stealth[harpoon]}] (0*\dx+\munuspacerx,\munuoffsety+\munugapy) to node[midway, above]{${\pushleft}_1$} (1*\dx-\munuspacerx,\munuoffsety+\munugapy);
    \draw[line width=0.3mm, arrows={-Stealth[harpoon]}] (1*\dx-\munuspacerx,\munuoffsety) to node[midway, below]{${\pushright}_1$} (0*\dx+\munuspacerx,\munuoffsety);

    \draw[line width=0.3mm, arrows={-Stealth[harpoon]}] (1*\dx+\munuspacerx,\munuoffsety+\munugapy) to node[midway, above]{${\pushleft}_2$} (2*\dx-\munuspacerx,\munuoffsety+\munugapy);
    \draw[line width=0.3mm, arrows={-Stealth[harpoon]}] (2*\dx-\munuspacerx,\munuoffsety) to node[midway, below]{${\pushright}_2$} (1*\dx+\munuspacerx,\munuoffsety);

    \draw[line width=0.3mm, arrows={-Stealth[harpoon]}] (2*\dx+\munuspacerx,\munuoffsety+\munugapy) to node[midway, above]{${\pushleft}_3$} (3*\dx-\munuspacerx,\munuoffsety+\munugapy);
    \draw[line width=0.3mm, arrows={-Stealth[harpoon]}] (3*\dx-\munuspacerx,\munuoffsety) to node[midway, below]{${\pushright}_3$} (2*\dx+\munuspacerx,\munuoffsety);

    \draw[line width=0.3mm, arrows={-Stealth[harpoon]}] (3*\dx+\munuspacerx,\munuoffsety+\munugapy) to node[midway, above]{${\pushleft}_4$} (4*\dx-\munuspacerx,\munuoffsety+\munugapy);
    \draw[line width=0.3mm, arrows={-Stealth[harpoon]}] (4*\dx-\munuspacerx,\munuoffsety) to node[midway, below]{${\pushright}_4$} (3*\dx+\munuspacerx,\munuoffsety);

    \draw[line width=0.3mm, arrows={-Stealth[harpoon]}] (0*\dx+\munuoffsety+\munugapy,0-\munuspacerx) to node[midway, right]{${\vb{\eta}}_1$} (0*\dx+\munuoffsety+\munugapy,-\dya+\munuspacerx);
    \draw[line width=0.3mm, arrows={-Stealth[harpoon]}] (0*\dx+\munuoffsety,-\dya+\munuspacerx) to node[midway, left]{${\vb{\xi}}_1$} (0*\dx+\munuoffsety,0-\munuspacerx);

    \draw[line width=0.3mm, arrows={-Stealth[harpoon]}] (1*\dx+\munuoffsety+\munugapy,0-\munuspacerx) to node[midway, right]{${\vb{\eta}}_2$} (1*\dx+\munuoffsety+\munugapy,-\dya+\munuspacerx);
    \draw[line width=0.3mm, arrows={-Stealth[harpoon]}] (1*\dx+\munuoffsety,-\dya+\munuspacerx) to node[midway, left]{${\vb{\xi}}_2$} (1*\dx+\munuoffsety,0-\munuspacerx);

    \draw[line width=0.3mm, arrows={-Stealth[harpoon]}] (2*\dx+\munuoffsety+\munugapy,0-\munuspacerx) to node[midway, right]{${\vb{\eta}}_3$} (2*\dx+\munuoffsety+\munugapy,-\dya+\munuspacerx);
    \draw[line width=0.3mm, arrows={-Stealth[harpoon]}] (2*\dx+\munuoffsety,-\dya+\munuspacerx) to node[midway, left]{${\vb{\xi}}_3$} (2*\dx+\munuoffsety,0-\munuspacerx);

    \draw[line width=0.3mm, arrows={-Stealth[harpoon]}] (3*\dx+\munuoffsety+\munugapy,0-\munuspacerx) to node[midway, right]{${\vb{\eta}}_4$} (3*\dx+\munuoffsety+\munugapy,-\dya+\munuspacerx);
    \draw[line width=0.3mm, arrows={-Stealth[harpoon]}] (3*\dx+\munuoffsety,-\dya+\munuspacerx) to node[midway, left]{${\vb{\xi}}_4$} (3*\dx+\munuoffsety,0-\munuspacerx);

    \draw[line width=0.3mm, arrows={-Stealth[harpoon]}] (0*\dx+\munuoffsety+\munugapy,-\dya-\munuspacerx) to node[midway, right]{$\vb{z}_1$} (0*\dx+\munuoffsety+\munugapy,-\dya-\dyb+\munuspacerx);
    \draw[line width=0.3mm, arrows={-Stealth[harpoon]}] (0*\dx+\munuoffsety,-\dya-\dyb+\munuspacerx) to node[midway, left]{$\vb{w}_1$} (0*\dx+\munuoffsety,-\dya-\munuspacerx);

    \draw[line width=0.3mm, arrows={-Stealth[harpoon]}] (1*\dx+\munuoffsety+\munugapy,-\dya-\munuspacerx) to node[midway, right]{$\vb{z}_2$} (1*\dx+\munuoffsety+\munugapy,-\dya-\dyb+\munuspacerx);
    \draw[line width=0.3mm, arrows={-Stealth[harpoon]}] (1*\dx+\munuoffsety,-\dya-\dyb+\munuspacerx) to node[midway, left]{$\vb{w}_2$} (1*\dx+\munuoffsety,-\dya-\munuspacerx);

    \draw[line width=0.3mm, arrows={-Stealth[harpoon]}] (2*\dx+\munuoffsety+\munugapy,-\dya-\munuspacerx) to node[midway, right]{$\vb{z}_3$} (2*\dx+\munuoffsety+\munugapy,-\dya-\dyb+\munuspacerx);
    \draw[line width=0.3mm, arrows={-Stealth[harpoon]}] (2*\dx+\munuoffsety,-\dya-\dyb+\munuspacerx) to node[midway, left]{$\vb{w}_3$} (2*\dx+\munuoffsety,-\dya-\munuspacerx);

    \draw[line width=0.3mm, arrows={-Stealth[harpoon]}] (3*\dx+\munuoffsety+\munugapy,-\dya-\munuspacerx) to node[midway, right]{$\vb{z}_4$} (3*\dx+\munuoffsety+\munugapy,-\dya-\dyb+\munuspacerx);
    \draw[line width=0.3mm, arrows={-Stealth[harpoon]}] (3*\dx+\munuoffsety,-\dya-\dyb+\munuspacerx) to node[midway, left]{$\vb{w}_4$} (3*\dx+\munuoffsety,-\dya-\munuspacerx);

\end{tikzpicture}
\caption{Edge variables associated with probes of a Tucker tensor train. The contracted subnetwork associated with an edge variable consists of everything ``behind'' the arrow.
}
\label{fig:tt_pushthroughs}
\end{figure}

\algoref{alg:probe_t3} computes all $d$ probes of a Tucker tensor train by sweeping contractions, at cost $O(dNn + dnr^2 + Mm)$ (\thmref{thm:probe_t3}). Let $\tuple{T}=((\vb{U}_i)_{i=1}^d,(\vb{G}_i)_{i=1}^d)$ represent a scalar-valued multilinear function $T$, and let $\vb{w}_1,\dots,\vb{w}_d$ be probing vectors. We treat the output mode as an additional scalar tensor mode; the vector-valued case follows by the scalar/vector correspondence established in \secref{sec:tensor_correspondences}. The $i$th probe is the vector $\vb{z}_i$ representing the linear functional $\vb{t} \mapsto T(\vb{w}_1, \dots, \vb{t}, \dots, \vb{w}_d)$:
$$
\vb{z}_i^T \vb{t} = T(\vb{w}_1, \dots, \vb{w}_{i-1}, \vb{t}, \vb{w}_{i+1}, \dots, \vb{w}_d)
\qquad \text{for all } \vb{t}.
$$
The sweeps are organized around the \emph{edge variables}
\begin{align*}
\vb{\xi}_i :=& \vb{U}_i^T \vb{w}_i, \\
\pushleft_{i}^T :=& G_{1}(\vb{\xi}_1) G_{2}(\vb{\xi}_2) \dots G_{i}(\vb{\xi}_i), \\
\pushright_{i} :=& G_{i+1}(\vb{\xi}_{i+1})G_{i+2}(\vb{\xi}_{i+2}) \dots G_{d}(\vb{\xi}_d), \\
\eta_i(\vb{s}):=& \pushleft_{i-1}^T G_i(\vb{s}) \pushright_i,
\end{align*}
where $G_i(\vb{s})$ is the matrix formed by contracting the Tucker (middle) index of $\vb{G}_i$ against $\vb{s}$. Each edge variable is a partial contraction of the network: contracting $\tuple{T}$ against all $\vb{w}_i$ reduces it to a scalar, and each internal edge splits the network into two subnetworks whose separate contractions give the vectors associated with that edge (\figref{fig:tt_pushthroughs}). For the edge between $\vb{G}_i$ and $\vb{G}_{i+1}$ these are $\pushleft_i$ and $\pushright_i$; for the edge between $\vb{U}_i$ and $\vb{G}_i$ they are $\vb{\xi}_i$ (below) and $\vb{\eta}_i$ (above); and for the edge between $\vb{w}_i$ and $\vb{U}_i$, contracting the network above the edge gives the $i$th probe $\vb{z}_i$. \algoref{alg:probe_t3} evaluates these by contracting up for $\vb{\xi}_i$, sweeping left-to-right for $\pushleft_i$, right-to-left for $\pushright_i$, centrally for $\vb{\eta}_i$, and down for the probes $\vb{z}_i$.

\begin{algorithm}[t]
\caption{Probing a Tucker tensor train.}
\label{alg:probe_t3}
\settowidth{\algforw}{$\pushleft_{i}^T \gets \pushleft_{i-1}^T G_i(\vb{\xi}_i)$\quad}
\begin{algorithmic}[1]
\Require Tucker tensor train $\tuple{T}=((\vb{U}_i)_{i=1}^d,(\vb{G}_i)_{i=1}^d)$; probing vectors $\vb{w}_1,\dots,\vb{w}_d$.
\Ensure Probes $\vb{z}_1,\dots,\vb{z}_d$.
\State\label{ln:t3-up} \makebox[\algforw][l]{$\vb{\xi}_i \gets \vb{U}_i^T \vb{w}_i$}\textbf{for } $i=1,\dots,d$ \Comment{contract up}
\State\label{ln:t3-bdy} $\pushleft_{0}^T \gets 1$, \ $\pushright_{d} \gets 1$ \Comment{boundary terms}
\State\label{ln:t3-left} \makebox[\algforw][l]{$\pushleft_{i}^T \gets \pushleft_{i-1}^T G_i(\vb{\xi}_i)$}\textbf{for } $i=1,\dots,d-1$ \Comment{left sweep}
\State\label{ln:t3-right} \makebox[\algforw][l]{$\pushright_{i} \gets G_i(\vb{\xi}_i)\,\pushright_{i+1}$}\textbf{for } $i=d-1,\dots,1$ \Comment{right sweep}
\State\label{ln:t3-central} \makebox[\algforw][l]{$\vb{\eta}_i \gets \pushleft_{i-1}^T \vb{G}_i \pushright_i$}\textbf{for } $i=1,\dots,d$ \Comment{central contraction}
\State\label{ln:t3-down} \makebox[\algforw][l]{$\vb{z}_i \gets \vb{U}_i \vb{\eta}_i$}\textbf{for } $i=1,\dots,d$ \Comment{contract down}
\State \textbf{return} $\vb{z}_1,\dots,\vb{z}_d$
\end{algorithmic}
\end{algorithm}

\begin{theorem}[Probing a Tucker tensor train]
\label{thm:probe_t3}
\algoref{alg:probe_t3} returns the probes of $\tuple{T}$ on $\vb{w}_1,\dots,\vb{w}_d$; that is, its outputs satisfy
$$
\vb{z}_i^T \vb{t} = T(\vb{w}_1, \dots, \vb{w}_{i-1}, \vb{t}, \vb{w}_{i+1}, \dots, \vb{w}_d)
\qquad \text{for all } \vb{t} \text{ and } i=1,\dots,d.
$$
If $\tuple{T}$ has shape $(N,\dots,N,M)$, Tucker ranks $(n,\dots,n,m)$, and tensor train ranks $(1,r,\dots,r,1)$, then \algoref{alg:probe_t3} requires $O(dNn + d n r^2 + Mm)$ operations.
\end{theorem}

The proof of \thmref{thm:probe_t3} is given in \appref{app:finite_to_infinite}.

\subsubsection{Probing a tangent vector}
\label{sec:perturbed_action}

Applying $\jac^{(s)}$ to $\delta \tuple{V}$ means computing the probes $\delta\vb{z} = (\delta\vb{z}_i)_{i=1}^d$ of the tangent vector $\vb{v}$ which $\delta \tuple{V}$ represents. This computation proceeds in two stages, \algoref{alg:probe_basepoint} and \algoref{alg:probe_tangent}, each at cost $O(dNn + dnr^2 + Mm)$ (\thmref{thm:probe_tangent}). Fix orthogonal representations \eqref{eq:p_rep_Btilde} and \eqref{eq:p_rep_Gtilde} of the base point $\vb{p} \in \mathcal{M}_{\vb{n},\vb{r}}$ (\appref{sec:tensor_train_manifold}), and let $\delta \tuple{V} = ((\delta \vb{U}_i)_{i=1}^d, (\delta \vb{G}_i)_{i=1}^d)$ represent $\vb{v} \in T_\vb{p}\mathcal{M}_{\vb{n},\vb{r}}$ with respect to them.

\algoref{alg:probe_basepoint} is the probing sweep of \secref{sec:tt_actions} applied to the base point, using its gauge-appropriate cores $\vb{P}_i$ (left sweep), $\vb{Q}_i$ (right sweep), and $\vb{O}_i$ (central contraction) in place of $\vb{G}_i$. It returns the base-point probes $\widehat{\vb{z}}_i = \vb{U}_i \widehat{\vb{\eta}}_i$ and the base-point edge variables $\widehat{\vb{\xi}}_i$, $\widehat{\pushleft}_i$, $\widehat{\pushright}_i$, $\widehat{\vb{\eta}}_i$.

\algoref{alg:probe_tangent} then computes the tangent probes. The tangent vector $\vb{v}$ is represented by a doubled-rank Tucker tensor train (\appref{app:tangent_vectors_doubled_ranks}), so the internal edge variables of \secref{sec:tt_actions} double in dimension, each splitting into a base-point component and a perturbation component:
\begin{equation*}
\vb{\xi}_i \to (\widehat{\vb{\xi}}_i, \delta \vb{\xi}_i), \quad
\pushleft_i \to (\ppsleft_i, \widehat{\pushleft}_i), \quad
\pushright_i \to (\ppsright_i, \widehat{\pushright}_i), \quad
\vb{\eta}_i \to (\delta \vb{\eta}_i, \widehat{\vb{\eta}}_i).
\end{equation*}
The base-point components are supplied by \algoref{alg:probe_basepoint}; \algoref{alg:probe_tangent} computes the perturbation components $\delta \vb{\xi}_i$, $\ppsleft_i$, $\ppsright_i$, $\delta \vb{\eta}_i$ and assembles the tangent probes $\delta\vb{z}=\jac^{(s)}(\delta\tuple{V})$.

\begin{algorithm}[t]
\caption{Probing a base point via its orthogonal representations.}
\label{alg:probe_basepoint}
\settowidth{\algforw}{$\widehat{\pushleft}_{i}^T \gets \widehat{\pushleft}_{i-1}^T P_i(\widehat{\vb{\xi}}_i)$\quad}
\begin{algorithmic}[1]
\Require Base-point cores $(\vb{U}_i)_{i=1}^d$ and $(\vb{P}_i,\vb{Q}_i,\vb{O}_i)_{i=1}^d$ of the representations \eqref{eq:p_rep_Btilde}--\eqref{eq:p_rep_Gtilde}; probing vectors $\vb{w}_1,\dots,\vb{w}_d$.
\Ensure Base-point probes $\widehat{\vb{z}}_1,\dots,\widehat{\vb{z}}_d$ and edge variables $\widehat{\vb{\xi}}_i,\widehat{\pushleft}_i,\widehat{\pushright}_i,\widehat{\vb{\eta}}_i$.
\State\label{ln:bp-up} \makebox[\algforw][l]{$\widehat{\vb{\xi}}_i \gets \vb{U}_i^T \vb{w}_i$}\textbf{for } $i=1,\dots,d$ \Comment{contract up}
\State\label{ln:bp-bdy} $\widehat{\pushleft}_0^T \gets 1$, \ $\widehat{\pushright}_d \gets 1$ \Comment{boundary terms}
\State\label{ln:bp-left} \makebox[\algforw][l]{$\widehat{\pushleft}_i^T \gets \widehat{\pushleft}_{i-1}^T P_i(\widehat{\vb{\xi}}_i)$}\textbf{for } $i=1,\dots,d-1$ \Comment{left sweep}
\State\label{ln:bp-right} \makebox[\algforw][l]{$\widehat{\pushright}_i \gets Q_i(\widehat{\vb{\xi}}_i)\,\widehat{\pushright}_{i+1}$}\textbf{for } $i=d-1,\dots,1$ \Comment{right sweep}
\State\label{ln:bp-central} \makebox[\algforw][l]{$\widehat{\vb{\eta}}_i \gets \widehat{\pushleft}_{i-1}^T \vb{O}_i \widehat{\pushright}_i$}\textbf{for } $i=1,\dots,d$ \Comment{central contraction}
\State\label{ln:bp-down} \makebox[\algforw][l]{$\widehat{\vb{z}}_i \gets \vb{U}_i \widehat{\vb{\eta}}_i$}\textbf{for } $i=1,\dots,d$ \Comment{contract down}
\State \textbf{return} $\widehat{\vb{z}}_1,\dots,\widehat{\vb{z}}_d$ and $\widehat{\vb{\xi}}_i,\widehat{\pushleft}_i,\widehat{\pushright}_i,\widehat{\vb{\eta}}_i$
\end{algorithmic}
\end{algorithm}

\begin{algorithm}[t]
\caption{Probing a tangent vector: application of $\jac^{(s)}$.}
\label{alg:probe_tangent}
\settowidth{\algforw}{$\delta\vb{\eta}_i \gets \ppsleft_{i-1}^T \rightcore_i \widehat{\pushright}_i + \widehat{\pushleft}_{i-1} \leftcore_i \ppsright_i + \widehat{\pushleft}_{i-1}^T \delta\vb{G}_i \widehat{\pushright}_i$\quad}
\begin{algorithmic}[1]
\Require Base edge variables $\widehat{\vb{\xi}}_i,\widehat{\pushleft}_i,\widehat{\pushright}_i,\widehat{\vb{\eta}}_i$ (\algoref{alg:probe_basepoint}); base-point cores $(\vb{U}_i)_{i=1}^d$ and $(\vb{P}_i,\vb{Q}_i,\vb{O}_i)_{i=1}^d$; variation $\delta\tuple{V}=((\delta\vb{U}_i)_{i=1}^d,(\delta\vb{G}_i)_{i=1}^d)$; probing vectors $\vb{w}_1,\dots,\vb{w}_d$.
\Ensure Tangent probes $\delta\vb{z}=(\delta\vb{z}_1,\dots,\delta\vb{z}_d)=\jac^{(s)}(\delta\tuple{V})$.
\State\label{ln:tan-up} \makebox[\algforw][l]{$\delta\vb{\xi}_i \gets \delta\vb{U}_i^T \vb{w}_i$}\textbf{for } $i=1,\dots,d$ \Comment{contract up}
\State\label{ln:tan-bdy} $\ppsleft_0^T \gets 0$, \ $\ppsright_d \gets 0$ \Comment{boundary terms}
\State\label{ln:tan-left} $\ppsleft_i^T \gets \ppsleft_{i-1}^T Q_i(\widehat{\vb{\xi}}_i) + \widehat{\pushleft}_{i-1}^T \delta G_i(\widehat{\vb{\xi}}_i) + \widehat{\pushleft}_{i-1}^T O_i(\delta\vb{\xi}_i)$ \quad \textbf{for } $i=1,\dots,d-1$ \Comment{left sweep}
\State\label{ln:tan-right} $\ppsright_i \gets \delta G_{i+1}(\widehat{\vb{\xi}}_{i+1})\widehat{\pushright}_{i+1} + O_{i+1}(\delta\vb{\xi}_{i+1})\widehat{\pushright}_{i+1} + P_{i+1}(\widehat{\vb{\xi}}_{i+1})\ppsright_{i+1}$ \quad \textbf{for } $i=d-1,\dots,1$ \Comment{right sweep}
\State\label{ln:tan-central} \makebox[\algforw][l]{$\delta\vb{\eta}_i \gets \ppsleft_{i-1}^T \rightcore_i \widehat{\pushright}_i + \widehat{\pushleft}_{i-1}^T \leftcore_i \ppsright_i + \widehat{\pushleft}_{i-1}^T \delta\vb{G}_i \widehat{\pushright}_i$}\textbf{for } $i=1,\dots,d$ \Comment{central contraction}
\State\label{ln:tan-down} \makebox[\algforw][l]{$\delta\vb{z}_i \gets \vb{U}_i \delta\vb{\eta}_i + \delta\vb{U}_i \widehat{\vb{\eta}}_i$}\textbf{for } $i=1,\dots,d$ \Comment{contract down}
\State \textbf{return} $\delta\vb{z}=(\delta\vb{z}_1,\dots,\delta\vb{z}_d)$
\end{algorithmic}
\end{algorithm}

\begin{theorem}[Probing the base point and the tangent vector]
\label{thm:probe_tangent}
\algoref{alg:probe_basepoint} returns the probes $\widehat{\vb{z}}$ of the base point $\vb{p}$ on $\vb{w}_1,\dots,\vb{w}_d$, and \algoref{alg:probe_tangent} returns $\delta\vb{z}=\jac^{(s)}(\delta\tuple{V})$, the probes of the tangent vector $\vb{v}$ represented by $\delta\tuple{V}$. If $\vb{p}$ has shape $(N,\dots,N,M)$, Tucker ranks $(n,\dots,n,m)$, and tensor train ranks $(1,r,\dots,r,1)$, then each of Algorithms~\ref{alg:probe_basepoint} and~\ref{alg:probe_tangent} requires $O(dNn + d n r^2 + Mm)$ operations.
\end{theorem}

The proof of \thmref{thm:probe_tangent} is given in \appref{app:finite_to_infinite}.

The base-point edge variables $\widehat{\vb{\xi}}_i$, $\widehat{\pushleft}_i$, $\widehat{\pushright}_i$, $\widehat{\vb{\eta}}_i$ depend on the base point $\vb{p}$ and the sample $s$, but not on the tangent vector, so they can be precomputed once per base point and reused across all tangent vectors there. Because they depend on $s$, this requires storing them for every sample; when training data are plentiful, recomputing them per tangent vector trades computation for reduced memory.

\subsubsection{Transpose of tangent vector to probes map}
\label{sec:transpose_of_action_map}

Let $\widetilde{\vb{z}}=(\widetilde{\vb{z}}_1, \dots, \widetilde{\vb{z}}_d)$ be a tuple of probe-like perturbation vectors (e.g.\ residuals). \algoref{alg:transpose_tangent} applies the transpose map $\widetilde{\vb{z}} \mapsto \left(\jac^{(s)}\right)^T(\widetilde{\vb{z}})$ by an adjoint sweep over the edge variables of \secref{sec:perturbed_action}, again at cost $O(dNn + dnr^2 + Mm)$ (\thmref{thm:transpose_tangent}). In \algoref{alg:transpose_tangent}, $\otimes$ is the tensor product, $(\vb{a} \otimes \vb{b} \otimes \vb{c})[j,k,l]=\vb{a}[j] \vb{b}[k] \vb{c}[l]$, not the Kronecker product $\kron$.

\begin{algorithm}[t]
\caption{Transpose of the tangent vector to probes map: application of $\left(\jac^{(s)}\right)^T$.}
\label{alg:transpose_tangent}
\settowidth{\algforw}{$\incrleft_i \gets Q_{i+1}(\widetilde{\delta\vb{\eta}}_{i+1})\widehat{\pushright}_{i+1} + Q_{i+1}(\widehat{\vb{\xi}}_{i+1})\incrleft_{i+1}$\quad}
\begin{algorithmic}[1]
\Require Base edge variables $\widehat{\vb{\xi}}_i,\widehat{\pushleft}_i,\widehat{\pushright}_i,\widehat{\vb{\eta}}_i$ (\algoref{alg:probe_basepoint}); base-point cores $(\vb{U}_i)_{i=1}^d$ and $(\vb{P}_i,\vb{Q}_i,\vb{O}_i)_{i=1}^d$ of the representations \eqref{eq:p_rep_Btilde}--\eqref{eq:p_rep_Gtilde}; probing vectors $\vb{w}_1,\dots,\vb{w}_d$; perturbation vectors $\widetilde{\vb{z}}_1,\dots,\widetilde{\vb{z}}_d$.
\Ensure $\left(\jac^{(s)}\right)^T(\widetilde{\vb{z}}) = ((\widetilde{\delta\vb{U}}_i)_{i=1}^d,(\widetilde{\delta\vb{G}}_i)_{i=1}^d)$.
\State\label{ln:tr-up} \makebox[\algforw][l]{$\widetilde{\delta\vb{\eta}}_i \gets \vb{U}_i^T \widetilde{\vb{z}}_i$}\textbf{for } $i=1,\dots,d$ \Comment{contract up}
\State\label{ln:tr-bdy} $\incrright_0^T \gets 0$, \ $\incrleft_d \gets 0$ \Comment{boundary terms}
\State\label{ln:tr-left} \makebox[\algforw][l]{$\incrright_i^T \gets \incrright_{i-1}^T P_i(\widehat{\vb{\xi}}_i) + \widehat{\pushleft}_{i-1}^T P_i(\widetilde{\delta\vb{\eta}}_i)$}\textbf{for } $i=1,\dots,d-1$ \Comment{left sweep}
\State\label{ln:tr-right} \makebox[\algforw][l]{$\incrleft_i \gets Q_{i+1}(\widetilde{\delta\vb{\eta}}_{i+1})\widehat{\pushright}_{i+1} + Q_{i+1}(\widehat{\vb{\xi}}_{i+1})\incrleft_{i+1}$}\textbf{for } $i=d-1,\dots,1$ \Comment{right sweep}
\State\label{ln:tr-central} \makebox[\algforw][l]{$\widetilde{\delta\vb{\xi}}_i \gets \incrright_{i-1}^T \vb{O}_i \widehat{\pushright}_i + \widehat{\pushleft}_{i-1}^T \vb{O}_i \incrleft_i$}\textbf{for } $i=1,\dots,d$ \Comment{central contraction}
\State\label{ln:tr-assembleU} \makebox[\algforw][l]{$\widetilde{\delta\vb{U}}_i \gets \widetilde{\vb{z}}_i \widehat{\vb{\eta}}_i^T + \vb{w}_i \widetilde{\delta\vb{\xi}}_i^T$}\textbf{for } $i=1,\dots,d$ \Comment{assemble $\widetilde{\delta\vb{U}}_i$}
\State\label{ln:tr-assembleG} $\widetilde{\delta\vb{G}}_i \gets \incrright_{i-1}^T \otimes \widehat{\vb{\xi}}_i \otimes \widehat{\pushright}_i + \widehat{\pushleft}_{i-1}^T \otimes \widehat{\vb{\xi}}_i \otimes \incrleft_i + \widehat{\pushleft}_{i-1}^T \otimes \widetilde{\delta\vb{\eta}}_i \otimes \widehat{\pushright}_i$ \quad \textbf{for } $i=1,\dots,d$ \Comment{assemble $\widetilde{\delta\vb{G}}_i$}
\State \textbf{return} $((\widetilde{\delta\vb{U}}_i)_{i=1}^d,(\widetilde{\delta\vb{G}}_i)_{i=1}^d)$
\end{algorithmic}
\end{algorithm}

\begin{theorem}[Transpose of tangent vector to probes map]
\label{thm:transpose_tangent}
\algoref{alg:transpose_tangent} returns $\left(\jac^{(s)}\right)^T(\widetilde{\vb{z}})$, the gradient of $\delta\tuple{V} \mapsto \langle \widetilde{\vb{z}}, \jac^{(s)}(\delta\tuple{V})\rangle$ with respect to $\delta\tuple{V}$. If $\vb{p}$ has shape $(N,\dots,N,M)$, Tucker ranks $(n,\dots,n,m)$, and tensor train ranks $(1,r,\dots,r,1)$, then \algoref{alg:transpose_tangent} requires $O(dNn + d n r^2 + Mm)$ operations.
\end{theorem}

The proof of \thmref{thm:transpose_tangent} is given in \appref{app:finite_to_infinite}.

\subsection{Corewise (non-manifold) simplification}
\label{sec:corewise_jac}

Readers using the manifold methods of this paper can skip this subsection. Standard optimizers such as Adam and L-BFGS treat the Tucker tensor train cores $((\vb{U}_i)_{i=1}^d, (\vb{G}_i)_{i=1}^d)$ as independent optimization variables. This has no intrinsic meaning from the manifold perspective, where the optimization variable is the full tensor and the cores merely provide a representation: individual core entries of the gradient depend on the base-point representation, and combining gradients from different tangent spaces (as in L-BFGS) is ill-defined. For optimizers that nonetheless require it, the corewise (non-manifold) Jacobian and its transpose follow from the manifold sweeping methods by a substitution.

Because a tensor network is multilinear in its cores, perturbing the cores by a small amount yields a sum of tensor trains, each agreeing with the original except for one core replaced by its perturbation. This sum has the form of the tangent-vector sum \eqref{eq:t3_tangent_sum_lru} under the substitutions
\begin{equation*}
\vb{P}_i \rightarrow \vb{G}_i, \qquad
\vb{Q}_i \rightarrow \vb{G}_i, \qquad
\vb{O}_i \rightarrow \vb{G}_i,
\end{equation*}
with the basis cores $\vb{U}_i$ no longer required to be orthogonal. Making these substitutions throughout \secref{sec:perturbed_action} and \secref{sec:transpose_of_action_map}, the $\delta\vb{z}$ computed by \algoref{alg:probe_tangent} is the corewise Jacobian applied to the corewise tangent vector $((\delta \vb{U}_i), (\delta \vb{G}_i))$, and the $((\widetilde{\delta \vb{U}}_i)_{i=1}^d, (\widetilde{\delta \vb{G}}_i)_{i=1}^d)$ computed by \algoref{alg:transpose_tangent} is the corewise Jacobian transpose applied to $\widetilde{\vb{z}}$.


\section{Error bounds}
\label{sec:theory}

In this section, we prove that $D^j \ptoW(0)$ may
be approximated by a Tucker tensor train, and provide upper bounds on the Tucker and tensor train ranks based on the induced norm of $D^j \pto(\paramzero)$ and the spectral properties of $\sqrtcov$.  The basic argument goes as follows. 

First, we show that finite-dimensional tensors that are preconditioned by matrices with decaying eigenvalues in all but one index may be approximated by tensor trains (\propref{thm:preconditioned_are_tt}). 
This is done with a peeling process that constructs tensor train cores one at a time. The crux of the argument is the proof of \lemref{lem:helper_for_tt_thm}, which shows how to construct one tensor train core by peeling off the last two modes of the tensor. 
Then, we show that approximately symmetric tensor trains may be approximated by Tucker tensor trains (\lemref{lem:sym_tt_to_ttt}). 
\lemref{lem:sym_tt_to_ttt} is combined with \propref{thm:preconditioned_are_tt} to show that finite-dimensional symmetric preconditioned tensors may be approximated by Tucker tensor trains (\propref{thm:finite_ttt}). 
Finally, we generalize the results in \propref{thm:finite_ttt} from finite dimensions to infinite dimensions in \thmref{thm:ttt_main_thm}, and combine \thmref{thm:ttt_main_thm} with a standard hyperbolic cross result to provide upper bounds on the Tucker and tensor train ranks in the case where the eigenvalues of $\sqrtcov$ decay according to a power law (\corref{cor:error_vs_rank_poly_decay}).


Two caveats must be made regarding this theory. First, these are representational guarantees: they establish the existence of accurate low-rank Tucker tensor train approximants of the covariance-whitened derivative tensors, but the performance of the fitting algorithms introduces an additional source of error which we do not consider here. In particular, the results do not by themselves imply finite-sample recovery of the derivative tensors from finitely many random probes, nor convergence of the nonconvex fitting algorithms; these statistical and optimization questions are left for future work. Although the TR-RMGN fitting algorithm can be shown to converge to a local minimizer, little can be said about convergence to a global minimizer because optimization of tensors is a well-known difficult problem. Second, these bounds only use spectral decay of the covariance operator and boundedness of the derivative tensors, and are therefore highly conservative. Bounds that also consider intrinsic low rank structure of the derivative tensors themselves would provide sharper estimates that better reflect true performance of the method. However, such bounds cannot be given in general, because they depend on the structure of the derivative tensors which is problem-specific.

\subsection{Tensor train approximation in finite dimensions}
\label{sec:tt_finite_dim}

\begin{prop}[Preconditioned multilinear maps admit tensor-train approximation]
\label{thm:preconditioned_are_tt}
    Let $\vb{B} \in \mathbb{R}^{N_1 \times \dots \times N_k \times M}$, let $B:\mathbb{R}^{N_1} \times \dots \times \mathbb{R}^{N_k} \rightarrow \mathbb{R}^{M}$ denote the vector-valued multilinear function corresponding to $\vb{B}$, let $\{\vb{C}_i \in \mathbb{R}^{N_i \times N_i}\}_{i=1}^k$ be symmetric positive-semidefinite matrices, let $F$ be the multilinear function
    \begin{equation*}
        F(\vb{x}_1, \dots, \vb{x}_{k}) := B(\vb{C}_1 \vb{x}_1, \dots, \vb{C}_{k} \vb{x}_{k}),
    \end{equation*}
    and let $\vb{F}$ be the array corresponding to $F$. 
    Further, let $\gamma_{j,1} \ge \gamma_{j,2} \ge \dots \ge \gamma_{j,N_1 \dots N_{j}}$ denote the eigenvalues of $\vb{C}_1 \kron \dots \kron \vb{C}_j$,
    and choose $r'_j$ so that
    \begin{equation}
    \label{eq:gamma_eigs_assumption}
        \left(\sum_{\alpha=r'_j+1}^{N_1 \dots N_{j}} \gamma_{j,\alpha}^2 \right)^{1/2} \le \epsilon_j \left(\sum_{\alpha=1}^{N_1 \dots N_{j}} \gamma_{j,\alpha}^2 \right)^{1/2},
    \end{equation}
    for $j=1,\dots,k$. Then there exists a tensor train $\tuple{S}$ with ranks $\vb{r}=(1, r_1, \dots, r_k, 1)$ with $r_i \le r'_i$ for $i=1,\dots,k$, such that the vector-valued multilinear function $S$ corresponding to $\tuple{S}$ satisfies
    \begin{equation}
    \label{eq:tt_overall_bound}
        \norm{F - S} \le 2 \norm{\vb{B}} \left(\prod_{l=1}^k \norm{\vb{C}_l}_{\mathrm{HS}}\right) \sum_{j=1}^{k} \epsilon_j.
    \end{equation}
\end{prop}

The full proof of \propref{thm:preconditioned_are_tt} is given in \appref{app:finite_to_infinite}. The idea is to perform a rescaling to set $\norm{C}_{\mathrm{HS}}=1$, then build the tensor train one core at a time by repeatedly peeling modes off $F$ using \lemref{lem:helper_for_tt_thm} (below). The total error is bounded by summing the error incurred at each peeling step. For \lemref{lem:helper_for_tt_thm}, we need the following standard machinery: \lemref{lem:unfoldings_and_kron}, which relates matrix unfoldings and the Kronecker product, and \corref{cor:unfolding_norm_bound}, which bounds induced tensor norms by matrix unfolding norms.

\begin{lemma}[Matrix unfoldings and the Kronecker product]
\label{lem:unfoldings_and_kron}
Suppose $\vb{A}$ is an $N_1 \times \dots \times N_\nummodes$ array with corresponding scalar-valued multilinear function $A$, $\vb{M}_i$ are $N_i \times q_i$ matrices, and $F$ is the scalar-valued multilinear function
$$
F(\vb{x}_1, \dots, \vb{x}_\nummodes) := A(\vb{M}_1 \vb{x}_1, \dots, \vb{M}_\nummodes \vb{x}_\nummodes)
$$
with corresponding $q_1 \times \dots \times q_\nummodes$ array $\vb{F}$. Further, let $\vb{F}_\text{mat}$ and $\vb{A}_\text{mat}$ denote the $i$th matrix unfoldings of $\vb{F}$ and $\vb{A}$, respectively. Then
\begin{equation}
\label{eq:unfolding_and_kron}
\vb{F}_\text{mat} = (\vb{M}_1 \kron \dots \kron \vb{M}_i)^T \vb{A}_\text{mat} (\vb{M}_{i+1} \kron \dots \kron \vb{M}_\nummodes).
\end{equation}
\end{lemma}

\begin{corollary}[Bounding tensor norm by matrix unfolding norm]
\label{cor:unfolding_norm_bound}
If $\vb{B}$ is a tensor and $\vb{B}_\text{mat}$ is any matrix unfolding of $\vb{B}$, then $\norm{\vb{B}} \le \norm{\vb{B}_\text{mat}}$.
\end{corollary}

\lemref{lem:unfoldings_and_kron} follows from direct calculation using the definitions. \corref{cor:unfolding_norm_bound} follows immediately from \lemref{lem:unfoldings_and_kron}: $\norm{\vb{B}}$ may be viewed as testing $\vb{B}_\text{mat}$ against vectors that have a Kronecker product structure, while $\norm{\vb{B}_\text{mat}}$ tests $\vb{B}_\text{mat}$ against arbitrary vectors and is therefore at least as large.

\begin{figure}
\centering
\begin{tikzpicture}[scale=0.75, every node/.style={scale=0.75*0.75}]


    \node[draw, rounded corners, minimum size=1.55cm, line width=0.3mm] (A) at (0.0,0) {\Large $B$};
    \node[draw, rounded corners, minimum size=0.9cm, inner sep=0.0, line width=0.3mm] (C3) at (-0.38627124296868437, 3.811179354631058-5.0) {\large $C_3$};
    \node[draw, rounded corners, minimum size=0.9cm, inner sep=0.0, line width=0.3mm] (C4) at (1.011271242968684, 4.265268434634408-5.0) {\large $C_4$};
    \node[draw, rounded corners, minimum size=0.9cm, inner sep=0.0, line width=0.3mm] (C1) at (-0.386271242968684, 6.188820645368942-5.0) {\large $C_1$};
    \node[draw, rounded corners, minimum size=0.9cm, inner sep=0.0, line width=0.3mm] (C2) at (-1.25, 0.0) {\large $C_2$};

    \draw[line width=0.3mm] (A) -- (C3);
    \draw[line width=0.3mm] (A) -- (C4);
    \draw[line width=0.3mm] (A) -- (C1);
    \draw[line width=0.3mm] (A) -- (C2);

    \node (C3e) at (-0.6489356881873898, 3.002781315780178-5.0) {};
    \node (C4e) at (1.6989356881873894, 3.765650970185806-5.0) {};
    \node (Ae) at (0.75 * 1.69893568818739, 0.85 * 6.234349029814194 - 0.85 * 5.0) {};
    \node (C1e) at (-0.6489356881873891, 6.997218684219822-5.0) {};
    \node (C2e) at (-2.1, 5.000000000000001-5.0) {};

    \draw[line width=0.3mm] (C3e) -- (C3);
    \draw[line width=0.3mm] (C4e) -- (C4);
    \draw[line width=0.3mm] (Ae) -- (A);
    \draw[line width=0.3mm] (C1e) -- (C1);
    \draw[line width=0.3mm] (C2e) -- (C2);

    \node (approx) at (7.75 - 5,0) {\Huge $\approx$};

    \node[draw, rounded corners, minimum size=1.5cm, line width=0.3mm] (R) at (5.75,0) {\Large $B'$};
    \node[draw, rounded corners, minimum size=0.9cm, inner sep=0.0, line width=0.3mm] (C3_b) at (-0.38627124296868437+5.75, 3.811179354631058-5.0) {\large $C_3$};
    \node[draw, rounded corners, minimum size=0.9cm, inner sep=0.0, line width=0.3mm] (C1_b) at (-0.386271242968684+5.75, 6.188820645368942-5.0) {\large $C_1$};
    \node[draw, rounded corners, minimum size=0.9cm, inner sep=0.0, line width=0.3mm] (C2_b) at (-1.25+5.75, 0.0) {\large $C_2$};

    \draw[line width=0.3mm] (R) -- (C3_b);
    \draw[line width=0.3mm] (R) -- (C1_b);
    \draw[line width=0.3mm] (R) -- (C2_b);

    \node (C3e_b) at (-0.6489356881873898+5.75, 3.002781315780178-5.0) {};
    \node (C4e_b) at (1.6989356881873894+5.75, 3.765650970185806-5.0) {};
    \node (C1e_b) at (-0.6489356881873891+5.75, 6.997218684219822-5.0) {};
    \node (C2e_b) at (-2.1+5.75, 5.000000000000001-5.0) {};

    \draw[line width=0.3mm] (C3e_b) -- (C3_b);
    \draw[line width=0.3mm] (C1e_b) -- (C1_b);
    \draw[line width=0.3mm] (C2e_b) -- (C2_b);

    \node[draw, rounded corners, minimum size=1.15cm, line width=0.3mm] (U) at (5.75+1.5,0) {\Large $G$};

    \node (U1) at (5.75+1.5, -8.75 + 5.75 + 1.5 + 0.25) {};
    \node (U2) at (8.5, 0.0) {};

    \draw[line width=0.3mm] (R) -- (U);
    \draw[line width=0.3mm] (U) -- (U1);
    \draw[line width=0.3mm] (U) -- (U2);

\end{tikzpicture}
\caption{Peeling off the last two modes from a preconditioned tensor per \lemref{lem:helper_for_tt_thm}.}
\label{fig:tt_peel_one_mode}
\end{figure}

\begin{lemma}[Helper for \propref{thm:preconditioned_are_tt}: peeling off the last two modes of $T$]
\label{lem:helper_for_tt_thm}
    Under the same setup as \propref{thm:preconditioned_are_tt} and assuming $\norm{\vb{C}_i}_{\mathrm{HS}}=1$ for $i=1,\dots,k$, there exists a multilinear function $B':\mathbb{R}^{N_1} \times \dots \times \mathbb{R}^{N_{k-1}} \rightarrow \mathbb{R}^{r_{k-1}}$ satisfying 
    \begin{equation}
    \label{eq:R_less_than_A_onestep}
    \norm{B'} \le \norm{B},
    \end{equation}
    and a matrix-valued linear function $G:\mathbb{R}^{N_k} \rightarrow \mathbb{R}^{r_{k-1}\times M}$  satisfying 
    \begin{equation}
    \label{eq:U_func_bounded_by_1}
    \norm{G} \le 1
    \end{equation}
    such that the multilinear function $F'$ defined by
    \begin{equation}
    \label{eq:tt_one_step_factorization}
        F'(\vb{x}_1, \dots, \vb{x}_{k-1}, \vb{x}_{k}) := B'(\vb{C}_1 \vb{x}_1, \dots, \vb{C}_{k-1} \vb{x}_{k-1})^T G(\vb{x}_{k})
    \end{equation}
    satisfies
    \begin{equation*}
        \norm{F - F'} \le 2\epsilon_{k-1} \norm{B}.
    \end{equation*}
    This is illustrated in \figref{fig:tt_peel_one_mode}.
\end{lemma}

\begin{proof}
For this proof, we assume without loss of generality that all vectors, arrays, and multilinear functions are expressed with respect to the bases of eigenvectors of the matrices $\vb{C}_1, \vb{C}_2,\dots, \vb{C}_k$. 
Under this assumption, $\vb{C}_j = \operatorname{diag}(\lambda_1, \lambda_2, \dots, \lambda_{N_j})$.

The proof proceeds in three parts. In the first part, we define two approximations, $F \approx \widetilde{F}'$ and $\widetilde{F}' \approx F'$. We then prove bound $\norm{F - F'} \le 2 \norm{\vb{F}_\text{mat} - \widetilde{\vb{F}}_\text{mat}'}$, where $\vb{F}_\text{mat}$ and $\widetilde{\vb{F}}_\text{mat}'$ are certain matrix unfoldings of the arrays corresponding to $F$ and $\widetilde{F}'$, respectively. In the second part, we bound $\norm{\vb{F}_\text{mat} - \widetilde{\vb{F}}_\text{mat}'}$. In the third part, we show that $F'$ arises from a factorization of the form given in \eqref{eq:tt_one_step_factorization}.

\paragraph{Part 1} 
Let $\vb{F}$ be the $N_1 \times \dots \times N_k \times M$ array corresponding to $F$. Let $\vb{B}_\text{mat}$ be the $(N_1 \dots N_{k-1}) \times (N_k M)$ matrix formed by reshaping $\vb{B}$ such that the rows of $\vb{B}_\text{mat}$ correspond to the first $k-1$ input modes of $\vb{B}$, and the columns of $\vb{B}_\text{mat}$ correspond to the last input and the output modes of $\vb{B}$, and let $\vb{F}_\text{mat}$ be the $(N_1 \dots N_{k-1}) \times (N_k M)$ matrix formed by reshaping $\vb{F}$ in the same way.
Using the relationship between matrix unfoldings and the Kronecker product (\lemref{lem:unfoldings_and_kron}) yields
\begin{equation*}
\vb{F}_\text{mat} = \underbrace{\left(\vb{C}_1 \kron \dots \kron \vb{C}_{k-1}\right)}_{\vb{\Gamma}} \vb{B}_\text{mat} \left(\vb{C}_k \kron \vb{I}\right) = \vb{\Gamma} \vb{B}_\text{mat} \left(\vb{C}_k \kron \vb{I}\right),
\end{equation*}
where $\vb{I}$ is the $M \times M$ identity matrix and 
\begin{equation*}
    \vb{\Gamma} := \vb{C}_1 \kron \dots \kron \vb{C}_{k-1}.
\end{equation*}
In the eigenvector bases we are using, $\vb{\Gamma}$ is diagonal, and its diagonal entries consist of the eigenvalue products $\gamma_{k-1,\alpha}$ in a permuted order. Let $\vb{\Gamma}'$ be the diagonal matrix formed from $\vb{\Gamma}$ by keeping the largest $r_{k-1}'$ diagonal entries of $\vb{\Gamma}$, and setting the rest of the diagonal entries to zero.
Set
\begin{equation*}
\widetilde{\vb{F}}_\text{mat}' := \vb{\Gamma}' \vb{B}_\text{mat} \left(\vb{C}_k \kron \vb{I}\right).
\end{equation*}
Define $r_{k-1}$ to be the rank of $\widetilde{\vb{F}}_\text{mat}'$. By construction, $r_{k-1} \le r_{k-1}'$. Let $\vb{G}^R$ be a $r_{k-1} \times N_k M$ matrix with rows that form an orthonormal basis for the row space of $\widetilde{\vb{F}}_\text{mat}'$, and define
\begin{equation*}
\vb{F}_\text{mat}' := \vb{F}_\text{mat} \left(\vb{G}^R\right)^T \vb{G}^R.
\end{equation*}
Let $\widetilde{\vb{F}}'$ and $\vb{F}'$ denote the $N_1 \times \dots \times N_k \times M$ arrays formed by reshaping $\widetilde{\vb{F}}_\text{mat}$ and $\vb{F}_\text{mat}'$, respectively, and let $\widetilde{F}':\mathbb{R}^{N_1} \times \dots \times \mathbb{R}^{N_k}\rightarrow \mathbb{R}^M$ and $F':\mathbb{R}^{N_1} \times \dots \times \mathbb{R}^{N_k}\rightarrow \mathbb{R}^M$ denote the vector-valued multilinear functions corresponding to $\widetilde{\vb{F}}'$ and $\vb{F}'$, respectively.

Using the triangle inequality and the fact that the induced norm of an array is always bounded by the induced norm of its matrix unfoldings 
yields
\begin{equation*}
\norm{F - F'} \le \norm{\vb{F} - \widetilde{\vb{F}}'} + \norm{\vb{F}' - \widetilde{\vb{F}}'} 
\le \norm{\vb{F}_\text{mat} - \widetilde{\vb{F}}_\text{mat}'} + \norm{\vb{F}_\text{mat}' - \widetilde{\vb{F}}_\text{mat}'}.
\end{equation*}
Using the fact that the rows of $\vb{G}^R$ form an orthonormal basis for the rows of $\vb{F}_\text{mat}'$, $\widetilde{\vb{F}}_\text{mat}'$ and the definitions of $\vb{F}_\text{mat}'$, $\widetilde{\vb{F}}_\text{mat}'$, we have
\begin{align*}
\norm{\vb{F}_\text{mat}' - \widetilde{\vb{F}}_\text{mat}'} &= \norm{\left(\vb{F}_\text{mat}' - \widetilde{\vb{F}}_\text{mat}'\right)\left(\vb{G}^R\right)^T \vb{G}^R} \\
&= \norm{\left(\vb{F}_\text{mat} - \widetilde{\vb{F}}_\text{mat}'\right)\left(\vb{G}^R\right)^T \vb{G}^R}
\le \norm{\vb{F}_\text{mat} - \widetilde{\vb{F}}_\text{mat}'},
\end{align*}
and hence $\norm{F - F'} \le 2 \norm{\vb{F}_\text{mat} - \widetilde{\vb{F}}_\text{mat}'}$. 

\paragraph{Part 2} By the definition of the induced norm, we have
\begin{equation*}
\norm{\vb{F}_\text{mat} - \widetilde{\vb{F}}_\text{mat}'}
= 
\sup_{\substack{
\vb{a} \in \mathbb{R}^{N_1 \dots N_{k-1}} \\
\norm{\vb{a}} = 1
}}
\sup_{\substack{
\vb{b} \in \mathbb{R}^{N_k M} \\
\norm{\vb{b}} = 1
}}
\underbrace{\vb{a}^T (\vb{\Gamma} - \vb{\Gamma}') \vb{B}_\text{mat} (\vb{C}_k \kron \vb{I}) \vb{b}}_{=:\rho}.
\end{equation*}
We will now proceed to bound $\rho:=\vb{a}^T (\vb{\Gamma} - \vb{\Gamma}') \vb{B}_\text{mat} (\vb{C}_k \kron I) \vb{b}$. For convenience, define $\vb{g}$ to be the diagonal of the diagonal matrix $\vb{\Gamma} - \vb{\Gamma}'$, so that $\left(\vb{a}^T (\vb{\Gamma} - \vb{\Gamma}')\right)\![\alpha] = \vb{a}[\alpha] \vb{g}[\alpha]$. Notice that by assumption \eqref{eq:gamma_eigs_assumption}, we have $\norm{\vb{g}}\le \epsilon_{k-1}$. 
Let $(\sigma_j, \vb{\phi}_j, \vb{\psi}_j)$ be the singular values and singular vectors of the $N_k \times M$ matrix formed by reshaping $\vb{b}$.
Here, we take the full singular value decomposition, so that the singular vectors $\vb{\phi}_j$ are defined for $j=1,\dots,N_k$ and form an orthonormal basis for $\mathbb{R}^{N_k}$, even if $M < N_k$. By the relationship between the Kronecker product and matrix unfoldings, we have
$$
\vb{b} = \sum_{j=1}^K \sigma_j \vb{\phi}_j \kron \vb{\psi}_j,
$$
where $K=\min(N_k, M)$. Since the induced norm of a vector is the Hilbert-Schmidt norm of its matrix reshaping, we have $1=\norm{\vb{b}}^2 = \sum_{j=1}^K \sigma_j^2$.
Writing the matrix products in $\rho$ in terms of component sums, using the triangle inequality twice, and the non-negativity of $\sigma_j$, yields
\begin{align*}
\abs{\rho}
&= \left\lvert\sum_{\alpha=1}^{N_1 \dots N_{k-1}} \sum_{\beta=1}^{N_k M} \underbrace{\vb{a}[\alpha] \vb{g}[\alpha]}_{\left(\vb{a}^T(\vb{\Gamma} - \vb{\Gamma}')\right)[\alpha]} \vb{B}_\text{mat}[\alpha, \beta] \underbrace{\left(\sum_{j=1}^K \sigma_j \vb{C}_k\vb{\phi}_j \kron \vb{\psi}_j\right)\![\beta]}_{\left((\vb{C}_k \kron \vb{I}) \vb{b}\right)[\beta]}\right\rvert \\
&= \left\lvert\sum_{j=1}^K \sigma_j \sum_{\alpha=1}^{N_1 \dots N_{k-1}} \sum_{\beta=1}^{N_k M} \vb{a}[\alpha] \vb{g}[\alpha] \vb{B}_\text{mat}[\alpha,\beta] \left( \vb{C}_k\vb{\phi}_j \kron \vb{\psi}_j\right)\![\beta]\right\rvert \\
&\le \sum_{j=1}^K \sigma_j \sum_{\alpha=1}^{N_1 \dots N_{k-1}} \abs{\vb{a}[\alpha] \vb{g}[\alpha]} \left\lvert \sum_{\beta=1}^{N_k M}  \vb{B}_\text{mat}[\alpha,\beta]  \left(\vb{C}_k\vb{\phi}_j \kron \vb{\psi}_j\right)\![\beta]\right\rvert.
\end{align*}
Let $i_1, \dots, i_{k-1}$ be the sub-indices for each mode corresponding to the grouped-mode index $\alpha$. We have
\begin{align*}
\left\lvert \sum_{\beta=1}^{N_k M}  \vb{B}_\text{mat}[\alpha, \beta]  \left(\vb{C}_k\vb{\phi}_j \kron \vb{\psi}_j\right)\![\beta]\right\rvert &= \abs{\vb{\psi}_j^T B(\vb{e}_{1,i_1}, \dots, \vb{e}_{k-1,i_{k-1}}, \vb{C}_k\vb{\phi}_j)} \\
&\le \norm{\vb{\psi}_j} \norm{B} \norm{\vb{C}_k\vb{\phi}_j} = \norm{B} \norm{\vb{C}_k\vb{\phi}_j},
\end{align*}
where $\vb{e}_{l,i}$ is the $i$th standard unit basis vector for $\mathbb{R}^{N_l}.$
Substituting this into our previous expression for $\abs{\rho}$ and using Cauchy--Schwarz twice yields
\begin{align*}
\abs{\rho} &\le \norm{B} \sum_{j=1}^K \sigma_j \norm{\vb{C}_k\vb{\phi}_j} \sum_{\alpha=1}^{N_1 \dots N_{k-1}} \abs{\vb{a}[\alpha] \vb{g}[\alpha]} \\
&\le \norm{B} \left(\sum_{j=1}^K \sigma_j^2\right)^{1/2} \left(\sum_{j=1}^K\norm{\vb{C}_k\vb{\phi}_j}^2\right)^{1/2} \norm{\vb{a}} \norm{\vb{g}} \\
&\le \epsilon_{k-1}\norm{B} \left(\sum_{j=1}^K\norm{\vb{C}_k\vb{\phi}_j}^2\right)^{1/2}.
\end{align*}
Since $K \le N_k$, we have
\begin{equation*}
\sum_{j=1}^{K} \norm{\vb{C}_k \vb{\phi}_j}^2 \le \sum_{j=1}^{N_k} \norm{\vb{C}_k \vb{\phi}_j}^2 = \norm{\vb{C}_k}_{\mathrm{HS}}^2 = 1,
\end{equation*}
so $\abs{\rho} \le \epsilon_{k-1} \norm{B}$. Hence, $\norm{F - \widetilde{F}'} \le  \epsilon_{k-1} \norm{B}$, which implies $\norm{F - F'} \le  2\epsilon_{k-1} \norm{B} $.

\paragraph{Part 3} Let $G:\mathbb{R}^{N_k} \rightarrow \mathbb{R}^{r_{k-1} \times M}$ be the matrix-valued multilinear function corresponding to the $r_{k-1} \times (N_k M)$ array $\vb{G}^R$ (note the regrouping of indices). 
Let $\vb{B}'$ be the $N_1 \times \dots \times N_{k-1} \times r_{k-1}$ array formed by reshaping $\vb{B}_\text{mat} (\vb{C}_k \kron \vb{I}) \left(\vb{G}^R\right)^T$, and let $B':\mathbb{R}^{N_1} \times \dots \times \mathbb{R}^{N_{k-1}} \rightarrow \mathbb{R}^{r_{k-1}}$ be the vector-valued multilinear function corresponding to $\vb{B}'$.
We have
\begin{align*}
&B'(\vb{C}_1\vb{x}_1, \dots, \vb{C}_{k-1}\vb{x}_{k-1})^T G(\vb{x}_k)\vb{\nu} \\
&= (\vb{C}_1 \vb{x}_1 \kron \dots \kron \vb{C}_{k-1} \vb{x}_{k-1})^T \vb{B}_\text{mat} (\vb{C}_k \kron \vb{I}) \left(\vb{G}^R\right)^T G(\vb{x}_k)\vb{\nu} \\
&= (\vb{x}_1 \kron \dots \kron \vb{x}_{k-1})^T \vb{\Gamma} \vb{B}_\text{mat} (\vb{C}_k \kron \vb{I}) \left(\vb{G}^R\right)^T \vb{G}^R(\vb{x}_k \kron \vb{\nu}) \\
&= (\vb{x}_1 \kron \dots \kron \vb{x}_{k-1})^T \vb{F}_\text{mat}'(\vb{x}_k \kron \vb{\nu}) \\
&= F'(\vb{x}_1, \dots, \vb{x}_{k-1}, \vb{x}_k)^T \vb{\nu},
\end{align*}
which implies factorization \eqref{eq:tt_one_step_factorization} since $\vb{\nu}$ is arbitrary. All that remains is to verify bounds \eqref{eq:R_less_than_A_onestep} and \eqref{eq:U_func_bounded_by_1}. For \eqref{eq:U_func_bounded_by_1}, notice that 
\begin{equation*}
\norm{\vb{\mu}^T G(\vb{x}_k)\vb{\nu}} = \norm{\vb{\mu}^T\vb{G}^R (\vb{x}_k \kron \vb{\nu})} 
\le \norm{\vb{\mu}}\norm{\vb{x}_k \kron \vb{\nu}} = \norm{\vb{\mu}}\norm{\vb{x}_k} \norm{\vb{\nu}},
\end{equation*}
and so $\norm{G} \le 1$ as required. For \eqref{eq:R_less_than_A_onestep}, let $\vb{\mu}$ be an arbitrary vector in $\mathbb{R}^{r_{k-1}}$ satisfying $\norm{\vb{\mu}}=1$. 
Let $\vb{X}$
be the $N_k \times M$ matrix formed by reshaping $\left(\vb{G}^R\right)^T\vb{\mu}$, and let $(\sigma_j, \vb{\phi}_j, \vb{\psi}_j)$ be the singular values and singular vectors in the full singular value decomposition of $\vb{X}$. We have
\begin{equation*}
\left(\vb{G}^R\right)^T\vb{\mu} = \sum_{j=1}^{K} \sigma_j \vb{\phi}_j \kron \vb{\psi}_j.
\end{equation*}
Since the rows of $\vb{G}^R$ are orthonormal, 
we have 
\begin{equation*}
\sum_{j=1}^{K} \sigma_j^2 = \norm{\vb{X}}_{\mathrm{HS}}^2 = \norm{\left(\vb{G}^R\right)^T\vb{\mu}}^2 \le 1.
\end{equation*}
Let $\vb{\theta}_i \in \mathbb{R}^{N_i}$, $i=1,\dots,k-1$, be vectors satisfying $\norm{\vb{\theta}_i}=1$. We compute
\begin{align*}
B'(\vb{\theta}_1, \dots, \vb{\theta}_{k-1})^T \vb{\mu} &= (\vb{\theta}_1 \kron \dots \kron \vb{\theta}_{k-1})^T \vb{B}_\text{mat} (\vb{C}_k \kron \vb{I}) \left(\vb{G}^R\right)^T \vb{\mu} \\
&= \sum_{j=1}^{K} \sigma_j (\vb{\theta}_1 \kron \dots \kron \vb{\theta}_{k-1})^T \vb{B}_\text{mat} (\vb{C}_k \vb{\phi}_j \kron \vb{\psi}_j) \\
&= \sum_{j=1}^{K} \sigma_j B(\vb{\theta}_1, \dots, \vb{\theta}_{k-1}, \vb{C}_k \vb{\phi}_j)^T \vb{\psi}_j.
\end{align*}
Taking the absolute value, using the triangle inequality, the definition of the induced norm, and Cauchy--Schwarz yields
\begin{align}
\left\lvert B'(\vb{\theta}_1, \dots, \vb{\theta}_{k-1})^T \vb{\mu} \right\rvert &\le \sum_{j=1}^{K} \abs{\sigma_j} \left\lvert B(\vb{\theta}_1, \dots, \vb{\theta}_{k-1}, \vb{C}_k \vb{\phi}_j)^T \vb{\psi}_j \right\rvert \nonumber\\
&\le \norm{B} \sum_{j=1}^{K} \abs{\sigma_j} \norm{\vb{C}_k \vb{\phi}_j} \nonumber\\
&\le \norm{B} \left(\sum_{j=1}^{K} \norm{\vb{C}_k \vb{\phi}_j}^2\right)^{1/2}
\le \norm{B} \label{eq:bound_for_Rabs}.
\end{align}
where in the last inequality we used the fact that $\norm{\vb{C}_k}_{\mathrm{HS}} = 1$.
Taking the supremum over the vectors $\vb{\mu}$ and $\vb{\theta}_i$ yields the bound $\norm{B'} \le \norm{B} $, which completes the proof.
\end{proof}



\subsection{From tensor train to Tucker tensor train}
\label{sec:sym_tt_to_ttt}

\begin{lemma}[Symmetric tensor trains are Tucker tensor trains]
\label{lem:sym_tt_to_ttt}
    Let $S:\mathbb{R}^{N} \times \dots \times \mathbb{R}^{N}\rightarrow \mathbb{R}^M$ be a vector-valued multilinear function represented by a tensor train $\tuple{S}=(\vb{G}_i)_{i=1}^{k+1}$ with TT-ranks $\vb{r}=(1, r_1, \dots, r_k,1)$.
    Further, assume that $S$ is nearly symmetric to tolerance $\epsilon \ge 0$ in its $k$ inputs. That is, if $\pi$ is a permutation function which re-orders $k$ variables\footnote{
    If $k=3$ there are six possible permutation functions. Two of which, for example, are $(S \circ \pi)(\vb{x}_1, \vb{x}_2, \vb{x}_3) = S(\vb{x}_2, \vb{x}_1, \vb{x}_3)$ and $(S \circ \pi)(\vb{x}_1, \vb{x}_2, \vb{x}_3) = S(\vb{x}_3, \vb{x}_1, \vb{x}_2)$.
    }, then
    \begin{equation}
    \label{eq:permutation_error_assumption}
    \norm{S - S \circ \pi} \le \epsilon.
    \end{equation}
    Then there exists an $N \times r_1$ matrix $\vb{\inputbasis}$ with orthonormal columns, a $M \times r_k$ matrix $\vb{\outputbasis}$ with orthonormal columns, and a tensor train $\tuple{S}'=(\vb{G}'_1, \dots, \vb{G}'_{k+1})$ with shape $\vb{n}=(r_1,\dots,r_1,r_k)$ and ranks $\vb{r}$ such that the multilinear function $T$ defined by
    \begin{equation}
    \label{eq:T_tilde_PSQ}
        T(\vb{x}_1, \dots, \vb{x}_k) := \vb{\outputbasis} S'(\vb{\inputbasis}^T\vb{x}_1, \dots, \vb{\inputbasis}^T\vb{x}_k)
    \end{equation}
    satisfies
    $$
    \norm{S - T} \le (k-1)\epsilon.
    $$
    Here $S'$ is the vector-valued multilinear function corresponding to $\tuple{S}'$.
\end{lemma}

\begin{proof}
Since $r_0=r_d=1$, we may interpret the $1 \times N \times r_1$ array $\vb{G}_1$ as an $N \times r_1$ matrix, and we may interpret the $r_k \times M \times 1$ array $\vb{G}_{k+1}$ as an $r_k \times M$ matrix. Let $\vb{\inputbasis}$ be an $N \times r_1$ matrix whose columns form an orthonormal basis for the column space of $\vb{G}_1$, and let $\vb{\outputbasis}$ be an $M \times r_k$ matrix whose columns form an orthonormal basis for the row space of $\vb{G}_{k+1}$, with $\vb{G}_1$ and $\vb{G}_{k+1}$ interpreted as matrices. For $i=1,\dots,k$ define $\vb{G}'_i$ to be the $r_{i-1} \times r_1 \times r_i$ array formed by contracting the middle mode of $\vb{G}_i$ with the columns of $\vb{\inputbasis}$, and define $\vb{G}'_{k+1}$ to be the $r_k \times r_k \times 1$ array formed by contracting the middle mode of $\vb{G}_{k+1}$ with the columns of $\vb{\outputbasis}$. I.e., the arrays $\vb{G}'_i$ have entries
\begin{align*}
    \vb{G}'_i[a,b,c] &= \sum_{\beta=1}^N \vb{G}_i[a,\beta,c] \vb{\inputbasis}[\beta,b], \qquad i=1,\dots,k, \\
    \vb{G}'_{k+1}[a,b,c] &= \sum_{\beta=1}^M \vb{G}_{k+1}[a,\beta,c] \vb{\outputbasis}[\beta,b].
\end{align*}
Further, define $\tuple{S}'=(\vb{G}'_1, \dots, \vb{G}'_{k+1})$. The matrices $\vb{\inputbasis}$ and $\vb{\outputbasis}$ and the tensor train $\tuple{S}'$ have the appropriate shapes and ranks required by the theorem, and by construction
\begin{equation*}
S'(\vb{\xi}_1, \dots, \vb{\xi}_k) = \vb{\outputbasis}^T S(\vb{\inputbasis} \vb{\xi}_1, \dots, \vb{\inputbasis} \vb{\xi}_k)
\end{equation*}
and so equation \eqref{eq:T_tilde_PSQ} becomes 
\begin{equation*}
    T(\vb{x}_1, \dots, \vb{x}_k) = \vb{\outputbasis}\vb{\outputbasis}^T S(\vb{\inputbasis}\vb{\inputbasis}^T \vb{x}_1, \dots, \vb{\inputbasis}\vb{\inputbasis}^T \vb{x}_k).
\end{equation*}
Since $\vb{\inputbasis}$ and $\vb{\outputbasis}$ are orthonormal bases for the externally facing indices of the first and last cores of $\tuple{S}$, we have
\begin{equation}
\label{eq:project_first_and_last}
    S(\vb{x}_1, \dots, \vb{x}_k) = \vb{\outputbasis}\vb{\outputbasis}^T S(\vb{\inputbasis}\vb{\inputbasis}^T \vb{x}_1, \vb{x}_2, \dots, \vb{x}_k).
\end{equation}
The remainder of the proof is to show that we can insert the projector $\vb{\inputbasis}\vb{\inputbasis}^T$ into the $2$nd through $k$th indices while incurring less than $(k-1)\epsilon$ error. We will need to manipulate versions of $T$ with $\vb{\inputbasis}\vb{\inputbasis}^T$ inserted in to subsets of the $k$ input indices, and $\vb{\outputbasis}\vb{\outputbasis}^T$ always contracted with the output index. For convenience of notation, we represent tensors of this form by formal strings of order $k$ consisting of the letters `$a$' and `$b$'. The letter $b$ occurs in the $i$th entry of the string if $\vb{\inputbasis}\vb{\inputbasis}^T$ is inserted in index $i$, and $a$ occurs otherwise. 
If the same letter is repeated several times in a row, we may shorten the string with exponent notation, but note that the letters do not commute. 
For example, if $k=4$, we have the correspondences
\begin{alignat*}{5}
a^4 = aaaa \quad &\leftrightarrow \quad \vb{\outputbasis}\vb{\outputbasis}^T S(&~\cdot~,&& ~\cdot~,&& ~\cdot~,&& ~\cdot~) \\
b a^3 = baaa \quad &\leftrightarrow \quad \vb{\outputbasis}\vb{\outputbasis}^T S(&\vb{\inputbasis}\vb{\inputbasis}^T ~\cdot~,&& ~\cdot~,&& ~\cdot~,&& ~\cdot~) \\
a b a^2 = abaa \quad &\leftrightarrow \quad \vb{\outputbasis}\vb{\outputbasis}^T S(&~\cdot~,&&~ \vb{\inputbasis}\vb{\inputbasis}^T ~\cdot~,&& ~\cdot~,&& ~\cdot~) \\
b a b^2 = babb \quad &\leftrightarrow \quad \vb{\outputbasis}\vb{\outputbasis}^T S(&\vb{\inputbasis}\vb{\inputbasis}^T ~\cdot~,&& ~\cdot~,&&~ \vb{\inputbasis}\vb{\inputbasis}^T~\cdot~,&&~ \vb{\inputbasis}\vb{\inputbasis}^T~\cdot~)
\end{alignat*}
and so on. Using this notation, $S=a^k$, $T=b^k$, and Equation \eqref{eq:project_first_and_last} becomes $a^k = ba^{k-1}$. Assumption \eqref{eq:permutation_error_assumption} means that permuting the letters in a string incurs error $\epsilon$, e.g., $\norm{abba - baba} \le \epsilon$. 

By the same reasoning used to derive Equation \eqref{eq:project_first_and_last}, two strings are equal if they differ in only the first entry, e.g., $abba=bbba$. We therefore have
\begin{equation}
\label{eq:aabb_sum_1}
0 = \sum_{i=1}^k ab^{i-1}a^{k-i} - bb^{i-1}a^{k-i}
\end{equation}
because $ab^{i-1}a^{k-i}$ and $bb^{i-1}a^{k-i}$ differ only in their first entry.
If expanded out fully, this sum has $2k$ terms, with first term $a^k$ and last term $b^k$. Regrouping the terms to match the second term with the third, the fourth term with the fifth, and so on, yields
\begin{equation}
\label{eq:aabb_sum_2}
0 = a^k - \left(\sum_{i=1}^{k-1}bb^{i-1}aa^{k-i-1} - ab^{i-1}ba^{k-i-1}\right) - b^k.
\end{equation}
Moving $a^k - b^k$ to the right-hand side, taking the induced norm of each side, and using the triangle inequality yields
\begin{align*}
\norm{a^k - b^k} =& \norm{\sum_{i=1}^{k-1} bb^{i-1}aa^{k-i-1} - ab^{i-1}ba^{k-i-1}} \\
\le& \sum_{i=1}^{k-1} \norm{bb^{i-1}aa^{k-i-1} - ab^{i-1}ba^{k-i-1}} \le (k-1) \epsilon.
\end{align*}
The last inequality follows from \eqref{eq:permutation_error_assumption} since $bb^{i-1}aa^{k-i-1}$ and $ab^{i-1}ba^{k-i-1}$ are permutations of each other (the first and $i$th letters are swapped). The theorem follows from recalling that $a^k = S$ and $b^k = T$.
\end{proof}

The Tucker tensor train constructed in \lemref{lem:sym_tt_to_ttt} uses a single shared input basis $\vb{\inputbasis}$ for all $k$ input modes (and a separate output basis $\vb{\outputbasis}$), reflecting the approximate symmetry of $S$; this shared-input structure is precisely the symmetry-aware parameterization noted in \secref{sec:t4s_model}.

\begin{prop}[Tucker tensor train in finite dimensions]
\label{thm:finite_ttt}
    Let the conditions of \propref{thm:preconditioned_are_tt} hold, and impose further constraints that $N_1=\dots=N_k=N$, that $F$ is symmetric in its $k$ input arguments, and that $\vb{C}_1=\dots=\vb{C}_k=\vb{C}$. Then there exists an $N \times r_1$ matrix $\vb{\inputbasis}$ with orthonormal columns, an $M \times r_k$ matrix $\vb{\outputbasis}$ with orthonormal columns, and a tensor train $\tuple{S}'$ with shape $\vb{n}=(r_1,\dots,r_1,r_k)$ and ranks $\vb{r}=(1, r_1, \dots, r_k, 1)$ such that the multilinear function $T$ defined by
    \begin{equation}
    \label{eq:widetildet_5456}
        T(\vb{x}_1, \dots, \vb{x}_k) := \vb{\outputbasis} S'(\vb{\inputbasis}^T\vb{x}_1, \dots, \vb{\inputbasis}^T\vb{x}_k)
    \end{equation}
    satisfies
    \begin{equation*}
        \norm{F - T} \le 2(2k-1) \norm{B} \norm{\vb{C}}_{\mathrm{HS}}^k \sum_{j=1}^k \epsilon_j.
    \end{equation*}
    Here, $S'$ is the vector-valued multilinear function corresponding to $\tuple{S}'$.
\end{prop}

The proof of \propref{thm:finite_ttt} is given in \appref{app:finite_to_infinite}. The idea of the proof is to use \propref{thm:preconditioned_are_tt} to form a tensor train approximation of $F$, then use \lemref{lem:sym_tt_to_ttt} to approximate the tensor train with a Tucker tensor train.

\subsection{Extension to infinite dimensions: main theorem}
\label{sec:infinite_and_main_thm}

\begin{theorem}[Tucker tensor train approximation of derivative tensors]
\label{thm:ttt_main_thm}
    Let $X$ and $Y$ be Hilbert spaces, let $q:X \rightarrow Y$ be a function with $k$ continuous derivatives, let $\sqrtcov:X \rightarrow X$ be a Hilbert--Schmidt self-adjoint positive-semidefinite operator, let $\theta_0 \in X$, and define $f:X \rightarrow Y$,
    \begin{equation*}
        f(x) := q(\theta_0 + \sqrtcov x).
    \end{equation*}
    Further, let $\gamma_{j,1} \ge \gamma_{j,2} \ge \dots$ denote the eigenvalues of $\kron^j \sqrtcov$
    (the Kronecker product of $j$ copies of $C$)
    and 
    choose $r_j$ so that
    \begin{equation}
    \label{eq:gamma_eigs_assumption2}
    \left(\sum_{\alpha=r_j+1}^{\infty} \gamma_{j,\alpha}^2 \right)^{1/2} \le \epsilon_j \left(\sum_{\alpha=1}^{\infty} \gamma_{j,\alpha}^2 \right)^{1/2},
    \end{equation}
    for $j=1,\dots,k$. Then there exists a Tucker tensor train $\tuple{T}$ with tensor train ranks $\vb{r}=(1,r_1, r_2, \dots, r_k, 1)$ and Tucker ranks $\vb{n}=(r_1, \dots, r_1, r_k)$ such that
    \begin{equation*}
        \norm{T - D^k \ptoW(0)} \le 2(2k-1)\norm{D^k \pto(\paramzero)} \norm{C}_{\mathrm{HS}}^k \sum_{j=1}^k \epsilon_j,
    \end{equation*}
    where $T$ is the vector-valued multilinear function corresponding to $\tuple{T}$, and $\norm{C}_{\mathrm{HS}}$ is the Hilbert--Schmidt norm of $C$.
\end{theorem}

\begin{proof}
    The theorem follows from using \lemref{lem:inf_to_finite} (\appref{app:finite_to_infinite}) with $B=D^k \pto(\paramzero)$ to reduce the problem to finite dimensions, then using \propref{thm:finite_ttt} with $B=B_N$ and $\vb{C}=\vb{\Lambda}$ from \lemref{lem:inf_to_finite} to construct a Tucker tensor train in finite dimensions. Two observations make the bound of \propref{thm:finite_ttt} reproduce the stated bound. First, $\norm{B_N} \le \norm{B} = \norm{D^k \pto(\paramzero)}$ by \lemref{lem:inf_to_finite}, and $\norm{\vb{\Lambda}}_{\mathrm{HS}} \le \norm{C}_{\mathrm{HS}}$ since $\vb{\Lambda}$ holds eigenvalues of $C$. Second, the eigenvalues of $\kron^j \vb{\Lambda}$ are a subset of those of $\kron^j C$ and exhaust them as $N \to \infty$, so the truncation tolerances $\epsilon_j$ fixed by \eqref{eq:gamma_eigs_assumption2} are met in finite dimensions for $N$ large enough. Since the reduction error of \lemref{lem:inf_to_finite} likewise vanishes as $N \to \infty$, the bound $2(2k-1)\norm{D^k \pto(\paramzero)} \norm{C}_{\mathrm{HS}}^k \sum_{j=1}^k \epsilon_j$ follows. The first $k$ Tucker bases are copies of the composition of $\inputbasis$ from \lemref{lem:inf_to_finite} with $\vb{\inputbasis}$ from \propref{thm:finite_ttt}, and the last Tucker basis is the composition of $\outputbasis$ from \lemref{lem:inf_to_finite} with $\vb{\outputbasis}$ from \propref{thm:finite_ttt}.
\end{proof}

\begin{corollary}[Asymptotic error bound for power law covariance]
\label{cor:error_vs_rank_poly_decay}
Let the conditions of \thmref{thm:ttt_main_thm} hold, and further
suppose that the eigenvalues $\lambda_i$ of $C$ are given by
\begin{equation*}
\lambda_i = i^{-\beta}
\end{equation*}
for some $\beta > 1/2$. 
For any sufficiently large choice of tensor train ranks $\vb{r} = (1, r_1, \dots, r_k, 1)$, there exists a Tucker tensor train $\tuple{T}$ with tensor train ranks $\vb{r}$ and Tucker ranks $\vb{n} = (r_1, \dots, r_1, r_k)$ such that
\begin{equation}
\label{eq:corollary_main_bound}
    \norm{T - D^k \ptoW(0)} \le c \norm{D^k \pto(\paramzero)} \sum_{j=1}^k r_j^{-\beta + \frac{1}{2}}\left(\log r_j\right)^{\beta(j-1)},
\end{equation}
where $T$ is the vector-valued multilinear function corresponding to $\tuple{T}$. Here,
$c=c(k, \beta)$ is a constant which depends on $k$ and $\beta$.
\end{corollary}

The proof of \corref{cor:error_vs_rank_poly_decay} is provided in \appref{app:finite_to_infinite}. The idea of the proof is that 
the leading eigenvectors and eigenvalues of $\kron^j C$ correspond to points on a hyperbolic cross, with the eigenvalues of $\kron^j C$ being products of those of $C$. Since the eigenvalues of $C$ decay according to a power law, the result is proven by combining \thmref{thm:ttt_main_thm} with standard results for bounding power law sums on the hyperbolic cross.

The Tucker ranks $\vb{n}=(r_1,\dots,r_1,r_k)$ in \corref{cor:error_vs_rank_poly_decay}, with $r_1$ repeated over all input modes, are a conservative artifact of the proof, which builds the shared input basis from the first tensor train core together with the symmetry argument of \lemref{lem:sym_tt_to_ttt}; the TT ranks $r_1,\dots,r_k$ may differ across edges. In practice, T3 approximations may use nonuniform Tucker ranks selected by validation or rank continuation (\secref{sec:rank_adaptivity}).

Strictly speaking, the condition that the ranks be sufficiently large in \corref{cor:error_vs_rank_poly_decay} could be removed: since only finitely many ranks are excluded, any violation of \eqref{eq:corollary_main_bound} on this finite set can be accounted for by increasing the constant $c$. However, the resulting value of $c$ would then be dictated by the small rank regime and would no longer reflect the large rank asymptotic behavior.

In summary, the theory supports the use of Tucker tensor train representations for covariance-whitened derivative tensors when the input covariance has a rapidly decaying spectrum. The bounds are intentionally conservative: they use only covariance decay and boundedness of the derivative tensors, and ignore any additional smoothness, locality, or low-rank structure of the underlying solution map.


\ifdefined\kdefulltw\else\newlength{\kdefulltw}\fi

\section{Numerical experiments}
\label{sec:numerical_results}

We present three sets of numerical results. In \secref{sec:numerical_random}, we fit symmetric random preconditioned tensors with Tucker tensor trains using TR-RMGN and MC-SGD, comparing against the quasi-optimal T3-SVD benchmark. In \secref{sec:numerical_darcy}, we use the complete method (with MC-SGD for the fitting step) 
to form T4S approximations of two mappings defined implicitly through a Poisson PDE. The first (\secref{sec:transmission}), which contains the bulk of our results, is a standard test case: a linear equation, a Gaussian parameter, and a smooth state-trace observable. The second (\secref{sec:outflow}) keeps the same Poisson PDE but hardens the problem on every other axis: a nonlinear equation, a non-Gaussian (logistic) parameter, and a less regular boundary-flux observable. It also illustrates how to construct T4S models when $\theta$ is not normally distributed.

\subsection{Random tensor}
\label{sec:numerical_random}

We fit random preconditioned tensors with Tucker tensor trains to evaluate our optimization methods. Random tensors are used here so that we can compare our methods against the T3-SVD baseline, which is quasi-optimal in the Hilbert-Schmidt norm, but expensive and applicable only to dense tensors. 

Let $\vb{A} \in \mathbb{R}^{N \times \dots \times N \times M}$,
\begin{equation*}
\vb{A}[i_1, \dots, i_k, i_{k+1}] \sim N(0, 1),
\end{equation*}
be a $(k+1)$th-order random tensor with i.i.d.\ standard-normal entries, and let $\vb{A}_\text{sym}$ be its symmetrization, averaging over all permutations of the first $k$ indices. 
Let $\vb{C}$ be the diagonal $N \times N$ matrix with diagonal entries obeying the power law 
$$
\vb{C}[i,i]= i^{-2}
$$ 
We approximate the tensor $\vb{T} \in \mathbb{R}^{N \times \dots \times N \times M}$ formed by contracting $\vb{A}_\text{sym}$ with copies of $\vb{C}$ in the first $k$ indices, i.e., the array of the vector-valued multilinear function
\begin{equation*}
T(x_1, \dots, x_k) := A_\text{sym}(\vb{C}x_1, \dots, \vb{C}x_k).
\end{equation*}
We fit $\vb{T}$ using $n_s$ random symmetric probes (\secref{sec:symmetric_action_data}) as training data.

\begin{figure}[tb]
\centering
\begin{tikzpicture}

\newcommand{\addcmsgd}[4]{
    \addplot [
        thick,
        color of colormap={#2 of tabgreens},
        mark=*,
        mark size=1pt,
        mark options={solid}, 
        dash pattern=on #3pt off #4pt
    ] table [x index=0, y index=1] {\main/figures/random_rank_adaptivity_cm-sgd/cm-sgd-#1.txt};
    
    \addlegendentry{MC-SGD(#1)}
}

\begin{groupplot}[
    group style={
        group size=2 by 1,
        y descriptions at=edge left, 
        horizontal sep=6pt,
    },
    xmin=1e2,
    xmax=8e4,
    ymax=1,
    xmode=log,
    ymode=log,
    width=0.52\textwidth,        
    height=0.52\textwidth, 
    title style={yshift=-6pt}
]

    \nextgroupplot[
        title={TR-RMGN($n_\text{chunk}$)}, 
        ylabel={Relative Forward Error},
        legend pos=south west,
        legend cell align=left,
        legend style={
            font=\footnotesize, 
            draw=none, 
            fill=none
        },
        reverse legend]
        
    \addplot [
        black, 
        thick, 
        mark=square*, 
        only marks,
        mark size=1pt,
        mark options={solid}
    ] table [x index=0, y index=1] {\main/figures/random_rank_adaptivity_tr-rmgn/t3svd.txt};
    \addlegendentry{T3-SVD}

    \addplot [
        thick,
        color of colormap={200 of tabblues},
        mark=triangle*,
        only marks,
        mark size=1pt,
        mark options={solid}, 
    ] table [x index=0, y index=1] {\main/figures/random_rank_adaptivity_tr-rmgn/rmgn.txt};
    \addlegendentry{TR-RMGN}
    
    \addplot [
        thick,
        color of colormap={800 of tabblues},
        mark=triangle*,
        mark size=1pt,
        mark options={solid}, 
        dash pattern=on 1pt off 1pt
    ] table [x index=0, y index=1] {\main/figures/random_rank_adaptivity_tr-rmgn/rmgn-1.txt};
    
    \addlegendentry{TR-RMGN(1)}

    \nextgroupplot[
        title={MC-SGD($n_\text{chunk}$)}, 
        legend pos=south west,
        legend cell align=left,
        legend style={
            font=\footnotesize, 
            draw=none, 
            fill=none
        },
        reverse legend]
        
    \addplot [
        black, 
        thick, 
        mark=square*, 
        only marks,
        mark size=1pt,
        mark options={solid}
    ] table [x index=0, y index=1] {\main/figures/random_rank_adaptivity_cm-sgd/t3svd.txt};
    \addlegendentry{T3-SVD}

    \addplot [
        thick,
        color of colormap={200 of tabgreens},
        mark=*,
        only marks,
        mark size=1pt,
        mark options={solid},
    ] table [x index=0, y index=1] {\main/figures/random_rank_adaptivity_cm-sgd/cm-sgd.txt};
    \addlegendentry{MC-SGD}
    
    \addcmsgd{10}{400}{5.66}{2.38}
    \addcmsgd{5}{600}{4}{2}
    \addcmsgd{1}{800}{2.83}{1.68}

\end{groupplot}

\node at ($(group c1r1.south)!0.5!(group c2r1.south) - (0, 24pt)$) {Manifold Dimension};

\end{tikzpicture}
\caption{(\textbf{Random tensor: rank continuation}) Error vs. manifold dimension with and without rank continuation. 
Left: TR-RMGN. T3-SVD and non-adaptive TR-RMGN use the same ranks as TR-RMGN(1). Right: MC-SGD. T3-SVD and non-adaptive MC-SGD use the same ranks as MC-SGD(1).
}
\label{fig:random_tensor_1}
\end{figure}

\paragraph{Rank continuation} We compare TR-RMGN (\secref{sec:trrmgn}) and MC-SGD (\secref{sec:cauchy_sgd}), with and without rank continuation (\secref{sec:rank_adaptivity}), for several values of the continuation parameter $n_\text{chunk}$. We write MC-SGD($n_\text{chunk}$) and TR-RMGN($n_\text{chunk}$) for the continuation variants and MC-SGD, TR-RMGN for the non-continuation variants.

The tensor has shape parameters $k=4$, $N=30$, $M=25$; training uses $n_s=400$ random probes and testing uses $n_t = 1000$ independent probes. We plot relative forward error against the manifold dimension---the number of independent degrees of freedom in the model, analogous to the number of weights in a neural network. The initial model is a Tucker tensor train with all ranks one; ranks then increase per our continuation scheme until the manifold dimension exceeds the training-data size (more ``unknowns'' than ``equations'').

The relative forward error for an approximation $\widetilde{\vb{T}}$ of $\vb{T}$ is
\begin{equation*}
\text{Relative Forward Error} ~~ := ~~ \sqrt{\frac{\sum_{i=1}^{n_t}\norm{\vb{y}^{(\sampleind)}(\vb{T}) - \vb{y}^{(\sampleind)}(\widetilde{\vb{T}})}^2}{\sum_{i=1}^{n_t}\norm{\vb{y}^{(i)}(\vb{T})}^2}},
\end{equation*}
where 
$$\vb{y}^{(\sampleind)}(\vb{T}) := T(\vb{x}^{(i)}, \dots, \vb{x}^{(i)})$$
and the test vectors $\vb{x}^{(i)}$ are drawn uniformly at random from the unit sphere in $\mathbb{R}^N$ (normalized white noise vectors).

The TR-RMGN error curve is V-shaped: it tracks the T3-SVD baseline as rank increases, then rises as overfitting sets in. TR-RMGN(1) follows the same descent but levels off at the minimum instead of rising. Rank continuation does not lower the minimal error, but it avoids overfitting and shifts computation toward smaller, cheaper models.

The MC-SGD curve depends strongly on rank continuation: small $n_\text{chunk}$ yields better models than large $n_\text{chunk}$, and MC-SGD without continuation overfits badly, so MC-SGD should be used with rank continuation. With it, MC-SGD(1), TR-RMGN(1), and well-tuned TR-RMGN achieve comparable accuracy.

\begin{figure}[tb]
\centering
\begin{tikzpicture}

\newcommand{\addtrrmgn}[4]{
    \addplot [
        thick,
        color of colormap={#2 of tabblues},
        mark=triangle*,
        mark size=1pt,
        mark options={solid}, 
        dash pattern=on #3pt off #4pt
    ] table [x index=0, y index=1] {\main/figures/random_data_scaling_tr-rmgn/rmgn-1_n-#1.txt};
    
    \addlegendentry{$\numactions=#1$} 
}

\newcommand{\addcmsgd}[4]{
    \addplot [
        thick,
        color of colormap={#2 of tabgreens},
        mark=*,
        mark size=1pt,
        mark options={solid}, 
        dash pattern=on #3pt off #4pt
    ] table [x index=0, y index=1] {\main/figures/random_data_scaling_cm-sgd/cm-sgd-1_n-#1.txt};
    
    \addlegendentry{$\numactions=#1$} 
}

\begin{groupplot}[
    group style={
        group size=2 by 1,
        y descriptions at=edge left, 
        horizontal sep=6pt,
    },
    xmin=1e2,
    xmax=4e5,
    ymax=1.1,
    xmode=log,
    ymode=log,
    width=0.52\textwidth,        
    height=0.52\textwidth, 
    title style={yshift=-6pt}
]

    \nextgroupplot[
        title={TR-RMGN(1)}, 
        ylabel={Relative Forward Error},
        legend pos=south west,
        legend cell align=left,
        legend style={
            font=\footnotesize, 
            draw=none, 
            fill=none
        },
        reverse legend]
        
    \addplot [
        black, 
        thick, 
        mark=square*, 
        only marks,
        mark size=1pt,
        mark options={solid}
    ] table [x index=0, y index=1] {\main/figures/random_data_scaling_tr-rmgn/t3svd.txt};
    \addlegendentry{T3-SVD}

    \addtrrmgn{1600}{200}{5.66}{2.38}
    \addtrrmgn{800}{400}{4}{2}
    \addtrrmgn{400}{600}{2.83}{1.68}
    \addtrrmgn{200}{800}{2}{1.41}
    \addtrrmgn{100}{1000}{1.41}{1.19}

    \nextgroupplot[
        title={MC-SGD(1)}, 
        legend pos=south west,
        legend cell align=left,
        legend style={
            font=\footnotesize, 
            draw=none, 
            fill=none
        },
        reverse legend]
        
    \addplot [
        black, 
        thick, 
        mark=square*, 
        only marks,
        mark size=1pt,
        mark options={solid}
    ] table [x index=0, y index=1] {\main/figures/random_data_scaling_cm-sgd/t3svd.txt};
    \addlegendentry{T3-SVD}

    \addcmsgd{1600}{200}{5.66}{2.38}
    \addcmsgd{800}{400}{4}{2}
    \addcmsgd{400}{600}{2.83}{1.68}
    \addcmsgd{200}{800}{2}{1.41}
    \addcmsgd{100}{1000}{1.41}{1.19}

\end{groupplot}

\node at ($(group c1r1.south)!0.5!(group c2r1.south) - (0, 24pt)$) {Manifold Dimension};

\end{tikzpicture}
\caption{(\textbf{Random tensor: varying training data}) Error vs.\ manifold dimension with 100, 200, 400, 800, and 1600 training samples. Left: TR-RMGN(1). Right: MC-SGD(1). T3-SVD uses the same ranks as TR-RMGN(1) and MC-SGD(1) in the left and right plots, respectively.
}
\label{fig:random_tensor_2}
\end{figure}

\paragraph{Amount of training data} \figref{fig:random_tensor_2} varies the number of training probes $\numactions$ for TR-RMGN(1) and MC-SGD(1); all other parameters match \figref{fig:random_tensor_1}.
Both error curves follow the T3-SVD benchmark until they plateau, and more training data extends this agreement to larger manifold dimensions before flattening. The only exception is the data-limited $n_s=100$ MC-SGD(1) case, which plateaus early and is therefore unable to match T3-SVD at any ranks. The curves occasionally dip below T3-SVD despite its Hilbert-Schmidt quasi-optimality, because the forward error, which differs from the Hilbert-Schmidt norm, is a quantity that TR-RMGN(1) and MC-SGD(1) actually optimize. As expected, more training data reduces the final error for both methods.

\paragraph{Tensor size and number of indices} \figref{fig:random_tensor_3} and \figref{fig:random_tensor_4} compare TR-RMGN(1), MC-SGD(1), and T3-SVD for varying $N$ and $k$, respectively; all other parameters match \figref{fig:random_tensor_1}.

\begin{figure}
\centering
\ifdefined\panelside\else\newlength{\panelside}\fi
\setlength{\panelside}{\dimexpr (\textwidth - 120pt)/3 \relax}

\begin{tikzpicture}

\newcommand{\addNplot}[1]{
    \addplot [
        black, 
        mark=square*, 
        only marks, 
        mark size=1pt
    ] table [x index=0, y index=1] {\main/figures/random_size_scaling/t3svd_N-#1.txt};

    \addplot [
        color=tabblue, 
        thick, 
        mark=triangle*, 
        mark options={solid}, 
        dash pattern=on 1pt off 1pt, 
        mark size=1pt
    ] table [x index=0, y index=1] {\main/figures/random_size_scaling/rmgn-1_N-#1.txt};

    \addplot [
        color=tabgreen, 
        thick, 
        mark=*, 
        mark options={solid}, 
        dash pattern=on 1pt off 1pt, 
        mark size=1pt
    ] table [x index=0, y index=1] {\main/figures/random_size_scaling/cm-sgd-1_N-#1.txt};
}

\begin{groupplot}[
    group style={
        group size=3 by 2,
        vertical sep=30pt,
        horizontal sep=30pt,
    },
    ymax=1.1,
    xmode=log,
    ymode=log,
    scale only axis=true,        
    width=\panelside,            
    height=\panelside,           
    title style={yshift=-6pt},
    tick label style={font=\scriptsize}
]

    \nextgroupplot[
        title={$N=10$},
        legend pos=south west,
        legend cell align=left,
        legend style={
            font=\scriptsize, 
            draw=none, 
            fill=none
        },
        reverse legend
    ] 
    \addNplot{10}

    \addlegendentry{T3-SVD}
    \addlegendentry{TR-RMGN(1)}
    \addlegendentry{MC-SGD(1)}
    
    \nextgroupplot[title={$N=20$}]
    \addNplot{20}
    
    \nextgroupplot[title={$N=30$}]
    \addNplot{30}
    
    \nextgroupplot[title={$N=40$}] 
    \addNplot{40}
    
    \nextgroupplot[title={$N=50$}] 
    \addNplot{50}

    \nextgroupplot[title={$N=60$}]
    \addNplot{60}
\end{groupplot}

\node at ($(group c2r2.south) - (0, 24pt)$) {Manifold Dimension};
\node[rotate=90] at ($(group c1r1.west)!0.5!(group c1r2.west) - (36pt, 0)$) {Relative Forward Error};

\end{tikzpicture}
\caption{(\textbf{Random tensor: varying index size}) Comparison of optimization methods for fitting the preconditioned random $N \times N \times N \times N \times (N-5)$ tensor for different $N$.
}
\label{fig:random_tensor_3}
\end{figure}

In \figref{fig:random_tensor_3}, the TR-RMGN(1) and MC-SGD(1) curves follow T3-SVD until they plateau\footnote{For the smallest tensor $N=10$, TR-RMGN(1) departs from T3-SVD early; likely due to the Newton convergence criterion which is set at a relative tolerance of $10^{-4}$.}, flattening earlier for larger $N$. Larger tensors are thus harder to fit and require more data. 

\begin{figure}
\centering
\ifdefined\panelside\else\newlength{\panelside}\fi
\setlength{\panelside}{\dimexpr (\textwidth - 120pt)/3 \relax}

\begin{tikzpicture}

\newcommand{\addkplot}[1]{
    \addplot [
        black, 
        mark=square*, 
        only marks, 
        mark size=1pt
    ] table [x index=0, y index=1] {\main/figures/random_order_scaling/t3svd_k-#1.txt};

    \addplot [
        color=tabblue, 
        thick, 
        mark=triangle*, 
        mark options={solid}, 
        dash pattern=on 1pt off 1pt, 
        mark size=1pt
    ] table [x index=0, y index=1] {\main/figures/random_order_scaling/rmgn-1_k-#1.txt};

    \addplot [
        color=tabgreen, 
        thick, 
        mark=*, 
        mark options={solid}, 
        dash pattern=on 1pt off 1pt, 
        mark size=1pt
    ] table [x index=0, y index=1] {\main/figures/random_order_scaling/cm-sgd-1_k-#1.txt};
}

\begin{groupplot}[
    group style={
        group size=3 by 2,
        vertical sep=30pt,
        horizontal sep=30pt,
    },
    ymax=1.1,
    xmode=log,
    ymode=log,
    scale only axis=true,        
    width=\panelside,            
    height=\panelside,           
    title style={yshift=-6pt},
    tick label style={font=\scriptsize}
]

    \nextgroupplot[
        title={$k=1$},
        legend pos=south west,
        legend cell align=left,
        legend style={
            font=\scriptsize, 
            draw=none, 
            fill=none
        },
        reverse legend
    ] 
    \addkplot{1}

    \addlegendentry{T3-SVD}
    \addlegendentry{TR-RMGN(1)}
    \addlegendentry{MC-SGD(1)}
    
    \nextgroupplot[title={$k=2$}]
    \addkplot{2}
    
    \nextgroupplot[title={$k=3$}]
    \addkplot{3}
    
    \nextgroupplot[title={$k=4$}] 
    \addkplot{4}
    
    \nextgroupplot[title={$k=5$}]
    \addkplot{5}

\end{groupplot}

\node at ($(group c2r2.south) - (0, 24pt)$) {Manifold Dimension};
\node[rotate=90] at ($(group c1r1.west)!0.5!(group c1r2.west) - (36pt, 0)$) {Relative Forward Error};

\end{tikzpicture}
\caption{(\textbf{Random tensor: varying number of indices}) Comparison of optimization methods for fitting the preconditioned random $40^{\times k} \times 35$ tensor for different $k$.
}
\label{fig:random_tensor_4}
\end{figure}

In \figref{fig:random_tensor_4}, behavior depends on $k$. For $k=1$ and $k=2$, the training data exceeds the number of entries in $\vb{T}$, so $\vb{T}$ is fully determined; the TR-RMGN(1), MC-SGD(1), and T3-SVD curves coincide and drop to machine epsilon once the ranks grow large enough that the Tucker-tensor-train submanifold fills the whole space, $\mathcal{M}_{\vb{n}, \vb{r}}=\left(\mathbb{R}^N\right)^{\times k} \times \mathbb{R}^M$. For $k=3$, $4$, and $5$, TR-RMGN(1) follows T3-SVD until it plateaus, flattening earlier for larger $k$; MC-SGD(1) tracks it for $k=3$ and $4$ but flattens earlier at $k=5$. For large $k$, TR-RMGN(1) may therefore be preferable: more accurate than MC-SGD(1), though more expensive.

\begin{figure}[tb]
\centering
\begin{tikzpicture}
\begin{groupplot}[
    group style={group size=3 by 1, horizontal sep=46pt},
    width=0.30\textwidth,
    height=0.30\textwidth,
    title style={yshift=-4pt},
    legend cell align=left,
    legend style={font=\scriptsize, draw=none, fill=none},
    label style={font=\footnotesize},
    tick label style={font=\scriptsize},
]

    \nextgroupplot[
        title={Probe Normalization},
        xlabel={Order $k$}, ylabel={GN Condition Number},
        ymode=log, legend pos=north west,
    ]
    \addplot [thick, color=black!75, mark=*, mark size=1pt]
        table [x index=0, y index=2] {\main/figures/conditioning/panelA_normalization.txt};
    \addlegendentry{Raw}
    \addplot [thick, color=black!75, mark=*, mark size=1pt, dash pattern=on 2pt off 2pt]
        table [x index=0, y index=1] {\main/figures/conditioning/panelA_normalization.txt};
    \addlegendentry{Unit}

    \nextgroupplot[
        title={Conditioning vs.\ Rank},
        xlabel={Manifold Dimension}, ylabel={GN Condition Number},
        xmode=log, ymode=log,
    ]
    \addplot [thick, color=black!75, mark=square*, mark size=1pt]
        table [x index=0, y index=1] {\main/figures/conditioning/panelB_conditioning.txt};

    \nextgroupplot[
        title={Warm-start Deflation},
        xlabel={Rank}, ylabel={Captured Gradient Energy (\%)},
        ymin=-5, ymax=105, legend style={at={(0.97,0.5)}, anchor=east},
    ]
    \addplot [thick, color=black!75, mark=*, mark size=1pt]
        table [x index=0, y index=3] {\main/figures/conditioning/panelC_deflation.txt};
    \addlegendentry{Cold}
    \addplot [thick, color=black!75, mark=*, mark size=1pt, dash pattern=on 2pt off 2pt]
        table [x index=0, y index=2] {\main/figures/conditioning/panelC_deflation.txt};
    \addlegendentry{Warm}

\end{groupplot}
\end{tikzpicture}
\caption{(\textbf{Random tensor: conditioning, normalization, and continuation}) Left: Gauss-Newton Hessian condition number on the tangent space at the rank-$2$ base versus order $k$, for unnormalized and unit-norm probe directions. Middle: Condition number versus manifold dimension as the rank grows, with normalized probes. Right: percentage of the gradient's energy lying in the captured (rank-$(r-1)$) subspace, for cold and warm starts, versus rank.
}
\label{fig:conditioning}
\end{figure}

\paragraph{Conditioning, normalization, and rank continuation} \figref{fig:conditioning} investigates the Hessian spectral properties which underlie the algorithm design decisions in \secref{sec:symmetric_action_data} and \secref{sec:initial_guess_new_ranks}, on the same tensor ($N=12$, $M=8$). We isolate this investigation from the details of the fitting algorithm by choosing uniform ranks (all Tucker and TT ranks the same), constructing the quasi-optimal Tucker tensor train with those ranks using T3-SVD, and forming the fully dense Riemannian Gauss-Newton Hessian at that point. \emph{Left:} the Gauss-Newton condition number at the rank-$2$ base versus order $k$, unnormalized versus unit-norm directions; the $\norm{\vb{x}_0}^k$ weight inflates it with order ($1.7\times$ at $k=2$ to $5.4\times$ at $k=5$) and cuts the effective sample fraction from $0.58$ to $0.12$. This motivates normalization (\secref{sec:symmetric_action_data}). \emph{Middle:} even normalized, conditioning worsens monotonically with rank ($\kappa$ from $10$ to $70$ over ranks $1$--$6$), as each added rank resolves a smaller eigenvalue; continuation does not cure this. \emph{Right:} continuation succeeds anyway because the warm-start gradient carries almost no energy in the already-resolved (captured) subspace---the rank-$(r-1)$ tangent space inside the rank-$r$ one---whereas a cold start puts about $90\%$ there. The warm refit thus acts only on the new directions and is far smaller in norm (cold-to-warm ratio $9$ at rank $2$, $770$ at rank $10$). These new directions correspond to a much narrower band of eigenvalues, which helps mitigate the impact of the Hessian ill-conditioning.

\subsection{Poisson PDE}
\label{sec:numerical_darcy}

We apply MC-SGD(1) to two mappings whose state equation is the Poisson PDE
\begin{align*}
-\nabla \cdot \kappa(\theta,u) \nabla u = s& \quad \text{in } \Omega.
\end{align*}
on the unit square $\Omega=[0,1]^2$, with sides denoted $\Gamma_\text{left}$, $\Gamma_\text{right}$, $\Gamma_\text{top}$, and $\Gamma_\text{bot}$.
The source
\begin{equation*}
s(b,c) \propto \exp\left(-\frac{1}{2}\frac{\left(b-0.5\right)^2 + \left(c-0.5\right)^2}{0.025^2}\right),
\end{equation*}
is a Gaussian centered in the domain, normalized so that $\int_\Omega s(b,c)~db~dc = 1$; we write spatial coordinates as $b$ and $c$ since $x$ and $y$ denote other quantities elsewhere. The PDE models steady-state heat conduction, with temperature $u$, conductivity $\kappa$, and source $s$. 
The two mappings differ in $\kappa$, boundary conditions, quantity of interest $q$, and the distribution of $\theta$. 

\subsubsection{Example 1: Linear PDE, Gaussian parameter, Dirichlet Q.o.I.}
\label{sec:transmission}

\begin{figure}
\centering
\newcommand{\adddolfinplot}[5][]{
    \nextgroupplot[
        title={#3},
        point meta min=#4,
        point meta max=#5,
        colorbar style/.append style={
            xtick={#4, #5} 
        },
        #1
    ]
    \addplot graphics [xmin=0, xmax=1, ymin=0, ymax=1] {\main/figures/poisson_components/#2};
    \addplot [draw=none, forget plot, point meta=y] coordinates {(0,0)};
}

\begin{tikzpicture}

\begin{groupplot}[
    group style={
        group size=5 by 1,      
        horizontal sep=8pt     
    },
    scale only axis,
    enlargelimits=false,
    width=0.1625\textwidth,    
    height=0.1625\textwidth,
    colormap name=grey_r,           
    colorbar horizontal,         
    colorbar style={
        width=0.1625\textwidth,
        height=8pt,
        xticklabel style={font=\footnotesize}, 
        yshift=6pt,
        grid=major,   
        major grid style={black, solid}
    },
    title style={yshift=-6pt},      
    xtick=\empty, 
    yticklabel style={
        font=\footnotesize,
        /pgf/number format/fixed, 
        /pgf/number format/precision=0
    }
]

    \adddolfinplot[
        ytick=\empty
    ]{noise_nocbar_1.png}{$x$}{-3.4}{3.8}
    
    \adddolfinplot[ytick=\empty]{theta_nocbar_1.png}{$\theta$}{-1.76}{1.52}

    \nextgroupplot[
        title={$\kappa$},
        ytick=\empty,
        colormap name=grey_r,           
        colorbar horizontal,         
        colorbar style={
            width=0.1625\textwidth,
            height=8pt,
            yshift=6pt,
            grid=major,   
            major grid style={black, solid},
            xtick={0.15, 4.35}, 
            xticklabels={
                \makebox[0pt][l]{0.15}, 
                \makebox[0pt][r]{4.35}
            },
            xticklabel style={
                font=\footnotesize,
                inner xsep=0pt 
            }
        },
        point meta min=0.15,
        point meta max=4.35,
        colorbar style/.append style={
            xtick={0.15, 4.35} 
        },
    ]
    \addplot graphics [xmin=0, xmax=1, ymin=0, ymax=1] {\main/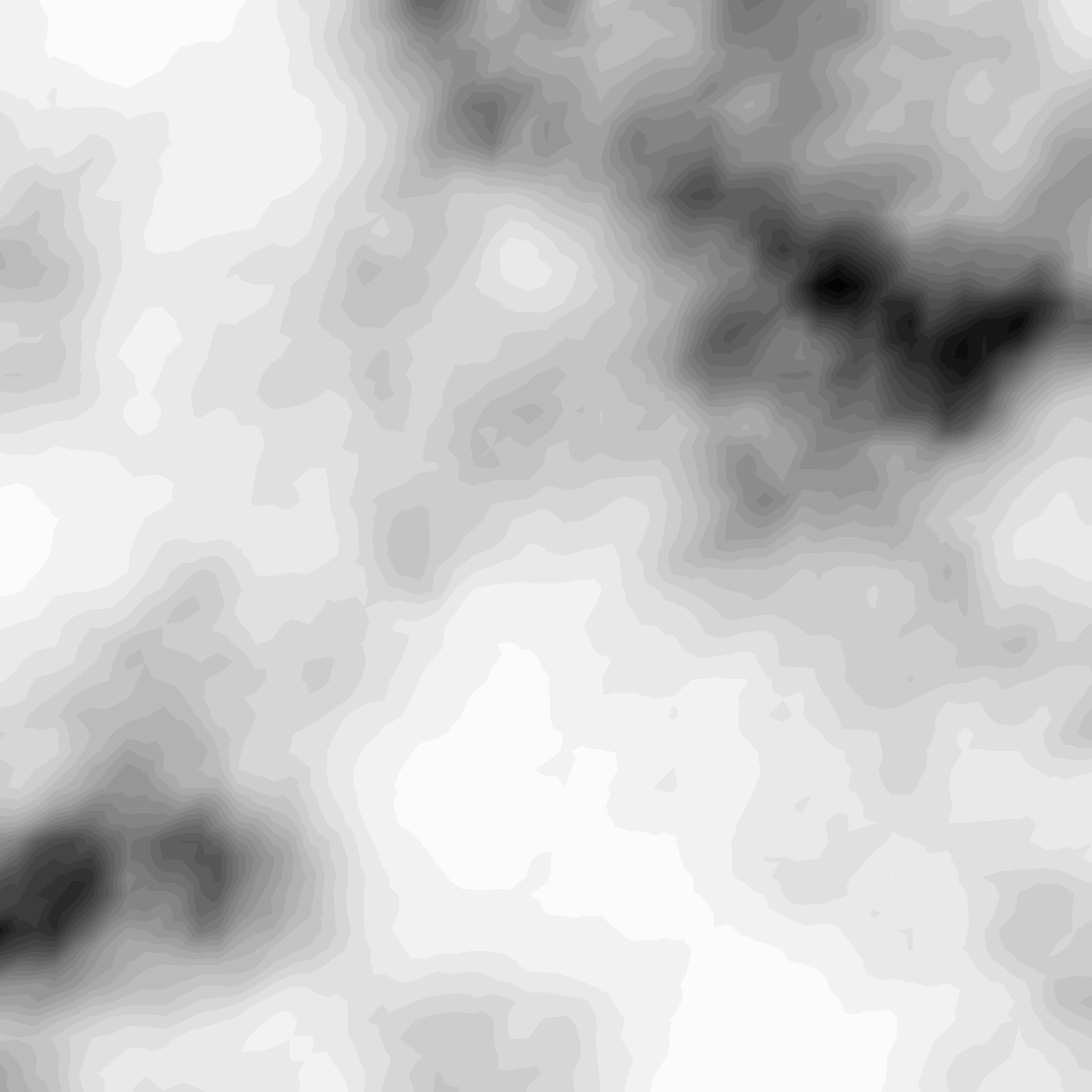};
    \addplot [draw=none, forget plot, point meta=y] coordinates {(0,0)};
    
    \adddolfinplot[
        ytick=\empty,
        clip=false,
        execute at end axis={
            \draw[black, thick] (-0.02, 0.98) rectangle (1.02, 1.02);
        }
    ]{state_nocbar_1.png}{$u$}{0}{0.78}

    \nextgroupplot[
        title={$q$},
        hide axis=false,
        colorbar=false,
        xmin=0,
        xmax=1,
        xtick={0, 1},
        minor xtick={0.25, 0.5, 0.75},
        xticklabel style={
            font=\footnotesize,
            /pgf/number format/fixed,
            /pgf/number format/precision=0
        },
        ymin=0, ymax=0.21,
        ytick={0, 0.05, 0.10, 0.15, 0.20},
        yticklabel pos=right,    
        yticklabel style={   
            /pgf/number format/fixed,
            /pgf/number format/fixed zerofill, 
            /pgf/number format/precision=2     
        },
        xminorgrids=true,
        ymajorgrids=true,
        grid style={gray!20, solid},       
        minor grid style={gray!20, solid}
    ]
    \addplot [
        thick,
        black,
        mark=none
    ] table [x index=0, y index=1] {\main/figures/poisson_components/fig2a_1.txt};
    
\end{groupplot}


\end{tikzpicture}
\caption{(\textbf{Example 1: illustration}) Illustration of the first Poisson example in \secref{sec:transmission}. From left to right: noise vector $x = C^{-1} \theta$, parameter $\theta$, conductivity $\kappa=\exp(\theta)$, temperature $u$, and quantity of interest $q$. We highlight the top boundary in the plot of $u$ to indicate that $q$ is the trace of $u$ along the top boundary. We approximate the overall mapping $f: x \mapsto q$, which is the composition of mappings $x \mapsto \theta \mapsto \kappa \mapsto u \mapsto q$.
}
\label{fig:transmission_overall_mapping}
\end{figure}

In the first example we impose homogeneous Dirichlet conditions on $\Gamma_\text{left}$, $\Gamma_\text{right}$, $\Gamma_\text{bot}$ and homogeneous Neumann on $\Gamma_\text{top}$, with quantity of interest
\begin{equation*}
q = u|_{\Gamma_\text{top}},
\end{equation*}
the trace of the temperature along the top boundary.
The parameter
$$\theta \sim N(\theta_0, C^2)$$
has mean $\theta_0=0$ and Mat\'ern covariance operator $C$. Following \cite{lindgren2011explicit}, we implement $C$ as the inverse of $-\gamma \Delta + \delta$ with homogeneous Neumann boundary conditions, choosing $\gamma$ and $\delta$ by the formulas of \cite{daon2016mitigating,villa2024note} to match a correlation length of $0.5$ and pointwise standard deviation $\sigma$. We take $\sigma \in \{0.5, 0.75, 1.0\}$ in \figref{fig:taylor_samples_pointwise_poisson} and \figref{fig:hist_poisson_taylor_tucker} and $\sigma=0.75$ in all other figures. We use
\begin{equation*}
\kappa = \exp(\theta),
\end{equation*}
yielding a log-normal conductivity field. We approximate the mapping
$
f(x) = q(Cx)
$
by a T4S Taylor series about $x=0$. \figref{fig:transmission_overall_mapping} illustrates the composite mapping
$$x \mapsto \theta \mapsto \kappa \mapsto u \mapsto q.$$
Note that while the PDE is linear as a mapping $s \mapsto u$, this is not the mapping being approximated here. The relevant mapping is the coefficient-to-solution operator $\kappa \mapsto u$, which is nonlinear. Intuitively, the nonlinearity comes from ``blocking'': changing $\kappa$ at one location modifies the shape and magnitude of downstream flow, which changes the sensitivity of the map to further changes in $\kappa$ in those downstream locations, i.e., a nontrivial second derivative. The impacts of multiple blocks combine, yielding nontrivial third and higher derivatives. The larger the pointwise variations are in $\kappa$, the more significant the blocking, and hence the more substantial the nonlinearity. The pointwise exponentiation $\theta \mapsto \kappa$ also adds nonlinearity.

We discretize $\theta$ and $\kappa$ with first-order continuous Galerkin (CG1) elements on a triangular mesh on a $60 \times 61$ grid, $u$ with CG2 (second-order) elements on the same mesh, and $q$ as the trace of the CG1 elements on $\Gamma_\text{top}$.

\begin{figure}
\centering
\begin{tikzpicture}

\newcommand{\addtaylorplot}[2]{
    \nextgroupplot[title={#2}]
    
    \pgfplotsinvokeforeach{6,5,4,3,2,1}{
        \pgfmathsetmacro{\dashlen}{2*sqrt(##1)}
        \pgfmathsetmacro{\skiplen}{1.5*sqrt(##1)}
        \pgfmathsetmacro{\colorpos}{(##1-1)*190}
        \pgfmathtruncatemacro{\kindex}{##1 - 1}
        
        \edef\temp{
            \noexpand\addplot [
                forget plot,
                very thick,
                line cap=round,
                color=tab##1,
                mark=none,
                opacity=0.7
            ] table [x index=0, y index=##1] {\main/figures/poisson_taylor_samples/fig2b_stdev#1_1.txt};
        }
        \temp
    }

    \addplot [
        forget plot,
        thick,
        line cap=round,
        black,
        dashed,
        mark=none
    ] table [x index=0, y index=1] {\main/figures/poisson_taylor_samples/fig2b_obs_1.txt};
}

\begin{groupplot}[
    group style={
        group size=3 by 1,
        horizontal sep=8pt,
        y descriptions at=edge left
    },
    width=0.225\textwidth,
    height=0.225\textwidth,
    scale only axis,
    xmin=0, xmax=1,
    ymin=0, ymax=0.21,
    xtick={0, 1},
    minor xtick={0.25, 0.5, 0.75},
    ytick={0, 0.05, 0.10, 0.15, 0.20},
    xticklabel style={
        font=\footnotesize,
        /pgf/number format/fixed,
        /pgf/number format/precision=0
    },
    yticklabel style={
        font=\footnotesize,
        /pgf/number format/fixed,
        /pgf/number format/fixed zerofill, 
        /pgf/number format/precision=2     
    },
    xminorgrids=true,
    ymajorgrids=true,
    grid style={gray!20, solid},       
    minor grid style={gray!20, solid},
    title style={yshift=-6pt},  
    legend columns=-1,
    legend style={
        at={(0.4125, -0.1)},
        anchor=north,
        font=\footnotesize,
        draw=none,                  
        fill=none,
        /tikz/every even column/.append style={column sep=6pt}
    },
]

    \addtaylorplot{0.5}{$\sigma = 0.5$}
    \addtaylorplot{0.75}{$\sigma = 0.75$}

    \pgfplotsinvokeforeach{1,2,3,4,5,6}{
        \pgfmathtruncatemacro{\kk}{#1 - 1}
        \addlegendimage{very thick, line cap=round, color=tab#1, mark=none, opacity=0.7}
        \edef\temp{\noexpand\addlegendentry{$k=\kk$}}\temp
    }
    \addlegendimage{thick, line cap=round, black, dashed}
    \addlegendentry{Truth}

    \addtaylorplot{1.0}{$\sigma = 1.0$}

\end{groupplot}


\end{tikzpicture}
\caption{(\textbf{Example 1: Taylor approximation of one sample}) Taylor series approximations (using the Taylor series of the map $f$, with no dimension reduction or T4S approximation) of $q$ corresponding to scaled versions of $\theta$ shown in \figref{fig:transmission_overall_mapping}. On the left, middle, and right, we show the samples of $q$ corresponding to $\theta$ drawn from distributions with pointwise standard deviations $\sigma=0.5$, $\sigma=0.75$, and $\sigma=1.0$, respectively.
}
\label{fig:taylor_samples_pointwise_poisson}
\end{figure}

\paragraph{Influence of pointwise standard deviation on the Taylor series} \figref{fig:taylor_samples_pointwise_poisson} plots Taylor series approximations of $q$ (from the Taylor series of $f$, with no dimension reduction or T4S approximation) for the parameter $\theta$ of \figref{fig:transmission_overall_mapping}, at $\sigma=0.5$, $0.75$, and $1.0$. For this sample, the ratio of the largest to smallest $\kappa$ in $\Omega$ is $8.6$, $25.0$, and $73.2$, respectively. As expected, the Taylor series converges to $q$ faster for smaller $\sigma$ and slower for larger $\sigma$.

\begin{figure}[tb]
\centering
\subfile{figures/poisson_taylor_tucker_combined_hist/poisson_taylor_tucker_combined_hist}
\caption{(\textbf{Example 1: Taylor error distributions}) KDEs of normalized error in the quantity of interest $q$, organized by Taylor series order $k$. \emph{Left:} exact Taylor series, for parameters $\theta$ with pointwise standard deviation $\sigma\in\{0.5, 0.75, 1.0\}$. \emph{Right:} dimension-reduced Taylor series, for reduced dimensions $N=31$, $82$, and $209$ (tolerance parameters $\epsilon=0.25$, $0.05$, and $0.01$ in the dimension reduction algorithm); the true full-space Taylor series is shown in black for reference.
}
\label{fig:hist_poisson_taylor_tucker}
\end{figure}

For an approximation $\widetilde{q} \approx q$, the normalized error is defined as:
\begin{equation*}
\text{Normalized Error} ~~ := ~~ \frac{\norm{q(\theta) - \widetilde{q}(\theta)}}{\norm{q(\theta_0)}}.
\end{equation*}
\figref{fig:hist_poisson_taylor_tucker} (left) shows KDEs of the normalized error in approximating $q$ by Taylor series, computed from $500$ test samples with the default \texttt{stats.gaussian\_kde}() function in SciPy \citep{2020SciPy-NMeth}. As the order increases, the error distribution shifts down and spreads out on the log scale. Smaller $\sigma$ decreases faster with order than larger $\sigma$, since smaller $\sigma$ yields samples closer to the expansion point.

\paragraph{Dimension reduction} \figref{fig:hist_poisson_taylor_tucker} (right) shows KDEs of the relative forward error when $q$ is approximated by the dimension-reduced Taylor series of \secref{sec:initial_dimension_reduction}, for tolerances $\epsilon= 0.25$, $0.05$, and $0.01$ (reduced input dimensions $N=31$, $82$, and $209$). Each reduced curve tracks the exact Taylor error as the order increases, then stagnates at a fixed level for all higher orders---levels that match the true series up to orders roughly $2$, $3$, and $5$, respectively. Dimension reduction thus caps the achievable error: higher orders require larger reduced dimensions.

\begin{figure}
\centering
\ifdefined\panelside\else\newlength{\panelside}\fi
\setlength{\panelside}{\dimexpr (\textwidth - 120pt)/3 \relax}

\begin{tikzpicture}

\newcommand{\addcmsgd}[5]{
    \addplot [
        thick,
        color of colormap={#3 of tabgreens},
        mark=*,
        mark size=1pt,
        mark options={solid}, 
        dash pattern=on #4pt off #5pt
    ] table [x index=0, y index=1] {\main/figures/poisson_fitting_convergence/fig2f_k#1_train#2.txt};
}

\newcommand{\addkplot}[1]{
    \addcmsgd{#1}{6400}{200}{5.66}{2.38}
    \addcmsgd{#1}{1600}{467}{5.66}{2.38}
    \addcmsgd{#1}{400}{733}{2.83}{1.68}
    \addcmsgd{#1}{100}{1000}{1.41}{1.19}
}

\begin{groupplot}[
    group style={
        group size=3 by 2,
        vertical sep=30pt,
        horizontal sep=30pt,
    },
    ymax=1.1,
    xmode=log,
    ymode=log,
    scale only axis=true,        
    width=\panelside,            
    height=\panelside,           
    title style={yshift=-6pt},
    tick label style={font=\scriptsize}
]

    \nextgroupplot[
        title={$k=1$},
        legend pos=south west,
        legend cell align=left,
        legend style={
            font=\scriptsize, 
            draw=none, 
            fill=none
        },
        reverse legend
    ] 
    \addkplot{1}

    \addlegendentry{$\numactions=6400$}
    \addlegendentry{$\numactions=1600$}
    \addlegendentry{$\numactions=400$}
    \addlegendentry{$\numactions=100$}
    
    \nextgroupplot[title={$k=2$}]
    \addkplot{2}
    
    \nextgroupplot[title={$k=3$}]
    \addkplot{3}
    
    \nextgroupplot[title={$k=4$}] 
    \addkplot{4}
    
    \nextgroupplot[title={$k=5$}] 
    \addkplot{5}

\end{groupplot}

\node at ($(group c2r2.south) - (0, 24pt)$) {Manifold Dimension};
\node[rotate=90] at ($(group c1r1.west)!0.5!(group c1r2.west) - (36pt, 0)$) {Relative Forward Error};

\end{tikzpicture}
\caption{(\textbf{Example 1: derivative fitting}) Convergence of MC-SGD(1) for each derivative tensor. Curves show the relative forward error for fitting the derivatives with different numbers of training samples, $n_s$.
}
\label{fig:poisson_t3_convergence}
\end{figure}

\paragraph{T4S approximation} \figref{fig:poisson_t3_convergence} plots relative forward error versus manifold dimension for fitting the derivative tensors with MC-SGD(1), at $n_s=100$, $400$, $1600$, and $6400$. The reduced bases use $\epsilon=0.01$ (reduced input dimension 209), and rank continuation stops when the manifold dimension exceeds $50\%$ of the training-data size (twice as many ``equations'' as ``unknowns'') or the ranks exceed 125, whichever comes first. Final error increases with derivative order $k$ and decreases with $n_s$: higher-order derivatives are harder to fit but remain accurate given sufficient training data. 

\begin{figure}[tb]
\centering
\setlength{\kdefulltw}{\textwidth}%
\begin{minipage}[c]{0.56\textwidth}
\centering
{\setlength{\textwidth}{\kdefulltw}%
 \resizebox{\linewidth}{!}{\begin{tikzpicture} 

\begin{groupplot}[
    group style={
        group size=1 by 6,             
        x descriptions at=edge bottom, 
        y descriptions at=edge left,   
        vertical sep=-1.8em,
        group name=hists
    },
    colormap name=tabgreens,
    xmin=5e-4,
    xmax=1e0,
    ymin=0,
    ymax=1.4,
    axis x line*=bottom,
    tick label style={font=\footnotesize},
    ytick=\empty,
    y axis line style={draw=none},
    every axis y label/.style={
        at={(axis description cs:-0.02, 0.33)},
        rotate=0, 
        anchor=west,
        font=\footnotesize
    },
    scale only axis=true,                       
    width=0.6\textwidth,                       
    height={(0.95 * 1.618 / 13) * \textwidth},   
    enlarge x limits=true,
    xmode=log,
]

\nextgroupplot[
    ylabel={$k=0$},
]

\addplot [
    thick, 
    black,
    opacity=0.8
] table [x index=0, y index=1] {\main/figures/poisson_t4s_hist/fig2g_taylor_k0.txt};

\addplot [
    thick, 
    line cap=round,
    color=black,
    dash pattern=on 2 pt off 1.5 pt,
    opacity=0.8
] table [x index=0, y index=1] {\main/figures/poisson_t4s_hist/fig2g_tucker_k0.txt};

\foreach \ns/\cname in {6400/tab4, 1600/tab3, 400/tab2, 100/tab1} {
    \edef\temp{
        \noexpand\addplot [
            draw=none, 
            fill=\cname, 
            fill opacity=0.2, 
            forget plot
        ] table [x index=0, y index=1] {\noexpand\main/figures/poisson_t4s_hist/fig2g_t4s_k0_train\ns.txt} \noexpand\closedcycle;

        \noexpand\addplot [
            thick,
            color=\cname,
            draw opacity=0.8
        ] table [x index=0, y index=1] {\noexpand\main/figures/poisson_t4s_hist/fig2g_t4s_k0_train\ns.txt};
    }
    \temp
}

\nextgroupplot[
    ylabel={$k=1$},
    legend style={
        at={(0.10, 1.25)}, 
        anchor=north west,
        font=\footnotesize,
        /tikz/every node/.style={anchor=west},
        rounded corners=2pt,
    },
    legend cell align=left,
    reverse legend
]

\addplot [
    thick, 
    black,
    opacity=0.8
] table [x index=0, y index=1] {\main/figures/poisson_t4s_hist/fig2g_taylor_k1.txt};
\addlegendentry{Taylor}

\addplot [
    thick, 
    line cap=round,
    color=black,
    dash pattern=on 2 pt off 1.5 pt,
    opacity=0.8
] table [x index=0, y index=1] {\main/figures/poisson_t4s_hist/fig2g_tucker_k1.txt};
\addlegendentry{Reduced Taylor}

\foreach \ns/\cname in {6400/tab4, 1600/tab3, 400/tab2, 100/tab1} {
    \edef\temp{
        \noexpand\addplot [
            draw=none, 
            fill=\cname, 
            fill opacity=0.2, 
            forget plot
        ] table [x index=0, y index=1] {\noexpand\main/figures/poisson_t4s_hist/fig2g_t4s_k1_train\ns.txt} \noexpand\closedcycle;
    
        \noexpand\addplot [
            thick,
            color=\cname,
            draw opacity=0.8,
        ] table [x index=0, y index=1] {\noexpand\main/figures/poisson_t4s_hist/fig2g_t4s_k1_train\ns.txt};

        \noexpand\addlegendentry{T4S ($\numactions\!=\!\noexpand\pgfmathprintnumber[fixed, precision=0, 1000 sep={}]{\ns}$)\!\!}
    }
    \temp
}

\nextgroupplot[ylabel={$k=2$}]

\addplot [
    thick, 
    black,
    opacity=0.8
] table [x index=0, y index=1] {\main/figures/poisson_t4s_hist/fig2g_taylor_k2.txt};

\addplot [
    thick, 
    line cap=round,
    color=black,
    dash pattern=on 2 pt off 1.5 pt,
    opacity=0.8
] table [x index=0, y index=1] {\main/figures/poisson_t4s_hist/fig2g_tucker_k2.txt};

\foreach \ns/\cname in {6400/tab4, 1600/tab3, 400/tab2, 100/tab1} {
    \edef\temp{
        \noexpand\addplot [
            draw=none, 
            fill=\cname, 
            fill opacity=0.2, 
            forget plot
        ] table [x index=0, y index=1] {\noexpand\main/figures/poisson_t4s_hist/fig2g_t4s_k2_train\ns.txt} \noexpand\closedcycle;
    
        \noexpand\addplot [
            thick,
            color=\cname,
            draw opacity=0.8
        ] table [x index=0, y index=1] {\noexpand\main/figures/poisson_t4s_hist/fig2g_t4s_k2_train\ns.txt};
    }
    \temp
}

\nextgroupplot[ylabel={$k=3$}]

\addplot [
    thick, 
    black,
    opacity=0.8
] table [x index=0, y index=1] {\main/figures/poisson_t4s_hist/fig2g_taylor_k3.txt};

\addplot [
    thick, 
    line cap=round,
    color=black,
    dash pattern=on 2 pt off 1.5 pt,
    opacity=0.8
] table [x index=0, y index=1] {\main/figures/poisson_t4s_hist/fig2g_tucker_k3.txt};

\foreach \ns/\cname in {6400/tab4, 1600/tab3, 400/tab2, 100/tab1} {
    \edef\temp{
        \noexpand\addplot [
            draw=none, 
            fill=\cname, 
            fill opacity=0.2, 
            forget plot
        ] table [x index=0, y index=1] {\noexpand\main/figures/poisson_t4s_hist/fig2g_t4s_k3_train\ns.txt} \noexpand\closedcycle;
    
        \noexpand\addplot [
            thick,
            color=\cname,
            draw opacity=0.8
        ] table [x index=0, y index=1] {\noexpand\main/figures/poisson_t4s_hist/fig2g_t4s_k3_train\ns.txt};
    }
    \temp
}

\nextgroupplot[ylabel={$k=4$}]

\addplot [
    thick, 
    black,
    opacity=0.8
] table [x index=0, y index=1] {\main/figures/poisson_t4s_hist/fig2g_taylor_k4.txt};

\addplot [
    thick, 
    line cap=round,
    color=black,
    dash pattern=on 2 pt off 1.5 pt,
    opacity=0.8
] table [x index=0, y index=1] {\main/figures/poisson_t4s_hist/fig2g_tucker_k4.txt};

\foreach \ns/\cname in {6400/tab4, 1600/tab3, 400/tab2, 100/tab1} {
    \edef\temp{
        \noexpand\addplot [
            draw=none, 
            fill=\cname, 
            fill opacity=0.2, 
            forget plot
        ] table [x index=0, y index=1] {\noexpand\main/figures/poisson_t4s_hist/fig2g_t4s_k4_train\ns.txt} \noexpand\closedcycle;
    
        \noexpand\addplot [
            thick,
            color=\cname,
            draw opacity=0.8
        ] table [x index=0, y index=1] {\noexpand\main/figures/poisson_t4s_hist/fig2g_t4s_k4_train\ns.txt};
    }
    \temp
}

\nextgroupplot[ylabel={$k=5$}]

\addplot [
    thick, 
    black,
    opacity=0.8
] table [x index=0, y index=1] {\main/figures/poisson_t4s_hist/fig2g_taylor_k5.txt};

\addplot [
    thick, 
    line cap=round,
    color=black,
    dash pattern=on 2 pt off 1.5 pt,
    opacity=0.8
] table [x index=0, y index=1] {\main/figures/poisson_t4s_hist/fig2g_tucker_k5.txt};

\foreach \ns/\cname in {6400/tab4, 1600/tab3, 400/tab2, 100/tab1} {
    \edef\temp{
        \noexpand\addplot [
            draw=none, 
            fill=\cname, 
            fill opacity=0.2, 
            forget plot
        ] table [x index=0, y index=1] {\noexpand\main/figures/poisson_t4s_hist/fig2g_t4s_k5_train\ns.txt} \noexpand\closedcycle;
    
        \noexpand\addplot [
            thick,
            color=\cname,
            draw opacity=0.8
        ] table [x index=0, y index=1] {\noexpand\main/figures/poisson_t4s_hist/fig2g_t4s_k5_train\ns.txt};
    }
    \temp
}

\end{groupplot}

\node[anchor=north, yshift=-14pt] at (hists c1r6.south) {Normalized Error};
\node[anchor=south, xshift=-4pt, rotate=90] at (hists c1r4.west) {Taylor Series Order, $k$};

\end{tikzpicture}}}%
\end{minipage}\hfill
\begin{minipage}[c]{0.42\textwidth}
\centering
\small
{%
\renewcommand{\arraystretch}{1.35}%
\begin{tabular}{@{}c @{} c @{} c @{\hspace{1.6em}} c @{} c @{} c @{\hspace{1.6em}} c @{}}
\toprule
 & & \multicolumn{2}{c}{Taylor} & & \multicolumn{2}{c}{T4S} \\
\cmidrule(lr){3-4}\cmidrule(lr){6-7}
$k$ & \rule{1.4em}{0pt} & P50 & P90 & \rule{2.2em}{0pt} & P50 & P90 \\
\midrule
0 & & $4.0\mathrm{e}{-1}$ & $7.9\mathrm{e}{-1}$ & & $4.0\mathrm{e}{-1}$ & $7.9\mathrm{e}{-1}$ \\
1 & & $1.4\mathrm{e}{-1}$ & $4.1\mathrm{e}{-1}$ & & $1.4\mathrm{e}{-1}$ & $4.1\mathrm{e}{-1}$ \\
2 & & $5.5\mathrm{e}{-2}$ & $1.9\mathrm{e}{-1}$ & & $5.6\mathrm{e}{-2}$ & $1.8\mathrm{e}{-1}$ \\
3 & & $2.3\mathrm{e}{-2}$ & $1.1\mathrm{e}{-1}$ & & $2.2\mathrm{e}{-2}$ & $1.1\mathrm{e}{-1}$ \\
4 & & $9.2\mathrm{e}{-3}$ & $6.6\mathrm{e}{-2}$ & & $1.0\mathrm{e}{-2}$ & $6.7\mathrm{e}{-2}$ \\
5 & & $4.0\mathrm{e}{-3}$ & $3.2\mathrm{e}{-2}$ & & $6.7\mathrm{e}{-3}$ & $3.3\mathrm{e}{-2}$ \\
\bottomrule
\end{tabular}%
}
\end{minipage}
\caption{(\textbf{Example 1: T4S errors}) \emph{Left:} KDEs of the normalized error for the
T4S approximations of \figref{fig:poisson_t3_convergence}, organized by Taylor order $k$, for
models fit with $n_s\in\{100,400,1600,6400\}$ training samples. \emph{Right:} median (P50) and
90th-percentile (P90) 
normalized error of the exact Taylor series and of the best-data T4S model ($n_s=6400$), each measured against the true $q$.
}
\label{fig:t4s_hist_poisson}
\end{figure}

The left side of \figref{fig:t4s_hist_poisson} shows normalized error KDEs for the T4S models from the fitting process of \figref{fig:poisson_t3_convergence}.
These errors are measured against the true output $q$. The exact-Taylor curves are the truncation error of the full-tensor Taylor series, and the T4S curves add the low-rank fitting error on top. T4S matching the exact-Taylor curves therefore means the low-rank Tucker tensor-train representation is nearly lossless; it reproduces the accuracy of the exact Taylor series relative to the true PDE.
More training data lets the T4S models match higher-order Taylor series before stagnating; at $n_s=6400$ the KDE for the T4S model matches that of the exact Taylor series for orders $k=0,1,\dots,4$ but not $k=5$.
The table on the right side of \figref{fig:t4s_hist_poisson} gives the corresponding median (P50) and 90th-percentile (P90) errors for the exact Taylor series and the T4S models with $n_s=6400$. The P90 results for T4S track the Taylor series for $k=0,\dots,5$. The P50 results for T4S track the Taylor series for $k=1,\dots,4$ but differ for $k=5$. The T4S P50 error still decreases from $k=4$ to $k=5$, just not as much as the Taylor series.
This deviation of T4S from Taylor at $k=5$ is caused by the accumulated low-rank fitting error across the retained terms, not a poorly fit fifth derivative (Figure~\ref{fig:poisson_t3_convergence} shows it is captured nearly as accurately as the fourth): the models match the exact series on the samples hardest to approximate, where truncation error dominates the floor, and fall short on easier samples, where the exact series already reaches errors below it. Because each derivative fit improves with more probes (Figure~\ref{fig:poisson_t3_convergence}), this floor is data-limited, not intrinsic.



T4S contains the standard low-order low-rank local surrogates as special cases, so the $k=1$ and $k=2$ curves are those models themselves, not merely lower-order T4S variants. At $k=1$, T4S is a low-rank affine model of $f$: the low-rank linearized (Gauss–Newton) surrogate that is the workhorse of large-scale inverse problems and uncertainty quantification. Its Jacobian is fit to machine precision within the reduced space (\figref{fig:poisson_t3_convergence}, left), and at our tolerance $\epsilon$ the dimension-reduction error is negligible (\figref{fig:hist_poisson_taylor_tucker}, right), so the $k=1$ curve is the strongest baseline a low-rank linearization could provide. At $k=2$, T4S adds a Tucker decomposition of the second derivative (a 3-tensor); such surrogates are used only rarely in PDE-constrained outer-loop problems and take no single established form, but once the $k=2$ tensor is fit to Taylor accuracy for $n_s \ge 400$ the $k=2$ curve likewise stands in for this class of Tucker-based quadratic surrogates. The $k=3,4,5$ terms have no established counterpart; the accuracy they add over the $k=1$ and $k=2$ curves is the gain from carrying the local expansion beyond these existing models.

\subsubsection{Example 2: Nonlinear PDE, Logistic parameter, Neumann Q.o.I.}
\label{sec:outflow}

\begin{figure}
\centering
\newcommand{\adddolfinplot}[3][]{
    \nextgroupplot[
        title={#3},
        #1
    ]
    \addplot graphics [xmin=0, xmax=1, ymin=0, ymax=1] {\main/figures/darcy_components/#2};
    \addplot [draw=none, forget plot, point meta=y] coordinates {(0,0)};
}

\begin{tikzpicture}

\begin{groupplot}[
    group style={
        group size=3 by 1,
        horizontal sep=8pt
    },
    scale only axis,
    enlargelimits=false,
    width=0.225\textwidth,
    height=0.225\textwidth,
    title style={yshift=-6pt},
    xtick=\empty, 
    yticklabel style={
        font=\footnotesize,
        /pgf/number format/fixed,
        /pgf/number format/precision=0
    }
]

    \adddolfinplot[ytick=\empty,
        colormap name=grey_r,
        colorbar horizontal,
        colorbar style={
            width=0.225\textwidth,
            height=8pt,
            xticklabel style={font=\footnotesize},
            yshift=6pt,
            grid=major,
            major grid style={black, solid},
            xtick={1.0, 10.0},
            xticklabels={
                \makebox[0pt][l]{1.0},
                \makebox[0pt][r]{10.0}
            },
            xticklabel style={
                font=\footnotesize,
                inner xsep=0pt
            }
        },
        point meta min=1.0,
        point meta max=10.0,
    ]{kappa_p25.png}{$\theta$}

    \adddolfinplot[
        ytick=\empty,
        clip=false,
        execute at end axis={
            \draw[black, thick] (-0.02, 0.98) rectangle (1.02, 1.02);
        }
    ]{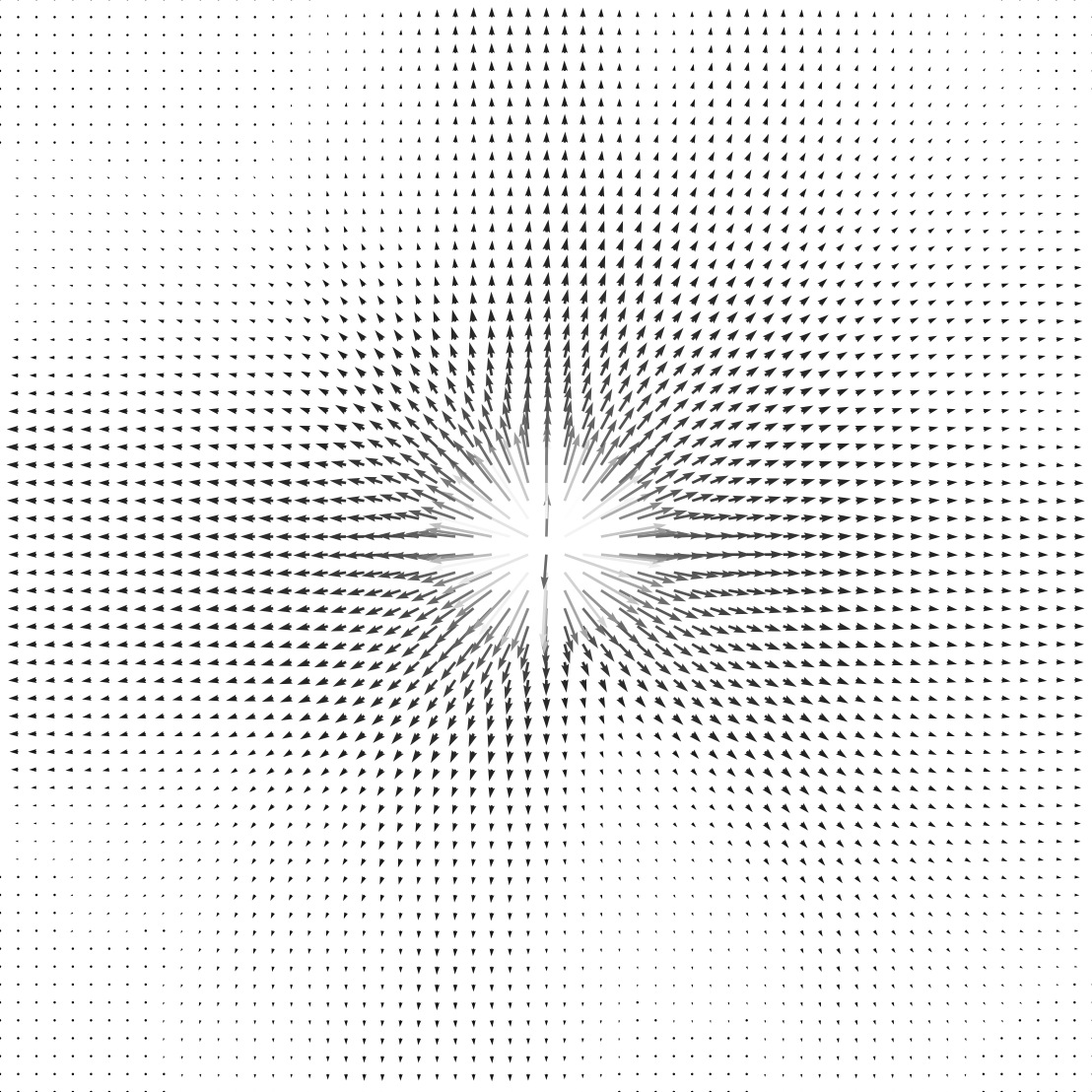}{$-\kappa \nabla u$}

    \nextgroupplot[
        title={$q$},
        hide axis=false,
        colorbar=false,
        xmin=0,
        xmax=1,
        xtick={0, 1},
        minor xtick={0.25, 0.5, 0.75},
        xticklabel style={
            font=\footnotesize,
            /pgf/number format/fixed,
            /pgf/number format/precision=0
        },
        ymin=0, ymax=0.55,
        ytick={0.0, 0.2, 0.4},
        minor ytick={0.1, 0.3, 0.5},
        yticklabel pos=right,
        yticklabel style={
            /pgf/number format/fixed,
            /pgf/number format/fixed zerofill,
            /pgf/number format/precision=2
        },
        xminorgrids=true,
        ymajorgrids=true,
        yminorgrids=true,
        grid style={gray!20, solid},
        minor grid style={gray!20, solid},
    ]
    \addplot [
        line width=1.0pt,
        black,
        mark=none
    ] table [x index=0, y index=1] {\main/figures/darcy_components/obs_p25.txt};

\end{groupplot}

\end{tikzpicture}
\caption{(\textbf{Example 2: illustration}) Illustration of the second Poisson example in \secref{sec:outflow}. Left: parameter, $\theta$. Middle: heat flux covector field, $-\kappa \nabla u$. Right: normal heat flux along the top boundary, $q$.}
\label{fig:darcy_overall_mapping}
\end{figure}

As previewed in \secref{sec:numerical_results}, this example keeps the Poisson PDE of Example 1 but the PDE becomes nonlinear, the parameter distribution becomes non-Gaussian, and the observable becomes a boundary flux.

Here, we use Dirichlet boundary conditions $u|_{\partial \Omega} = 0$.
The conductivity is
\begin{equation*}
\kappa(\theta, u) = (1+u)^2\theta.
\end{equation*}
The quantity of interest, 
\begin{equation*}
q = \left. -\nu \cdot \kappa \nabla u \right|_{\Gamma_\text{top}},
\end{equation*}
is the trace of the normal heat flux along the top surface, where $\nu$ is the unit outward normal. Because this boundary flux includes a differential operator applied to the state, the output of $f$ is less regular than Example 1's state trace, with sharp, localized features visible in the flux samples of \figref{fig:darcy_taylor_gallery}. 
The parameter $\theta$ follows the logistic distribution
\begin{equation}
\label{eq:logistic_param}
\theta = \theta_\text{min} + \frac{\theta_\text{max} - \theta_\text{min}}{1 + \exp(-2 \vartheta)}
\end{equation}
where $\theta_\text{min}=1$, $\theta_\text{max}=10$, and $\vartheta \sim N(0, C^2)$; for $C$ we reuse the Mat\'ern covariance square root of Example 1 (\secref{sec:transmission}) with $\sigma=1.0$. Draws vary pointwise between $\theta_\text{min}$ and $\theta_\text{max}$ and inherit their spatial correlation from $C$ (\figref{fig:darcy_overall_mapping}). We approximate a preconditioned version of $\theta \mapsto q$ rather than $\vartheta \mapsto q$, since the sigmoid $\vartheta\mapsto \theta$ is poorly approximated by Taylor series.

The PDE is discretized with mixed finite elements on the Example 1 mesh: zeroth-order Raviart--Thomas elements for the flux (one degree of freedom per triangle edge) and piecewise-constant discontinuous Galerkin elements for the temperature. We discretize $\theta$ with CG1 elements and $q$ with one-dimensional piecewise-constant elements along $\Gamma_\text{top}$.

To generate training data, locally $\vartheta \mapsto \theta$ is approximated by $\theta \approx \overline{\theta} + s\,\vartheta$, with $\overline{\theta}=(\theta_\text{min} + \theta_\text{max})/2$ and $s$ the slope of the sigmoid at zero. Since $\vartheta$ has covariance square root $C$, a local Gaussian approximation of $\theta$ then has mean $\overline{\theta}$ and covariance square root $s C$. Absorbing $s$ into the input coordinate, we approximate the mapping
$$
f(x) = q(\overline{\theta} + Cx)
$$
with T4S. This rescaling changes the Taylor coefficients' values but not their tensor ranks or the probe-fitting procedure. The resulting $f_k \approx f$ yields an approximation $q_k\approx q$ via
\begin{equation*}
q_k(\theta) := f_k(C^{-1}(\theta - \overline{\theta})).
\end{equation*}
We evaluate quality on 500 random test samples $\theta^{(i)}$ from the logistic distribution, comparing $q(\theta^{(i)})$ to $q_k(\theta^{(i)})$.

\begin{figure}[tb]
\centering
\begin{tikzpicture}

\newcommand{\addtaylorplot}[2]{
    \nextgroupplot[title={#2}]

    \pgfplotsinvokeforeach{6,5,4,3,2,1}{
        \pgfmathtruncatemacro{\qindex}{##1 + 1}
        \edef\temp{
            \noexpand\addplot [
                very thick,
                line cap=round,
                color=tab##1,
                mark=none,
                opacity=0.7
            ] table [x index=0, y index=\qindex] {\main/figures/darcy_components/#1};
        }
        \temp
    }

    \addplot [
        thick,
        line cap=round,
        black,
        dashed
    ] table [x index=0, y index=1] {\main/figures/darcy_components/#1};
}

\newcommand{\addthetaplot}[2][]{
    \nextgroupplot[
        xmin=0, xmax=1, ymin=0, ymax=1,
        xtick=\empty, minor xtick=\empty, ytick=\empty,
        xticklabels=\empty, yticklabels=\empty,
        xmajorgrids=false, ymajorgrids=false, xminorgrids=false, yminorgrids=false,
        enlargelimits=false, axis on top, title={},
        #1
    ]
    \addplot [forget plot] graphics [xmin=0, xmax=1, ymin=0, ymax=1]
        {\main/figures/darcy_components/#2};
}

\begin{groupplot}[
    group style={
        group size=4 by 2,
        horizontal sep=8pt,
        vertical sep=14pt,
        y descriptions at=edge left
    },
    width=0.2\textwidth,
    height=0.2\textwidth,
    scale only axis,
    xmin=0, xmax=1,
    ymin=0, ymax=0.85,
    xtick={0, 1},
    minor xtick={0.25, 0.5, 0.75},
    ytick={0, 0.2, 0.4, 0.6, 0.8},
    xticklabel style={
        font=\footnotesize,
        /pgf/number format/fixed,
        /pgf/number format/precision=0
    },
    yticklabel style={
        font=\footnotesize,
        /pgf/number format/fixed,
        /pgf/number format/fixed zerofill,
        /pgf/number format/precision=1
    },
    xminorgrids=true,
    ymajorgrids=true,
    grid style={gray!20, solid},
    minor grid style={gray!20, solid},
    title style={yshift=-6pt},
    legend columns=-1,
    legend style={
        at={(-1.13, -0.05)},    
        anchor=north,
        font=\footnotesize,
        draw=none,
        fill=none,
        /tikz/every even column/.append style={column sep=6pt}
    }
]

    \addtaylorplot{obs_p01.txt}{1st percentile}
    \addtaylorplot{obs_p25.txt}{25th percentile}
    \addtaylorplot{obs_p75.txt}{75th percentile}
    \addtaylorplot{obs_p99.txt}{99th percentile}

    \addthetaplot[
        colormap name=grey_r,
        colorbar,
        colorbar style={
            width=8pt,
            at={(-0.05, 0)}, anchor=south east,
            ytick={1.0, 10.0},
            yticklabel pos=left,
            yticklabel style={font=\footnotesize, /pgf/number format/fixed,
                              /pgf/number format/fixed zerofill, /pgf/number format/precision=1},
        },
        point meta min=1.0, point meta max=10.0,
    ]{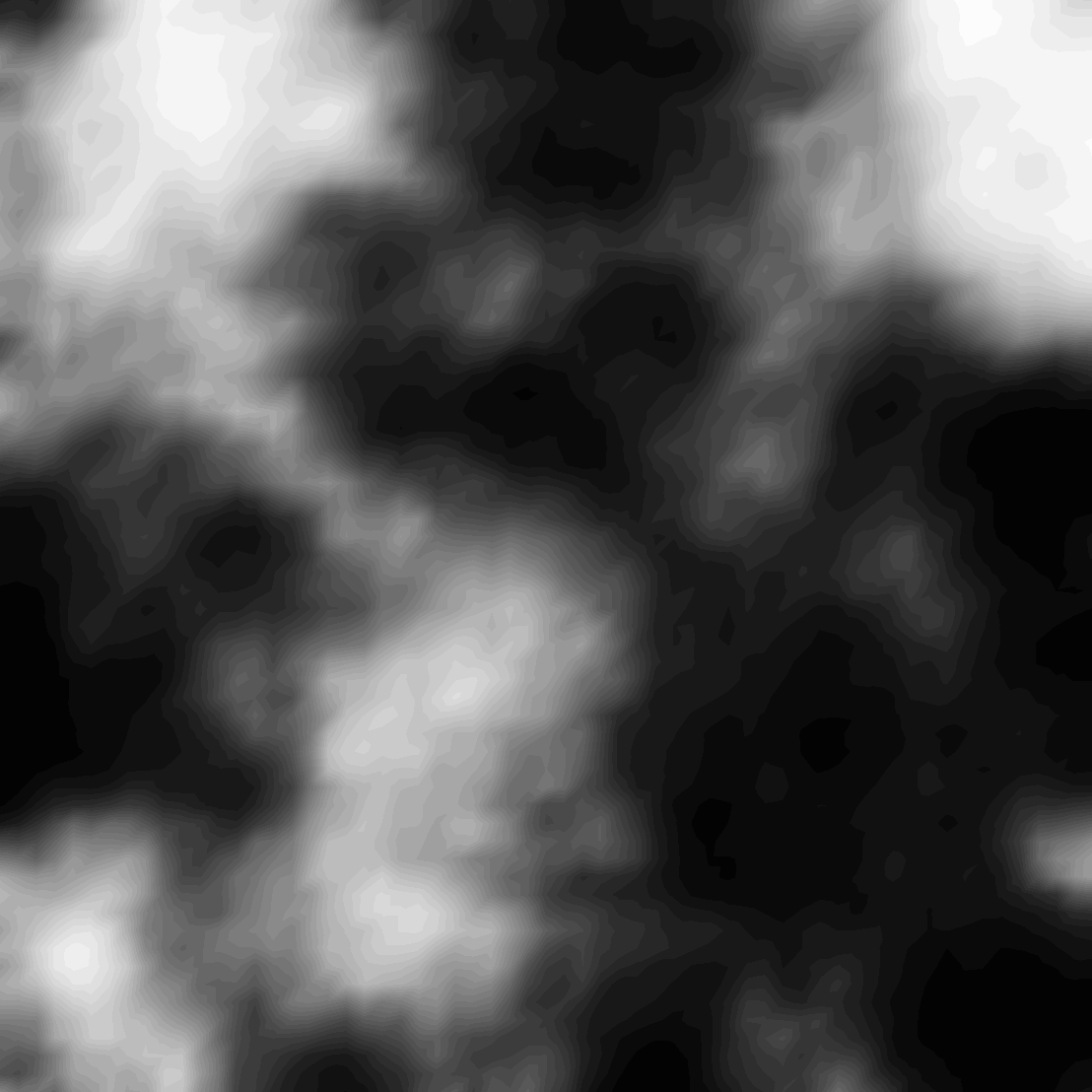}
    \addplot [draw=none, forget plot, point meta=y] coordinates {(0,0)};   
    \addthetaplot{kappa_p25.png}
    \addthetaplot{kappa_p75.png}
    \addthetaplot{kappa_p99.png}

    \pgfplotsinvokeforeach{1,2,3,4,5,6}{
        \pgfmathtruncatemacro{\kk}{#1 - 1}
        \addlegendimage{very thick, line cap=round, color=tab#1, mark=none, opacity=0.7}
        \edef\temp{\noexpand\addlegendentry{$k=\kk$}}\temp   
    }
    \addlegendimage{thick, line cap=round, black, dashed}
    \addlegendentry{Truth}

\end{groupplot}

\end{tikzpicture}
\caption{(\textbf{Example 2: Taylor series across samples}) Taylor approximations $q_k$ of the boundary flux $q$ (no dimension reduction or T4S) for four test parameters, at the $1$st, $25$th, $75$th, and $99$th percentiles of the order-$5$ normalized error, left to right (errors $0.0014$, $0.0040$, $0.014$, $0.048$). The $25$th-percentile sample is that of \figref{fig:darcy_overall_mapping}.}
\label{fig:darcy_taylor_gallery}
\end{figure}

\figref{fig:darcy_taylor_gallery} shows the Taylor series $q_k$ across the range of sample difficulty. Convergence is rapid where $q$ tracks the basic ``bump''-like shape of the mean and slow where $q$ deviates from its mean via sharp, localized features. The left side of \figref{fig:darcy_histogram_t4s} shows normalized-error KDEs for $n_s \in \{100,400,1600,6400\}$. Larger $n_s$ matches higher orders before stagnating: at $n_s=6400$ the model matches orders $k=0,1,2,3$ well, $k=4$ reasonably, and not $k=5$. As in Example 1, the $k=1$ curve is the linearized baseline and the higher-order curves show the gain over it. 
The percentile table on the right hand side of \figref{fig:darcy_histogram_t4s} shows the same pattern as Example 1. For P90 (90th percentile error level) T4S matches the exact Taylor series for orders $k=0,\dots,5$. For P50 (median) T4S matches the exact Taylor series through order $4$, and at order $5$ departs from the exact Taylor series but still improves over order $4$.
The T4S models thus work well on logistic samples despite being built from $\theta_0$ and $C$ of a local Gaussian approximation.  

\begin{figure}
\centering
\setlength{\kdefulltw}{\textwidth}%
\begin{minipage}[c]{0.56\textwidth}
\centering
{\setlength{\textwidth}{\kdefulltw}%
 \resizebox{\linewidth}{!}{\begin{tikzpicture} 

\begin{groupplot}[
    group style={
        group size=1 by 6,             
        x descriptions at=edge bottom, 
        y descriptions at=edge left,   
        vertical sep=-1.8em,
        group name=hists
    },
    xmin=1e-3,
    xmax=1,
    ymin=0,
    ymax=1.6,
    axis x line*=bottom,
    tick label style={font=\footnotesize},
    ytick=\empty,
    y axis line style={draw=none},
    every axis y label/.style={
        at={(axis description cs:-0.02, 0.33)},
        rotate=0, 
        anchor=west,
        font=\footnotesize
    },
    scale only axis=true,                       
    width=0.6\textwidth,                       
    height={(0.95 * 1.618 / 13) * \textwidth},   
    enlarge x limits=true,
    xmode=log,
]

\nextgroupplot[ylabel={$k=0$}]

\addplot [
    thick, 
    black,
    opacity=0.8
] table [x index=0, y index=1] {\main/figures/darcy_t4s_hist/t4s_hist_taylor_k0.txt};

\addplot [
    thick, 
    line cap=round,
    color=black,
    dash pattern=on 2 pt off 1.5 pt,
    opacity=0.8
] table [x index=0, y index=1] {\main/figures/darcy_t4s_hist/t4s_hist_tucker_k0.txt};

\foreach \ns/\cname in {6400/tab4, 1600/tab3, 400/tab2, 100/tab1} {
    \edef\temp{
        \noexpand\addplot [
            draw=none, 
            fill=\cname, 
            fill opacity=0.2, 
            forget plot
        ] table [x index=0, y index=1] {\noexpand\main/figures/darcy_t4s_hist/t4s_hist_t4s_k0_train\ns.txt} \noexpand\closedcycle;

        \noexpand\addplot [
            thick,
            color=\cname,
            draw opacity=0.8
        ] table [x index=0, y index=1] {\noexpand\main/figures/darcy_t4s_hist/t4s_hist_t4s_k0_train\ns.txt};
    }
    \temp
}

\nextgroupplot[
    ylabel={$k=1$},
    legend style={
        at={(0.0875, 1.25)}, 
        anchor=north west,
        font=\footnotesize,
        /tikz/every node/.style={anchor=west},
        rounded corners=2pt,
    },
    legend cell align=left,
    reverse legend
]

\addplot [
    thick, 
    black,
    opacity=0.8
] table [x index=0, y index=1] {\main/figures/darcy_t4s_hist/t4s_hist_taylor_k1.txt};
\addlegendentry{Taylor}

\addplot [
    thick, 
    line cap=round,
    color=black,
    dash pattern=on 2 pt off 1.5 pt,
    opacity=0.8
] table [x index=0, y index=1] {\main/figures/darcy_t4s_hist/t4s_hist_tucker_k1.txt};
\addlegendentry{Reduced Taylor\!\!}

\foreach \ns/\cname in {6400/tab4, 1600/tab3, 400/tab2, 100/tab1} {
    \edef\temp{
        \noexpand\addplot [
            draw=none, 
            fill=\cname, 
            fill opacity=0.2, 
            forget plot,
        ] table [x index=0, y index=1] {\noexpand\main/figures/darcy_t4s_hist/t4s_hist_t4s_k1_train\ns.txt} \noexpand\closedcycle;
    
        \noexpand\addplot [
            thick,
            color=\cname,
            draw opacity=0.8,
        ] table [x index=0, y index=1] {\noexpand\main/figures/darcy_t4s_hist/t4s_hist_t4s_k1_train\ns.txt};

        \noexpand\addlegendentry{T4S ($\numactions\!=\!\noexpand\pgfmathprintnumber[fixed, precision=0, 1000 sep={}]{\ns}$)\!\!}
    }
    \temp
}

\nextgroupplot[ylabel={$k=2$}]

\addplot [
    thick, 
    black,
    opacity=0.8
] table [x index=0, y index=1] {\main/figures/darcy_t4s_hist/t4s_hist_taylor_k2.txt};

\addplot [
    thick, 
    line cap=round,
    color=black,
    dash pattern=on 2 pt off 1.5 pt,
    opacity=0.8
] table [x index=0, y index=1] {\main/figures/darcy_t4s_hist/t4s_hist_tucker_k2.txt};

\foreach \ns/\cname in {6400/tab4, 1600/tab3, 400/tab2, 100/tab1} {
    \edef\temp{
        \noexpand\addplot [
            draw=none, 
            fill=\cname, 
            fill opacity=0.2, 
            forget plot
        ] table [x index=0, y index=1] {\noexpand\main/figures/darcy_t4s_hist/t4s_hist_t4s_k2_train\ns.txt} \noexpand\closedcycle;
    
        \noexpand\addplot [
            thick,
            color=\cname,
            draw opacity=0.8
        ] table [x index=0, y index=1] {\noexpand\main/figures/darcy_t4s_hist/t4s_hist_t4s_k2_train\ns.txt};
    }
    \temp
}

\nextgroupplot[ylabel={$k=3$}]

\addplot [
    thick, 
    black,
    opacity=0.8
] table [x index=0, y index=1] {\main/figures/darcy_t4s_hist/t4s_hist_taylor_k3.txt};

\addplot [
    thick, 
    line cap=round,
    color=black,
    dash pattern=on 2 pt off 1.5 pt,
    opacity=0.8
] table [x index=0, y index=1] {\main/figures/darcy_t4s_hist/t4s_hist_tucker_k3.txt};

\foreach \ns/\cname in {6400/tab4, 1600/tab3, 400/tab2, 100/tab1} {
    \edef\temp{
        \noexpand\addplot [
            draw=none, 
            fill=\cname, 
            fill opacity=0.2, 
            forget plot
        ] table [x index=0, y index=1] {\noexpand\main/figures/darcy_t4s_hist/t4s_hist_t4s_k3_train\ns.txt} \noexpand\closedcycle;
    
        \noexpand\addplot [
            thick,
            color=\cname,
            draw opacity=0.8
        ] table [x index=0, y index=1] {\noexpand\main/figures/darcy_t4s_hist/t4s_hist_t4s_k3_train\ns.txt};
    }
    \temp
}

\nextgroupplot[ylabel={$k=4$}]

\addplot [
    thick, 
    black,
    opacity=0.8
] table [x index=0, y index=1] {\main/figures/darcy_t4s_hist/t4s_hist_taylor_k4.txt};

\addplot [
    thick, 
    line cap=round,
    color=black,
    dash pattern=on 2 pt off 1.5 pt,
    opacity=0.8
] table [x index=0, y index=1] {\main/figures/darcy_t4s_hist/t4s_hist_tucker_k4.txt};

\foreach \ns/\cname in {6400/tab4, 1600/tab3, 400/tab2, 100/tab1} {
    \edef\temp{
        \noexpand\addplot [
            draw=none, 
            fill=\cname, 
            fill opacity=0.2, 
            forget plot
        ] table [x index=0, y index=1] {\noexpand\main/figures/darcy_t4s_hist/t4s_hist_t4s_k4_train\ns.txt} \noexpand\closedcycle;
    
        \noexpand\addplot [
            thick,
            color=\cname,
            draw opacity=0.8
        ] table [x index=0, y index=1] {\noexpand\main/figures/darcy_t4s_hist/t4s_hist_t4s_k4_train\ns.txt};
    }
    \temp
}

\nextgroupplot[ylabel={$k=5$}]

\addplot [
    thick, 
    black,
    opacity=0.8
] table [x index=0, y index=1] {\main/figures/darcy_t4s_hist/t4s_hist_taylor_k5.txt};

\addplot [
    thick, 
    line cap=round,
    color=black,
    dash pattern=on 2 pt off 1.5 pt,
    opacity=0.8
] table [x index=0, y index=1] {\main/figures/darcy_t4s_hist/t4s_hist_tucker_k5.txt};

\foreach \ns/\cname in {6400/tab4, 1600/tab3, 400/tab2, 100/tab1} {
    \edef\temp{
        \noexpand\addplot [
            draw=none, 
            fill=\cname, 
            fill opacity=0.2, 
            forget plot
        ] table [x index=0, y index=1] {\noexpand\main/figures/darcy_t4s_hist/t4s_hist_t4s_k5_train\ns.txt} \noexpand\closedcycle;
    
        \noexpand\addplot [
            thick,
            color=\cname,
            draw opacity=0.8
        ] table [x index=0, y index=1] {\noexpand\main/figures/darcy_t4s_hist/t4s_hist_t4s_k5_train\ns.txt};
    }
    \temp
}

\end{groupplot}

\node[anchor=north, yshift=-14pt] at (hists c1r6.south) {Normalized Error};
\node[anchor=south, xshift=-4pt, rotate=90] at (hists c1r4.west) {Taylor Series Order, $k$};

\end{tikzpicture}}}%
\end{minipage}\hfill
\begin{minipage}[c]{0.42\textwidth}
\centering
\small
{%
\renewcommand{\arraystretch}{1.35}%
\begin{tabular}{@{}c @{} c @{} c @{\hspace{1.6em}} c @{} c @{} c @{\hspace{1.6em}} c @{}}
\toprule
 & & \multicolumn{2}{c}{Taylor} & & \multicolumn{2}{c}{T4S} \\
\cmidrule(lr){3-4}\cmidrule(lr){6-7}
$k$ & \rule{1.4em}{0pt} & P50 & P90 & \rule{2.2em}{0pt} & P50 & P90 \\
\midrule
0 & & $4.2\mathrm{e}{-1}$ & $6.5\mathrm{e}{-1}$ & & $4.2\mathrm{e}{-1}$ & $6.5\mathrm{e}{-1}$ \\
1 & & $8.9\mathrm{e}{-2}$ & $2.1\mathrm{e}{-1}$ & & $8.9\mathrm{e}{-2}$ & $2.1\mathrm{e}{-1}$ \\
2 & & $5.4\mathrm{e}{-2}$ & $1.2\mathrm{e}{-1}$ & & $5.3\mathrm{e}{-2}$ & $1.1\mathrm{e}{-1}$ \\
3 & & $2.1\mathrm{e}{-2}$ & $6.2\mathrm{e}{-2}$ & & $2.2\mathrm{e}{-2}$ & $6.4\mathrm{e}{-2}$ \\
4 & & $1.4\mathrm{e}{-2}$ & $4.1\mathrm{e}{-2}$ & & $1.7\mathrm{e}{-2}$ & $4.0\mathrm{e}{-2}$ \\
5 & & $7.0\mathrm{e}{-3}$ & $2.6\mathrm{e}{-2}$ & & $1.3\mathrm{e}{-2}$ & $2.5\mathrm{e}{-2}$ \\
\bottomrule
\end{tabular}%
}
\end{minipage}
\caption{(\textbf{Example 2: T4S errors}) \emph{Left:} KDEs of the normalized error, organized
by Taylor order $k$, for models fit with $n_s\in\{100,400,1600,6400\}$ training samples.
\emph{Right:} median (P50) and 90th-percentile (P90) 
normalized error of the exact Taylor series and of the best-data T4S model ($n_s=6400$), each measured against the true output q.
}
\label{fig:darcy_histogram_t4s}
\end{figure}

\FloatBarrier 



\section{Conclusion}
\label{sec:conclusion}

We introduced the Tucker tensor train Taylor series (T4S), a local derivative-accurate surrogate for smooth high-dimensional maps defined implicitly by a simulator. T4S represents each derivative tensor of a truncated Taylor expansion as a Tucker tensor train, fit directly from directionally symmetric derivative probes at a single expansion point. The construction pairs derivative-informed dimension reduction with Riemannian manifold optimization and rank continuation, using fast sweeping routines to apply the Riemannian Jacobian and its transpose. We introduced two fitting methods, a trust-region Riemannian Gauss–Newton method (TR-RMGN) and a manifold stochastic gradient method with Cauchy step sizes (MC-SGD), and proved that, under spectral decay of the input covariance, the derivative tensors admit Tucker tensor train approximations with quantified ranks and errors. The use of directionally symmetric probes is central: for implicit models, these probes reuse the same linearized state and adjoint operators and reduce the solve count from exponential to linear growth in derivative order.

The numerical experiments show what the method achieves. On random-tensor benchmarks, the proposed fitting methods track quasi-optimal T3-SVD accuracy from probes alone, up to data-limited ranks (Section 8.1). On two Poisson PDE problems of increasing difficulty, the complete method recovers high-order Taylor structure for both Gaussian and non-Gaussian parameters (Section 8.2). 
More broadly, these results show that derivative probes are a practical training-data modality,
especially when function evaluations are expensive but reusable linearized solves are available.


The main limitations of the method are inherent to the approach. First, for problems with slowly decaying covariance spectra and high-rank derivative tensor structure (see, e.g., \cite{alger2024point,hu2025accelerating}), the required Tucker and tensor train ranks become large, which increases the model's construction and evaluation costs. 
Second, T4S is local by construction: it is accurate near the expansion point but does not provide global coverage, which would require multipoint or mixture extensions discussed below.
Third, T4S requires access to derivative-probing machinery, such as linearized and adjoint solvers. Fourth, the approximation results are representational: finite-probe sample complexity and optimization guarantees for the fitted model remain open.

We see several directions for future work. First, the representational guarantees use only the input covariance spectrum and are therefore conservative, prescribing larger ranks than are needed in practice; sharper bounds that exploit additional structure of the derivative tensors, together with finite-probe sample-complexity results for the fitted model, would more closely match observed performance. Second, model selection could be made more automatic, including adaptive choices of expansion point, Taylor order, and Tucker and tensor train ranks, as well as symmetry-aware parameterizations that share input factors across modes. Third, global coverage will require combining local expansions through multipoint or mixture constructions, or coupling T4S anchors with a global operator-learning surrogate. Finally, because T4S provides low-cost, accurate access to a map and its derivatives at a point, it is a natural candidate for derivative-driven outer-loop tasks such as Bayesian inverse problems, optimal experimental design, and optimization under uncertainty; integrating and evaluating it within these pipelines is future work. High-order Taylor surrogates have long been considered intractable in high dimensions; T4S shows that, with a suitable low-rank representation and probe-based training, they need not be.

\section{Author contributions}

Development of the numerical methods and their implementation was done jointly by Nick Alger and Blake Christierson. Theoretical results were proved by Nick Alger with help from Blake Christierson. The paper was drafted by Nick Alger with help from Blake Christierson. Peng Chen proposed the idea of using Tucker tensor trains instead of just tensor trains. Peng Chen and Omar Ghattas provided guidance and served in an advisory role throughout the project, and edited the paper. 

\section{Acknowledgments}

This research was partially supported by the U.S.\ Department of Energy under award DE-SC0023171, the U.S.\ Air Force Office of Scientific Research under award FA9550-24-1-0327, and the U.S.\ National Science Foundation under awards DMS-2324643, OAC-2313033, and DMS-2245111.

We thank Dingcheng Luo, Thomas O'Leary-Roseberry, Nicole Aretz, Rachel Ward, and Alex Gorodetsky for helpful discussions. We thank J.J. Alger for helpful discussions and extensive editing suggestions.


\appendix

\section{Tucker tensor train details}
\label{app:tensor_appendix}

This section collects details on Tucker tensor trains and the tools they require in \secref{sec:method} and \secref{sec:theory}. We define orthogonalization, relate Tucker, tensor train, and Tucker tensor train ranks to the ranks of matricizations and matrix unfoldings, and describe the fixed-rank Tucker tensor train manifold together with an efficient parameterization of its tangent space.

\subsection{Orthogonalization}

Let $\vb{G}$ be a 3-tensor of shape $a \times b \times c$.  
We define three matrix versions of $\vb{G}$:
$$
\leftunfolding{\vb{G}} \in \mathbb{R}^{(ab)\times c}, \qquad
\rightunfolding{\vb{G}} \in \mathbb{R}^{a\times (bc)}, \qquad
\vb{G}^O \in \mathbb{R}^{(ac)\times b}.
$$
Here $\vb{G}^L$ and $\vb{G}^R$ are reshapings of $\vb{G}$ into tall and wide matrices; $\vb{G}^O$ is formed by permuting and reshaping $\vb{G}$ so that its first and third indices form the rows.

A 3-tensor $\vb{G}$ is  
\emph{left orthogonal} if  
$
\left(\leftunfolding{\vb{G}}\right)^T \leftunfolding{\vb{G}} = \vb{I},
$
\emph{right orthogonal} if  
$
\rightunfolding{\vb{G}}\left(\rightunfolding{\vb{G}}\right)^T = \vb{I},
$
and \emph{outer orthogonal} if  
$
(\vb{G}^O)^T \vb{G}^O = \vb{I}.
$
A Tucker basis matrix $\vb{U}$ is orthogonal when  
$\vb{U}^T\vb{U} = \vb{I}$.  
A tensor train $(\vb{G}_i)_{i=1}^d$ is \emph{left orthogonal} if $\vb{G}_i$ is left orthogonal for $i=1,\dots,d-1$,  
and \emph{right orthogonal} if $\vb{G}_i$ is right orthogonal for $i=2,\dots,d$.  
A Tucker tensor train $((\vb{U}_i)_{i=1}^d,(\vb{G}_i)_{i=1}^d)$  
is said to have \emph{orthogonal Tucker bases} when each $\vb{U}_i$ is orthogonal.

Tensor trains can be converted to left or right orthogonal form without changing the represented tensor. To left orthogonalize, sweep left to right across the cores, replacing each $\vb{G}_i$ with the reshaped $\vb{Q}$ factor of a thin QR decomposition of $\leftunfolding{\vb{G}_i}$ and absorbing $\vb{R}$ into the first index of $\vb{G}_{i+1}$; right orthogonalization is analogous, sweeping right to left (see \cite[Section 4.2.1]{steinlechner2016riemannianphdthesis}). The Tucker bases can likewise be orthogonalized: replace each $\vb{U}_i$ with the $\vb{Q}$ factor of a thin QR decomposition of $\vb{U}_i$ and absorb $\vb{R}$ into the middle index of $\vb{G}_i$.

\subsection{Connection with matrix unfoldings and matricizations}
\label{app:t3_unfoldings}
The ranks of the Tucker, tensor train, and Tucker tensor train decompositions of a tensor are determined by the ranks of its matricizations and matrix unfoldings. 
For a tensor $\vb{T} \in \mathbb{R}^{N_1 \times \dots \times N_d}$, let $n_i$ denote the rank of its $i$th matricization and $r_i$ the rank of its $i$th matrix unfolding, and define the rank vectors
$
\vb{n} = (n_1, \dots, n_d)
$,
$
\vb{r} = (1, r_1, \dots, r_{d-1}, 1)
$.
$\vb{T}$ admits a low-rank Tucker decomposition if and only if its matricizations are low-rank, with minimal Tucker ranks $\vb{n}$. It admits a low-rank tensor train decomposition if and only if its matrix unfoldings are low-rank, with minimal TT-ranks $\vb{r}$. It admits a Tucker tensor train decomposition if and only if both are low-rank, with minimal Tucker and TT-ranks $\vb{n}$ and $\vb{r}$, respectively.
The dense T3-SVD algorithm (\algoref{alg:t3_svd_dense}) produces such a representation, generalizing the dense TT-SVD algorithm (see \cite[Algorithm 2]{oseledets.2011}). Any Tucker tensor train representation of $\vb{T}$ with Tucker ranks $\vb{n}'$ and TT-ranks $\vb{r}'$ satisfies $\vb{n}' \ge \vb{n}$ and $\vb{r}' \ge \vb{r}$; it is \emph{degenerate} when $\vb{n}' \neq \vb{n}$ or $\vb{r}' \neq \vb{r}$.

Let $\vb{\Sigma}_i^{\text{TT}}$ and $\vb{\Sigma}_i^{\text{Tucker}}$ denote the diagonal matrices of singular values from the thin SVDs of the $i$th matrix unfolding and $i$th matricization of $\vb{T}$. When $\vb{T}$ is represented by a Tucker tensor train, these can be computed by the implicit T3-SVD algorithm (\algoref{alg:t3_svd_implicit}), which mirrors the dense procedure but operates directly on the smaller cores.

Truncating the SVDs in \algoref{alg:t3_svd_implicit}, either by tolerance or by imposing maximum Tucker or TT-ranks, produces a lower-rank Tucker tensor train approximating the original. This is \emph{rounding}, a general mechanism for approximating one Tucker tensor train by another of lower rank (see \cite[Section 3]{oseledets.2011}).
 
\begin{algorithm}
\caption{Dense T3-SVD.}
\label{alg:t3_svd_dense}
\begin{algorithmic}[1]
\Require Dense tensor $\vb{T} \in \mathbb{R}^{N_1 \times \dots \times N_d}$
\Ensure Tucker tensor train $((\vb{U}_i)_{i=1}^d,(\vb{G}_i)_{i=1}^d)$
\State $\vb{X} \gets \vb{T}$, \quad $r \gets 1$
\For{$i=1,\dots,d$}
    \State Permute and reshape $\vb{X}$ into $N_i \times (r N_{i+1}\dots N_d)$ matrix
    \State Compute 
    $\vb{X} = \vb{U}_i \vb{\Sigma}_i^\text{Tucker} \vb{V}^T$
    \Comment{truncated SVD}
    \State $\vb{X} \gets \vb{\Sigma}_i^\text{Tucker} \vb{V}^T$, \quad $n \gets \#\text{rows}~\vb{\Sigma}_i^\text{Tucker}$
    \If{$i < d$}
        \State Permute and reshape $\vb{X}$ into $(rn) \times (N_{i+1} \times \dots \times N_d)$ matrix
        \State Compute 
        $\vb{X} = \vb{G}_i^L \vb{\Sigma}_i^\text{TT} \vb{V}^T$
        \Comment{truncated SVD}
        \State $r' \gets \#\text{rows}~\vb{\Sigma}_i^\text{TT}$
        \State Form $\vb{G}_i$ by reshaping $\vb{G}_i^L$ into $r \times n \times r'$ tensor
        \State $\vb{X} \gets \vb{\Sigma}_i^\text{TT} \vb{V}^T$
        \State Reshape $\vb{X}$ into $r' \times N_{i+1} \times \dots \times N_d$ tensor
        \State $r \gets r'$
    \Else
        \State Form $\vb{G}_d$ by reshaping $\vb{X}$ into $r' \times N_d \times 1$ tensor
    \EndIf
\EndFor
\end{algorithmic}
\end{algorithm}

\begin{algorithm}
\caption{Implicit T3-SVD 
}
\label{alg:t3_svd_implicit}
\begin{algorithmic}[1]
\Require {Tucker tensor train $\tuple{T}=((\vb{U}_i)_{i=1}^d,(\vb{G}_i)_{i=1}^d)$}
\Ensure Singular values $\{\vb{\Sigma}_i^\text{Tucker}\}_{i=1}^d$, $\{\vb{\Sigma}_i^\text{TT}\}_{i=1}^{d-1}$. 
\State Orthogonalize the Tucker bases of $\tuple{T}$
\State Right orthogonalize the tensor train $(\vb{G}_i)_{i=1}^d$
\For{$i=1,\dots,d$}
    \State Compute
    $\vb{G}_i^O = \vb{W} \vb{\Sigma}_i^\text{Tucker} \vb{V}^T$ 
    \Comment{truncated SVD}
    \State $\vb{G}_i^{O} \gets \vb{W} \vb{\Sigma}_i^\text{Tucker}$
    \State $\vb{U}_i \gets \vb{U}_i \vb{V}$
    \If{$i < d$}
        \State Compute 
        $\vb{G}_i^L = \vb{W} \vb{\Sigma}_i^\text{TT} \vb{V}^T$
        \Comment{truncated SVD}
        \State Replace $\vb{G}_i$ by appropriate 3-tensor reshaping of $\vb{W}$
        \State Replace $\vb{G}_{i+1}$ with the contraction $\vb{G}_{i+1}'$ defined by
        $$\vb{G}_{i+1}'[a,x,c] = \sum_b\left(\vb{\Sigma}_i^\text{TT} \vb{V}^T\right)[a,b]\vb{G}_{i+1}[b,x,c]$$
    \EndIf
\EndFor
\end{algorithmic}
\end{algorithm}

\subsection{T3 manifold and tangent space}
\label{sec:tensor_train_manifold}

We describe the manifold of non-degenerate fixed-rank Tucker tensor trains and a memory-efficient parameterization of its tangent space that supports fast linear algebra (addition, scaling, inner products) on tangent vectors. These results extend standard ones for fixed-rank tensor train manifolds (\citealp{haegeman2014geometry,holtz.rohwedder.ea.2012,khoromskij2012efficient,steinlechner2016riemannian,uschmajew2020geometric}; \citealp[Chapter 4]{steinlechner2016riemannianphdthesis}).

Let $\mathcal{M}_{\vb{n},\vb{r}} \subset \mathbb{R}^{N_1 \times \dots \times N_d}$ denote the set of all tensors $\vb{p}$ representable by non-degenerate Tucker tensor trains with Tucker ranks $\vb{n}=(n_1,\dots,n_d)$ and TT-ranks $\vb{r}=(1,r_1,\dots,r_{d-1},1)$. If non-degenerate Tucker tensor trains with these ranks exist, then $\mathcal{M}_{\vb{n},\vb{r}}$
forms an embedded submanifold of $\mathbb{R}^{N_1 \times \dots \times N_d}$.
To parameterize the tangent space $T_\vb{p} \mathcal{M}_{\vb{n},\vb{r}}$, we extend the tensor train formulation of \cite{khoromskij2012efficient,steinlechner2016riemannian}. We first construct $2d$ representations of $\vb{p}$ of the forms
\begin{equation}
\label{eq:p_rep_Btilde}
\begin{tikzpicture}[scale=1.0, every node/.style={scale=0.9},baseline=(current  bounding  box.center)]

    \node at (0,0.0) {$\vb{p}$};

    \node at (1,0.0) {$=$};

    \node[draw, rounded corners, minimum size=0.5cm, inner sep=0.0] (oneL) at (2,0.5) {$1$};
    \node[draw, rounded corners, minimum size=0.5cm, inner sep=0.0] (oneR) at (10,0.5) {$1$};
    
    \node[draw, rounded corners, minimum size=0.8cm, inner sep=0.5] (G1) at (3,0.5) {$\vb{P}_{1}$};
    \node (dots1) at (4,0.5) {$\dots$};
    \node[draw, rounded corners, minimum size=0.8cm, inner sep=0.5] (Gprev) at (5,0.5) {$\vb{P}_{i-1}$};
    \node[draw, rounded corners, minimum size=0.8cm, inner sep=0.5] (Gcur) at (6,0.5) {$\vb{O}_{i}$};
    \node[draw, rounded corners, minimum size=0.8cm, inner sep=0.5] (Gnext) at (7,0.5) {$\vb{Q}_{i+1}$};
    \node (dots2) at (8,0.5) {$\dots$};
    \node[draw, rounded corners, minimum size=0.8cm, inner sep=0.5] (Gd) at (9,0.5) {$\vb{Q}_{d}$};

    \draw (oneL) -- (G1) -- (dots1) -- (Gprev) -- (Gcur) -- (Gnext) -- (dots2) -- (Gd) -- (oneR);
    
    \node[draw, rounded corners, minimum size=0.8cm, inner sep=0.5] (B1) at (3,-0.5) {$\vb{U}_1$};
    \node[draw, rounded corners, minimum size=0.8cm, inner sep=0.5] (Bprev) at (5,-0.5) {$\vb{U}_{i-1}$};
    \node[fill=lightgray, draw, rounded corners, minimum size=0.8cm, inner sep=0.5] (Bcur) at (6,-0.5) {$\widetilde{\vb{U}}_{i}$};
    \node[draw, rounded corners, minimum size=0.8cm, inner sep=0.5] (Bnext) at (7,-0.5) {$\vb{U}_{i+1}$};
    \node[draw, rounded corners, minimum size=0.8cm, inner sep=0.5] (Bd) at (9,-0.5) {$\vb{U}_{d}$};
			
    \draw (G1) -- (B1);
    \draw (Gprev) -- (Bprev);
    \draw (Gcur) -- (Bcur);
    \draw (Gnext) -- (Bnext);
    \draw (Gd) -- (Bd);

    \node (B1x) at (3,-1.25) {};
    \node (Bprevx) at (5,-1.25) {};
    \node (Bcurx) at (6,-1.25) {};
    \node (Bnextx) at (7,-1.25) {};
    \node (Bdx) at (9,-1.25) {};

    \draw (B1) -- (B1x);
    \draw (Bprev) -- (Bprevx);
    \draw (Bcur) -- (Bcurx);
    \draw (Bnext) -- (Bnextx);
    \draw (Bd) -- (Bdx);
\end{tikzpicture}
\end{equation}
\noindent
and
\begin{equation}
\label{eq:p_rep_Gtilde}
\begin{tikzpicture}[scale=1.0, every node/.style={scale=0.9},baseline=(current  bounding  box.center)]

    \node at (0,0.0) {$\vb{p}$};

    \node at (1,0.0) {$=$};

    \node[draw, rounded corners, minimum size=0.5cm, inner sep=0.0] (oneL) at (2,0.5) {$1$};
    \node[draw, rounded corners, minimum size=0.5cm, inner sep=0.0] (oneR) at (10,0.5) {$1$};
    
    \node[draw, rounded corners, minimum size=0.8cm, inner sep=0.5] (G1) at (3,0.5) {$\vb{P}_{1}$};
    \node (dots1) at (4,0.5) {$\dots$};
    \node[draw, rounded corners, minimum size=0.8cm, inner sep=0.5] (Gprev) at (5,0.5) {$\vb{P}_{i-1}$};
    \node[fill=lightgray, draw, rounded corners, minimum size=0.8cm, inner sep=0.5] (Gcur) at (6,0.5) {$\widetilde{\vb{G}}_{i}$};
    \node[draw, rounded corners, minimum size=0.8cm, inner sep=0.5] (Gnext) at (7,0.5) {$\vb{Q}_{i+1}$};
    \node (dots2) at (8,0.5) {$\dots$};
    \node[draw, rounded corners, minimum size=0.8cm, inner sep=0.5] (Gd) at (9,0.5) {$\vb{Q}_{d}$};

    \draw (oneL) -- (G1) -- (dots1) -- (Gprev) -- (Gcur) -- (Gnext) -- (dots2) -- (Gd) -- (oneR);
    
    \node[draw, rounded corners, minimum size=0.8cm, inner sep=0.5] (B1) at (3,-0.5) {$\vb{U}_1$};
    \node[draw, rounded corners, minimum size=0.8cm, inner sep=0.5] (Bprev) at (5,-0.5) {$\vb{U}_{i-1}$};
    \node[draw, rounded corners, minimum size=0.8cm, inner sep=0.5] (Bcur) at (6,-0.5) {$\vb{U}_{i}$};
    \node[draw, rounded corners, minimum size=0.8cm, inner sep=0.5] (Bnext) at (7,-0.5) {$\vb{U}_{i+1}$};
    \node[draw, rounded corners, minimum size=0.8cm, inner sep=0.5] (Bd) at (9,-0.5) {$\vb{U}_{d}$};
			
    \draw (G1) -- (B1);
    \draw (Gprev) -- (Bprev);
    \draw (Gcur) -- (Bcur);
    \draw (Gnext) -- (Bnext);
    \draw (Gd) -- (Bd);

    \node (B1x) at (3,-1.25) {};
    \node (Bprevx) at (5,-1.25) {};
    \node (Bcurx) at (6,-1.25) {};
    \node (Bnextx) at (7,-1.25) {};
    \node (Bdx) at (9,-1.25) {};

    \draw (B1) -- (B1x);
    \draw (Bprev) -- (Bprevx);
    \draw (Bcur) -- (Bcurx);
    \draw (Bnext) -- (Bnextx);
    \draw (Bd) -- (Bdx);
\end{tikzpicture}
\end{equation}
\noindent
for $i=1,\dots,d$. These must satisfy:
\begin{itemize}
\item $\mathbf{U}_i$ has orthonormal columns for $i=1,\dots,d$.
\item $\vb{P}_i$ is left orthogonal for $i=1,\dots,d-1$.
\item $\vb{Q}_i$ is right orthogonal for $i=2,\dots,d$.
\item $\vb{O}_i$ is outer orthogonal for $i=1,\dots,d$.
\end{itemize}
Each representation designates one special core ($\widetilde{\vb{U}_i}$ or $\widetilde{\vb{G}_i}$, lightly shaded), relative to which the surrounding sub-networks are orthogonal.
They can be constructed from any non-degenerate representation $\tuple{T}=((\mathbf{U}_i)_{i=1}^d, (\mathbf{G}_i)_{i=1}^d)$ of $\vb{p}$ by the sweeping orthogonalization in \algoref{alg:t3_sweep_orth}, which is similar to the implicit T3-SVD (\algoref{alg:t3_svd_implicit}).

\begin{algorithm}
\caption{Sweeping orthogonalization.}
\label{alg:t3_sweep_orth}
\begin{algorithmic}[1]
\Require Non-degenerate Tucker tensor train $\tuple{T}=((\vb{U}_i)_{i=1}^d,(\vb{G}_i)_{i=1}^d)$ of $\vb{p}$
\Ensure Cores $(\vb{U}_i)_{i=1}^d$, $(\vb{P}_i)_{i=1}^d$, $(\vb{Q}_i)_{i=1}^d$, $(\vb{O}_i)_{i=1}^d$ and special cores $(\widetilde{\vb{U}}_i)_{i=1}^d$, $(\widetilde{\vb{G}}_i)_{i=1}^d$
\For{$i=1,\dots,d$} \Comment{orthogonalize Tucker bases}
    \State Factor $\vb{U}_i = \vb{V}_i \vb{X}$ with $\vb{V}_i^T \vb{V}_i = \vb{I}$
    \State $\vb{U}_i \gets \vb{V}_i$ and contract $\vb{X}$ into $\vb{G}_i$
\EndFor
\State Form $(\vb{Q}_i)_{i=1}^d$ by right orthogonalizing the tensor train, sweeping right-to-left
\State Form $(\vb{P}_i)_{i=1}^d$ by left orthogonalizing $(\vb{Q}_i)_{i=1}^d$, sweeping left-to-right; let $\widetilde{\vb{G}}_i$ be the $i$th core before its orthogonalization
\For{$i=1,\dots,d$} \Comment{outer orthogonalize cores}
    \State Factor $\widetilde{\vb{G}}_i[a,y,b]=\sum_{x} \vb{O}_i[a,x,b]\,\vb{X}[x,y]$ so that the $(r_{i-1}r_i) \times n_i$ matrix formed by permuting and reshaping $\vb{O}_i$ has orthonormal columns
    \State Contract $\vb{X}$ into $\vb{U}_i$ to form $\widetilde{\vb{U}}_i$
\EndFor
\end{algorithmic}
\end{algorithm}

Under these conditions, any $\vb{v} \in T_\vb{p} \mathcal{M}_{\vb{n}, \vb{r}}$ can be expressed as
\begin{equation}
\label{eq:t3_tangent_sum_lru}
\begin{tikzpicture}[scale=1.0, every node/.style={scale=0.9},baseline=(current  bounding  box.center)]

    \node at (-1,0.0) {$\vb{v}$};

    \node at (0,0.0) {$=$};

    \node at (1,0.0) {\Large $\displaystyle \sum_{i=1}^d$};

    \node[draw, rounded corners, minimum size=0.5cm, inner sep=0.0] (oneL) at (2,0.5) {$1$};
    \node[draw, rounded corners, minimum size=0.5cm, inner sep=0.0] (oneR) at (10,0.5) {$1$};
    
    \node[draw, rounded corners, minimum size=0.8cm, inner sep=0.5] (G1) at (3,0.5) {$\vb{P}_{1}$};
    \node (dots1) at (4,0.5) {$\dots$};
    \node[draw, rounded corners, minimum size=0.8cm, inner sep=0.5] (Gprev) at (5,0.5) {$\vb{P}_{i-1}$};
    \node[draw, rounded corners, minimum size=0.8cm, inner sep=0.5] (Gcur) at (6,0.5) {$\vb{O}_{i}$};
    \node[draw, rounded corners, minimum size=0.8cm, inner sep=0.5] (Gnext) at (7,0.5) {$\vb{Q}_{i+1}$};
    \node (dots2) at (8,0.5) {$\dots$};
    \node[draw, rounded corners, minimum size=0.8cm, inner sep=0.5] (Gd) at (9,0.5) {$\vb{Q}_{d}$};

    \draw (oneL) -- (G1) -- (dots1) -- (Gprev) -- (Gcur) -- (Gnext) -- (dots2) -- (Gd) -- (oneR);
    
    \node[draw, rounded corners, minimum size=0.8cm, inner sep=0.5] (B1) at (3,-0.5) {$\vb{U}_1$};
    \node[draw, rounded corners, minimum size=0.8cm, inner sep=0.5] (Bprev) at (5,-0.5) {$\vb{U}_{i-1}$};
    \node[fill=lightgray, draw, rounded corners, minimum size=0.8cm, inner sep=0.5] (Bcur) at (6,-0.5) {$\delta \vb{U}_{i}$};
    \node[draw, rounded corners, minimum size=0.8cm, inner sep=0.5] (Bnext) at (7,-0.5) {$\vb{U}_{i+1}$};
    \node[draw, rounded corners, minimum size=0.8cm, inner sep=0.5] (Bd) at (9,-0.5) {$\vb{U}_{d}$};
			
    \draw (G1) -- (B1);
    \draw (Gprev) -- (Bprev);
    \draw (Gcur) -- (Bcur);
    \draw (Gnext) -- (Bnext);
    \draw (Gd) -- (Bd);

    \node (B1x) at (3,-1.25) {};
    \node (Bprevx) at (5,-1.25) {};
    \node (Bcurx) at (6,-1.25) {};
    \node (Bnextx) at (7,-1.25) {};
    \node (Bdx) at (9,-1.25) {};

    \draw (B1) -- (B1x);
    \draw (Bprev) -- (Bprevx);
    \draw (Bcur) -- (Bcurx);
    \draw (Bnext) -- (Bnextx);
    \draw (Bd) -- (Bdx);

    %


    \node at (1,0.0-2.5) {\Large $+$};

    \node[draw, rounded corners, minimum size=0.5cm, inner sep=0.0] (oneL) at (2,0.5-2.5) {$1$};
    \node[draw, rounded corners, minimum size=0.5cm, inner sep=0.0] (oneR) at (10,0.5-2.5) {$1$};
    
    \node[draw, rounded corners, minimum size=0.8cm, inner sep=0.5] (G1) at (3,0.5-2.5) {$\vb{P}_{1}$};
    \node (dots1) at (4,0.5-2.5) {$\dots$};
    \node[draw, rounded corners, minimum size=0.8cm, inner sep=0.5] (Gprev) at (5,0.5-2.5) {$\vb{P}_{i-1}$};
    \node[fill=lightgray, draw, rounded corners, minimum size=0.8cm, inner sep=0.5] (Gcur) at (6,0.5-2.5) {$\delta \vb{G}_{i}$};
    \node[draw, rounded corners, minimum size=0.8cm, inner sep=0.5] (Gnext) at (7,0.5-2.5) {$\vb{Q}_{i+1}$};
    \node (dots2) at (8,0.5-2.5) {$\dots$};
    \node[draw, rounded corners, minimum size=0.8cm, inner sep=0.5] (Gd) at (9,0.5-2.5) {$\vb{Q}_{d}$};

    \draw (oneL) -- (G1) -- (dots1) -- (Gprev) -- (Gcur) -- (Gnext) -- (dots2) -- (Gd) -- (oneR);
    
    \node[draw, rounded corners, minimum size=0.8cm, inner sep=0.5] (B1) at (3,-0.5-2.5) {$\vb{U}_1$};
    \node[draw, rounded corners, minimum size=0.8cm, inner sep=0.5] (Bprev) at (5,-0.5-2.5) {$\vb{U}_{i-1}$};
    \node[draw, rounded corners, minimum size=0.8cm, inner sep=0.5] (Bcur) at (6,-0.5-2.5) {$\vb{U}_{i}$};
    \node[draw, rounded corners, minimum size=0.8cm, inner sep=0.5] (Bnext) at (7,-0.5-2.5) {$\vb{U}_{i+1}$};
    \node[draw, rounded corners, minimum size=0.8cm, inner sep=0.5] (Bd) at (9,-0.5-2.5) {$\vb{U}_{d}$};
			
    \draw (G1) -- (B1);
    \draw (Gprev) -- (Bprev);
    \draw (Gcur) -- (Bcur);
    \draw (Gnext) -- (Bnext);
    \draw (Gd) -- (Bd);

    \node (B1x) at (3,-1.25-2.5) {};
    \node (Bprevx) at (5,-1.25-2.5) {};
    \node (Bcurx) at (6,-1.25-2.5) {};
    \node (Bnextx) at (7,-1.25-2.5) {};
    \node (Bdx) at (9,-1.25-2.5) {};

    \draw (B1) -- (B1x);
    \draw (Bprev) -- (Bprevx);
    \draw (Bcur) -- (Bcurx);
    \draw (Bnext) -- (Bnextx);
    \draw (Bd) -- (Bdx);
    
\end{tikzpicture}
\end{equation}
\noindent
for some perturbation cores $\delta \tuple{V}=((\delta \vb{U}_i)_{i=1}^d, (\delta \vb{G}_i)_{i=1}^d)$ (lightly shaded), which we call \emph{variations}. This representation is not unique, but $\vb{v}$ has a unique representation by variations satisfying the \emph{gauge conditions}:
\begin{align}
\vb{U}_i^T \delta \vb{U}_i &= 0, \quad i=1,\dots,d, \label{eq:gauge_condition_tucker}\\
\left(\leftunfolding{\vb{\leftcore}_i}\right)^T \leftunfolding{\delta \vb{G}_i} &= 0, \quad i=1,\dots,d-1. \label{eq:gauge_condition_tt}
\end{align}

For variations $\delta \tuple{V}=((\delta \vb{U}_i)_{i=1}^d, (\delta \vb{G}_i)_{i=1}^d)$ and $\delta \tuple{V}'=((\delta \vb{U}_i')_{i=1}^d, (\delta \vb{G}_i')_{i=1}^d)$ representing $\vb{v}, \vb{v}'\in T_\vb{p} \mathcal{M}_{\vb{n},\vb{r}}$, and for scalars $a, b$, the linear combination $a \vb{v} + b \vb{v}'$ is represented by variations
$
((a\delta \vb{U}_i + b \delta \vb{U}_i')_{i=1}^d, (a\delta \vb{G}_i + b \delta \vb{G}_i')_{i=1}^d).
$
If $\delta \tuple{V}, \delta \tuple{V}'$ satisfy the gauge conditions, then
$$
\innerproduct{\vb{v}}{ \vb{v}'}_\text{HS} = \sum_{i=1}^d \innerproduct{\delta \vb{U}_i}{\delta \vb{U}_i'}_\text{HS} + \sum_{i=1}^d \innerproduct{\delta \vb{G}_i}{\delta \vb{G}_i'}_\text{HS}.
$$

\begin{figure}
\centering
\begin{subfigure}[b]{0.98\textwidth}
    \centering
    \begin{tikzpicture}[scale=0.5, every node/.style={scale=0.7}]

\node at (-1.5,0.0,0) {$\vb{W}_i ~~= $};

\coordinate (Az) at (0,2,-1);
\coordinate (Bz) at (0,2,1);
\coordinate (Cz) at (0,-2,1);
\coordinate (Dz) at (2,-2,-1);
\coordinate (Ez) at (2,2,-1);
\coordinate (Fz) at (2,2,1);
\coordinate (Gz) at (2,-2,1);
\coordinate (Hz) at (0,-2,-1);

\draw[black] (Az) -- (Ez) -- (Dz) -- (Hz) -- cycle; 
\draw[black] (Dz) -- (Ez) -- (Fz) -- (Gz) -- cycle; 
\draw[black] (Cz) -- (Bz) -- (Fz) -- (Gz) -- cycle; 
\draw[black] (Az) -- (Bz) -- (Fz) -- (Ez) -- cycle; 
\draw[black] (Cz) -- (Gz) -- (Dz) -- (Hz) -- cycle; 

\coordinate (A4) at (0,-2,0);
\coordinate (B4) at (0,-2,1);
\coordinate (C4) at (0,0,1);
\coordinate (D4) at (1,0,0);
\coordinate (E4) at (1,-2,0);
\coordinate (F4) at (1,-2,1);
\coordinate (G4) at (1,0,1);

\draw[black,fill=gray!80,opacity=1.0] (D4) -- (E4) -- (F4) -- (G4) -- cycle; 
\draw[black,fill=gray!40,opacity=1.0] (C4) -- (B4) -- (F4) -- (G4) -- cycle; 


%

\coordinate (A2) at (0,2,-1);
\coordinate (B2) at (0,2,0);
\coordinate (C2) at (0,0,0);
\coordinate (D2) at (1,0,-1);
\coordinate (E2) at (1,2,-1);
\coordinate (F2) at (1,2,0);
\coordinate (G2) at (1,0,0);

\draw[black,fill=gray!80,opacity=1.0] (D2) -- (E2) -- (F2) -- (G2) -- cycle; 
\draw[black,fill=gray!40,opacity=1.0] (C2) -- (B2) -- (F2) -- (G2) -- cycle; 
\draw[black,fill=gray!20,opacity=1.0] (A2) -- (B2) -- (F2) -- (E2) -- cycle; 


%

\coordinate (A1) at (0,2,0);
\coordinate (B1) at (0,2,1);
\coordinate (C1) at (0,0,1);
\coordinate (D1) at (1,0,0);
\coordinate (E1) at (1,2,0);
\coordinate (F1) at (1,2,1);
\coordinate (G1) at (1,0,1);

\draw[black,fill=gray!80,opacity=1.0] (D1) -- (E1) -- (F1) -- (G1) -- cycle; 
\draw[black,fill=gray!40,opacity=1.0] (C1) -- (B1) -- (F1) -- (G1) -- cycle; 
\draw[black,fill=gray!20,opacity=1.0] (A1) -- (B1) -- (F1) -- (E1) -- cycle; 


%

\coordinate (A3) at (1,2,0);
\coordinate (B3) at (1,2,1);
\coordinate (C3) at (1,0,1);
\coordinate (D3) at (2,0,0);
\coordinate (E3) at (2,2,0);
\coordinate (F3) at (2,2,1);
\coordinate (G3) at (2,0,1);

\draw[black,fill=gray!80,opacity=1.0] (D3) -- (E3) -- (F3) -- (G3) -- cycle; 
\draw[black,fill=gray!40,opacity=1.0] (C3) -- (B3) -- (F3) -- (G3) -- cycle; 
\draw[black,fill=gray!20,opacity=1.0] (A3) -- (B3) -- (F3) -- (E3) -- cycle; 


%

\node at (0.125,0.5,0) {$\delta \mathbf{G}_i$};
\node at (1.125,0.5,0) {$\mathbf{Q}_i$};
\node at (0.125,-1.5,0) {$\mathbf{O}_i$};
\node (P) at (-0.75,2.15,0) {$\mathbf{P}_i$};

\node (Px) at (0.85, 2.05, 0.0) {};
\draw[black,->,thick] (P) edge [bend left] (Px);


\coordinate (Aq) at (0+7,2+0.5,0);
\coordinate (Bq) at (2+7,2+0.5,0);
\coordinate (Cq) at (0+7,-3+0.5,0);
\coordinate (Dq) at (2+7,-3+0.5,0);

\draw[black,fill=gray!40,opacity=1.0] (Aq) -- (Cq) -- (Dq) -- (Bq) -- cycle;

\coordinate (Bw) at (2+7+2,2+0.5,0);
\coordinate (Dw) at (2+7+2,-3+0.5,0);

\draw[black,fill=gray!40,opacity=1.0] (Bq) -- (Bw) -- (Dw) -- (Dq) -- cycle;

\node at (-1.25+7,-0.5+0.5,0) {$\vb{B}_i ~~= $};

\node at (9-1,-0.5+0.5,0) {$\mathbf{U}_i$};
\node at (9+1,-0.5+0.5,0) {$\delta \mathbf{U}_i$};

\end{tikzpicture}
\end{subfigure}
\caption{Middle tensor train (left) and Tucker (right) cores for the representation of a Tucker tensor train tangent vector as a Tucker tensor train with doubled ranks. Empty portions are filled with zeros.
}
\label{fig:t3_tangent_2r}
\end{figure}

\subsubsection{Tangent vectors as doubled rank Tucker tensor trains} 
\label{app:tangent_vectors_doubled_ranks}
One can verify that $\vb{v}$ in \eqref{eq:t3_tangent_sum_lru} is represented by the Tucker tensor train $((\vb{B}_i)_{i=1}^d,(\vb{W}_i)_{i=1}^d)$ with Tucker bases
\begin{equation}
\label{eq:tucker_doubled_tangent}
\vb{B}_i = \begin{bmatrix}
\vb{U}_i & \delta \vb{U}_i
\end{bmatrix} \in \mathbb{R}^{N_i \times 2 n_i}, \quad \text{for} \quad i=1,\dots,d,
\end{equation}
and tensor train cores $\vb{W}_1 \in \mathbb{R}^{1 \times 2n_1 \times 2r_1}$, $\vb{W}_i\in \mathbb{R}^{2r_i \times 2n_i \times 2r_{i+1}}$, $i=2,\dots,d-1$, and $\vb{W}_d \in \mathbb{R}^{2r_{d-1} \times 2n_d \times 1}$ which are defined by
\begin{align}
\vb{W}_1\!\left(\begin{bmatrix}\vb{\xi}_1 \\ \delta\vb{\xi}_1\end{bmatrix}\right) &= \begin{bmatrix}
\delta \vb{G}_1(\vb{\xi}_1) & \leftcore_1(\vb{\xi}_1)
\end{bmatrix} + \begin{bmatrix}
\vb{O}_1(\delta\vb{\xi}_1) & 0
\end{bmatrix}, \label{eq:t3_tangent_2r_first} \\
\vb{W}_i\!\left(\begin{bmatrix}\vb{\xi}_i \\ \delta\vb{\xi}_i\end{bmatrix}\right) &= \begin{bmatrix}\rightcore_i(\vb{\xi}_i) & 0 \\ \delta \vb{G}_i(\vb{\xi}_i) & \leftcore_i(\vb{\xi}_i) \end{bmatrix} + \begin{bmatrix}0 & 0 \\ \vb{O}_i(\delta\vb{\xi}_i) & 0 \end{bmatrix}, \quad \quad i=2,\dots,d-1, \label{eq:t3_tangent_2r_mid}\\
\vb{W}_d\!\left(\begin{bmatrix}\vb{\xi}_d \\ \delta\vb{\xi}_d\end{bmatrix}\right) &= \begin{bmatrix}\rightcore_d(\vb{\xi}_d) \\ \delta \vb{G}_d(\vb{\xi}_d) \end{bmatrix} + \begin{bmatrix}0 \\ \vb{O}_d(\delta\vb{\xi}_d) \end{bmatrix}. \label{eq:t3_tangent_2r_last}
\end{align}
This is illustrated in \figref{fig:t3_tangent_2r}.
Thus, each tangent vector may be represented by a Tucker tensor train with Tucker ranks $(2n_1, \dots, 2n_d)$ and tensor train ranks $(1, 2r_1, \dots, 2r_{d-1}, 1)$. 
Vectors in $\vb{p} + T_\vb{p} \mathcal{M}_{\vb{n},\vb{r}}$ likewise have doubled-rank representations: $\vb{p}+\vb{u}$ is represented by the Tucker tensor train $((\vb{B}_1, \dots, \vb{B}_d), (\vb{W}_1, \dots, \vb{W}_{d-1}, \widetilde{\vb{W}}_d))$,
where $\widetilde{\vb{W}}_d$ is defined by
\begin{equation*}
\widetilde{\vb{W}}_d\!\left(\begin{bmatrix}\vb{\xi}_d \\ \delta\vb{\xi}_d\end{bmatrix}\right) = \vb{W}_d\!\left(\begin{bmatrix}\vb{\xi}_d \\ \delta\vb{\xi}_d\end{bmatrix}\right) +  \begin{bmatrix}0 \\ \leftcore_d(\vb{\xi}_d) \end{bmatrix}.
\end{equation*}
This is a generalization of the formula in \cite[Section 1.6.5]{voorhaartensor} for tensor trains.

\section{Additional proofs}
\label{app:finite_to_infinite}

\begin{proof}[Proof of \thmref{thm:probe_t3}]
Since $T$ is the multilinear function represented by $\tuple{T}$, contracting each input index against the corresponding probing vector gives the scalar
$$
T(\vb{w}_1,\dots,\vb{w}_d) = G_1(\vb{U}_1^T\vb{w}_1)\,G_2(\vb{U}_2^T\vb{w}_2)\cdots G_d(\vb{U}_d^T\vb{w}_d) = G_1(\vb{\xi}_1)\cdots G_d(\vb{\xi}_d).
$$
Replacing the $i$th probing vector $\vb{w}_i$ by a free argument $\vb{t}$ replaces the $i$th factor $G_i(\vb{\xi}_i)$ by $G_i(\vb{U}_i^T\vb{t})$, so
\begin{align*}
z_i(\vb{t})
&= G_1(\vb{\xi}_1) \cdots G_{i-1}(\vb{\xi}_{i-1})\, G_i(\vb{U}_i^T\vb{t})\, G_{i+1}(\vb{\xi}_{i+1}) \cdots G_d(\vb{\xi}_d) \\
&= \pushleft_{i-1}^T G_i(\vb{U}_i^T\vb{t}) \pushright_i
= \eta_i(\vb{U}_i^T \vb{t}).
\end{align*}
Because $G_i(\cdot)$ is linear in its contracted index, $\eta_i(\vb{U}_i^T\vb{t}) = (\pushleft_{i-1}^T \vb{G}_i \pushright_i)^T \vb{U}_i^T \vb{t} = (\vb{U}_i \vb{\eta}_i)^T \vb{t}$, where $\vb{\eta}_i := \pushleft_{i-1}^T \vb{G}_i \pushright_i$. Hence the $i$th probe is $\vb{z}_i = \vb{U}_i \vb{\eta}_i$.

The edge variables, and hence the probes, satisfy the sweeping recursions evaluated by \algoref{alg:probe_t3}. The contract-up step (\algoref{alg:probe_t3}, line~\ref{ln:t3-up}) is the definition of $\vb{\xi}_i$; the products defining $\pushleft_i$ and $\pushright_i$ telescope into the left and right sweeps (lines~\ref{ln:t3-left} and~\ref{ln:t3-right}), with boundary terms $\pushleft_0^T=\pushright_d=1$ (line~\ref{ln:t3-bdy}); and the central contraction $\vb{\eta}_i=\pushleft_{i-1}^T \vb{G}_i \pushright_i$ (line~\ref{ln:t3-central}) and contract-down step $\vb{z}_i=\vb{U}_i \vb{\eta}_i$ (line~\ref{ln:t3-down}) are the identities established above. Here $\pushleft_{i-1}^T \vb{G}_i \pushright_i$ denotes the vector obtained by contracting $\vb{\mu}_{i-1}$ and $\vb{\nu}_i$ with the left and right tensor train indices of $\vb{G}_i$, leaving its Tucker index uncontracted. \algoref{alg:probe_t3} evaluates these in an order consistent with their dependencies, which establishes correctness.

For the complexity, account for each loop under the stated shape and ranks. Reducing inputs (\algoref{alg:probe_t3}, line~\ref{ln:t3-up}) costs $O(Nn)$ for each of the $d-1$ input modes and $O(Mm)$ for the output mode, i.e.\ $O(dNn + Mm)$. In the sweeps (lines~\ref{ln:t3-left}--\ref{ln:t3-right}), each step forms a matrix $G_i(\vb{w}_i)\in\mathbb{R}^{r\times r}$ by contracting the Tucker index of an $r\times n\times r$ core at cost $O(nr^2)$, followed by a vector--matrix product at cost $O(r^2)$; over $O(d)$ steps this is $O(d n r^2)$. The central contractions (line~\ref{ln:t3-central}) contract each $r\times n\times r$ core against two length-$r$ vectors at cost $O(nr^2)$, giving $O(d n r^2)$. Reducing outputs (line~\ref{ln:t3-down}) costs $O(dNn + Mm)$, as in line~\ref{ln:t3-up}. Operations at the output mode involve $m$ and $r$ in place of $n$ and contribute $O(rm)\le O(Mm)$. Summing gives $O(dNn + d n r^2 + Mm)$.
\end{proof}

\begin{proof}[Proof of \thmref{thm:probe_tangent}]
\emph{Base point (\algoref{alg:probe_basepoint}).} Applying the sweeping scheme of \algoref{alg:probe_t3} to the orthogonal representations of $\vb{p}$, with the gauge-appropriate cores $\vb{P}_i$, $\vb{Q}_i$, and $\vb{O}_i$ in place of $\vb{G}_i$, yields the base-point edge variables and probes computed in \algoref{alg:probe_basepoint} (lines~\ref{ln:bp-up}--\ref{ln:bp-down}). Since each orthogonal representation represents $\vb{p}$, Theorem~\ref{thm:probe_t3} gives $\widehat{\vb{z}}_i = \vb{U}_i\widehat{\vb{\eta}}_i$ as the $i$th probe of $\vb{p}$.

\emph{Tangent vector (\algoref{alg:probe_tangent}).} By \appref{app:tangent_vectors_doubled_ranks}, $\vb{v}$ is represented by the doubled-rank Tucker tensor train $\tuple{V}$, so by Theorem~\ref{thm:probe_t3}, applying the sweeping scheme of \algoref{alg:probe_t3} to $\tuple{V}$ computes $\delta\vb{z}=\jac^{(s)}(\delta\tuple{V})$. The cores of $\tuple{V}$ have the block-triangular structure of \eqref{eq:tucker_doubled_tangent} and \eqref{eq:t3_tangent_2r_first}--\eqref{eq:t3_tangent_2r_last}, under which each edge variable splits into a base-point and a perturbation component. Reading off the perturbation component of each sweep gives the assignments of \algoref{alg:probe_tangent} (lines~\ref{ln:tan-up}, \ref{ln:tan-left}--\ref{ln:tan-down}), with boundary values $\ppsleft_0 = 0$ and $\ppsright_d = 0$ (line~\ref{ln:tan-bdy}). Evaluated in an order consistent with their dependencies (the base-point quantities being supplied by \algoref{alg:probe_basepoint}), \algoref{alg:probe_tangent} returns $\jac^{(s)}(\delta\tuple{V})$.

For the complexity, each line of either algorithm performs a constant number of operations of the same shapes counted in the proof of Theorem~\ref{thm:probe_t3}: a length-$N$ (or length-$M$) reduction by $\vb{U}_i$ or $\delta\vb{U}_i$, a contraction of an $r\times n\times r$ core against a vector at cost $O(nr^2)$, or a length-$r$ vector--matrix product at cost $O(r^2)$. Summing over modes, each of Algorithms~\ref{alg:probe_basepoint} and~\ref{alg:probe_tangent} requires $O(dNn + d n r^2 + Mm)$ operations.
\end{proof}

\begin{proof}[Proof of \thmref{thm:transpose_tangent}]
The transpose satisfies the identity
\begin{equation}
\left(\jac^{(s)}\right)^T(\widetilde{\vb{z}}) = D_{\delta \tuple{V}} \left\langle \widetilde{\vb{z}}, \jac^{(s)}(\delta \tuple{V})\right\rangle = D_{\delta \tuple{V}}
\left\langle \widetilde{\vb{z}}, \delta\vb{z}\right\rangle,
\label{eq:tangent_transpose_ATY_TYY}
\end{equation}
where $\delta\vb{z} = (\delta\vb{z}_1, \dots, \delta\vb{z}_d)$ are the probes of the tangent vector represented by $\delta \tuple{V}$ on $\vb{w}_1, \dots, \vb{w}_d$, and $D_{\delta \tuple{V}}$ is the total derivative with respect to $\delta \tuple{V}$.

To evaluate this derivative, we apply the adjoint state method.
Define the state and adjoint variables  
\begin{align*}
\Xi &= ((\delta\vb{\xi}_i)_{1}^d, (\ppsleft_i)_{1}^{d-1}, (\ppsright_i)_{1}^{d-1}, (\delta \vb{\eta}_i)_{1}^d), \\
\widetilde{\Xi} &= ((\widetilde{\delta\vb{\xi}}_i)_{1}^d, (\incrleft_i)_{1}^{d-1}, (\incrright_i)_{1}^{d-1}, (\widetilde{\delta \vb{\eta}}_i)_{1}^d),
\end{align*}
and the Lagrangian  
$$
\mathcal{L}(\delta \tuple{V}, \Xi, \widetilde{\Xi}) = \Psi(\delta \tuple{V}, \Xi) + \langle \widetilde{\Xi}, E(\Xi) \rangle.
$$
The state variables $\Xi$ include all edge variables that depend on $\delta \tuple{V}$. The edge variables $\widehat{\vb{\xi}}_i$, $\widehat{\pushleft}_i$, $\widehat{\pushright}_i$, and $\widehat{\vb{\eta}}_i$ are independent of $\delta \tuple{V}$ and are therefore
treated as constants.

The constraint $E(\Xi)$ consists of the residuals of the tangent sweep \algoref{alg:probe_tangent} (lines~\ref{ln:tan-up}, \ref{ln:tan-left}, \ref{ln:tan-right}, and~\ref{ln:tan-central}), which determine the edge variables in $\Xi$.
The objective function is defined by
$$
\Psi(\delta \tuple{V}, \Xi) = \sum_{i=1}^d
\widetilde{\vb{z}}_i^T (\vb{U}_i \delta \vb{\eta}_i + \delta \vb{U}_i \widehat{\vb{\eta}}_i),
\qquad i = 1, \dots, d.
$$
When \(\Xi\) satisfies \(E(\Xi) = 0\), we have  
$$
\Psi(\delta \tuple{V}, \Xi) = \sum_{i=1}^d \widetilde{\vb{z}}_i^T \delta\vb{z}_i
= \langle \widetilde{\vb{z}}, \delta\vb{z} \rangle,
$$
which is the quantity differentiated in the transpose identity.
Hence,
\begin{equation*}
\left(\jac^{(s)}\right)^T(\widetilde{\vb{z}}) = \partial_{\delta_V} \mathcal{L}(\delta \tuple{V}, \Xi, \widetilde{\Xi}),
\end{equation*}
where $\Xi$ and $\widetilde{\Xi}$ solve
\begin{align*}
0 &= \partial_\Xi \mathcal{L}(\delta \tuple{V}, \Xi, \widetilde{\Xi}), \\
0 &= \partial_{\widetilde{\Xi}} \mathcal{L}(\delta \tuple{V}, \Xi, \widetilde{\Xi}).
\end{align*}

Substituting the definitions of $E$ and $\Psi$ into $\mathcal{L}$ yields the large sum:
\begin{align*}
&\mathcal{L}(\delta \tuple{V}, \Xi, \widetilde{\Xi})
 = \\
&~~~\sum_{i=1}^d\widetilde{\vb{z}}_i^T\vb{U}_i \delta \vb{\eta}_i + \widetilde{\vb{z}}_i^T\delta \vb{U}_i \widehat{\vb{\eta}}_i \\
&+ \sum_{i=1}^d \widetilde{\delta \vb{\xi}_i}^T\delta \vb{U}_i^T \vb{w}_i - \widetilde{\delta \vb{\xi}_i}^T \delta \vb{\xi}_i \\
&+ \sum_{i=1}^{d-1}  \incrright_i^T\delta G_{i+1}(\widehat{\vb{\xi}}_{i+1}) \widehat{\pushright}_{i+1} + \incrright_i^T O_{i+1}(\delta\vb{\xi}_{i+1}) \widehat{\pushright}_{i+1} + \incrright_i^TP_{i+1}(\widehat{\vb{\xi}}_{i+1}) \ppsright_{i+1} - \incrright_i^T\ppsright_{i}\\
&+\sum_{i=1}^{d-1}  \ppsleft_{i-1}^T Q_i(\widehat{\vb{\xi}}_i) \incrleft_i + \widehat{\pushleft}_{i-1}^T \delta G_i(\widehat{\vb{\xi}}_i)\incrleft_i + \widehat{\pushleft}_{i-1}^T O_i(\delta\vb{\xi}_i)\incrleft_i - \ppsleft_{i}^T\incrleft_i \\
&+ \sum_{i=1}^d \ppsleft_{i-1}^T Q_i(\widetilde{\delta \vb{\eta}_i}) \widehat{\pushright}_i
+ \widehat{\pushleft}_{i-1}^T P_i(\widetilde{\delta \vb{\eta}_i}) \ppsright_i
+ \widehat{\pushleft}_{i-1}^T \delta G_i(\widetilde{\delta \vb{\eta}_i}) \widehat{\pushright}_{i} - \widetilde{\delta \vb{\eta}_i}^T\delta \vb{\eta}_i.
\end{align*}
The condition $0=\partial_{\widetilde{\Xi}} \mathcal{L}$ is the state equation, $0=E(\Xi)$. 
Differentiating $\mathcal{L}$ with respect to $\Xi$ yields the adjoint equations, which are given componentwise as follows:
\begin{itemize}
\item $0=\partial_{\delta \vb{\eta}_i} \mathcal{L}$ gives the contract-up step (\algoref{alg:transpose_tangent}, line~\ref{ln:tr-up}).
\item $0=\partial_{\ppsright_i} \mathcal{L}$ gives the left sweep (line~\ref{ln:tr-left}), with boundary $\incrright_0^T=0$ (line~\ref{ln:tr-bdy}).
\item $0=\partial_{\ppsleft_i} \mathcal{L}$ gives the right sweep (line~\ref{ln:tr-right}), with boundary $\incrleft_d=0$ (line~\ref{ln:tr-bdy}).
\item $0=\partial_{\delta \vb{\xi}_i} \mathcal{L}$ gives the central contraction (line~\ref{ln:tr-central}).
\end{itemize}
Differentiating $\mathcal{L}$ with respect to the variations yields the components of the gradient:
\begin{itemize}
\item $\widetilde{\delta \vb{U}}_i = \partial_{\delta \vb{U}_i} \mathcal{L}$ is the assembly of $\widetilde{\delta\vb{U}}_i$ (line~\ref{ln:tr-assembleU}).
\item $\widetilde{\delta \vb{G}}_i = \partial_{\delta \vb{G}_i} \mathcal{L}$ is the assembly of $\widetilde{\delta\vb{G}}_i$ (line~\ref{ln:tr-assembleG}).
\end{itemize}

The gradient is therefore
$$
((\widetilde{\delta \vb{U}}_i)_{i=1}^d, (\widetilde{\delta \vb{G}}_i)_{i=1}^d) = D_{\delta \tuple{V}}
\left\langle \widetilde{\vb{z}}, \delta\vb{z}\right\rangle= \left(\jac^{(s)}\right)^T(\widetilde{\vb{z}}).
$$
These stationarity conditions are precisely the assignments of \algoref{alg:transpose_tangent} (which first forms the base edge variables $\widehat{\vb{\xi}}_i$, $\widehat{\vb{\eta}}_i$, $\widehat{\pushleft}_i$, and $\widehat{\pushright}_i$ via \algoref{alg:probe_basepoint}), evaluated in an order consistent with their dependencies. Hence \algoref{alg:transpose_tangent} returns $\left(\jac^{(s)}\right)^T(\widetilde{\vb{z}})$.

For the complexity, the base edge variables are formed in $O(dNn + d n r^2 + Mm)$ operations by Theorem~\ref{thm:probe_tangent}. The reductions (line~\ref{ln:tr-up}) and the assembly of $\widetilde{\delta\vb{U}}_i$ (line~\ref{ln:tr-assembleU}) cost $O(dNn + Mm)$, the latter through outer products $\widetilde{\vb{z}}_i\widehat{\vb{\eta}}_i^T$ and $\vb{w}_i\widetilde{\delta\vb{\xi}}_i^T$. The adjoint sweeps (lines~\ref{ln:tr-left}--\ref{ln:tr-right}) and the central contraction (line~\ref{ln:tr-central}) each contract $r\times n\times r$ cores against vectors a constant number of times per mode, at cost $O(d n r^2)$; the assembly of $\widetilde{\delta\vb{G}}_i$ (line~\ref{ln:tr-assembleG}) forms three $r\times n\times r$ tensor products per mode, also $O(d n r^2)$. Summing gives $O(dNn + d n r^2 + Mm)$.
\end{proof}

\begin{proof}[Proof of \propref{thm:preconditioned_are_tt}]
For this proof, we assume without loss of generality that $\norm{\vb{C}_i}_{\mathrm{HS}}=1$ for all $\vb{C}_i$. Under this assumption, we have $\left(\sum_{\alpha=1}^{N_1 \dots N_{j}} \gamma_{j,\alpha}^2 \right)^{1/2}=1$ on the right-hand side of \eqref{eq:gamma_eigs_assumption} and $\left(\prod_{l=1}^k \norm{\vb{C}_l}_{\mathrm{HS}}\right)=1$ on the right-hand side of \eqref{eq:tt_overall_bound}. The stated result for $\norm{\vb{C}_i}_{\mathrm{HS}}\neq 1$ then follows from scaling and multilinearity.
The central components of this proof are \lemref{lem:helper_for_tt_thm} (given below, illustrated in \figref{fig:tt_peel_one_mode}) and its special case \lemref{lem:helper2_for_tt_thm} (given in \appref{app:finite_to_infinite}), which show that one may ``peel'' indices from $S$ in such a way that the tensor remaining after peeling has similar structure to $S$. The peeling is done repeatedly until there is only one input index left. The tensors that are peeled off during this process are the cores of the desired tensor train approximation. The first and last cores are special cases.

\lemref{lem:helper2_for_tt_thm}
provides a multilinear function $B_{k}$ (which is $B'$ in the lemma) with output dimension $r_k$ and a linear function $G_{k+1}:\mathbb{R}^M \rightarrow \mathbb{R}^{r_k \times 1}$ such that
\begin{align}
&\abs{B(\vb{C}_1\vb{x}_1, \dots, \vb{C}_k\vb{x}_k)^T \vb{\omega} - B_k(\vb{C}_1\vb{x}_1, \dots, \vb{C}_k\vb{x}_k)^T G_{k+1}(\vb{\omega})} \nonumber \\
&\le 2 \epsilon_k \norm{B} \norm{\vb{\omega}} \prod_{i=1}^{k} \norm{\vb{x}_i}. \label{eq:tt_first_factorization_special_case}
\end{align}
Furthermore, $\norm{B_k} \le \norm{B}$ and $\norm{G_{k+1}} \le 1$.
Iteratively applying \lemref{lem:helper_for_tt_thm} with $B_k$ as $B$, then with the ``remainder'' $B_{k-1}$ as $B$, and so on, yields a sequence of vector-valued multilinear functions $B_{k-1}, \dots, B_1$ and matrix-valued linear functions $G_k, \dots, G_2$ such that 
\begin{align}
&\abs{B_{j+1}(\vb{C}_1\vb{x}_1, \dots, \vb{C}_{j+1}\vb{x}_{j+1})^T \vb{\nu} - B_{j}(\vb{C}_1\vb{x}_1, \dots, \vb{C}_j\vb{x}_{j})^T G_{j+1}(\vb{x}_{j+1}) \vb{\nu}} \nonumber \\
&\le 2 \epsilon_j \norm{B} \norm{\vb{\nu}} \prod_{i=1}^{j+1} \norm{\vb{x}_i} \label{eq:recursive_R_bound}
\end{align}
Here, the output of $B_{j}$ has dimension $r_j \le r_j'$, $G_{j+1}:\mathbb{R}^{N_{j+1}}\rightarrow \mathbb{R}^{r_{j} \times r_{j+1}}$, and we have $\norm{B_j} \le \norm{B_{j+1}}$ and $\norm{G_{j}} \le 1$.
We define the linear function $G_1:\mathbb{R}^{N_1}\rightarrow \mathbb{R}^{1 \times r_1}$ by $G_1(\vb{x}_1) := B_1(\vb{C}_1 \vb{x}_1)^T$ and note that there is no bound on $\norm{G_1}$.

Define the functions $S_{j}: \mathbb{R}^{N_1} \times \dots \times \mathbb{R}^{N_k} \rightarrow \mathbb{R}^M$ by
\begin{equation*}
\vb{\omega}^T S_{j}(\vb{x}_1, \dots, \vb{x}_k) := B_j(\vb{C}_1 \vb{x}_1, \dots, \vb{C}_j \vb{x}_j)^T G_{j+1}(\vb{x}_{j+1}) \dots G_{k}(\vb{x}_{k}) G_{k+1}(\vb{\omega}).
\end{equation*}
The functions $S_j$ are successive approximations of $F$ which are formed by contracting longer and longer tensor train ``right tails'' with smaller and smaller remainders. 
By construction, $S_{k+1}=F$ and $S:=S_1$ is the multilinear function corresponding to the tensor train $\tuple{S}=(G_1, \dots, G_{k+1})$ which has ranks $\vb{r}$.
The rest of the proof therefore consists of bounding $\norm{S_{k+1} - S_1}$.

Suppose $1 \le j \le k$. By the definition of the induced norm, we have
\begin{equation}
\label{eq:T_delta_j_norm}
\norm{S_j - S_{j+1}} = \sup_{\substack{
\norm{\vb{\omega}}=1 \\
\norm{\vb{x}_i}=1
}} 
\left\lvert \vb{\omega}^TS_j(\vb{x}_1, \dots, \vb{x}_{k}) - \vb{\omega}^TS_{j+1}(\vb{x}_1, \dots, \vb{x}_{k}) \right\rvert.
\end{equation}
Let $\vb{x}_i$ and $\vb{\omega}$ be vectors of the appropriate sizes satisfying $\norm{\vb{x}_i}=1$, $\norm{\vb{\omega}}=1$. Using the definitions, we have
\begin{align*}
&\vb{\omega}^TS_j(\vb{x}_1, \dots, \vb{x}_{k}) - \vb{\omega}^TS_{j+1}(\vb{x}_1, \dots, \vb{x}_{k}) \\
&= (B_{j+1}(\vb{C}_1 \vb{x}_1, \dots, \vb{C}_{j+1} \vb{x}_{j+1})^T - B_j(\vb{C}_1 \vb{x}_1, \dots, \vb{C}_j \vb{x}_j)^T G_{j+1}(\vb{x}_{j+1})) \vb{\nu}_{j+2}
\end{align*}
where 
$$\vb{\nu}_{j+2} := G_{j+2}(\vb{x}_i) \dots G_{k}(\vb{x}_k) G_{k+1}(\vb{\omega})
$$
for $j < k$, or $\vb{\nu}_{k+1}:=1$ in the case $j=k$. Hence, recalling the bounds in \eqref{eq:tt_first_factorization_special_case} and \eqref{eq:recursive_R_bound}, we have
\begin{equation*}
\norm{S_j - S_{j+1}} \le 2 \epsilon_j \norm{B} \norm{\vb{\nu}_{j+2}} \\
 \le 2 \epsilon_j \norm{B}.
\end{equation*}
In the second inequality, we used the fact that $\norm{\vb{\nu}_{j+2}}=1$, which follows from the bounds $\norm{G_i} \le 1$ and the assumptions that $\norm{\vb{x}_i}=1$ and $\norm{\vb{\omega}}=1$.
Repeatedly using the triangle inequality 
yields 
\begin{equation}
    \norm{F - S} \le 2 \norm{B} \sum_{j=1}^{k} \epsilon_j,
\end{equation}
which is \eqref{eq:tt_overall_bound} under the assumption $\norm{\vb{C}_l}_{\mathrm{HS}}=1$, as required
\end{proof}

\begin{lemma}[Helper for \propref{thm:preconditioned_are_tt}: first step of the peeling process]
\label{lem:helper2_for_tt_thm}
    Under the same setup as \lemref{lem:helper_for_tt_thm}, there exists a multilinear function $B':\mathbb{R}^{N_1} \times \dots \times \mathbb{R}^{N_{k}} \rightarrow \mathbb{R}^{r_{k}}$ satisfying 
    $
    \norm{B'} \le \norm{B},
    $
    and a linear function $G:\mathbb{R}^M \rightarrow \mathbb{R}^{r_{k} \times 1}$ satisfying 
    $
    \norm{G} \le 1
    $
    such that the multilinear function $F'$ defined by
    \begin{equation}
    \label{eq:tt_one_step_factorization2}
        F'(\vb{x}_1, \dots, \vb{x}_{k})^T \vb{w} := B'(\vb{C}_1 \vb{x}_1, \dots, \vb{C}_{k} \vb{x}_{k})^T G(\vb{\omega})
    \end{equation}
    satisfies
    $
        \norm{F - F'} \le 2\epsilon_{k} \norm{B}.
    $
\end{lemma}

\begin{proof}[Proof of \lemref{lem:helper2_for_tt_thm}]
    Let
    $
        \widetilde{B}(\vb{\theta}_1, \dots, \vb{\theta}_{k}, s) := sB(\vb{\theta}_1, \dots, \vb{\theta}_{k})
    $
    be the multilinear function which is the same as $B$, except with an extra scalar nuisance parameter $s$, and define the multilinear function $\widetilde{F}$ by
    $$\widetilde{F}(\vb{x}_1, \dots, \vb{x}_k, s)=\widetilde{B}(\vb{C}_1 \vb{x}_1, \dots, \vb{C}_k \vb{x}_k, s).$$
    Applying \lemref{lem:helper_for_tt_thm} with $\widetilde{B}$ in place of $B$ and with $\vb{C}_{k+1}:=1$ for this application of the lemma yields an approximate factorization
    \begin{equation*}
        \widetilde{F}'(\vb{x}_1, \dots, \vb{x}_k, s) = \widetilde{B}'(\vb{C}_1\vb{x}_1, \dots, \vb{C}_k\vb{x}_k) \widetilde{G}(s)
    \end{equation*}
    where $\norm{\widetilde{F} - \widetilde{F}'} \le 2 \epsilon_k \norm{B}$ and $\widetilde{G}:\mathbb{R} \rightarrow \mathbb{R}^{r_k \times M}$ satisfies $\norm{\widetilde{G}} \le 1$.
    The desired result follows from setting $F'(\vb{x}_1, \dots, \vb{x}_k):=\widetilde{F}'(\vb{x}_1, \dots, \vb{x}_k, 1)$ and defining $G$ by $\vb{\mu}^TG(\vb{\omega})s:= \vb{\mu}^T\widetilde{G}(s)\vb{\omega}$.
\end{proof}

\begin{proof}[Proof of \propref{thm:finite_ttt}]
\propref{thm:preconditioned_are_tt} furnishes a tensor train $\tuple{S}$ with shape $(N,\dots,N,M)$ and ranks $\vb{r}$ satisfying
    \begin{equation*}
        \norm{F - S} \le 2 \norm{B} \norm{\vb{C}}_{\mathrm{HS}}^k \sum_{j=1}^k \epsilon_j,
    \end{equation*}
    where $S$ is the vector-valued multilinear function corresponding to $\tuple{S}$. 
    
    Let $\pi$ be any permutation function of $k$ variables. 
    Since $F \circ \pi = F$, 
    \begin{equation*}
    S - S \circ \pi = (S - F) + (F - S \circ \pi) 
    = (S - F) + (F \circ \pi - S \circ \pi).
    \end{equation*}
    Applying the triangle inequality and using the fact that induced norm is invariant under permutations yields
    \begin{equation*}
        \norm{S - S \circ \pi}
        \le \norm{F - S} + \norm{S \circ \pi - F \circ \pi} 
        = 2 \norm{F - S}.
    \end{equation*}
    Applying \lemref{lem:sym_tt_to_ttt} to $\tuple{S}$ with $\epsilon=2 \norm{F - S}$ furnishes matrices $\vb{\inputbasis}$ and $\vb{\outputbasis}$ and a tensor train $\tuple{S}'$ of the appropriate shape and ranks such that $T$ from \eqref{eq:widetildet_5456} satisfies
    \begin{equation*}
        \norm{S - T} \le (k-1) \epsilon = (2k-2) \norm{F - S}.
    \end{equation*}
    Hence,
    \begin{align*}
        \norm{F - T} &\le \norm{F - S} + \norm{S - T} 
        \le (2k-1)\norm{F - S} 
        \le 2(2k-1) \norm{B} \norm{\vb{C}}_{\mathrm{HS}}^k \sum_{j=1}^k \epsilon_j
    \end{align*}
    as required.
\end{proof}

\begin{lemma}[Infinite to finite reduction]
\label{lem:inf_to_finite}
Let $X$ and $Y$ be Hilbert spaces, let $B: X^k \rightarrow Y$ be a bounded $k$-multilinear function, let $C:X \rightarrow X$ be a compact symmetric positive-semidefinite operator, and define the multilinear function $F:X^k \rightarrow Y$ by
\begin{equation*}
F(x_1, \dots, x_k) := B(C x_1, \dots, C x_k).
\end{equation*}
Let $(\lambda_i, u_i)$ be the eigenvalue/eigenvector pairs of $C$, let $\inputbasis:\mathbb{R}^N \rightarrow X$
be the semi-infinite matrix with orthonormal columns given by the $N$ leading eigenvectors of $C$, and let
$
\vb{\Lambda} = \operatorname{diag}(\lambda_1, \dots, \lambda_N).
$
Let $\outputbasis: \mathbb{R}^{N^k} \rightarrow Y$ be a semi-infinite matrix with columns that form an orthonormal basis for the span of the vectors
\begin{equation*}
\{F(u_{i_1}, u_{i_2}, \dots, u_{i_k}): (i_1, \dots, i_k) \in \{1,\dots, N\}^k \}.
\end{equation*}
Finally, define the multilinear functions
\begin{alignat*}{5}
B_N:&~ \left(\mathbb{R}^N\right)^k \rightarrow \mathbb{R}^M, \qquad&& B_N(\vb{\theta}_1, \dots, \vb{\theta}_k) := \outputbasis^* B(\inputbasis \vb{\theta}_1, \dots, \inputbasis \vb{\theta}_k), \\
F_N:&~ \left(\mathbb{R}^N\right)^k \rightarrow \mathbb{R}^M, \qquad&& F_N(\vb{x}_1, \dots, \vb{x}_k) := B_N(\vb{\Lambda} \vb{x}_1, \dots, \vb{\Lambda} \vb{x}_k),\\
F':&~ X^k \rightarrow Y, \qquad&& F'(x_1, \dots, x_k) := \outputbasis F_N(\inputbasis^* x_1, \dots, \inputbasis^* x_k)
\end{alignat*}
(the quantities $B_N$, $F_N$, and $\vb{\Lambda}$ are finite-dimensional analogs of $B$, $F$, and $\vb{C}$, respectively, and $F'$ is an approximation of $F$).
Then, for any $\epsilon>0$, there exists $N$
such that
\begin{equation}
\label{eq:inf_to_finite_errorbound}
\norm{F - F'} \le \epsilon.
\end{equation}
We also have $\norm{B_N} \le \norm{B}$.
\end{lemma}

\begin{proof}[Proof of \lemref{lem:inf_to_finite}]
By multilinearity, we may perform the rescaling $C \rightarrow C / \norm{C}$ and $\epsilon \rightarrow \norm{C}^k \epsilon$, so without loss of generality we may assume $\norm{C}=1$. The value of $\norm{C}$ will affect $N$, but for this proposition we are only interested in the existence of finite $N$, not its size. 

The fact that $\norm{B_N} \le \norm{B}$ follows directly from the definition of $B_N$ and orthonormality of $\inputbasis$ and $\outputbasis$. The rest of the proof therefore consists of 
proving bound \eqref{eq:inf_to_finite_errorbound}. 
We proceed in the following two part process. In Part 1, we quantify the error in the approximation
\begin{equation*}
F(x_1, \dots, x_k) \approx F(\inputbasis\inputbasis^* x_1, \dots, \inputbasis\inputbasis^* x_k).
\end{equation*}
In Part 2, we show that
\begin{equation*}
F(\inputbasis\inputbasis^* x_1, \dots, \inputbasis\inputbasis^* x_k) = \outputbasis \outputbasis^* F(\inputbasis\inputbasis^* x_1, \dots, \inputbasis\inputbasis^* x_k).
\end{equation*}
The lemma will then follow from the identity
\begin{align}
F'(x_1, \dots, x_k) &= \outputbasis\outputbasis^*B(\inputbasis \vb{\Lambda} \inputbasis^* x_1, \dots, \inputbasis \vb{\Lambda} \inputbasis^* x_k) \nonumber\\
&= \outputbasis\outputbasis^*B(C \inputbasis \inputbasis^* x_1, \dots, C \inputbasis \inputbasis^* x_k) \nonumber\\
&= \outputbasis\outputbasis^* F(\inputbasis\inputbasis^* x_1, \dots, \inputbasis\inputbasis^* x_k) \label{eq:projector_Ttilde_expression}
\end{align}
The equality in the first line of this identity follows from substituting the definition of $F_N$ into the expression for $F'$, then substituting the definition of $B_N$ into the result. Going from the first to the second line, we used the fact that $\inputbasis \vb{\Lambda} \inputbasis^* = C \inputbasis\inputbasis^*$, which follows from the fact that $\inputbasis \vb{\Lambda} \inputbasis^*$ is a truncated eigenvalue decomposition of $C$. Going from the second line to the third line we used the definition of $F$.


\paragraph{Part 1} For convenience of notation, for any subset $\sigma \subset \{1,2,\dots,k\}$, let $F^{\sigma}$ denote the version of $F$ with $\inputbasis\inputbasis^*$ inserted in the arguments corresponding to elements of $\sigma$, and $I - \inputbasis\inputbasis^*$ inserted into the remaining arguments. For example, the limiting cases are $F^{\{\}}(x_1, \dots, x_k) = F((I-\inputbasis\inputbasis^*) x_1, \dots, (I-\inputbasis\inputbasis^*) x_k)$ and $F^{\{1,\dots,k\}}(x_1, \dots, x_k) = F(\inputbasis\inputbasis^* x_1, \dots, \inputbasis\inputbasis^* x_k)$.

Inserting the identity in the form $I = \inputbasis\inputbasis^* + (I- \inputbasis\inputbasis^*)$ into all arguments of $F$ and expanding using multilinearity, we have
\begin{equation*}
F = \sum_{\sigma \subset \{1,\dots,k\}} F^{\sigma},
\end{equation*}
where the sum is taken over all subsets $\sigma \subset \{1,\dots,k\}$. 
Subtracting $F^{\{1,\dots,k\}}$ from both sides, taking the norm, and using the triangle inequality yields the bound
\begin{equation*}
\norm{F - F^{\{1,\dots,k\}}} \le \sum_{\substack{\sigma \subset \{1,\dots,k\} \\ \sigma \neq \{1,\dots,k\}}} \norm{F^\sigma},
\end{equation*}
where the sum is taken over all subsets of $\{1,\dots,k\}$ except for $\{1,\dots,k\}$ itself. For any particular term $\sigma$, let $i=\abs{\sigma}$. We have
\begin{equation*}
\norm{F^\sigma} \le \norm{B} \norm{C \inputbasis\inputbasis^*}^{i} \norm{C (I - \inputbasis\inputbasis^*)}^{k-i} = \norm{B} \lambda_{N+1}^{k-i}.
\end{equation*}
In the inequality, we used the definition of $F$ in terms of $B$ and $C$, and the submultiplicative property of the induced norm. In the equality, we used the facts that $\norm{C (I - \inputbasis\inputbasis^*)}=\lambda_{N+1}$ and $\norm{C \inputbasis\inputbasis^*} = \lambda_1$, and the assumption that $\lambda_1 = \norm{C} = 1$.
Applying this bound to all terms and using the fact that there are $\binom{k}{i}$ subsets $\sigma$ with $\abs{\sigma}=i$ yields
\begin{align*}
\norm{F - F^{\{1,\dots,k\}}} &\le \sum_{i=0}^{k-1} \binom{k}{i} \norm{B} \lambda_{N+1}^{k-i} \\
&= \norm{B} \left(\sum_{i=0}^{k} \binom{k}{i} \lambda_{N+1}^{k-i} - 1\right) 
= \norm{B}\left((\lambda_{N+1} + 1)^k - 1\right).
\end{align*}
The function $f(\lambda)=\left((\lambda + 1)^k - 1\right)$ is a continuous nonnegative function on $(0,\infty)$ with $f(0)=0$, hence there is some $\delta > 0$ such that $\norm{B}\left((\eta + 1)^k - 1\right) \le \epsilon$ for all $\eta < \delta$. The sequence $\lambda_i$, $i=1,2,\dots$ is a decreasing sequence with limit zero, so we may choose $N$ such that $\lambda_{N+1} < \delta$. With this choice of $N$, we therefore have
$$
\norm{F - F^{\{1,\dots,k\}}} \le \epsilon.
$$

\paragraph{Part 2} 
First, define 
$\widetilde{F}_N:\left(\mathbb{R}^N\right)^k \rightarrow Y$,
$
\widetilde{F}_N(\vb{x}_1, \dots, \vb{x}_k) = F(\inputbasis \vb{x}_1, \dots, \inputbasis \vb{x}_k),
$
so that
\begin{align}
F^{\{1,\dots,k\}}(x_1, \dots, x_k) &= F(\inputbasis\inputbasis^* x_1, \dots, \inputbasis\inputbasis^* x_k) 
= \widetilde{F}_N(\inputbasis^* x_1, \dots, \inputbasis^* x_k). \label{eq:C0_reduction_equivalence}
\end{align}
By construction,
\begin{align*}
\widetilde{F}_N(\vb{e}_{i_1}, \vb{e}_{i_2}, \dots, \vb{e}_{i_k}) &= F(u_{i_1}, u_{i_2}, \dots, u_{i_k}) \\
&= \outputbasis \outputbasis^*F(u_{i_1}, u_{i_2}, \dots, u_{i_k}) 
= \outputbasis \outputbasis^* \widetilde{F}_N(\vb{e}_{i_1}, \vb{e}_{i_2}, \dots, \vb{e}_{i_k}),
\end{align*}
and extending by multilinearity yields
\begin{equation}
\label{eq:S_zi_981}
\widetilde{F}_N(\vb{x}_1, \dots, \vb{x}_k) = \outputbasis \outputbasis^* \widetilde{F}_N(\vb{x}_1, \dots, \vb{x}_k)
\end{equation}
for arbitrary vectors $\vb{x}_i$.
Combining \eqref{eq:S_zi_981} and \eqref{eq:C0_reduction_equivalence}, we have
$$
F^{\{1,\dots,k\}}(x_1, \dots, x_k) = \outputbasis \outputbasis^*F(\inputbasis\inputbasis^* x_1, \dots, \inputbasis\inputbasis^* x_k).
$$
Setting $F_N=\outputbasis^* \widetilde{F}_N$ and
combining this with the bound at the end of Part 1 yields the required bound for this theorem.
\end{proof}

\begin{lemma}[Power law sum on hyperbolic cross]
\label{lem:hcross_standard_sum}
Let 
\begin{equation*}
\Upsilon(\rho) := \{\vb{k} \in \mathbb{Z}^j: 1 \le \vb{k}, ~ \text{and} ~ \prod_{i=1}^j \vb{k}[i] \le \rho\},
\end{equation*}
denote the hyperbolic cross of radius $\rho$,
let $r=|\Upsilon(\rho)|$ denote the number of points in $\Upsilon(\rho)$, and define
\begin{equation}
\label{eq:S_of_rho_defn}
S(\rho) = \sum_{\vb{k} \in \Upsilon(\rho)} \prod_{i=1}^j \vb{k}[i]^{-2\beta},
\end{equation}
for $\beta > 1/2$. We have
\begin{equation*}
S(\infty) - S(\rho) \asymp r^{1-2\beta}(\log r)^{2 \beta(j-1)}.
\end{equation*}
Here, the asymptotic relation $u \asymp v$ means that there exist positive constants $c_1, c_2$ such that $c_1 u(\rho) \le v(\rho) \le c_2 u(\rho)$ for sufficiently large $\rho$. 
\end{lemma}

\begin{proof}[Proof of \lemref{lem:hcross_standard_sum}]
By standard hyperbolic cross arguments (see, e.g., \cite[Section 2.3]{temlyakov2016hyperbolic}), we have
$r \asymp \rho (\log \rho)^{j-1}$ and $\Delta S \asymp \rho^{1-2\beta} (\log \rho)^{j-1}$,
where $\Delta S := S(\infty) - S(\rho)$.
Let us write these asymptotic relationships in the form
\begin{align}
r &= \rho (\log \rho)^{j-1} e_1(\rho), \label{eq:r_e1_456} \\
\Delta S &= \rho^{1-2\beta} (\log \rho)^{j-1} e_2(\rho), \label{eq:s_e2_456}
\end{align}
where $e_1$ and $e_2$ are functions that are bounded between positive constants for all large $\rho$.

Defining $h=\log \rho$, \eqref{eq:s_e2_456} may be written as
\begin{equation}
\Delta S = e^{h(1-2\beta)}h^{j-1} e_2(\rho). \label{eq:S_of_h_2223}
\end{equation}
and \eqref{eq:r_e1_456} may be manipulated into the form
\begin{equation}
\label{eq:h_logV_relation}
h = \log r - (j-1) \log h - \log e_1(\rho)
\end{equation}
by substituting in $h$, taking the logarithm of both sides of the equation, and solving for $h$. 
Substituting the right-hand side of \eqref{eq:h_logV_relation} for the $h$ in the exponent of \eqref{eq:S_of_h_2223} yields
\begin{align}
\label{eq:ehrh987}
\Delta S &= e^{(1-2\beta) \log r} e^{(2\beta-1) (j-1) \log(h)} e^{(2\beta-1)\log e_1(\rho)} h^{j-1}e_2(\rho) \\
&= r^{1-2\beta} h^{2\beta(j-1)} e_2(\rho)\, e_1(\rho)^{2\beta-1}.
\end{align}
Writing the remaining $h$ in the form $h=\log r \frac{h}{\log r}$ yields
\begin{equation}
\label{eq:S_rlogr_with_es}
\Delta S = r^{1-2\beta} (\log r)^{2\beta(j-1)} e_2(\rho)\, e_1(\rho)^{2\beta-1} \left(\frac{h}{\log r}\right)^{2 \beta(j-1)}.
\end{equation}
Using \eqref{eq:h_logV_relation} again, we have
\begin{equation}
\label{eq:logr_over_h}
\frac{\log r}{h} = 1 + \frac{(j-1) \log h}{h} + \frac{\log e_1(\rho)}{h}.
\end{equation}
Since $h \rightarrow \infty$ as $\rho \rightarrow \infty$ and since $e_1$ is bounded between positive numbers, the denominators in the second and third terms dominate the numerators for large $h$, so we have
$
\frac{\log r}{h} \rightarrow 1
$
as $\rho \rightarrow \infty$, and consequently 
\begin{equation*}
\left(\frac{h}{\log r}\right)^{2 \beta(j-1)} \rightarrow 1 \quad \text{as} \quad \rho \rightarrow \infty.
\end{equation*}
Combining this with the fact that $e_1(\rho)$ and $e_2(\rho)$ are both bounded between positive constants for large $\rho$, it follows that
\begin{equation*}
e_2(\rho)\, e_1(\rho)^{2\beta-1} \left(\frac{h}{\log r}\right)^{2 \beta(j-1)}
\end{equation*}
is bounded between positive constants for large $\rho$, so from \eqref{eq:S_rlogr_with_es} we conclude
$
\Delta S \asymp r^{1-2\beta}(\log r)^{2 \beta(j-1)}.
$
\end{proof}

\begin{proof}[Proof of \corref{cor:error_vs_rank_poly_decay}]
The eigenvectors of $\kron^j C$ are Kronecker products of eigenvectors of $C$ and the eigenvalues are products of eigenvalues of $C$. The (eigenvalue, eigenvector) pairs may therefore be indexed by points in the lattice
\begin{equation*}
\mathbb{Z}_+^j=\{\vb{k} \in \mathbb{Z}^j: 1 \le \vb{k}\},
\end{equation*} 
and with this indexing the corresponding eigenvalues $\gamma_\vb{k}$ are given by
\begin{equation}
\label{eq:gamma_kvec}
    \gamma_\vb{k} = \prod_{i=1}^j \lambda_{\vb{k}[i]} = \prod_{i=1}^j \vb{k}[i]^{-\beta}.
\end{equation}
By monotonicity, the super-level sets of the function $\vb{k} \mapsto \gamma_\vb{k}$ are sub-level sets of the function $\vb{k} \mapsto \gamma_\vb{k}^{-1/\beta}=\prod_{i=1}^j \vb{k}[i]$. 
Hence, the leading (eigenvalue, eigenvector) pairs correspond to points in the hyperbolic cross 
$\Upsilon(\rho)$,
with all eigenvalues corresponding to points outside of $\Upsilon(\rho)$ being smaller than all eigenvalues corresponding to points in $\Upsilon(\rho)$ for any chosen $\rho$. 

Given a desired rank $r_j$, let $\Upsilon(\rho)$ be the largest hyperbolic cross containing at most $r_j$ points, and set $r_j^-:=|\Upsilon(\rho)| \le r_j$. Let $S(\rho)$ be the sum defined in \lemref{lem:hcross_standard_sum}. If $r_j$ is sufficiently large so that \lemref{lem:hcross_standard_sum} may be applied, the truncated eigenvalue sum on the left side of \eqref{eq:gamma_eigs_assumption2} in \thmref{thm:ttt_main_thm} may be bounded as follows:
\begin{align*}
\sum_{\alpha=r_j+1}^{\infty} \gamma_{j,\alpha}^2 &\le \sum_{\vb{k} \in \mathbb{Z}_+^j} \gamma_\vb{k}^2 - \sum_{\vb{k} \in \Upsilon(\rho)} \gamma_\vb{k}^2 \\
&= S(\infty) - S(\rho) \\
& \le c_1 (r_j^-)^{1-2 \beta}(\log r_j^-)^{2 \beta(j-1)} \\
&\le c_1 r_j^{1-2 \beta}(\log r_j)^{2 \beta(j-1)}
\end{align*}
for some positive constant $c_1$ that depends on $j$ and $\beta$.
The inequality in the first line holds because $\Upsilon(\rho)$ collects the $r_j^-$ largest eigenvalues with $r_j^- \le r_j$; the terms $\gamma_\vb{k}^2$ remaining on the right-hand side (those indexed outside $\Upsilon(\rho)$) therefore contain every term $\gamma_{j,\alpha}^2$ of the left-hand tail, which discards the $r_j \ge r_j^-$ largest. The equality in the second line follows from the definition of $S$ and the expression for the eigenvalues $\gamma_\vb{k}$ in \eqref{eq:gamma_kvec}. The inequality in the third line is the upper bound in the asymptotic relation of \lemref{lem:hcross_standard_sum}, applied with $r=r_j^-$. The inequality in the fourth line uses $r_j^- \asymp r_j$: the number of lattice points added to the cross when the radius crosses an integer is sub-polynomial in $\rho$, while $|\Upsilon(\rho)| \ge \lfloor \rho \rfloor$, so consecutive cross cardinalities have ratio tending to one (see, e.g., \cite[Section~2.3]{temlyakov2016hyperbolic}); hence $r_j \le (1+o(1))\, r_j^-$ and the resulting bounded factor is absorbed into $c_1$.
The result therefore follows from \thmref{thm:ttt_main_thm}, taking 
$$
\epsilon_j=\left(r_j^{1-2 \beta}(\log r_j)^{2 \beta(j-1)}\right)^{1/2}
$$
and absorbing all constants other than $\norm{D^k \pto(\paramzero)}$ into $c$.
\end{proof}




 \bibliographystyle{plain} 
 \bibliography{tttt}

\end{document}